\documentclass[a4paper,11pt,oneside]{amsbook}

\usepackage{etex}
\usepackage[T1]{fontenc}%
\usepackage[latin1]{inputenc}%
\usepackage{xspace}%
\usepackage{amsfonts,amsmath,amssymb,amsthm}%
\usepackage{mathrsfs}%
\usepackage{graphicx,mathdots}%
\usepackage{enumerate}%
\usepackage{subfigure}%
\usepackage{multirow}%

\input{xy}
\xyoption{all}
\objectmargin={3mm}

\theoremstyle{plain}
\newtheorem{theorem}{Theorem}
\newtheorem{proposition}[theorem]{Proposition}

\newtheorem{corollary}[theorem]{Corollary}
\newtheorem{lemma}[theorem]{Lemma}
\newtheorem{fact}[theorem]{Fact}

\newtheorem{conjecture}[theorem]{Conjecture}
\newtheorem{problems}[theorem]{Questions}

\theoremstyle{definition}
\newtheorem{definition}[theorem]{Definition}
\newtheorem{definition-proposition}[theorem]{Definition-Proposition}
\newtheorem{notation}[theorem]{Notation}
\newtheorem{observation}[theorem]{Observation}
\newtheorem{remark}[theorem]{Remark}
\newtheorem{remarks}[theorem]{Remarks}
\newtheorem{example}[theorem]{Example}
\newtheorem{examples}[theorem]{Examples}
\newtheorem{spher-setting}[theorem]{Main setting}
\newtheorem{gen-setting}[theorem]{General setting}

\newcommand{\C}{\mathbb{C}}
\newcommand{\R}{\mathbb{R}}
\newcommand{\Z}{\mathbb{Z}}
\newcommand{\N}{\mathbb{N}}

\newcommand{\SL}{\mathrm{SL}}

\newcommand{\SO}{\mathrm{SO}}
\newcommand{\SU}{\mathrm{SU}}
\newcommand{\U}{\mathrm{U}}
\newcommand{\Sp}{\mathrm{Sp}}
\newcommand{\Spin}{\mathrm{Spin}}
\newcommand{\g}{\mathfrak{g}}
\newcommand{\h}{\mathfrak{h}}
\newcommand{\kk}{\mathfrak{k}}
\newcommand{\p}{\mathfrak{p}}
\newcommand{\q}{\mathfrak{q}}
\newcommand{\aaa}{\mathfrak{a}}

\newcommand{\jj}{\mathfrak{j}}

\newcommand{\llll}{\mathfrak{l}}
\newcommand{\rr}{\mathfrak{r}}

\newcommand{\ssl}{\mathfrak{sl}}
\newcommand{\so}{\mathfrak{so}}
\newcommand{\spin}{\mathfrak{spin}}

\newcommand{\Hom}{\mathrm{Hom}}
\newcommand{\End}{\operatorname{End}}
\newcommand{\D}{\mathbb{D}}
\newcommand{\Diag}{\mathrm{Diag}}
\newcommand{\I}{\mathrm{\mathbf{I}}}
\newcommand{\II}{\mathrm{\mathbf{II}}}
\newcommand{\M}{\mathcal{M}}
\newcommand{\NN}{\mathcal{N}}

\newcommand{\F}{\mathcal{F}}
\newcommand{\A}{\mathcal{A}}
\newcommand{\Spec}{\mathrm{Spec}}

\newcommand{\Ad}{\operatorname{Ad}}

\newcommand{\HH}{\mathbb{H}}
\newcommand{\HHH}{\mathcal{H}}

\newcommand{\rank}{\operatorname{rank}}
\newcommand{\V}{\mathcal{V}}

\newcommand{\AdS}{\mathrm{AdS}}
\newcommand{\Disc}{\mathrm{Disc}}

\newcommand{\DD}{\mathcal{D}}
\newcommand{\ii}{\mathbf{i}}
\newcommand{\pp}{\mathbf{p}}
\newcommand{\dd}{\mathrm{d}}
\newcommand{\resp}{resp.\ }
\newcommand{\ie}{i.e.\ }
\newcommand{\eg}{e.g.\ }
\newcommand{\nnu}{\boldsymbol\nu}
\newcommand{\llambda}{\boldsymbol\lambda}
\newcommand{\Char}{\mathrm{Supp}}
\newcommand{\Ker}{\mathrm{Ker}}
\newcommand{\Symm}{\mathrm{Symm}}
\newcommand{\curlyEnd}{\mathcal{E}\!\mbox{\large $\mathpzc{nd}$}}

\newcommand{\longhookrightarrow}{\ensuremath{\lhook\joinrel\relbar\joinrel\rightarrow}}

\newcommand{\longtwoheadrightarrow}{\relbar\joinrel\twoheadrightarrow}

\newcommand{\Longlongleftarrow}{\ensuremath{\Longleftarrow\joinrel\Relbar\joinrel\Relbar\joinrel\Relbar}}
\newcommand{\Longlongrightarrow}{\ensuremath{\Relbar\joinrel\Relbar\joinrel\Relbar\joinrel\Longrightarrow}}

\newcommand{\sumplus}[1]{\underset{\hspace{-0.25cm}#1}{{\sum}^{\oplus}}}

\DeclareMathAlphabet{\mathpzc}{OT1}{pzc}{m}{it}

\newenvironment{changemargin}[2]{\begin{list}{}{%
\setlength{\topsep}{0pt}%
\setlength{\leftmargin}{0pt}%
\setlength{\rightmargin}{0pt}%
\setlength{\listparindent}{\parindent}%
\setlength{\itemindent}{\parindent}%
\setlength{\parsep}{0pt plus 1pt}%
\addtolength{\leftmargin}{#1}%
\addtolength{\rightmargin}{#2}%
}\item }{\end{list}}

\title[Spectral analysis on standard locally homogeneous spaces]{Spectral analysis on standard\\ locally homogeneous spaces}

\author{Fanny Kassel}
\address{CNRS and Laboratoire Alexander Grothendieck, Institut des Hautes \'Etudes Scientifiques, Universit\'e Paris-Saclay, 35 route de Chartres, 91440 Bures-sur-Yvette, France}
\email{kassel@ihes.fr}

\author{Toshiyuki Kobayashi}
\address{Graduate School of Mathematical Sciences, The University of Tokyo, 3-8-1 Komaba, Tokyo, 153-8914 Japan}
\email{toshi@ms.u-tokyo.ac.jp}

\thanks{This project received funding from the European Research Council (ERC) under the European Union's Horizon 2020 research and innovation programme (ERC starting grant DiGGeS, grant agreement No 715982).
FK was partially supported by the Agence Nationale de la Recherche through the grants DiscGroup (ANR-11-BS01-013), DynGeo (ANR-16-CE40-0025-01), and the Labex CEMPI (ANR-11-LABX-0007-01).
TK was partially supported by the JSPS under the Grant-in-Aid for Scientific Research (A) (18H03669, JP23H00084).}


\begin{document}

\frontmatter

\begin{abstract}
Let $X=G/H$ be a reductive homogeneous space with $H$ noncompact, endowed with a $G$-invariant pseudo-Riemannian structure.
Let $L$ be a reductive subgroup of~$G$ acting properly on~$X$ and $\Gamma$ a torsion-free discrete subgroup of~$L$.
Under the assumption that the complexification $X_{\C}$ is $L_{\C}$-spherical, we show that any compactly supported $C^{\infty}$ function on the standard locally homogeneous space $X_{\Gamma}=\Gamma\backslash X$ can be expanded into joint eigenfunctions for those ``intrinsic'' differential operators coming from $G$-invariant operators on~$X$. 
In particular, we prove that the pseudo-Riemannian Laplacian on~$X_{\Gamma}$ is essentially self-adjoint.
Furthermore, we exhibit an explicit correspondence between spectral analysis on $X_{\Gamma}$ and on $\Gamma\backslash L$ via branching laws for the restriction to the reductive subgroup~$L$ of infinite-dimensional irreducible representations of~$G$.
In particular, we prove that the pseudo-Riemannian Laplacian on~$X_{\Gamma}$ admits an infinite point spectrum when $X_{\Gamma}$ is compact or $\Gamma\subset L$ is arithmetic.
The proof builds on structural results for invariant differential operators on spherical homogeneous spaces with overgroups.
\end{abstract}

\maketitle
\newpage

\tableofcontents
\numberwithin{equation}{chapter}
\numberwithin{table}{chapter}
\numberwithin{theorem}{chapter}

\renewcommand{\thesection}{\thechapter.\arabic{section}}

\mainmatter
\chapter{Introduction}\label{sec:intro}

Let $H\subset G$ be two linear reductive Lie groups.
Classically, the space $X=G/H$ admits a $G$-invariant pseudo-Riemannian structure (Lemma~\ref{lem:pseudo-Riem-struct}); for semisimple~$G$, such a structure is induced for instance by the Killing form of the Lie algebra~$\g$.
If $\Gamma$ is a discrete subgroup of~$G$ acting properly discontinuously and freely on~$X$ (or ``discontinuous group for~$X$''), then the quotient space $X_{\Gamma}:=\Gamma\backslash X=\Gamma\backslash G/H$ is a manifold, and the covering map
\begin{equation} \label{eqn:p-Gamma}
p_{\Gamma} : G/H = X \longrightarrow X_{\Gamma} = \Gamma\backslash G/H.
\end{equation}
transports the pseudo-Riemannian structure of~$X$ to~$X_{\Gamma}$.
The \emph{Laplacian} of~$X_{\Gamma}$ is the second-order differential operator
\begin{equation}\label{eqn:defLaplacian}
\square_{X_{\Gamma}} =  \operatorname{div}\,\operatorname{grad},
\end{equation}
where the gradient and divergence are defined with respect to the pseudo-Riemannian structure and the induced volume form on~$X_{\Gamma}$.
When the pseudo-Riemannian structure is positive definite, this is the usual Laplacian on a Riemannian manifold, for which we also write $\Delta_{X_{\Gamma}}$ instead of~$\square_{X_\Gamma}$.
When the pseudo-Riemannian structure is not definite, the Laplacian $\square_{X_{\Gamma}}$ is not an elliptic differential operator. 
We are interested in the spectral analysis of $\square_{X_{\Gamma}}$ in that setting.

More generally, we consider ``intrinsic'' differential operators of higher order on~$X_{\Gamma}$, defined as follows.
Let $\D_G(X)$ be the $\C$-algebra of $G$-invariant differential operators on~$X$.
Any operator $D\in\D_G(X)$ induces a differential operator~$D_{\Gamma}$ on~$X_{\Gamma}$ such that
\begin{equation}\label{eqn:D-gamma}
D \circ p_{\Gamma}^{\ast} = p_{\Gamma}^{\ast} \circ D_{\Gamma},
\end{equation}
where $p_{\Gamma}^{\ast} :  C^\infty(X_{\Gamma})\rightarrow C^{\infty}(X)$ is the pull-back by~$p_{\Gamma}$.
In particular, the Laplacian $\square_X$ is $G$-invariant and $(\square_X)_{\Gamma}=\square_{X_{\Gamma}}$.
For $\F=\A$
(\resp $C^{\infty}$,
\resp $L^2$,
\resp $\DD'$),
let $\F(X_{\Gamma})$ be the space of real analytic (\resp smooth, \resp square-integrable, \resp distribution) functions on~$X_{\Gamma}$.
For any $\C$-algebra homomorphism
$$\lambda : \D_G(X) \longrightarrow \C,$$
we denote by $\F(X_{\Gamma};\M_{\lambda})$ the space of (weak) solutions $f\in\F(X_{\Gamma})$ to the system
$$D_{\Gamma} f = \lambda(D) f \quad\quad\mathrm{for\ all}\ D\in\D_G(X) \eqno{(\M_{\lambda})}.$$
For $\F=\A$ (\resp $C^{\infty}$, \resp $\DD'$), the space $\F(X_{\Gamma};\M_{\lambda})$ identifies with the set of analytic (\resp smooth, \resp distribution) $\Gamma$-periodic joint eigenfunctions for $\D_G(X)$ on~$X$ with respect to $\lambda\in\Hom_{\C\text{-}\mathrm{alg}}(\D_G(X),\C)$; for $\F=L^2$, there is an additional requirement that the eigenfunctions be square-integrable on the quotient~$X_{\Gamma}$ with respect to the natural measure induced by the pseudo-Riemannian structure.
By definition, the \emph{discrete spectrum} $\Spec_d(X_{\Gamma})$ of~$X_{\Gamma}$ is the set of homomorphisms $\lambda$ such that $L^2(X_{\Gamma};\M_{\lambda})\neq\nolinebreak\{ 0\} $ (``joint $L^2$-eigenvalues for $\D_G(X)$'').
Any element of $L^2(X_{\Gamma};\M_{\lambda})$ is in particular an $L^2$-eigenfunction (as a weak solution in~$L^2$) of the Laplacian $\square_{X_{\Gamma}}$ for the eigenvalue $\lambda(\square_X)$, yielding discrete spectrum (or point spectrum) of~$\square_{X_{\Gamma}}$.

Very little is known about $\F(X_{\Gamma};\M_{\lambda})$ when $H$ is noncompact and $\Gamma$ infinite.
For instance, the following questions are open in general.

\begin{problems} \label{problems}
\begin{enumerate}[(a)]
  \item Is $\Spec_d(X_{\Gamma})$ nonempty, \eg when $X_{\Gamma}$ is compact?
  \item Does $L^2(X_{\Gamma};\M_{\lambda})$ contain smooth eigenfunctions as a dense subspace?
  \item Does the Laplacian $\square_{X_{\Gamma}}$ defined on $C^{\infty}_c(X_{\Gamma})$ extend to a self-adjoint operator on $L^2(X_{\Gamma})$?
\end{enumerate}
\end{problems}

These questions have been studied extensively in the following two cases:
\begin{enumerate}[(i)]
  \item $H=K$ is a maximal compact subgroup of~$G$ (\ie $X_{\Gamma}$ is a Riemannian locally symmetric space);
  \item $(G,H)$ is a reductive symmetric pair and $\Gamma=\{e\}$ is trivial (\ie $X_{\Gamma}=X$ is a reductive symmetric space).
\end{enumerate}
In case~(i), the discrete spectrum $\Spec_d(X_{\Gamma})$ is infinite when $\Gamma$ is an arithmetic subgroup of~$G$ \cite{bg83}, whereas $\Spec_d(X)=\Spec(G/K)$ is empty, namely, $\Spec_d(X_{\Gamma})$ is empty when $\Gamma=\{e\}$; Questions~\ref{problems}.(b)--(c) always have affirmative answers by the general theory of the Laplacian on Riemannian manifolds (without the arithmeticity assumption on~$\Gamma$): see \cite[Th.\,3.4.4]{kkk86} (elliptic regularity theorem) for~(b) and \cite{gaf54,wol7273,str83} for~(c).
In case~(ii), the discrete spectrum $\Spec_d(X)$ is nonempty if and only if the rank condition
\begin{equation}\label{eqn:rank}
\rank G/H=\rank K/H\cap K
\end{equation}
is satisfied, in which case $\Spec_d(X)$ is in fact infinite \cite{fle80,mo84}; Questions \ref{problems}.(b)--(c) also have affirmative answers by the general theory of unitary representations and symmetric spaces: see \cite{gar47} for~(b) and \cite{ban87} for~(c).~However, the questions remain wide open when $H$ is noncompact and $\Gamma$ infinite.

In previous work \cite{kk11,kk16} we constructed nonzero generalized Poin\-car\'e series and obtained $L^2$-eigenfunctions on~$X_{\Gamma}$ corresponding to discrete spectrum (which we call of type~$\I$) under the assumption that $X$ satisfies the rank condition \eqref{eqn:rank} and the action of $\Gamma$ on~$X$ satisfies a strong properness condition called \emph{sharpness} (see \cite[Def.\,4.2]{kk16}); this provided a partial answer to Question~\ref{problems}.(a).

In the current paper, we study joint eigenfunctions for $\D_G(X)$ using a different approach.
We assume that $\Gamma$ is contained in a reductive subgroup $L$ of~$G$ acting properly on~$X$ (\ie $X_{\Gamma}$ is \emph{standard}, see Section~\ref{subsec:intro-stand}) and that $X_{\C}$ is \emph{$L_{\C}$-spherical} (see Section~\ref{subsec:intro-spher}), which ensures that the larger $\C$-algebra $\D_L(X) \supset \D_G(X)$ of $L$-invariant differential operators on~$X$ is commutative.
Using \cite{kkdiffop}, we introduce a pair of \emph{transfer maps} $\nnu$ and~$\llambda$ (see \eqref{eqn:nu-lambda-tau}), which are inverse to each other, and such that $\llambda$ sends spectrum from the classical Riemannian setting of $\Gamma\backslash L/(L\cap K)$ to the pseudo-Riemannian setting of $\Gamma\backslash G/H=X_{\Gamma}$.
From a representation-theoretic point of view, these transfer maps reflect the restriction of irreducible $G$-modules to the subgroup~$L$ (branching laws).
Using this, we obtain a description of the whole discrete spectrum of~$X_{\Gamma}$ (Theorem~\ref{thm:Specd-lambda}), and find new infinite spectrum (which we call of type~$\II$) when $X_{\Gamma}$ is compact or of an arithmetic nature (Theorem~\ref{thm:mainII}).
Moreover, via the transfer map~$\llambda$, we prove that any compactly supported smooth function on~$X_{\Gamma}$ can be developed into joint eigenfunctions of $\D_G(X)$ (Theorem~\ref{thm:Fourier}).
The assumptions on~$X_{\Gamma}$ here are different from \cite{kk11,kk16}: we do not assume the rank condition \eqref{eqn:rank} to be necessarily satisfied, but restrict ourselves to the case that $X_{\Gamma}$ is standard and $X_{\C}$ is $L_{\C}$-spherical.
In this setting we give affirmative answers to Questions~\ref{problems}.(a)--(c).
The main tool of the proof is analysis on spherical homogeneous spaces with overgroups, as developed in \cite{kkdiffop}.

Before we state our main results in a more precise way, let us introduce some definitions.

\section{Standard quotients} \label{subsec:intro-stand}

When the reductive homogeneous space $X=G/H$ is non-Riemannian, not all discrete subgroups of~$G$ act properly discontinuously on~$X$.
For instance, a lattice of~$G$ cannot act properly discontinuously if $H$ is noncompact, by the Howe--Moore ergodicity theorem.

An important class of examples is constructed as follows: a quotient $X_{\Gamma}=\Gamma\backslash X$ of~$X$ by a discrete subgroup $\Gamma$ of~$G$ is called \emph{standard} if $\Gamma$ is contained in some reductive subgroup $L$ of~$G$ acting properly on~$X$.
Then the action of $\Gamma$ on~$X$ is automatically properly discontinuous, and this action is free whenever $\Gamma$ is torsion-free; the quotient $X_{\Gamma}$ is compact if and only if $\Gamma$ is a uniform lattice in~$L$ and $L$ acts cocompactly on~$X$.

\begin{example} \label{ex:AdS-odd}
Let $X=\AdS^{2n+1}=\SO(2n,2)/\SO(2n,1)$ be the $(2n+1)$-dimensional anti-de Sitter space.
It is a reductive symmetric space with a $G$-invariant Lorentzian structure of constant negative sectional curvature, making it a Lorentzian analogue of the real hyperbolic space $\mathbb{H}^{2n+1}$.
The group $L=\U(n,1)$ acts properly and transitively on~$X$, and any torsion-free discrete subgroup $\Gamma\subset L$ yields a standard quotient manifold~$X_{\Gamma}$.
\end{example}

\begin{example} \label{ex:group-manifold}
Let $X=({}^{\backprime}G\times\!{}^{\backprime}G)/\Diag({}^{\backprime}G)$ be a \emph{group manifold}, where ${}^{\backprime}G$ is a noncompact reductive Lie group and $\Diag({}^{\backprime}G)$ denotes the diagonal of ${}^{\backprime}G\times\!{}^{\backprime}G$.
Let ${}^{\backprime}K$ be a maximal compact subgroup of~${}^{\backprime}G$.
The group $L={}^{\backprime}G\times\!{}^{\backprime}K$ acts properly and transitively on~$X$, and any torsion-free discrete subgroup $\Gamma\subset L$ yields a standard quotient manifold~$X_{\Gamma}$.
\end{example}

Almost all known examples of compact quotients of reductive homogeneous spaces are standard, and conjecturally \cite[Conj.\,3.3.10]{ky05} any reductive homogeneous space admitting compact quotients admits standard ones.
We refer to \cite[\S\,4]{kk16} for more details.

\begin{remark} \label{rem:Gamma-torsion}
For simplicity, in the statements of the theorems below, we shall assume the discontinuous $\Gamma$ to be torsion-free.
However, the theorems still hold, with the same proof, under the weaker assumption that $\Gamma$ acts freely on~$X$; indeed, the only thing we need is that the quotient $X_{\Gamma}=\Gamma\backslash X$ be a smooth manifold with covering map $X\to\Gamma\backslash X$.
One could also extend the theorems to the framework of orbifolds (or $V$-manifolds in the sense of Satake), allowing the discontinuous group $\Gamma$ to not act freely on~$X$.
\end{remark}

\section{Spherical homogeneous spaces} \label{subsec:intro-spher}

Recall that a connected complex manifold endowed with a holomorphic action of a complex reductive Lie group $G_{\C}$ is called \emph{$G_{\C}$-spherical} if it admits an open orbit of a Borel subgroup of~$G_{\C}$ (Definition~\ref{def:spherical}).
For instance, any complex reductive symmetric space is spherical \cite{wol74}.

One expects solutions to $(\M_{\lambda})$ on $X_{\Gamma}=\Gamma\backslash G/H$ for varying joint eigenvalues $\lambda\in\Hom_{\C\text{-}\mathrm{alg}}(\D_G(X),\C)$ to be abundant enough to expand arbitrary functions on~$X_{\Gamma}$ only if the algebra $\D_G(X)$ is commutative, or equivalently only if the complexification $X_{\C}=G_{\C}/H_{\C}$ is $G_{\C}$-spherical.
In this case, $C^{\infty}(X;\M_{\lambda})$ is a representation of~$G$ of finite length for any~$\lambda$ by \cite{ko13}, and we expect to relate spectral analysis on $X_{\Gamma}$ to representation theory of $G$ on $C^{\infty}(X)$.

In this paper, we shall consider spectral analysis on standard quotients~$X_{\Gamma}$ with $\Gamma\subset L$ in the following setting.

\begin{spher-setting} \label{spher-setting}
We consider a reductive homogeneous space $X=G/H$ with $G$ noncompact and simple, a reductive subgroup $L$ of~$G$ acting properly on~$X$, such that $X_{\C}=G_{\C}/H_{\C}$ is $L_{\C}$-spherical, and a torsion-free discrete subgroup $\Gamma$ of~$L$.
We assume $G$, $H$, and~$L$ to be connected.
\end{spher-setting}

Here we call a homogeneous space $X=G/H$ \emph{reductive} if $G$ is a real reductive Lie group and $H$ a closed subgroup which is reductive in~$G$.
A typical example of a reductive homogeneous space is a \emph{reductive symmetric space}, namely $G$ is a real reductive Lie group and $H$ an open subgroup of the group of fixed points of $G$ under some involutive automorphism~$\sigma$.

In the setting~\ref{spher-setting}, the complexification $X_{\C}$ is automatically $G_{\C}$-spherical, and the action of $L$ on~$X$ is transitive, by \cite[Lem.\,4.2]{ko13} and \cite[Lem.\,5.1]{kob94}.

\begin{remark} \label{rem:connected}
All the theorems in the paper remain true if we relax the assumption of the real reductive Lie groups $G$, $H$, $L$ being connected into $G$, $H$, $L$ being contained in connected complexifications $G_{\C}$, $H_{\C}$, $L_{\C}$.
Indeed, we can use \cite[Th.\,5.1 \& Prop.\,5.5]{kkdiffop} and replace everywhere the maximal compact subgroup $L_K$ of~$L$ by its identity component.
\end{remark}

Table~\ref{table1} below provides a full list of triples $(G,H,L)$ of the setting~\ref{spher-setting}, up to connected components and coverings.
It is obtained from Oni\v{s}\v{c}ik's list \cite{oni69} of triples $(G,H,L)$ with compact simple~$G$ such that $HL=G$ and from the classification \cite{kra79,bri87,mik87} of spherical homogeneous spaces.
Note that the pair $(\g,\h)$ of Lie algebras is a reductive symmetric pair in all cases except (ix).
The complexification $G_{\C}$ is simple in all cases except~(vii).
In all cases the action of $L$ on $X=G/H$ is cocompact, and so there exist finite-volume (\resp compact) quotients $X_{\Gamma}=\Gamma\backslash X$: one can just take $\Gamma$ to be a torsion-free lattice (\resp uniform lattice) in~$L$.

\begin{center}
\begin{table}[!h]
\centering
\subtable{
\begin{tabular}{|p{0.8cm}|p{2.1cm}|p{2.7cm}|p{2.9cm}|}
\hline
& \centering $G$ & \centering $H$ & \centering $L$\tabularnewline
\hline
\centering (i) & \centering $\SO(2n,2)$ & \centering $\SO(2n,1)$ & \centering$\U(n,1)$\tabularnewline
\centering (i)$'$ & \centering $\SO(2n,2)$ & \centering $\SO(2n,1)$ & \centering$\SU(n,1)$\tabularnewline
\centering (ii) & \centering $\SO(2n,2)$ & \centering $\U(n,1)$ & \centering $\SO(2n,1)$\tabularnewline
\centering (iii) & \centering $\SU(2n,2)$ & \centering $\U(2n,1)$ & \centering $\Sp(n,1)$\tabularnewline
\centering (iv) & \centering $\SU(2n,2)$ & \centering $\Sp(n,1)$ & \centering $\U(2n,1)$\tabularnewline
\centering (v) & \centering $\SO(4n,4)$ & \centering $\SO(4n,3)$ & \centering $\Sp(1)\cdot\Sp(n,1)$\tabularnewline
\centering (v)$'$ & \centering $\SO(4n,4)$ & \centering $\SO(4n,3)$ & \centering $\U(1)\cdot\Sp(n,1)$\tabularnewline
\centering (vi) & \centering $\SO(8,8)$ & \centering $\SO(8,7)$ & \centering $\Spin(8,1)$\tabularnewline
\centering (vii) & \centering $\SO(8,\C)$ & \centering $\SO(7,\C)$ & \centering $\Spin(7,1)$\tabularnewline
\centering (viii) & \centering $\SO(4,4)$ & \centering $\Spin(4,3)$ & \centering $\SO(4,1)\times\SO(3)$\tabularnewline
\centering (ix) & \centering $\SO(4,3)$ & \centering $G_{2(2)}$ & \centering $\SO(4,1)\times\SO(2)$\tabularnewline
\hline
\end{tabular}}
\subtable{
\begin{tabular}{|c|}
\hline
$\rank X$\tabularnewline
\hline
1\tabularnewline
1\tabularnewline
$\lceil n/2\rceil$\tabularnewline
1\tabularnewline
n\tabularnewline
1\tabularnewline
1\tabularnewline
1\tabularnewline
2\tabularnewline
1\tabularnewline
1\tabularnewline
\hline
\end{tabular}
}
\caption{Complete list of triples $(G,H,L)$ in the setting \ref{spher-setting}, up to a covering of~$G$ and up to connected components. In case~(i)$'$ we assume $n\geq 2$.}
\label{table1}
\end{table}
\end{center}

\section{Density of analytic eigenfunctions}\label{subsec:intro-main-result}

In our setting where $H$ is noncompact, the natural pseudo-Riemannian structure on~$X_{\Gamma}$ is not positive definitive, and the Laplacian $\square_{X_{\Gamma}}$ is not an elliptic differential operator.
Thus weak $L^2$-solutions (or distribution solutions) to $\M_{\lambda}$ are not necessarily smooth functions: see Section~\ref{subsec:torus-ex} for an elementary example.
Nevertheless, in the setting~\ref{spher-setting}, we give an affirmative answer to Question~\ref{problems}.(b) as follows.

\begin{theorem}[Density of analytic eigenfunctions]\label{thm:analytic}
In the setting~\ref{spher-setting}, for any $\lambda\in\Hom_{\C\text{-}\mathrm{alg}}(\D_G(X),\C)$, the space $(\A\cap L^2)(X_{\Gamma};\M_{\lambda})$ is dense in the Hilbert space $L^2(X_{\Gamma};\M_{\lambda})$, and $\A(X_{\Gamma};\M_{\lambda})$ is dense in $\DD'(X_{\Gamma};\M_{\lambda})$.
\end{theorem}

Theorem~\ref{thm:analytic} applies to the homogeneous spaces $X=G/H$ of Table~\ref{table1}.
We prove it in Section~\ref{subsec:proof-transfer} by constructing a dense analytic subspace of eigenfunctions via a transfer map~$\nnu$, see Theorem~\ref{thm:transfer} below.

\section[Self-adjointness and spectral decomposition]{Self-adjointness and spectral decomposition for the Laplacian} \label{subsec:intro-spec-decomp}

Let $(M,g)$ be a pseudo-Riemannian manifold.
The Laplacian $\square_M$, defined on the space $C_c^{\infty}(M)$ of compactly supported smooth functions on~$M$, is a symmetric operator, namely
$$(\square_Mf_1, f_2)_{L^2(M)} = (f_1, \square_Mf_2)_{L^2(M)}$$
for all $f_1,f_2\in C_c^{\infty}(M)$.
In this paper we consider the existence and uniqueness of a self-adjoint extension of the Laplacian $\square_M$ on $L^2(M)$.

More precisely, recall that the closure of $(\square_M,C_c^{\infty}(M))$ in the graph norm is defined on the set $\mathcal{S}$ of $f\in L^2(M)$ for which there exists a sequence $f_j\in C_c^{\infty}(M)$ such that $\Vert f_j-f\Vert_{L^2(M)}\to 0$ and $\square_Mf_j\in C_c^{\infty}(M)$ converges to an element of $L^2(M)$, which we can identify with the distribution $\square_Mf$.
The adjoint $\square_M^*$ of the symmetric operator $(\square_M,\mathcal{S})$ is defined on the set $\mathcal{S}^*$ of $f\in L^2(M)$ such that the distribution $\square_Mf$ belongs to $L^2(M)$.
Clearly, $\mathcal{S}\subset\mathcal{S}^*$.
The Laplacian $\square_M$ is called \emph{essentially self-adjoint} on $L^2(M)$ if $\mathcal{S}=\mathcal{S}^*$, or equivalently if there exists a unique self-adjoint extension of $(\square_M,C_c^{\infty}(M))$.
When $(M,g)$ is Riemannian and complete, the Laplacian $\square_M$ is always essentially self-adjoint, see \cite{str83} and references therein.
On the other hand, to the best of our knowledge there is no general theory ensuring the essential self-adjointness of the Laplacian $\square_M$ in the pseudo-Riemannian setting, even when $M$ is compact.

Here we prove that $\mathcal{S}=\mathcal{S}^*$ for $M=X_{\Gamma}$ for any standard $X_{\Gamma}$ with $\Gamma\subset L$ for $X=G/H$ and~$L$ as in Table~\ref{table1}.

\begin{theorem}[Self-adjoint extension] \label{thm:selfadj}
In the setting~\ref{spher-setting}, the pseudo-Rie\-mannian Laplacian $\square_{X_{\Gamma}}$ is essentially self-adjoint on $L^2(X_{\Gamma})$.
\end{theorem}

In particular, in this setting the Hilbert space $L^2(X_{\Gamma})$ admits a spectral decomposition with real spectrum for the Laplacian~$\square_{X_{\Gamma}}$.
Note that we do \emph{not} assume $X_{\Gamma}$ to have finite volume.
Theorem~\ref{thm:selfadj} will be proved in Chapter~\ref{sec:strategy}.

By combining the existence of a transfer map $\llambda$ (Proposition~\ref{prop:cond-Tf-A-B-satisfied}) with the representation theory of the subgroup $L$ on $L^2(\Gamma\backslash L)$, we obtain a spectral decomposition on $X_{\Gamma}=\Gamma\backslash G/H$ by joint eigenfunctions of $\D_G(X)$.

\begin{theorem}[Spectral decomposition] \label{thm:Fourier}
In the setting~\ref{spher-setting}, there exist a measure $\dd\mu$ on $\Hom_{\C\text{-}\mathrm{alg}}(\D_G(X),\C)$ and a measurable family of maps
$$\mathtt{F}_{\lambda} : C_c^{\infty}(X_{\Gamma}) \longrightarrow C^{\infty}(X_{\Gamma};\M_{\lambda}),$$
for $\lambda\in\Hom_{\C\text{-}\mathrm{alg}}(\D_G(X),\C)$, such that any $f\in C_c^{\infty}(X_{\Gamma})$ may be expanded into joint eigenfunctions on~$X_{\Gamma}$ as
\begin{equation} \label{eqn:Fourier}
f = \int_{\Hom_{\C\text{-}\mathrm{alg}}(\D_G(X),\C)} \mathtt{F}_{\lambda} f \ \dd\mu(\lambda),
\end{equation}
with a Parseval--Plancherel type formula
$$\Vert f\Vert_{L^2(X_{\Gamma})}^2 = \int_{\Hom_{\C\text{-}\mathrm{alg}}(\D_G(X),\C)} \Vert\mathtt{F}_{\lambda} f\Vert_{L^2(X_{\Gamma})}^2 \ \dd\mu(\lambda).$$
Moreover, \eqref{eqn:Fourier} is a discrete sum if $X_{\Gamma}$ is compact.
\end{theorem}

Theorem~\ref{thm:Fourier} will be proved in Section~\ref{subsec:proof-Fourier}.

\section{Square-integrable joint eigenfunctions}\label{subsec:intro-type-I-II}

We now focus on the \emph{discrete spectrum} of the Laplacian or more generally of the ``intrinsic'' differential operators $D_{\Gamma}$ on~$X_{\Gamma}$ coming from $\D_G(X)$, as given by \eqref{eqn:D-gamma}.
In Chapter~\ref{sec:type-I-II}, for joint $L^2$-eigenfunctions on~$X_{\Gamma}$, we introduce a Hilbert space decomposition
$$L^2(X_{\Gamma};\M_{\lambda}) = L^2(X_{\Gamma};\M_{\lambda})_{\I} \oplus L^2(X_{\Gamma};\M_{\lambda})_{\II}$$
according to the analysis on the homogeneous space $X=G/H$: namely, $L^2(X_{\Gamma};\M_{\lambda})_{\I}$ is associated with discrete series representations for the homogeneous space $X=G/H$, and $L^2(X_{\Gamma};\M_{\lambda})_{\II}$ is its orthogonal complement in $L^2(X_{\Gamma};\M_{\lambda})$.
For $i\in\{ \I,\II\}$, we set
$$\Spec_d(X_{\Gamma})_i := \{ \lambda\in\Spec_d(X_{\Gamma}) : L^2(X_{\Gamma};\M_{\lambda})_i\neq\{ 0\} \} ,$$
so that
$$\Spec_d(X_{\Gamma}) = \Spec_d(X_{\Gamma})_{\I} \cup \Spec_d(X_{\Gamma})_{\II}.$$

Discrete spectrum of type~$\I$ may exist only if the rank condition \eqref{eqn:rank} is satisfied.
In this case, in \cite{kk16} we constructed eigenfunctions of type~$\I$ for sufficiently regular~$\lambda$ when $\Gamma$ is \emph{sharp} --- a strong form of proper discontinuity \cite[Def.\,4.2]{kk16}.

In the classical setting where $H=K$, namely $X_{\Gamma}$ is a \emph{Riemannian} locally symmetric space $\Gamma\backslash G/K$, the discrete spectrum on~$X_{\Gamma}$ is always of type~$\II$.
In our pseudo-Riemannian setting it is not clear if there always exist $L^2$-eigenfunctions of type~$\II$ on $X_{\Gamma}$.
We shall prove the following in Chapter~\ref{sec:proof-mainII}.

\begin{theorem} \label{thm:mainII}
In the setting~\ref{spher-setting}, the set $\Spec_d(X_{\Gamma})_{\II}$ (hence $\Spec_d(X_{\Gamma})$) is infinite whenever $\Gamma$ is cocompact or arithmetic in~$L$.
\end{theorem}

Theorem~\ref{thm:mainII} applies to the triples $(G,H,L)$ of Table~\ref{table1}.
It gives an affirmative answer to Question~\ref{problems}.(a), and guarantees that the Laplacian $\Delta_{X_{\Gamma}}$ has infinitely many $L^2$-eigenvalues in this setting.
We note that $\Spec_d(X_{\Gamma})_{\I}$ is empty (\ie $\Spec_d(X_{\Gamma})=\Spec_d(X_{\Gamma})_{\II}$) for all~$\Gamma$ in case~(ii) with $n$ odd and in case~(vii) of Table~\ref{table1}: see Remark~\ref{rem:type-I-II}.(3).

\begin{remark}\label{rem:mainII}
S.~Mehdi and M.~Olbrich have announced that they can prove analogous results to Theorems \ref{thm:selfadj} and~\ref{thm:mainII} for the Laplacian for most triples $(G,H,L)$ in Table~\ref{table1}, by computing linear relations among the Casimir elements of $\g$, $\llll\cap\kk$, and~$\llll$ in the enveloping algebra $U(\g_{\C})$, similarly to Proposition~\ref{prop:rel-Lapl}.
As far as we understand, their method does not apply to higher-order differential operators on~$X_{\Gamma}$.
See their announcement \cite{mo21}, which appeared on the arXiv after we completed this work.
\end{remark}

\section{Group manifolds}

In this paper we also consider reductive symmetric spaces of the form $X = G/H = ({}^{\backprime}G\times\!{}^{\backprime}G)/\Diag({}^{\backprime}G)$ as in Example~\ref{ex:group-manifold}.
For $L={}^{\backprime}G\times\!{}^{\backprime}K$ where ${}^{\backprime}K$ be a maximal compact subgroup of~${}^{\backprime}G$, the complexification $X_{\C}$ is not always $L_{\C}$-spherical (see Example~\ref{ex:sph}.(3)), but we are still able to extend our techniques to prove the following analogue of Theorems \ref{thm:analytic}, \ref{thm:selfadj}, and~\ref{thm:mainII} for standard quotients $X_{\Gamma}$ with $\Gamma\subset L$.

\begin{theorem} \label{thm:GxG}
Let ${}^{\backprime}G$ be a noncompact reductive Lie group, ${}^{\backprime}K$ a maximal compact subgroup of~${}^{\backprime}G$, and $\Gamma$ a torsion-free discrete subgroup of $L={}^{\backprime}G\times\!{}^{\backprime}K$.
Then
\begin{enumerate}
  \item for any $\lambda\in\Spec_d(X_{\Gamma})$, the Hilbert space $L^2(X_{\Gamma};\M_{\lambda})$ contains real analytic eigenfunctions as a dense subset,
  \item the closure of the pseudo-Riemannian Laplacian~$\square_{X_{\Gamma}}$ on $C_c(X_{\Gamma})$ is a self-adjoint operator on $L^2(X_{\Gamma})$,
  \item $\Spec_d(X_{\Gamma})_{\II}$ (hence $\Spec_d(X_{\Gamma})$) is infinite whenever ${}^{\backprime}\Gamma$ is cocompact or arithmetic in~${}^{\backprime}G$.
\end{enumerate}
\end{theorem}

Beyond the classical case where $\Gamma$ is of the form ${}^{\backprime}\Gamma\times\{e\}$ and $L^2(X_{\Gamma})=L^2({}^{\backprime}\Gamma\backslash{}^{\backprime}G)$, we may obtain torsion-free discrete subgroups $\Gamma$ of $L={}^{\backprime}G\times\nolinebreak\!{}^{\backprime}K$ by considering a torsion-free discrete subgroup ${}^{\backprime}\Gamma$ of~${}^{\backprime}G$ and taking the graph of a homomorphism $\rho : {}^{\backprime}\Gamma\to{}^{\backprime}G$ with bounded image.
Nontrivial such homomorphisms exist in many situations, for instance when ${}^{\backprime}\Gamma$ is a free group, or when ${}^{\backprime}G=\SO(n,1)$ or $\SU(n,1)$ and ${}^{\backprime}\Gamma$ is a uniform lattice of~${}^{\backprime}G$ with $H^1({}^{\backprime}\Gamma;\R)\neq\{0\}$ (such lattices exist by \cite{mil76,kaz77}).

See Section \ref{subsec:proof-transfer} (\resp \ref{subsec:proof-selfadj}, \resp \ref{sec:proof-mainII}) for the proof of statement (1) (\resp (2), \resp (3)) of Theorem~\ref{thm:GxG}.

\section{The example of $\AdS^3$}

As one of the simplest examples, let $X$ be the $3$-dimensional \emph{anti-de Sitter space}
$$\AdS^3 = G/H = \SO(2,2)/\SO(2,1) \simeq (\SL(2,\R)\times\SL(2,\R))/\Diag(\SL(2,\R)),$$
which lies at the intersection of Examples \ref{ex:AdS-odd} and~\ref{ex:group-manifold}.
Then $\rank X=1$, and so the $\C$-algebra $\D_G(X)$ of $G$-invariant differential operators on~$X$ is generated by a single element, namely the Laplacian~$\square_X$.
Since $X$ is Lorentzian, the differential operator $\square_X$ is hyperbolic.
For any discrete subgroup $\Gamma$ of $\SO(2,2)$ acting properly discontinuously and freely on~$X$, we may identify $\Spec_d(X_{\Gamma})$ with the discrete spectrum of the Laplacian~$\square_{X_{\Gamma}}$.
With this identification, the following holds, independently of the fact that $X_{\Gamma}$ is compact or not.

\begin{proposition} \label{prop:AdS-I-II}
Let $\Gamma$ be a discrete subgroup of $\SO(2,2)$ acting properly discontinuously and freely on $X=\AdS^3$.
Then
\begin{enumerate}
  \item $\Spec_d(X_{\Gamma})_{\I} \subset \Spec_d(X) = \{ k(k+2)/4 : k\in\N\} \subset [0,+\infty)$,
  \item $0\in\Spec_d(X_{\Gamma})_{\II}$ if and only if $\mathrm{vol}(X_{\Gamma})<+\infty$,
  \item in the standard case where $\Gamma\subset L:=\U(1,1)$,
  \begin{itemize}
    \item $\Spec_d(X_{\Gamma})_{\I}$ is infinite, and contains $\{k(k+2)/4 : k\in\N,\ k\geq\nolinebreak k_0\}$ for some $k_0\in\N$ if $-1\notin\Gamma$,
    \item $\Spec_d(X_{\Gamma})_{\II}\subset (-\infty,0]$,
    \item $\Spec_d(X_{\Gamma})_{\II}$ is infinite whenever $\Gamma$ is cocompact or arithmetic in~$L$.
  \end{itemize}
\end{enumerate}
\end{proposition}

Proposition~\ref{prop:AdS-I-II} will be proved in Section~\ref{subsec:AdS3}.
It shows that\linebreak $\Spec_d(X_{\Gamma})_{\I} \cap \Spec_d(X_{\Gamma})_{\II} = \emptyset$ for all standard quotients $X_{\Gamma}$ of infinite volume when $X$ is the group manifold $({}^{\backprime}G\times\!{}^{\backprime}G)/\Diag({}^{\backprime}G)$ with ${}^{\backprime}G=\SL(2,\R)$.

\section{Organization of the paper}

In Chapter~\ref{sec:method} we explain the main method of proof for the results of Chapter~\ref{sec:intro}, and state refinements of Theorems \ref{thm:analytic} and~\ref{thm:mainII}, to be proved later in the paper.

Part~\ref{part:generalities} concerns generalities on invariant differential operators and their discrete spectrum on quotient manifolds $X_{\Gamma}=\Gamma\backslash X$ where $X=G/H$ is a spherical homogeneous space.
We start, in Chapter~\ref{sec:reminders}, by recalling some basic facts on $G$-invariant differential operators on~$X$, on joint eigenfunctions for $\D_G(X)$ on $X$ or~$X_{\Gamma}$, and on discrete series representations for~$X$.
Then, in Chapter~\ref{sec:type-I-II}, we introduce the notions of discrete spectrum of type~$\I$ and type~$\II$.
In Chapter~\ref{sec:dliota}, we consider a reductive subgroup $L$ of~$G$ acting properly and spherically on~$X$, and we discuss relations among the three subalgebras $\D_G(X)$, $\dd r(Z(\llll_{\C}\cap\kk_{\C}))$, and $\dd\ell(Z(\llll_{\C}))$ of $\D_L(X)$; we introduce several conditions, which we call (A), (B), ($\widetilde{\mathrm{A}}$), ($\widetilde{\mathrm{B}}$), and (Tf) for the existence of a pair of \emph{transfer maps} $\nnu$ and~$\llambda$ between eigenvalues on the Riemannian space~$Y_{\Gamma}$ and on the pseudo-Riemannian space~$X_{\Gamma}$, and we prove in particular that conditions (A), (B), and (Tf) are satisfied in the setting~\ref{spher-setting} (Proposition~\ref{prop:cond-Tf-A-B-satisfied}) and in the group manifold case (Proposition~\ref{prop:cond-Tf-A-B-satisfied-GxG}).

Part~\ref{part:main-proofs}, which is the core of the paper, provides proofs of the theorems of Chapter~\ref{sec:intro}.
We start, in Chapter~\ref{sec:strategy}, by establishing the essential self-adjointness of the pseudo-Riemannian Laplacian $\square_{X_{\Gamma}}$ (Theorems \ref{thm:selfadj} and~\ref{thm:GxG}.(2)); the proof, based on the important relation \eqref{eqn:rel-Lapl}, already illustrates the underlying idea of the transfer maps.
In Chapter~\ref{sec:transfer-Gamma}, using the transfer map~$\llambda$, we complete the proofs of Theorems \ref{thm:analytic} and~\ref{thm:GxG}.(1) (density of analytic eigenfunctions) and Theorem~\ref{thm:Fourier} (spectral decomposition).
Chapter~\ref{sec:transfer} is devoted to representation theory: we analyze the $G$-module structure together with the $L$-module structure (branching problem) on the space of distributions on $X=G/H$ under conditions (A) and~(B) (Theorem~\ref{thm:condB}).
In Chapter~\ref{sec:transfer-Gamma-type-I-II}, by using these results, we prove that the transfer maps preserve spectrum of type~$\I$ and type~$\II$ (Theorem~\ref{thm:transfer-spec}).
Thus we complete the proof of Theorems \ref{thm:mainII} and~\ref{thm:GxG}.(3) (existence of an infinite discrete spectrum of type~$\II$) in Chapter~\ref{sec:proof-mainII}.

Finally, Part~\ref{part:repr-spec} is devoted to the proof of Theorem~\ref{thm:Specd-lambda}, which describes the discrete spectrum of type~$\I$ and type~$\II$ of $X_{\Gamma}$ in terms of the representation theory of the reductive subgroup~$L$ via the transfer map~$\llambda$.
For this, we provide a general conjectural picture about $L^2$-spectrum on $X_{\Gamma}=\Gamma\backslash G/H$ and irreducible unitary representations of~$G$ in Chapter~\ref{sec:conj}, which we prove in two special cases in Chapter~\ref{sec:proof-Specd-lambda}.
These special cases combined with the results in Chapter~\ref{sec:transfer-Gamma-type-I-II} complete the proof of Theorem~\ref{thm:Specd-lambda}, see Section~\ref{subsec:proof-Specd-lambda}.

\section*{Convention}

In the whole paper, we assume that the respective Lie algebras $\g$, $\h$, $\llll$ of the real reductive Lie groups $G$, $H$, $L$ are defined algebraically over~$\R$.

\section*{Acknowledgements}

We are grateful to the University of Tokyo for its support through the GCOE program, and to the Institut des Hautes \'Etudes Scientifiques (Bures-sur-Yvette), Universit\'e Lille~1, Mathematical Sciences Research Institute (Berkeley), and Max Planck Institut f\"ur Mathematik (Bonn) for giving us opportunities to work together in very good conditions.

The results of the present paper were announced in \cite{kk20}.

\chapter{Method of proof} \label{sec:method}

In this chapter we explain our approach for proving the results of Chapter~\ref{sec:intro}.
We give refinements of Theorems \ref{thm:analytic} and~\ref{thm:mainII}, namely Theorems \ref{thm:transfer} and~\ref{thm:Specd-lambda}.

\section{Overview} \label{subsec:gen-setting}

In most of the paper (specifically, in Chapters \ref{sec:dliota} to~\ref{sec:proof-mainII}), we work in the following general setting.

\begin{gen-setting} \label{gen-setting}
We consider a reductive homogeneous space $X=G/H$ with $H$ noncompact, and a reductive subgroup $L$ of~$G$ acting properly and transitively on~$X$; then $L_H:=L\cap H$ is compact.
We also consider a maximal compact subgroup $K$ of~$G$ such that $L_K:=L\cap K$ is a maximal compact subgroup of~$L$ containing~$L_H$; then $X$ fibers over the Riemannian symmetric space $Y:=L/L_K$ of~$L$, with compact fiber $F:=L_K/L_H$:
\begin{equation} \label{eqn:bundleXY}
q : X = G/H \simeq L/L_H \overset{F}{\longrightarrow} L/L_K = Y.
\end{equation}
For simplicity we assume $G$, $H$, $L$ to be connected (see Remark~\ref{rem:connected}).
\end{gen-setting}

The main setting~\ref{spher-setting} of our theorems corresponds to the case that $X_{\C}=G_{\C}/H_{\C}$ is $L_{\C}$-spherical and $G$ simple, and the group manifold setting of Theorem~\ref{thm:GxG} to the case $(G,H,L) = ({}^{\backprime}G\times\!{}^{\backprime}G,\Diag({}^{\backprime}G),{}^{\backprime}G\times\!{}^{\backprime}K)$ where ${}^{\backprime}G$ is a noncompact reductive Lie group and ${}^{\backprime}K$ a maximal compact subgroup as in Example~\ref{ex:group-manifold}.

As in Chapter~\ref{sec:intro}, we consider standard pseudo-Riemannian locally homogeneous spaces $X_{\Gamma}=\Gamma\backslash X$ with $\Gamma\subset L$, and our goal is to analyze joint eigenfunctions on~$X_{\Gamma}$ for the differential operators $D_{\Gamma}$ with $D\in\D_G(X)$.
The groups involved here are summarized in the following diagram.
\begin{alignat*}{7}
     &
     &&
     &&\,G                                      
     && 
     &&\,\supset\,\,
     &&
     &&\,H\nonumber
\\
     &
     &&
     &&\,\text{\rotatebox{90}{$\subset$}}
     &&         
     &&
     && 
     &&\,\,\text{\rotatebox{90}{$\subset$}}
\\
     &\Gamma
     &&\quad\subset\quad
     &&\,L
     &&\quad\supset\quad
     &&L_K
     &&\quad\supset\quad
     &&L_H
\end{alignat*}

For this goal we consider, not only the algebra $\D_G(X)$, but also the larger $\C$-algebra $\D_L(X)$ of $L$-invariant differential operators on~$X$.
Let $Z(\llll_{\C})$ be the center of the enveloping algebra $U(\llll_{\C})$ and $\dd\ell : Z(\llll_{\C})\to\D_L(X)$ the natural $\C$-algebra homomorphism (see \eqref{eqn:dl-dr}).
Assuming that $X_{\C}=G_{\C}/H_{\C}$ is $L_{\C}$-spherical and $G$ simple, in \cite{kkdiffop} we proved the existence of \emph{transfer maps} $\nnu$ and $\llambda$ which relate the subalgebra $\D_G(X)$ and the image $\dd\ell(Z(\llll_{\C}))$ inside $\D_L(X)$, and thus relate the representations of the group $G$ and the subgroup~$L$ generated by eigenfunctions of $\D_G(X)$ and $\dd\ell(Z(\llll_{\C}))$ respectively.
In the current paper, these maps enable us to construct joint eigenfunctions on \emph{pseudo-Rieman\-nian} locally symmetric spaces $X_{\Gamma}$ with $\Gamma\subset L$ using (vector-bundle-valued) eigenfunctions on the corresponding \emph{Riemannian} locally symmetric space $Y_{\Gamma}=\Gamma\backslash L/L_K$.
We now explain this in more detail.

\section[Vector-bundle valued eigenfunctions]{Vector-bundle valued eigenfunctions in the Riemannian setting} \label{subsec:bundle}

In the whole paper, we denote by $\Disc(L_K/L_H)$ the set of equivalence classes of (finite-dimensional) irreducible representations $(\tau,V_{\tau})$ of the compact group~$L_K$ with nonzero $L_H$-fixed vectors.

For any $(\tau,V_{\tau})\in\Disc(L_K/L_H)$, the contragredient representation $V_{\tau}^{\vee}$ has a nonzero subspace $(V_{\tau}^{\vee})^{L_H}$ of $L_H$-fixed vectors.
For $\F=\A$, $C^{\infty}$, $L^2$, or~$\DD'$, let $\F(Y_{\Gamma},\V_{\tau})$ be the space of analytic, smooth, square-integrable, or distribution sections of the Hermitian vector bundle
\begin{equation} \label{eqn:mathcal-V-tau}
\V_{\tau}:=\Gamma\backslash L\times_{L_K} V_{\tau}
\end{equation}
over~$Y_{\Gamma}$.
There is a natural homomorphism
\begin{equation} \label{eqn:i-tau}
\ii_{\tau,\Gamma} : (V_{\tau}^{\vee})^{L_H} \otimes \F(Y_{\Gamma},\V_{\tau}) \longhookrightarrow \F(X_{\Gamma})
\end{equation}
sending $v_{\tau}\otimes\varphi$ to $\langle\varphi,v_{\tau}\rangle$, where we see $\varphi$ as an $L_K$-invariant element of $\F(\Gamma\backslash L)\otimes V_{\tau}$ under the diagonal action, and $\langle\varphi,v_{\tau}\rangle$ as an $L_H$-invariant element of $\F(\Gamma\backslash L)$.
The map $\ii_{\tau,\Gamma}$ is injective and continuous (see Remark~\ref{rem:top-distrib} for the topology on the space of distributions).
Under the assumption that $X_{\C}$ is $L_{\C}$-spherical, the space $(V_{\tau}^{\vee})^{L_H}\simeq\C$ is one-dimensional (see \cite[Lem.\,4.2.(4) \& Fact\,3.1.(iv)]{kkdiffop}), and so $\ii_{\tau,\Gamma}$ becomes a map from $\F(Y_{\Gamma},\V_{\tau})$ to $\F(X_{\Gamma})$; the algebraic direct sum $\bigoplus_{\tau} \F(Y_{\Gamma},\V_{\tau})$ is dense in $\F(X_{\Gamma})$ via the~$\ii_{\tau,\Gamma}$ (see Lemma~\ref{lem:p-tau-i-tau}).

When $\Gamma=\{e\}$ is trivial, we use the same notation $\mathcal{V}_{\tau}$ for the vector bundle over $Y$ associated to the representation $(\tau, V_\tau)$ as in \eqref{eqn:mathcal-V-tau}, and simply write $\ii_{\tau}$ for $\ii_{\tau,\Gamma}$.
Then $p_{\Gamma}^{\ast}\circ\ii_{\tau,\Gamma} = \ii_{\tau}\circ {p'_{\Gamma}}^{\!\!\ast}$ (see Observation~\ref{obs:extend-i-tau}), where $p_{\Gamma}^{\ast} : \F(X_{\Gamma})\to\DD'(X)$ and ${p'_{\Gamma}}^{\!\!\ast} : \F(Y_{\Gamma},\V_{\tau})\to\DD'(Y,\V_{\tau})$ denote the maps induced by the natural projections $p_{\Gamma} : X\to X_{\Gamma}$ and $p'_{\Gamma} : Y\to Y_{\Gamma}$, respectively.

The enveloping algebra $U(\llll_{\C})$ acts on the left on $\F(Y,\V_{\tau})$ as matrix-valued differential operators.
In particular, this gives a $\C$-algebra homomorphism $\dd\ell^{\tau}$ from $Z(\llll_{\C})$ to the $\C$-algebra $\D_L(Y,\V_{\tau})$ of $L$-invariant differential operators on $\F(Y,\V_{\tau})$.
In turn (by commutativity with the action of~$\Gamma$), for any $z\in Z(\llll_{\C})$ the operator $\dd\ell^{\tau}\!(z)$ induces a matrix-valued differential operator $\dd\ell^{\tau}\!(z)_{\Gamma}$ acting on $\F(Y_{\Gamma},\V_{\tau})$.
For any $\C$-algebra homomorphism $\nu : Z(\llll_{\C})\to\C$, let $\F(Y_{\Gamma},\V_{\tau};\NN_{\nu})$ be the space of solutions $\varphi\in\F(Y_{\Gamma},\V_{\tau})$ to the system
$$\dd\ell^{\tau}\!(z)_{\Gamma}\,\varphi = \nu(z) \varphi \quad\quad\mathrm{for\ all}\ z\in Z(\llll_{\C}) \eqno{(\NN_{\nu})}.$$
This space $\F(Y_{\Gamma},\V_{\tau};\NN_{\nu})$ is nonzero only if $\nu$ vanishes on $\Ker(\dd\ell^{\tau})$, in which case we can see $\nu$ as an element of $\Hom_{\C\text{-}\mathrm{alg}}(Z(\llll_{\C})/\Ker(\dd\ell^{\tau}),\C)$.
The space $\F(Y_{\Gamma},\V_{\tau};\NN_{\nu})$ is a subspace of $\A(Y_{\Gamma},\V_{\tau};\NN_{\nu})$ by the elliptic regularity theorem, since the system $(\NN_{\nu})$ contains an elliptic differential equation when $Y_{\Gamma}$ is Riemannian.

\begin{example} \label{ex:trivial-tau}
When $\tau$ is the trivial one-dimensional representation of $L_K$, the space $\F(Y,\V_\tau)$ identifies with $\F(Y)$, the $\C$-algebra $\D_L(Y,\V_\tau)$ with $D_L(Y)$, and the map $d\ell^\tau$ with the natural $\C$-algebra homomorphism $\dd\ell: Z(\llll_{\C}) \to \D_L(Y)$, see \eqref{eqn:dl-dr}; likewise, $\F(Y_\Gamma, \V_\tau; \NN_\nu)$ identifies with $\F(Y_\Gamma, \NN_\nu)$, and the map $\ii_{\tau,\Gamma}$ of \eqref{eqn:i-tau} is the pull-back of the projection map $q_{\Gamma} : X_{\Gamma}\to Y_{\Gamma}$.
\end{example}

\section{Transferring Riemannian eigenfunctions} \label{subsec:method-transfer}

We wish to understand joint eigenfunctions of the algebra $\D_G(X)$ on $\F(X_{\Gamma})$.
When $X_{\C}$ is $L_{\C}$-spherical, $\D_G(X)$ leaves the subspace $\ii_{\tau,\Gamma}(\F(Y_{\Gamma},\V_{\tau}))\subset\F(X_{\Gamma})$ invariant for every~$\tau$ (see \cite[Th.\,4.9]{kkdiffop}).
On the other hand, we have seen that the center $Z(\llll_{\C})$ naturally acts on $\F(Y_{\Gamma},\V_{\tau})$.
Although there is no direct map between the two algebras $\D_G(X)$ and $Z(\llll_{\C})$, the respective eigenvalues of $\D_G(X)$ and $Z(\llll_{\C})$ are still related as follows.

Assuming that $X_{\C}=G_{\C}/H_{\C}$ is $L_{\C}$-spherical and $G$ simple, in \cite{kkdiffop} we constructed ``transfer maps'' in a compact setting.
Via holomorphic continuation of invariant differential operators, we obtain for any $\tau\in\Disc(L_K/L_H)$ analogous maps
\begin{equation} \label{eqn:nu-lambda-tau}
\left\{ \begin{array}{l}
\nnu(\cdot,\tau) : \Hom_{\C\text{-}\mathrm{alg}}(\D_G(X),\C) \longrightarrow \Hom_{\C\text{-}\mathrm{alg}}(Z(\llll_{\C}),\C),\\
\llambda(\cdot,\tau) : \Hom_{\C\text{-}\mathrm{alg}}(Z(\llll_{\C})/\Ker(\dd\ell^{\tau}),\C) \longrightarrow \Hom_{\C\text{-}\mathrm{alg}}(\D_G(X),\C),
\end{array} \right.
\end{equation}
referred to also as \emph{transfer maps}, which are described explicitly in each case.
These maps have the following properties:
\begin{itemize}
  \item $\nnu(\llambda(\nu,\tau)) = \nu$ for all $\nu\in\Hom_{\C\text{-}\mathrm{alg}}(Z(\llll_{\C})/\Ker(\dd\ell^{\tau}),\C)$;
  \item $\llambda(\nnu(\lambda,\tau)) = \lambda$ for all $\lambda\in\mathrm{Spec}(X)_{\tau}$,
\end{itemize}
where $\mathrm{Spec}(X)_{\tau}$ denotes the set of $\lambda\in\Hom_{\C\text{-}\mathrm{alg}}(\D_G(X),\C)$ such that $\DD'(X;\M_{\lambda}) \cap \ii_{\tau}(\DD'(Y,\V_{\tau})) \neq \{ 0\}$.
We note that $\nnu(\lambda,\tau)$ vanishes on $\Ker(\dd\ell^{\tau})$ if $\lambda\in\mathrm{Spec}(X)_{\tau}$, see \cite[Prop.\,4.8]{kkdiffop}, and thus the composition $\llambda(\nnu(\lambda,\tau))$ is well defined for $\lambda\in\mathrm{Spec}(X)_{\tau}$.

We shall briefly review the existence of the maps $\nnu(\cdot,\tau)$ and $\llambda(\cdot,\tau)$ in Chapter~\ref{sec:dliota}, by introducing conditions (A) and~(B) on $G$- and $L$-submodules of $C^{\infty}(X)$, and conditions ($\widetilde{\mathrm{A}}$) and~($\widetilde{\mathrm{B}}$) on invariant differential operators on~$X$.
Using $\nnu(\cdot,\tau)$, we now construct a dense subspace of $\F(X_{\Gamma};\M_{\lambda})$ consisting of real analytic functions, in terms of the Riemannian data $\A(Y_{\Gamma},\V_{\tau};\NN_{\nu})$; this refines Theorem~\ref{thm:analytic}.

\begin{theorem}[Transfer of Riemannian eigenfunctions] \label{thm:transfer}
In the general setting~\ref{gen-setting}, suppose that $X_{\C}=G_{\C}/H_{\C}$ is $L_{\C}$-spherical and $G$ simple, and let $\Gamma$ be a torsion-free discrete subgroup of~$L$; in other words, we are in the main setting~\ref{spher-setting}.
For any  $\lambda \in \Hom_{\C\text{-}\mathrm{alg}}(\D_G(X),\C)$ and any $\tau\in\Disc(L_K/L_H)$, we have
$$\ii_{\tau,\Gamma}\big(C^{\infty}(Y_{\Gamma},\V_{\tau};\NN_{\nnu(\lambda,\tau)})\big) \subset C^{\infty}(X_{\Gamma};\M_{\lambda}).$$
Moreover, the algebraic direct sum
$$\bigoplus_{\tau\in\Disc(L_K/L_H)} \ii_{\tau,\Gamma}\big(\A(Y_{\Gamma},\V_{\tau};\NN_{\nnu(\lambda,\tau)})\big)$$
is dense in $\F(X_{\Gamma};\M_{\lambda})$ for $\F=\A$, $C^{\infty}$, or~$\DD'$ in each topology, and
$$\bigoplus_{\tau\in\Disc(L_K/L_H)} \ii_{\tau,\Gamma}\big((\A\cap L^2)(Y_{\Gamma},\V_{\tau};\NN_{\nnu(\lambda,\tau)})\big)$$
is dense in the Hilbert space $L^2(X_{\Gamma};\M_{\lambda})$.
\end{theorem}

Theorem~\ref{thm:transfer} applies to the triples $(G,H,L)$ of Table~\ref{table1}.
An analogous result does not hold anymore if we remove the assumption that $X_{\C}$ is $L_{\C}$-spherical: see Example~\ref{ex:not-spher-A-B-nonsym}.(1) for trivial~$\Gamma$.

We also obtain the following refinement of Theorem~\ref{thm:GxG}.(1).

\begin{theorem}[Transfer of Riemannian eigenfunctions in the group manifold case] \label{thm:transfer-GxG}
In the setting of Theorem~\ref{thm:GxG}, the same conclusion as Theorem~\ref{thm:transfer} holds with
$$(G,H,L,L_K,L_H) := \big({}^{\backprime}G\times\!{}^{\backprime}G,\Diag({}^{\backprime}G),{}^{\backprime}G\times\!{}^{\backprime}K,{}^{\backprime}K\times\!{}^{\backprime}K,\Diag({}^{\backprime}K)\big).$$
\end{theorem}

Theorems \ref{thm:transfer} and~\ref{thm:transfer-GxG} assert that the space $C^{\infty}(Y_{\Gamma},\V_{\tau};\NN_{\nu})$ of joint eigenfunctions for the center $Z(\llll_{\C})$ is transferred by~$\ii_{\tau,\Gamma}$ into a space of joint eigenfunctions for $\D_G(X)$.
They imply that, in their respective settings, $\Spec_d(X_{\Gamma})$ is infinite if $X_{\Gamma}$ is compact, giving an affirmative answer to Question~\ref{problems}.(a).
These theorems will be proved in Section~\ref{subsec:L2decomp} in the special case where $\rank G/H=1$, and in Section~\ref{subsec:proof-transfer} in general.

Here is a consequence of Theorem~\ref{thm:transfer} (see Section~\ref{subsec:proof-transfer}).

\begin{corollary} \label{cor:L2d(XGamma,Mlambda)-orth}
In the main setting \ref{spher-setting}, the spaces $L^2(X_{\Gamma};\M_{\lambda})$, for varying $\lambda\in\Hom_{\C\text{-}\mathrm{alg}}(\D_G(X),\C)$, are orthogonal to each other.
\end{corollary}

The algebras $\D_G(X)$ and $Z(\llll_{\C})$ are described by the Harish-Chandra isomorphism (see Section~\ref{subsec:DGH}), the set $\Disc(L_K/L_H)$ is described by (a variant of) the Cartan--Helgason theorem (see \cite[Th.\,3.3.1.1]{war72}), and the maps $\nnu(\cdot,\tau)$ are described explicitly in \cite[\S\,6--7]{kkdiffop} in each case of Table~\ref{table1}.
Here is one example; we refer to \cite{kkdiffop} for others.

\begin{example}[{\cite[\S\,6.2]{kkdiffop}}] \label{ex:SO-4n-2}
Let $X = G/H = \SO(4m,2)_0/\U(2m,1)$ where $m\geq 1$.
The space $X$ is a complex manifold of dimension $2m^2+m$, endowed with an indefinite K\"ahler structure of signature $(2m,2m^2-m)$ on which $G$ acts holomorphically by isometries.
The $\C$-algebra $\D_G(X)$ is generated by $m$ algebraically independent differential operators $P_k$ of order $2k$, for $1\leq k\leq m$, and $\Hom_{\C\text{-}\mathrm{alg}}(\D_G(X),\C)$ is parametrized by the Harish-Chandra isomorphism
$$\Hom_{\C\text{-}\mathrm{alg}}(\D_G(X),\C) \simeq \C^m/W(BC_m),$$
where $W(BC_m):=\mathfrak{S}_m\ltimes (\Z/2\Z)^m$ is the Weyl group for the root system of type $BC_m$.
Let $L:=\SO(4m,1)_0$.
By the Harish-Chandra isomorphism again, we have
$$\Hom_{\C\text{-}\mathrm{alg}}(Z(\llll_{\C}),\C) \simeq \C^{2m}/W(B_{2m}),$$
where $W(B_{2m}):=\mathfrak{S}_{2m}\ltimes (\Z/2\Z)^{2m}$.
The group $L$ acts properly and transitively on~$X$, and we have a diffeomorphism $X\simeq L/L_H$ where $L_H=\U(2m)$.
By taking $L_K=\SO(4m)$, the fiber $F=L_K/L_H$ is the compact symmetric space $\SO(4m)/\U(2m)$.
The Cartan--Helgason theorem gives a parametrization of $\Disc(L_K/L_H)$ by
$$\big\{ \mu = (j_1,j_1,\dots,j_m,j_m) \in \Z^{2m} : j_1\geq\dots\geq j_m\geq 0\big\},$$
where the vector $\mu$ corresponds to the irreducible finite-dimensional representation $(\tau,V_{\tau})\in\Disc(L_K/L_H)$ of $L_K=\SO(4m)$ with highest weight~$\mu$.
The transfer map
$$\nnu(\cdot,\tau) : \Hom_{\C\text{-}\mathrm{alg}}(\D_G(X),\C) \longrightarrow \Hom_{\C\text{-}\mathrm{alg}}(Z(\llll_{\C}),\C)$$
of \eqref{eqn:nu-lambda-tau}, as constructed in \cite[\S\,1.3]{kkdiffop}, sends $\lambda = (\lambda_1,\dots,\lambda_m) \!\!\mod W(BC_m)$ to
$$\frac{1}{2} \, (\lambda_1,2j_1+4m-3,\lambda_2,2j_2+4m-7,\dots,\lambda_m,2j_m+1) \!\!\mod W(B_{2m}).$$
\end{example}

\section{Describing the discrete spectrum} \label{subsec:method-transfer-map}

In the main setting \ref{spher-setting}, we now give a description of the discrete spectrum (of type $\I$ and~$\II$) for $L^2(X_{\Gamma}) = L^2(\Gamma\backslash G/H)$ by means of the data for the regular representation of the subgroup $L$ of~$G$ on $L^2(\Gamma\backslash L)$, via the transfer map $\llambda$ of \eqref{eqn:nu-lambda-tau}.
For this we use the following notation from representation theory.

For a real reductive Lie group~$L$, let $\widehat{L}$ be the unitary dual of $L$, the set of equivalence classes of irreducible unitary representations of~$L$.
Given a closed unimodular subgroup $M$ of~$L$, we denote by $\Disc(L/M)$ the set of $\vartheta\in\widehat{L}$ such that $\Hom_L(\vartheta,L^2(L/M))$ is nonzero.
When $M=\{e\}$ is trivial, we simply write $\Disc(L)$ for $\Disc(L/M)$.
The set $\Disc(L)$ coincides with $\widehat{L}$ when $L$ is compact, whereas it is the set of Harish-Chandra's discrete series representations when $L$ is noncompact.

Let $L_K$ be a maximal compact subgroup of~$L$.
For $\tau\in\widehat{L_K}$, we set
$$\widehat{L}(\tau) := \big\{ \vartheta\in\widehat{L} : \Hom_{L_K}(\tau,\vartheta|_{L_K})\neq\{ 0\}\big\}$$
and
$$\Disc(L)(\tau) := \Disc(L) \cap \widehat{L}(\tau).$$
When $L$ is noncompact, $\widehat{L}(\tau)$ contains continuously many elements for any $\tau\in\widehat{L_K}$, whereas $\Disc(L)(\tau)$ is either empty (\eg if $\tau$ is the trivial one-dimen\-sional representation of~$L_K$) or finite-dimensional by the Blattner formula \cite{hs75}.
If $\vartheta\in\widehat{L}(\tau)$, then the infinitesimal character $\chi_{\vartheta}\in\Hom_{\C\text{-}\mathrm{alg}}(Z(\llll_{\C}),\C)$ of~$\vartheta$ (see Section~\ref{subsec:notation-type-I-II}) vanishes on $\Ker(\dd\ell^{\tau})$, since $\vartheta$ is realized in $\DD'(Y,\V_{\tau})$ and since the action of $Z(\llll_{\C})$ factors through $\dd\ell^{\tau} : Z(\llll_{\C})\to\D_L(Y,\V_{\tau})$.
We regard $\chi_{\vartheta}$ as an element of $\Hom_{\C\text{-}\mathrm{alg}}(Z(\llll_{\C})/\Ker(\dd\ell^{\tau}),\C)$.

\begin{theorem} \label{thm:Specd-lambda}
In the general setting~\ref{gen-setting}, suppose that $X_{\C}=G_{\C}/H_{\C}$ is $L_{\C}$-spherical and $G$ simple, and let $\Gamma$ be a torsion-free discrete subgroup of~$L$; in other words, we are in the main setting~\ref{spher-setting}.
Then
\begin{eqnarray*}
\Spec_d(X_{\Gamma}) & \!\!=\!\!\!\! & \bigcup_{\tau\in\Disc(L_K/L_H)}\! \big\{ \llambda(\chi_{\vartheta},\tau) : \vartheta\in\Disc(\Gamma\backslash L)\cap\widehat{L}(\tau)\big\},\\
\Spec_d(X_{\Gamma})_{\I} & \!\!=\!\!\!\! & \bigcup_{\tau\in\Disc(L_K/L_H)}\! \big\{ \llambda(\chi_{\vartheta},\tau) : \vartheta\in\Disc(\Gamma\backslash L)\cap\Disc(L)(\tau)\big\},\\
\Spec_d(X_{\Gamma})_{\II} & \!\!=\!\!\!\! & \bigcup_{\tau\in\Disc(L_K/L_H)}\! \big\{ \llambda(\chi_{\vartheta},\tau) : \vartheta\in\Disc(\Gamma\backslash L)\!\cap\!\widehat{L}(\tau)\!\smallsetminus\!\Disc(L)(\tau)\big\}.
\end{eqnarray*}
\end{theorem}

Theorem~\ref{thm:Specd-lambda} will be proved in Section~\ref{subsec:proof-Specd-lambda}.
The point is to give a description of the full discrete spectrum of $X_{\Gamma}$, both of type~$\I$ and type~$\II$.

The approach here for discrete spectrum of type~$\I$ is very different from our earlier approach from \cite{kk16}.
Namely, in \cite{kk16}, using work of Flensted-Jensen \cite{fle80} and Matsuki--Oshima \cite{mo84}, we constructed (via generalized Poin\-car\'e series) an infinite subset of $\Spec_d(X_{\Gamma})_{\I}$ under the rank assumption \eqref{eqn:rank} whenever the action of $\Gamma$ on~$X$ is \emph{sharp} in the sense of \cite[Def.\,4.2]{kk16}, which includes the case of standard~$X_{\Gamma}$.
We showed that in many cases this infinite subset of $\Spec_d(X_{\Gamma})_{\I}$ is invariant under any small deformation of $\Gamma$ inside~$G$.
In Theorem~\ref{thm:Specd-lambda} we do not assume the rank condition \eqref{eqn:rank} to be satisfied.

\part{Generalities} \label{part:generalities}

In this Part~\ref{part:generalities}, we introduce some basic notions that are used in stating the main results, and set up some machinery for the proofs.

In Chapter~\ref{sec:reminders}, we start by briefly summarizing some basic facts on (real) spherical manifolds, on the algebra $\D_G(X)$ of $G$-invariant differential operators on~$X$, and on irreducible representations which are realized in the space $\DD'(X)$ of distributions on~$X$.

Any joint eigenfunction $f$ of $\D_G(X)$ on a quotient manifold $X_{\Gamma}$ generates a $G$-submodule $U_f$ of $\DD'(X)$, which is of finite length when $X$ is real spherical.
We wish to relate the spectrum for $\D_G(X)$ on~$X_{\Gamma}$ with the representation theory of~$G$.
In Chapter~\ref{sec:type-I-II}, we introduce a definition of $L^2$-eigenfunction of type~$\I$ and type~$\II$; averaging $L^2$-eigenfunctions on~$X$ by the discrete group~$\Gamma$ yields discrete spectrum of type~$\I$ \cite{kk16} (under some \emph{sharpness} assumption, see \cite[Def.\,4.2]{kk16}), whereas discrete spectrum in the classical setting where $X$ is Riemannian is of type~$\II$.

Chapter~\ref{sec:dliota} is devoted to the existence of transfer maps $\nnu$ and~$\llambda$.
Our strategy for spectral analysis of $\D_G(X)$ on standard locally homogeneous spaces $X_{\Gamma}$ is to use the representation theory of the subgroup $L$ of~$G$, which contains~$\Gamma$.
In general, spectral information coming from irreducible $L$-modules is described by the action of the center $Z(\llll_{\C})$ of the enveloping algebra $U(\llll_{\C})$, which is different from that by $\D_G(X)$.
In Chapter~\ref{sec:dliota}, we study how spectral information for $\D_G(X)$ and for $Z(\llll_{\C})$ are related.
For this, we use the $L$-equivariant fiber bundle structure $F\to X\to Y$ of \eqref{eqn:bundleXY}.
We formulate the ``abundance'' of differential operators coming from $Z(\llll_{\C})$ (\resp $\D_G(X)$) inside $\D_L(X)$ as condition ($\widetilde{\mathrm{A}}$) (\resp ($\widetilde{\mathrm{B}}$)) in terms of complexified Lie algebras, and also as condition (A) (\resp (B)) in terms of representation theory of the real Lie groups $G$ and~$L$.
We prove the existence of transfer maps relating the two algebras $\D_G(X)$ and $Z(\llll_{\C})$ (condition (Tf)) when $X_{\C}$ is $L_{\C}$-spherical and $G$ simple (Proposition~\ref{prop:cond-Tf-A-B-satisfied}) by reducing to the compact case established in \cite{kkdiffop}.

\chapter{Reminders: spectral analysis on spherical homogeneous spaces} \label{sec:reminders}

In this chapter we set up some terminology and notation, and recall basic facts on (real) spherical manifolds, on $G$-invariant differential operators on~$X$, on joint eigenfunctions for $\D_G(X)$ on quotient manifolds $X_{\Gamma}$, on discrete series representations for~$X$, and on generalized matrix coefficients for distribution vectors of unitary representations.

\section{An elementary example} \label{subsec:torus-ex}

Eigenfunctions of the Laplacian on a pseu\-do-Riem\-annian manifold are not always smooth, making Theorem~\ref{thm:analytic} nontrivial.
We illustrate this phenomenon with an elementary example where an analogue of the elliptic regularity theorem does not hold for the Laplacian in a pseudo-Riemannian setting.

Let $M$ be the torus $\R^2/\Z^2$ equipped with the pseudo-Riemannian metric $ds^2=dx^2-dy^2$.
The Laplacian $\square_M = \frac{\partial}{\partial x^2} - \frac{\partial}{\partial y^2}$ is a differential operator which is hyperbolic, not elliptic, and so the elliptic regularity theorem does not apply.
In fact, there exists a distribution eigenfunction of $\square_M$ which is not $C^{\infty}$.
Indeed, recall the Paley--Wiener theorem, which characterizes functions (or distributions) $f$ on the torus $X_{\Gamma}\simeq\mathbb{S}^1\times\mathbb{S}^1$ in terms of the Fourier series
$$f(x,y) = \sum_{n\in\Z} a_n e^{2i\pi (mx+ny)}$$
as follows:
\begin{itemize}
  \item $f\in\DD'(M)$ $\Leftrightarrow$ $\sup_{m,n\in\Z} |a_{m,n}| (1+m^2+n^2)^{-N} < +\infty$ for some $N>0$,
  \item $f\in L^2(M)$ $\Leftrightarrow$ $\sum_{m,n\in\Z} |a_{m,n}|^2 < +\infty$,
  \item $f\in C^{\infty}(M)$ $\Leftrightarrow$ $\sup_{m,n\in\Z} |a_{m,n}| (1+m^2+n^2)^{N} < +\infty$ for all $N>0$.
\end{itemize}
If we take $a_{m,n}=\delta_{m,n}$ (Kronecker symbol), then $f\in\DD'(M)$ and $\square_Mf=0$ in the distribution sense, but $f\notin C^{\infty}(M)$.
If we take $a_{m,n}=\delta_{m,n}/(|n|+1)$, then $f\in L^2(M)$ and $\square_Mf=0$ in the distribution sense, but $f\notin C^{\infty}(M)$.
This shows that weak solutions are not necessarily smooth solutions.

\section{Spherical manifolds}

We shall use the following terminology.

\begin{definition}\label{def:spherical}
\begin{itemize}
  \item A connected real manifold $X$ endowed with an action of a real reductive Lie group~$G$ is \emph{$G$-real spherical} if it admits an open orbit of a minimal parabolic subgroup $P$ of~$G$.
  \item A connected complex manifold $X_{\C}$ endowed with a holomorphic action of a complex reductive Lie group~$G_{\C}$ is \emph{$G_{\C}$-spherical} if it admits an open orbit of a Borel subgroup $B_{\C}$ of~$G_{\C}$.
\end{itemize}
\end{definition}

We now fix a real reductive Lie group $G$ and a complexification $G_{\C}$ of~$G$.

\begin{example}\label{ex:sph}
\begin{enumerate}
  \item Any complex reductive symmetric space $G_{\C}/H_{\C}$ is $G_{\C}$-sphe\-rical \cite{wol74}.
  \item If $X=G/H$ is a homogeneous space whose complexification $X_{\C}=G_{\C}/H_{\C}$ is $G_{\C}$-spherical, then $X$ is $G$-real spherical (see \cite[Lem.\,4.2]{ko13}).
  In particular, any complex reductive symmetric space $G_{\C}/H_{\C}$ is $G$-real spherical.
  \item Any homogeneous space $G/H$ where $G$ is a compact reductive Lie group is $G$-real spherical.
  \item Let ${}^{\backprime}G$ be a noncompact reductive Lie group and ${}^{\backprime}K$ a maximal compact subgroup of~${}^{\backprime}G$.
  Then $X=({}^{\backprime}G\times\!{}^{\backprime}K)/\Diag({}^{\backprime}K)$ is always $({}^{\backprime}G\times\!{}^{\backprime}K)$-real~sphe\-rical by the Iwasawa decomposition; its complexification $X_{\C}$ is $({}^{\backprime}G_{\C}\times\!{}^{\backprime}K_{\C})$-spherical if and only if each noncompact simple factor of~${}^{\backprime}G$ is locally isomorphic to $\SO(n,1)$ or $\SU(n,1)$, by Cooper \cite{coo75} and Kr\"amer \cite{kra76}.
\end{enumerate}
\end{example}

\section{Invariant differential operators on~$X$}\label{subsec:DGH}

Let $X=G/H$ be a reductive homogeneous space.
In this paragraph, we recall some classical results on the structure of the $\C$-algebra $\D_G(X)$ of $G$-invariant differential operators on~$X$.
We refer the reader to \cite[Ch.\,II]{hel00} for proofs and more details.

Let $U(\g_{\C})$ be the enveloping algebra of the complexified Lie algebra $\g_{\C}:=\g\otimes_{\R}\C$ and $U(\g_{\C})^H$ the subalgebra of $\Ad_G(H)$-invariant elements; the latter contains in particular the center $Z(\g_{\C})$ of $U(\g_{\C})$. 
Recall that $U(\g_{\C})$ acts on $C^{\infty}(G)$ by differentiation on the left, with
$$\big((Y_1\cdots Y_m)\cdot f\big)(g) = \frac{\partial}{\partial t_1}\Big|_{t_1=0}\,\cdots\ \frac{\partial}{\partial t_m}\Big|_{t_m=0}\ f\big(\exp(-t_mY_m)\cdots\exp(-t_1Y_1)g\big)$$
for all $Y_1,\dots,Y_m\in\g$, all $f\in C^{\infty}(G)$, and all $g\in G$.
It also acts on $C^{\infty}(G)$ by differentiation on the right, with
$$\big((Y_1\cdots Y_m)\cdot f\big)(g) = \frac{\partial}{\partial t_1}\Big|_{t_1=0}\,\cdots\ \frac{\partial}{\partial t_m}\Big|_{t_m=0}\ f\big(g\exp(t_1Y_1)\cdots\exp(t_mY_m)\big)$$
for all $Y_1,\dots,Y_m\in\g$, all $f\in C^{\infty}(G)$, and all $g\in G$.
By identifying the set of smooth functions on~$X$ with the set of right-$H$-invariant smooth functions on~$G$, we obtain a $\C$-algebra homomorphism
\begin{equation} \label{eqn:dl-dr}
\dd\ell \otimes \dd r : U(\g_{\C}) \otimes U(\g_{\C})^H \longrightarrow \D(X),
\end{equation}
where $\D(X)$ the full $\C$-algebra of differential operators on~$X$.
We have $\dd\ell(Z(\g_{\C}))=\dd r(Z(\g_{\C}))$.
The homomorphism $\dd r$ has image $\D_G(X)$ and kernel $U(\g_{\C})\h_{\C} \cap U(\g_{\C})^H$, hence induces a $\C$-algebra isomorphism
\begin{equation} \label{eqn:DX}
U(\g_{\C})^H/U(\g_{\C})\h_{\C} \cap U(\g_{\C})^H \,\overset{\scriptscriptstyle\sim}\longrightarrow\, \D_G(X)
\end{equation}
\cite[Ch.\,II, Th.\,4.6]{hel00}.

\begin{fact}\label{fact:spherical}
Let $X=G/H$ be a reductive homogeneous space.
The following conditions are equivalent:
\begin{enumerate}[(i)]
  \item the complexification $X_{\C}$ is $G_{\C}$-spherical,
  \item the $\C$-algebra $\D_G(X)$ is commutative,
  \item there exists $C>0$ such that $\dim\Hom_{\g,K}(\pi_K,C^{\infty}(X))\leq C$ for all irreducible $(\g,K)$-modules~$\pi_K$,
\end{enumerate}
where $\Hom_{\g,K}(\pi_K,C^{\infty}(X))$ is the set of $(\g,K)$-homomorphisms from $\pi_K$ to $C^{\infty}(X)$.
\end{fact}

For (i)$\,\Leftrightarrow\,$(ii), see \cite{vin01} for instance; for (i)$\,\Leftrightarrow\,$(iii), see \cite{ko13}.

If $X_{\C}$ is $G_{\C}$-spherical, then by work of Knop \cite{kno94} the $\C$-algebra $\D_G(X)$ is finitely generated as a $Z(\g_{\C})$-module and there is a $\C$-algebra isomorphism
\begin{equation}\label{eqn:Psi}
\Psi : \D_G(X) \overset{\scriptscriptstyle\sim\,}{\longrightarrow} S(\jj_{\C})^W,
\end{equation}
where $S(\jj_{\C})^W$ is the $\C$-algebra of $W$-invariant elements in the symmetric algebra $S(\jj_{\C})$ for some subspace $\jj_{\C}$ of a Cartan subalgebra of~$\g_{\C}$ and some finite reflection group $W$ acting on~$\jj_{\C}$.
The integer
$$r := \dim_{\C} \jj_{\C}$$
is called the \emph{rank} of $G/H$.

As a particular case, suppose that $X$ is a reductive symmetric space, defined by an involutive automorphism $\sigma$ of~$G$.
Let $\g=\h+\q$ be the decomposition of~$\g$ into eigenspaces of~$\dd\sigma$, with respective eigenvalues $+1$ and~$-1$.
Then in \eqref{eqn:Psi} we can take $\jj$ to be a maximal semisimple abelian subspace of~$\q$, and $W$ to be the Weyl group of the restricted root system $\Sigma(\g_{\C},\jj_{\C})$ of~$\jj_{\C}$ in~$\g_{\C}$.
In particular, $\D_G(X)$ is a polynomial algebra in $r$ generators.
The isomorphism \eqref{eqn:Psi} is known as the \emph{Harish-Chandra isomorphism}.

\section{Pseudo-Riemannian structure on~$X$} \label{subsec:pseudo-Riem-struct}

We recall the following classical fact on reductive homogeneous spaces $X=G/H$, for which we give a proof for the reader's convenience.

\begin{lemma} \label{lem:pseudo-Riem-struct}
Let $X=G/H$ be a reductive homogeneous space.
There exists a $G$-invariant pseudo-Riemannian structure $g_X$ on~$X$.
\end{lemma}

\begin{proof}
For a real vector space~$V$, we denote by $\Symm(V)$ the set of symmetric bilinear forms on~$V$, and by $\Symm(V)_{\mathrm{reg}}$ the set of nondegenerate ones.
If a group $H$ acts linearly on~$V$, then it also acts linearly on $\Symm(V)$, leaving $\Symm(V)_{\mathrm{reg}}$ invariant.
We take $V$ to be $\g/\h$, on which $H$ acts via the adjoint representation.
Then there is a natural bijection between $\Symm(\g/\h)_{\mathrm{reg}}^H$ and the set of $G$-invariant pseudo-Riemannian structures on~$X$, by the $G$-translation of an $H$-invariant nondegenerate symmetric bilinear form on $\g/\h\simeq T_{eH}X$.
Thus it is sufficient to see that $\Symm(\g/\h)_{\mathrm{reg}}^H$ is nonempty when $G$ and~$H$ are reductive.

By a theorem of Mostow \cite{mos55} and Karpelevich \cite{kar53}, there exists a Cartan involution $\theta$ of~$G$ that leaves $H$ stable.
Let $B$ be a $G$-invariant, nondegenerate, symmetric bilinear form on~$\g$ which is positive definite on $\p:=\g^{-\dd\theta}$, negative definite on $\kk:=\g^{\dd\theta}$, and for which $\p$ and~$\kk$ are orthogonal.
If $G$ is semisimple, we can take $B$ to be the Killing form of~$\g$.
The restriction of $B$ to~$\h$ is nondegenerate because $\h=(\h\cap\kk)+(\h\cap\p)$, and so $B$ induces an $H$-invariant, nondegenerate, symmetric bilinear form on $\g/\h$, \ie an element of $\Symm(\g/\h)_{\mathrm{reg}}^H$.
\end{proof}

The pseudo-Riemannian structure $g_X$ in Lemma~\ref{lem:pseudo-Riem-struct} determines the Laplacian~$\square_X$.
Let $\q$ be the orthogonal complement of $\h$ in~$\g$ with respect to~$B$.
Then $\q = (\q\cap\kk) + (\q\cap\p)$, and the pseudo-Riemannian structure has signature $(\dim(\q\cap\p),\dim(\q\cap\kk))$.

\begin{examples} \label{ex:pseudo-Riem-struct}
\begin{enumerate}
  \item If $X=G/H$ is a reductive symmetric space, then $\square_X\in\dd\ell(Z(\g))=\dd r(Z(\g))$.
  If moreover $X$ is irreducible, then the $G$-invariant pseudo-Riemannian structure on~$X$ is unique up to scale, and induced by the Killing form of~$\g$.
  \item For $(G,H)=(\SO(4,4)_0,\mathrm{Spin}(4,3))$ or $(\SO(4,3)_0,G_{2(2)})$, the homogeneous space $X=G/H$ is not a symmetric space.
  However, the $G$-invariant pseudo-Riemannian structure on~$X$ is still unique up to scale, and induced by the Killing form of~$\g$, because the representation of $H$ on $\g/\h$ is irreducible.
  \item If $X\!=\!G/H$ is not a symmetric space, then\,the $G$-invariant\,pseudo-Riemannian structure on~$X$ may not be unique (even if $G$ is simple) and $\square_X$ may not be contained in $\dd r(Z(\g))$.
  For instance, let $G$ be $\SL(3,\R)$ and let $H$ be the subgroup of~$G$ consisting of diagonal matrices.
  Then $\Symm(\g/\h)^H_{\mathrm{reg}}\simeq (\R^{\ast})^3$, giving rise to a $3$-parameter family of $G$-invariant pseudo-Riemannian structures on~$X$.
  On the other hand, there are only $2$ parameters worth of differential operators of order $\leq 2$ in $\dd r(Z(\g))$.
\end{enumerate}
\end{examples}

\section{Joint eigenfunctions for $\D_G(X)$ on quotient manifolds~$X_{\Gamma}$}\label{subsec:remind-disc-spec}

Let $\Gamma$ be a discrete subgroup of~$G$ acting properly discontinuously and freely on the reductive homogeneous space $X=G/H$.
Then $X_{\Gamma}=\Gamma\backslash X$ is a manifold with a covering $p_{\Gamma} : X\to X_{\Gamma}$, and any $D\in\D_G(X)$ induces a differential operator~$D_{\Gamma}$ on~$X_{\Gamma}$ satisfying \eqref{eqn:D-gamma}.
Let $\F=L^2$ (\resp $C^{\infty}$, \resp~$\DD'$).
For any $\C$-algebra homomorphism $\lambda : \D_G(X)\rightarrow\C$, we denote by $\F(X_{\Gamma};\M_{\lambda})$ the set of square-integrable (\resp smooth, \resp distribution) weak solutions on~$X_{\Gamma}$ to the system
$$D_{\Gamma} f = \lambda(D)f \quad\quad\mathrm{for\ all}\ D\in\D_G(X)
\eqno{(\M_{\lambda})}.$$

If $X_{\C}$ is $G_{\C}$-spherical, then $\D_G(X)$ is commutative (Fact~\ref{fact:spherical}) and we may use the spaces $\F(X_{\Gamma};\M_{\lambda})$ of joint eigenfunctions of $\D_G(X)$ to expand functions on~$X_{\Gamma}$.
Through the isomorphism $\Psi : \D_G(X)\overset{\scriptscriptstyle\sim\,}{\rightarrow} S(\jj_{\C})^W$ of \eqref{eqn:Psi}, we shall identify the space $\Hom_{\C\text{-}\mathrm{alg}}(\D_G(X),\C)$ of $\C$-algebra homomorphisms from $\D_G(X)$ to~$\C$ with $\jj_{\C}^{\ast}/W$.
We set
$$\Spec_d(X_{\Gamma}) = \big\{ \lambda\in\jj_{\C}^{\ast}/W :\, L^2(X_{\Gamma};\M_{\lambda})\neq\{ 0\} \big\} .$$

Suppose that $X$ is a reductive symmetric space, defined by an involutive automorphism~$\sigma$.
Let $\theta$ be a Cartan involution of~$G$ commuting with~$\sigma$, and let $B$ be a $G$-invariant, nondegenerate, symmetric bilinear form on~$\g$ which is positive definite on $\p:=\g^{-\dd\theta}$, negative definite on $\kk:=\g^{\dd\theta}$, and for which $\p$ and~$\kk$ are orthogonal.
If $G$ is semisimple, we can take $B$ to be the Killing form of~$\g$.
The restriction of $B$ to~$\h$ is nondegenerate because $\h=(\h\cap\kk)+(\h\cap\p)$, and so $B$ induces an $H$-invariant, nondegenerate, symmetric bilinear form on $\g/\h\simeq T_{eH}X$, which we extend to a $G$-invariant pseudo-Riemannian structure on~$X$.
In turn, this determines a Laplacian $\square_X$ as in \eqref{eqn:defLaplacian}.
The symmetric bilinear form~$B$ also induces the Casimir element $C_G\in Z(\g_{\C})$ for symmetric~$X$, and its image is the Laplacian~$\square_X$.

We now assume that $\jj$ is $\theta$-stable.
Then $B$ induces a nondegenerate $W$-invariant bilinear form $\langle\cdot,\cdot\rangle$ on~$\jj^{\ast}$, which we extend to a complex bilinear form $\langle\cdot,\cdot\rangle$ on~$\jj_{\C}^{\ast}$.
Let $\Sigma^+(\g_{\C},\jj_{\C})$ be a positive system and let $\rho\in\jj_{\C}^{\ast}$ be half the sum of the elements of $\Sigma^+(\g_{\C},\jj_{\C})$, counted with root multiplicities.
The following remark is a consequence of the description of the Harish-Chandra isomorphism $\Psi$ in Section~\ref{subsec:DGH}.

\begin{remark}\label{rem:eigenv-Lapl}
If $X$ is a reductive symmetric space, then for any $\lambda\in\linebreak\Spec_d(X_{\Gamma})$ the Laplacian $\square_{X_{\Gamma}}$ acts on $L^2(X_{\Gamma};\M_{\lambda})$ as the scalar $t_{\lambda}:=\langle\lambda,\lambda\rangle -\langle\rho,\rho\rangle\in\C$:
$$L^2(X_{\Gamma};\M_{\lambda}) \subset \mathrm{Ker}(\square_{X_{\Gamma}} - t_{\lambda}),$$
where the kernel $\mathrm{Ker}(\square_{X_{\Gamma}} - t_{\lambda})$ is understood as weak solutions in~$L^2$.
\end{remark}

\begin{remark}\label{rem:joint-eigenv-Lapl}
When $\rank G/H>1$, it may happen that some complex number $t$ is equal to~$t_{\lambda}$ for more than one $\lambda\in\Spec_d(X_{\Gamma})$: by \cite{kkI}, this is the case for infinitely many $t\in\C$ if $\rank G/H=\rank K/H\cap K>1$ and if $\Gamma$ is \emph{sharp} in the sense of \cite[Def.\,4.2]{kk16} (for instance if $X_{\Gamma}$ is standard in the sense of Section~\ref{subsec:intro-stand}, as in the setting of the present paper).
\end{remark}

\section{Discrete series representations for~$X$}\label{subsec:discreteseries}

Let $L^2(X)$ be the space of square-integrable functions on~$X$ with respect to the natural $G$-invariant measure.
Recall that an irreducible unitary representation~$\pi$ of~$G$ is called a \emph{discrete series representation} for the reductive homogeneous space $X=G/H$ if there exists a nonzero continuous $G$-intertwining operator from $\pi$ to the regular representation of~$G$ on~$L^2(X)$ or, equivalently, if $\pi$ can be realized as a closed $G$-invariant subspace of $L^2(X)$.
In the case that $X$ is a group manifold, \ie $X=({}^{\backprime}G\times\!{}^{\backprime}G)/\Diag({}^{\backprime}G)$ for some reductive group~${}^{\backprime}G$, the discrete series representations for~$X$ were classified by Harish-Chandra; in general, discrete series representations for $X=G/H$ are different from Harish-Chandra's discrete series representations for~$G$ since $L^2(X)\neq L^2(G)^H$ for noncompact~$H$.

Suppose that $X$ is a reductive symmetric space.
For $\lambda\in\Spec_d(X)$, the space $L^2(X;\M_{\lambda})$ of Section~\ref{subsec:remind-disc-spec} is preserved by the left regular representation of~$G$ and splits into a finite direct sum of discrete series representations for~$X$.
If the natural homomorphism $Z(\g_{\C})\rightarrow\D_G(X)$ is surjective (e.g.\ if $G$ is classical or $X$ is a group manifold), then any discrete series representation for~$X$ is contained in $L^2(X;\M_{\lambda})$ for some $\lambda\in\Spec_d(X)$, by Schur's lemma.
In general, Flensted-Jensen \cite{fle80} and Matsuki--Oshima \cite{mo84} proved that there exist discrete series representations for~$X$, or equivalently $\Spec_d(X)$ is nonempty, if and only if the rank condition \eqref{eqn:rank} is satisfied.
For $X=({}^{\backprime}G\times\!{}^{\backprime}G)/\Diag({}^{\backprime}G)$, this condition is equivalent to Harish-Chandra's rank condition $\rank{}^{\backprime}G=\rank{}^{\backprime}K$ where ${}^{\backprime}K$ is a maximal compact subgroup of~${}^{\backprime}G$.

In the case that $X=G/H$ is a \emph{compact} reductive homogeneous space, not necessarily symmetric, any irreducible representation of~$G$ occurring in $C^{\infty}(X)$ is automatically a discrete series representation for $X=G/H$, and the constant~$C$ in Fact~\ref{fact:spherical} is equal to one.
Thus Fact~\ref{fact:spherical} is refined as follows.

\begin{fact}\label{fact:spherical-cpt}
Let $X=G/H$ be a \emph{compact} reductive homogeneous space.
Then the following conditions are equivalent:
\begin{enumerate}[(i)]
  \item the complexification $X_{\C}=G_{\C}/H_{\C}$ is $G_{\C}$-spherical,
  \item the $\C$-algebra $\D_G(X)$ is commutative,
  \item the discrete series representations for $X=G/H$ have uniformly bounded multiplicities,
  \item $X=G/H$ is multiplicity-free (\ie all discrete series representations for $X=G/H$ occur exactly once).
\end{enumerate}
\end{fact}

For (iii)$\,\Leftrightarrow\,$(iv), see \cite{kra76}.
Note that the multiplicity-freeness in~(iv) does not hold in general for noncompact reductive groups~$G$.

\section{Regular representations on real spherical homogeneous spaces}

Real spherical homogeneous spaces are a class of spaces extending real forms of spherical complex homogeneous spaces (see Example~\ref{ex:sph}.(1)).
Recall that, for a complex reductive Lie algebra~$\g_{\C}$, a $\g_{\C}$-module $V$ is called \emph{$Z(\g_{\C})$-finite} if the annihilator $\mathrm{Ann}_{Z(\g_{\C})}(V)$ of $V$ in $Z(\g_{\C})$ has finite codimension in $Z(\g_{\C})$.
A $(\g_{\C},K)$-module is called $Z(\g_{\C})$-finite if the underlying $\g_{\C}$-module is $Z(\g_{\C})$-finite.
We shall use the following general lemma.

\begin{lemma}\label{lem:spher-finite-length}
Let $G$ be a real linear reductive Lie group, $H$ a closed subgroup, and $\V_{\tau}:=G\times_H V_{\tau}\rightarrow X$ the $G$-equivariant bundle over $X=G/H$ associated with a finite-dimensional representation $(\tau,V_{\tau})$ of~$H$.
\begin{enumerate}
  \item Any $Z(\g_{\C})$-finite $(\g,K)$-module in $\DD'(X,\V_{\tau})$ is contained in $\A(X,\V_{\tau})$.
  \item If $X$ is $G$-real spherical, then any $Z(\g_{\C})$-finite $(\g,K)$-module in $\DD'(X,\V_{\tau})$ is of finite length.
\end{enumerate}
\end{lemma}

\begin{proof}
\begin{enumerate}
  \item This is well known: since $\dd\ell(C_G)-2\dd r(C_K)$ is an elliptic operator, any generalized eigenfunction of $\dd\ell(C_G)-2\dd r(C_K)$ is real analytic by the elliptic regularity theorem (see \cite[Th.\,3.4.4]{kkk86} for instance).
  \item This was proved in \cite[Th.\,2.2]{ko13} under the slightly stronger assumption that $Z(\g_{\C})$ acts as scalars.
  Since the successive sequence of (hyperfunction-valued) boundary maps $\beta_{\mu}^i$ for $(i,\mu)$ in a poset in the proof (see \cite[p.\,931]{ko13}) is well defined for hyperfunctions (in particular smooth functions, distributions, etc.) that are annihilated by an ideal of finite codimension in $Z(\g_{\C})$ \cite{osh88b}, the proof goes similarly.\qedhere
\end{enumerate}
\end{proof}

\section[Smooth and distribution vectors of unitary representations]{Smooth and distribution vectors of unitary representations of~$G$} \label{subsec:H+-infty}

In the proof of Theorem~\ref{thm:Specd-lambda} (Chapter~\ref{sec:proof-Specd-lambda}), we shall need a generalized concept of matrix coefficient associated to distribution vectors of unitary representations of Lie groups, which we now summarize briefly.
See \cite{he14} and \cite[Vol.\,I, Ch.\,I]{wal88}.

Let $\pi$ be a continuous representation of a Lie group~$G$ on a Hilbert space~$\HHH$ with inner product $(\cdot,\cdot)_{\HHH}$ and norm $\Vert\cdot\Vert_{\HHH}$.
The space
$$\HHH^{\infty} := \{ v\in\HHH : \text{the map }G\ni g\mapsto\pi(g)v\in\HHH\text{ is }C^{\infty}\}$$
of smooth vectors of~$\HHH$ carries a Fr\'echet topology given by the seminorms $|v|_u:=\Vert\dd\pi(u)v\Vert_{\HHH}$ for $u\in U(\g_{\C})$.
The space $\HHH^{-\infty}$ of distribution vectors of $(\pi,\HHH)$ is defined to be the space of continuous, sesquilinear forms on~$\HHH^{\infty}$; it is isomorphic to the complex conjugate of the complex linear functional of~$\HHH^{\infty}$.
We write $\pi^{-\infty}$ for the natural representation of $G$ on $\HHH^{-\infty}$.

Suppose $(\pi,\HHH)$ is a unitary representation.
It can be seen as a subrepresentation of $(\pi^{-\infty},\HHH^{-\infty})$ by sending $v\in\HHH$ to $(v,\cdot)_{\HHH}\in\HHH^{-\infty}$.
The triple
\begin{equation} \label{eqn:Gelfand}
\HHH^{\infty} \subset \HHH \subset \HHH^{-\infty}
\end{equation}
is called the \emph{Gelfand triple} associated with the unitary representation $(\pi,\HHH)$.
We write $(\cdot,\cdot)$ for the pairing of $\HHH^{-\infty}$ and~$\HHH^{\infty}$; it coincides with the inner product $(\cdot,\cdot)_{\HHH}$ in restriction to $\HHH\times\HHH^{\infty}$.
For $u\in\HHH^{-\infty}$ and $v\in\HHH^{\infty}$ the \emph{matrix coefficient} associated with $u$ and~$v$ is the smooth function on~$G$ defined by
\begin{equation}\label{eqn:mat-coef}
g \longmapsto (\pi^{-\infty}(g^{-1})u,v) = (u,\pi(g)v).
\end{equation}

We now extend the notion of matrix coefficient to $\HHH^{-\infty}$.
For this, we observe that for $\varphi\in C_c^{\infty}(G)$ and $F\in\HHH^{-\infty}$ the G\r{a}rding vector
$$\pi^{-\infty}(\varphi) F := \int_G \varphi(g) \ \pi^{-\infty}(g) F \ \dd g$$
is a smooth vector of $(\pi,\HHH)$.
We also observe that if $F\in\HHH^{\infty}$, then the smooth function $g\mapsto (\pi^{-\infty}(g^{-1})u,F)$ may be regarded as a distribution on~$G$ by
\begin{eqnarray*}
\varphi & \longmapsto & \int_G \varphi(g) \, (\pi^{-\infty}(g^{-1})u,F) \, \dd g\\
& & = \left(\int_G \varphi(g) \, \pi^{-\infty}(g^{-1})u \, \dd g, F\right) = \left(u, \int_G \overline{\varphi(g)} \, \pi^{-\infty}(g) F \, \dd g\right) \in \C
\end{eqnarray*}
for $\varphi\in C^{\infty}_c(G)$.
We now define a sesquilinear map
\begin{equation}\label{eqn:matdist}
T : \HHH^{-\infty} \times \HHH^{-\infty} \longrightarrow \DD'(G)
\end{equation}
such that for any $u,F\in\HHH^{-\infty}$, the distribution $T(u,F)\in\DD'(G)$ is given by
$$\varphi \longmapsto \left(\int_G \varphi(g) \, \pi^{-\infty}(g^{-1})u \, \dd g, F\right) = \left(u, \int_G \overline{\varphi(g)} \, \pi^{-\infty}(g) F \, \dd g\right) \in \C$$
for $\varphi\in C_c^{\infty}(G)$.
Then
\begin{equation}\label{eqn:T-bi-equiv}
T(\pi^{-\infty}(g_1)u,\pi^{-\infty}(g_2)F) = T(u,F)(g_1^{-1}\cdot g_2)
\end{equation}
for all $g_1,g_2\in G$.
If $F\in\HHH^{\infty}$, then the distribution $T(u,F)$ identifies with a smooth function via the Haar measure, equal to the matrix coefficient \eqref{eqn:mat-coef} associated with $u$ and~$F$.

A distribution vector $u\in\HHH^{-\infty}$ is called \emph{cyclic} if any $v\in\HHH^{\infty}$ satisfying $(\pi^{-\infty}(g^{-1})u,v)=0$ for all $g\in G$ is zero.
If $u\in\HHH^{-\infty}$ is cyclic, then any $F\in\HHH^{-\infty}$ satisfying $T(\pi^{-\infty}(g^{-1})u,F)=0$ in $\DD'(G)$ is zero, because if $\pi(\overline{\varphi})F=0$ for all $\varphi\in C_c^{\infty}(G)$ then $F=0$.

\section[$(\g,K)$-modules of admissible representations]{$(\g,K)$-modules of admissible representations of real reductive Lie groups} \label{subsec:g-K-modules}

We now suppose that $G$ is a real linear reductive Lie group with maximal compact subgroup~$K$.
A continuous representation $\pi$ of~$G$ on a complete, locally convex topological vector space~$\HHH$ is called \emph{admissible} if\linebreak $\dim \Hom_K(\tau,\HHH) < +\infty$ for all $\tau\in\widehat{K}$.
Then the space
$$\HHH_K := \big\{ v\in\HHH : \dim_{\C} \C\text{-span}\{ \pi(k)v : k\in K\} < +\infty \big\}$$
of $K$-finite vectors is a dense subspace of~$\HHH$, and there is a natural $(\g,K)$-module structure on~$\HHH_K$, called the underlying \emph{$(\g,K)$-module} of~$\HHH$.
The following theorem of Harish-Chandra \cite{har53} creates a bridge between continuous representations and algebraic representations without any specific topology.

\begin{fact}\label{fact:gK}
Let $(\pi,\HHH)$ be a continuous admissible representation of~$G$ of finite length.
Then there is a lattice isomorphism between the closed invariant subspaces $V$ of~$\HHH$ and the $(\g,K)$-invariant subspaces $V_K$ of~$\HHH_K$, given by $V\leadsto V_K=V\cap\HHH_K$ and $V_K\leadsto V=\overline{V_K}$.
\end{fact}

\chapter{Discrete spectrum of type $\I$ and~$\II$} \label{sec:type-I-II}

In this chapter we consider joint $L^2$-eigenfunctions on quotient manifolds $X_{\Gamma}$, where $X=G/H$ is a reductive homogeneous space and $\Gamma$ a discrete~sub\-group of~$G$ acting properly discontinuously and freely on~$X$.
We assume the complexification $X_{\C}$ is $G_{\C}$-spherical (\eg $X$ is a reductive symmetric space).

Part of the discrete spectrum on the pseudo-Riemannian locally symmetric space~$X_{\Gamma}$ is built from discrete series representations for $X=G/H$ (``stable spectrum'' in \cite{kk11,kk16}, which neither varies nor disappears under small deformations of the discontinuous group~$\Gamma$).
However,\footnote{F 10/6: Rephrased, avoiding the word ``easy'': fine like this?} they may also exist another type of discrete spectrum, of a somewhat different nature.

\begin{example} \label{ex:negative-spec-AdS}
Let $X = (\SL(2,\R)\times\SL(2,\R))/\Diag(\SL(2,\R))$, which is isomorphic to the $3$-dimensional anti-de Sitter space $\AdS^3 = \SO(2,2)/\SO(2,1)$ of Example~\ref{ex:AdS-odd}.
By \cite[Th.\,9.9]{kk16}, there is a constant $R_X>0$ such that for any uniform lattice ${}^{\backprime}\Gamma$ of $\SL(2,\R)$ with $-I\notin{}^{\backprime}\Gamma$ (\resp with $-I\in{}^{\backprime}\Gamma$), if we set $\Gamma={}^{\backprime}\Gamma\times\{ e\}$, then the discrete spectrum of the Laplacian $\square_{X_{\Gamma}}$ on the Lorentzian $3$-manifold $X_{\Gamma}$ contains the infinite set
$$\Big\{ \frac{1}{4}k(k+2) : k\in\N,\ k\geq R_X\Big\} \quad \text{\Big(\resp }\Big\{ \frac{1}{4}k(k+2) : k\in 2\N,\ k\geq R_X\Big\}\text{\Big)}$$
coming from discrete series representations for $X=G/H$.
(See Section~\ref{subsec:AdS3} for normalization.)
However, we note that $L^2({}^{\backprime}\Gamma\backslash\HH^2)$ embeds into $L^2(X_{\Gamma})=L^2({}^{\backprime}\Gamma\backslash\SL(2,\R))$ and the restriction to $L^2({}^{\backprime}\Gamma\backslash\HH^2)$ of the Laplacian $\square_{X_{\Gamma}}$ corresponds to $-2$ times the usual Laplacian $\Delta_{\,{}^{\backprime}\Gamma\backslash\HH^2}$ on the hyperbolic surface ${}^{\backprime}\Gamma\backslash\HH^2$ (see \cite[Ch.\,X]{lan85}).
Therefore $\square_{X_{\Gamma}}$ is essentially self-adjoint and also admits infinitely many negative eigenvalues coming from eigenvalues of~$\Delta_{\,{}^{\backprime}\Gamma\backslash\HH^2}$.
These eigenvalues vary under small deformations of ${}^{\backprime}\Gamma$ inside $\SL(2,\R)$ (see \cite[Th.\,5.14]{wol94}).
\end{example}

Motivated by this observation, we now introduce a definition of discrete spectrum (and joint $L^2$-eigenfunctions) of type~$\I$ and type~$\II$ on~$X_{\Gamma}$.

\begin{remark}
In Section~\ref{subsec:compatibility-I-II}, for standard $\Gamma\subset L$, we shall introduce a similar notion of type~$\I$ and type~$\II$ for Hermitian vector bundles $\V_{\tau}$ over the Riemannian locally symmetric space $Y_{\Gamma}=\Gamma\backslash L/L_K$, by using Harish-Chandra's discrete series representations for~$L$.
We shall prove that the maps $\ii_{\tau,\Gamma}$ of Section~\ref{subsec:intro-main-result} preserve type~$\I$ and type~$\II$ when $X_{\C}$ is $L_{\C}$-spherical (Theorem~\ref{thm:transfer-spec}).
\end{remark}

\section{Definition of type~$\I$ and type~$\II$}\label{subsec:def-type-I-II}

We introduce Hilbert space decompositions
\begin{eqnarray}\label{eqn:L2decompI&II}
L^2(X_{\Gamma}) & = & L^2_d(X_{\Gamma}) \oplus L^2_{ac}(X_{\Gamma})\\
& = & \big(L^2_d(X_{\Gamma})_{\I} \oplus L^2_d(X_{\Gamma})_{\II}\big) \oplus L^2_{ac}(X_{\Gamma}),\nonumber
\end{eqnarray}
defined as follows.

Recall (Corollary~\ref{cor:L2d(XGamma,Mlambda)-orth}) that in the main setting~\ref{spher-setting} of the theorems of Chapters \ref{sec:intro} and~\ref{sec:method}, the subspaces $L^2(X_{\Gamma};\M_{\lambda})$, for varying $\lambda\in\Spec_d(X_{\Gamma})$, are orthogonal to each other inside $L^2(X_{\Gamma})$.
We set
$$L^2_d(X_{\Gamma}) \,:=\quad\! \sumplus{\lambda\in\Spec_d(X_{\Gamma})}\ L^2(X_{\Gamma};\M_{\lambda}),$$
where $\sum^{\oplus}$ denotes the Hilbert completion of the algebraic direct sum, and let $L^2_{ac}(X_{\Gamma})$ be the orthogonal complement of $L^2_d(X_{\Gamma})$ in $L^2(X_{\Gamma})$.

Next we introduce, for any $\lambda\in\Spec_d(X_{\Gamma})$, two subspaces $L^2(X_{\Gamma};\M_{\lambda})_{\I}$ and $L^2(X_{\Gamma};\M_{\lambda})_{\II}$ of $L^2(X_{\Gamma};\M_{\lambda})$.
For this we first specify the topology that we consider on the space $\DD'(X)$ of distributions on~$X$.

\begin{remark} \label{rem:top-distrib}
Let $C_c^{\infty}(X)$ be the space of compactly supported, smooth functions on~$X$, with the locally convex inductive limit topology.
Recall that a subset $B$ of a topological vector space is said to be (von Neumann) bounded if for every neighborhood $U$ of the origin, there exists a scalar $s>0$ such that $B \subset s U$.
We endow $\DD'(X)$ with the topology of uniform convergence on all bounded subsets $B$ of $C_c^{\infty}(X)$, \ie the topology defined by the family of seminorms
$$|\cdot|_B := \big( u \longmapsto \sup \{ \langle u,\varphi\rangle : \varphi\in B\} \big).$$
With this topology, $\DD'(X)$ is a complete, locally convex (though not metrizable) topological space and the map $p_{\Gamma}^{\ast} : L^2(X_{\Gamma})\to\DD'(X)$ induced by the projection $p_{\Gamma} : X\rightarrow X_{\Gamma}$ is continuous.
\end{remark}

\begin{notation} \label{def:type-I-II}
For any $\lambda\in\nolinebreak\Spec_d(X_{\Gamma})$, we let $L^2(X_{\Gamma};\M_{\lambda})_{\I}$ be the preimage, under~$p_{\Gamma}^{\ast}$, of the closure in $\DD'(X)$ of $L^2(X;\M_{\lambda})$, and $L^2(X_{\Gamma};\M_{\lambda})_{\II}$ be the orthogonal complement of $L^2(X_{\Gamma};\M_{\lambda})_{\I}$ in $L^2(X_{\Gamma};\M_{\lambda})$.
\end{notation}

This orthogonal complement is well defined by the following elementary\,observation.

\begin{lemma}\label{lem:L2Iclosed}
For any $\lambda\in\Spec_d(X_{\Gamma})$,
\begin{enumerate}
  \item $L^2(X_{\Gamma};\M_{\lambda})$ is a closed subspace of the Hilbert space~$L^2(X_{\Gamma})$,
  \item $L^2(X_{\Gamma};\M_{\lambda})_{\I}$ is a closed subspace of $L^2(X_{\Gamma};\M_{\lambda})$.
\end{enumerate}
\end{lemma}

\begin{proof}
All elements in $\D_G(X)$ are closed operators, which implies~(1).
The map $p_{\Gamma}^{\ast} : L^2(X_{\Gamma})\hookrightarrow\DD'(X)$ is continuous, and so $L^2(X_{\Gamma};\M_{\lambda})$ is a closed subspace of $L^2(X;\M_{\lambda})$.
Since $L^2(X;\M_{\lambda})$ is contained in $\DD'(X;\M_{\lambda})$, so is its closure $\overline{L^2(X;\M_{\lambda})}$, hence $(p_{\Gamma}^*)^{-1}(\overline{L^2(X;\M_{\lambda})})\subset L^2(X_{\Gamma};\M_{\lambda})$, which implies~(2).
\end{proof}

Thus for any $\lambda\in\Spec_d(X_{\Gamma})$ we have the Hilbert space decomposition
$$L^2(X_{\Gamma};\M_{\lambda}) = L^2(X_{\Gamma};\M_{\lambda})_{\I} \oplus L^2(X_{\Gamma};\M_{\lambda})_{\II}\,.$$
For $i=\I,\II$, we then set
\begin{equation} \label{eqn:L2-d-I-II}
L^2_d(X_{\Gamma})_i \,:=\quad\! \sumplus{\lambda\in\Spec_d(X_{\Gamma})}\ L^2(X_{\Gamma};\M_{\lambda})_i,
\end{equation}
where $\sum^{\oplus}$ denotes again the Hilbert completion of the algebraic direct sum.
We also set
\begin{eqnarray*}
\Spec_d(X_{\Gamma})_{\I} & := & \{ \lambda\in\Spec_d(X_{\Gamma}) : L^2(X_{\Gamma};\M_{\lambda})_{\I}\neq\{ 0\} \} ,\\
\Spec_d(X_{\Gamma})_{\II} & := & \{ \lambda\in\Spec_d(X_{\Gamma}) : L^2(X_{\Gamma};\M_{\lambda})_{\II}\neq\{ 0\} \} .
\end{eqnarray*}
Then
$$\Spec_d(X_{\Gamma}) = \Spec_d(X_{\Gamma})_{\I} \cup \Spec_d(X_{\Gamma})_{\II}.$$
This union is not disjoint a priori.
By construction, $\Spec_d(X_{\Gamma})_{\I}\subset\Spec_d(X)$.

For instance, for $X = G/H = (\SL(2,\R)\times\SL(2,\R))/\Diag(\SL(2,\R))$, the discrete spectrum $\Spec_d(X_{\Gamma})$ identifies with the discrete spectrum of the Laplacian $\square_{X_{\Gamma}}$; the positive (\resp negative) spectrum described in Example~\ref{ex:negative-spec-AdS} is of type~$\I$ (\resp type~$\II$); we refer to Section~\ref{subsec:AdS3} for more details on this example.

\begin{remarks} \label{rem:type-I-II}
(1) Take $\Gamma=\{ e\} $.
If the complexification $X_{\C}$ is $G_{\C}$-spherical, then $L^2_d(X)_{\II}=\{0\}$.
In fact, $L^2(X;\M_{\lambda})$ is a unitary representation of finite length by \cite{ko13}, hence it splits into a direct sum of irreducible unitary representations of~$G$.
If $X$ is a reductive symmetric space, an explicit decomposition of $L^2_{ac}(X)$ is given in \cite{del98} as a direct integral of irreducible unitary representations of~$G$ with continuous parameter.

(2) In \cite{kk16}, building on \cite{fle80}, we constructed discrete spectrum of type~$\I$ for $X_{\Gamma}$ when $X=G/H$ is a reductive symmetric space satisfying the rank condition $\rank G/H=\rank K/K\cap H$ and the action of $\Gamma$ on~$X$ is \emph{sharp} (a strong form of proper discontinuity, satisfied in many examples, see \cite[Def.\,4.2]{kk16}).

(3) If $X=G/H$ is a reductive symmetric space that does \emph{not} satisfy the rank condition above (\eg if $X=G/H=G/K$ is a Riemannian symmetric space), then $G/H$ does not admit discrete series \cite{mo84}, and so $L^2_d(X_{\Gamma})_{\I}=\nolinebreak\{0\}$ and $\Spec_d(X_{\Gamma})=\Spec_d(X_{\Gamma})_{\II}$.
\end{remarks}

\begin{lemma}\label{lem:type-I-positive}
If $X$ is a reductive symmetric space and $\Gamma$ a discrete subgroup of~$G$ acting properly discontinuously and freely on~$X$, then all the eigenvalues of the Laplacian $\square_{X_{\Gamma}}$ corresponding to $\Spec_d(X_{\Gamma})_{\I}$ are positive, except possibly for a finite number.
\end{lemma}

\begin{proof}
Let $\jj_{\R}^{\ast}$ be the $\R$-span of $\Sigma(\g_{\C},\jj_{\C})$: it is a real subspace of~$\jj_{\C}^{\ast}$, invariant under the Weyl group~$W$.
By the classification of discrete series representations for~$X$ by Matsuki--Oshima \cite{mo84}, we have $\Spec_d(X)\subset\jj_{\R}^{\ast}/W$, and $\Spec_d(X)$ is discrete in $\jj_{\R}^{\ast}/W$.
We then use the inclusion $\Spec_d(X_{\Gamma})_{\I}\subset\Spec_d(X)$ and the fact that the Laplacian $\square_{X_{\Gamma}}$ acts by the scalar $\langle\lambda,\lambda\rangle - \langle\rho,\rho\rangle$ on $L^2(X_{\Gamma};\M_{\lambda})$ for $\lambda\in\jj_{\C}^{\ast}/W$ (Remark~\ref{rem:eigenv-Lapl}).
\end{proof}

\section{Definition of type~$\I$ and type~$\II$ using $Z(\g_{\C})$ instead of $\D_G(X)$}\label{subsec:type-I-II-Zg}

In Chapter~\ref{sec:conj}, instead of using systems $(\M_{\lambda})$ for the $\C$-algebra $\D_G(X)$, we shall consider systems $(\NN_{\nu})$ for the center $Z(\g_{\C})$ of the enveloping algebra $U(\g_{\C})$, as follows.

Let $\F=\A$, $C^{\infty}$, $L^2$, or $\DD'$.
For any $\nu\in\Hom_{\C\text{-}\mathrm{alg}}(Z(\g_{\C}),\C)$, let $\F(X_{\Gamma};\NN_{\nu})$ be the space of (weak) solutions $\varphi\in\F(X_{\Gamma})$ to the system
$$\dd\ell(z)_{\Gamma}\,\varphi = \nu(z) \varphi \quad\quad\mathrm{for\ all}\ z\in Z(\g_{\C}). \eqno{(\NN_{\nu})}$$
Via the $\C$-algebra homomorphism $\dd\ell : Z(\g_{\C})\to\D_G(X)$, we have a natural inclusion
$$\F(X_{\Gamma};\M_{\lambda}) \longhookrightarrow \F(X_{\Gamma};\NN_{\lambda\circ\dd\ell})$$
which is bijective as soon as $\dd\ell$ is surjective.
In the general setting~\ref{gen-setting}, for $\Gamma\subset L$, we sometimes also use a similar map $\dd\ell : Z(\llll_{\C})\to\D_L(X)$.
We note that $\dd\ell : Z(\g_{\C})\to\D_G(X)$ is always surjective when $X_{\C}$ is $L_{\C}$-spherical and $G$ simple, whereas $\dd\ell : Z(\llll_{\C})\to\D_L(X)$ is not surjective in most cases.

Similarly to Notation~\ref{def:type-I-II}, we divide joint eigenfunctions on~$X_{\Gamma}$ into two types: type~$\I$ coming from discrete series representations for $G/H$, and type~$\II$ defined by taking an orthogonal complement.

\begin{notation} \label{def:Ntype-I-II}
For any $\nu\in\Hom_{\C\text{-}\mathrm{alg}}(Z(\g_{\C}),\C)$, we let $L^2(X_{\Gamma};\NN_{\nu})_{\I}$ be the preimage, under~$p_{\Gamma}^{\ast}$, of the closure of $L^2(X;\NN_{\nu})$ in $\DD'(X)$, and $L^2(X_{\Gamma};\NN_{\nu})_{\II}$ be the orthogonal complement of $L^2(X_{\Gamma};\NN_{\nu})_{\I}$ in $L^2(X_{\Gamma};\NN_{\nu})$.
For $i=\I$ or~$\II$, we set
$$\Spec_d^{Z(\g_{\C})}(X_{\Gamma})_i := \big\{ \nu\in\Hom_{\C\text{-}\mathrm{alg}}(Z(\g_{\C}),\C) : L^2(X_{\Gamma};\NN_{\nu})_i \neq \{ 0\}\big\}.$$
We also set
\begin{align*}
\Spec_d^{Z(\g_{\C})}(X_{\Gamma}) & := \big\{ \nu\in\Hom_{\C\text{-}\mathrm{alg}}(Z(\g_{\C}),\C) : L^2(X_{\Gamma};\NN_{\nu}) \neq \{ 0\}\big\}\\
& \ = \, \Spec_d^{Z(\g_{\C})}(X_{\Gamma})_{\I} \cup \Spec_d^{Z(\g_{\C})}(X_{\Gamma})_{\II},
\end{align*}
and
$$\Spec^{Z(\g_{\C})}(X_{\Gamma}) := \big\{ \nu\in\Hom_{\C\text{-}\mathrm{alg}}(Z(\g_{\C}),\C) : \DD'(X_{\Gamma};\NN_{\nu}) \neq \{ 0\}\big\}.$$
\end{notation}

Note that here we consider more general eigenfunctions which are not necessarily square-integrable.
Eigenfunctions are not automatically smooth, which is why we use a formulation with the space $\DD'$ of distributions.

In Chapter~\ref{sec:conj} we shall give constraints on $\Spec^{Z(\g_{\C})}(X_{\Gamma})$ (Proposition~\ref{prop:SpecZg}) and conjectural constraints on $\Spec_d^{Z(\g_{\C})}(X_{\Gamma})$ and $\Spec_d^{Z(\g_{\C})}(X_{\Gamma})_i$ for $i=\I,\II$ (Conjecture~\ref{conj:Specd-Zg}).

\chapter{Differential operators coming from~$L$ and from the fiber~$F$}\label{sec:dliota}

Our strategy for spectral analysis on a standard locally homogeneous space $X_{\Gamma} = \Gamma\backslash G/H$, where $\Gamma$ is contained in a reductive subgroup $L$ of~$G$ as in the general setting~\ref{gen-setting}, is to use the representation theory of~$L$, and for this it is desirable to have a control on:
\begin{itemize}
  \item the $G$-modules generated by irreducible $L$-modules in $C^{\infty}(X)$;
  \item the $L$-module structure of irreducible $G$-modules in $C^{\infty}(X)$.
\end{itemize}
In this chapter, we introduce two conditions (A) and~(B) to formulate these two types of control (Definition~\ref{def:cond-A-B}).
They use the $L$-equivariant fiber bundle structure $X=G/H\simeq L/L_H\to L/L_K = Y$ of \eqref{eqn:bundleXY}.

In order to verify conditions (A) and~(B), as well as the existence of \emph{transfer maps} $\nnu$ and~$\llambda$ as in \eqref{eqn:nu-lambda-tau} (condition (Tf), see Definition~\ref{def:cond-Tf} below), we consider two additional conditions ($\widetilde{\mathrm{A}}$) and~($\widetilde{\mathrm{B}}$) involving subalgebras $\dd\ell(Z(\llll_{\C}))$, $\D_G(X)$, and $\dd r(Z(\llll_{\C}\cap\kk_{\C}))$ of the algebra $\D_L(X)$ of $L$-invariant differential operators on~$X$, as in \cite[\S\,1.4]{kkdiffop}.
These conditions ($\widetilde{\mathrm{A}}$) and~($\widetilde{\mathrm{B}}$) indicate that the contribution of the fiber $F=L_K/L_H$ to $\D_L(X)$ is ``large''.
We observe (Lemma~\ref{lem:A-tilA-B-tilB}) that ($\widetilde{\mathrm{A}}$) implies~(A), that ($\widetilde{\mathrm{B}}$) implies~(B) whenever $X_{\C}$ is $G_{\C}$-spherical, and that ($\widetilde{\mathrm{A}}$) and~($\widetilde{\mathrm{B}}$) together imply (Tf) (existence of transfer maps).

Whereas conditions (A), (B), and (Tf) are defined using real forms, conditions ($\widetilde{\mathrm{A}}$) and~($\widetilde{\mathrm{B}}$) are formulated simply in terms of \emph{complex} Lie algebras; in particular, it is sufficient to check them on one real form in order for them to hold on any other real form.
This allows us to use \cite[Cor.\,1.12]{kkdiffop}, which concerns compact real forms, to prove that condition (Tf) is satisfied in the setting~\ref{spher-setting} (Proposition~\ref{prop:cond-Tf-A-B-satisfied}).

We shall discuss applications of conditions (A) and~(B) to the relation between spectral theory for the pseudo-Riemannian space~$X$ and the Riemannian symmetric space~$Y$ in Chapter~\ref{sec:transfer}, and their quotients $X_{\Gamma}$ and~$Y_{\Gamma}$ in Chapter~\ref{sec:transfer-Gamma-type-I-II}.
Conditions (A), (B), and (Tf) will play a crucial role in the proof of our main theorems.

\section{$L$-invariant differential operators on~$X$}\label{subsec:DLX}

We work in the general setting~\ref{gen-setting}.
As in \eqref{eqn:dl-dr}, the differentiations of the left and right regular representations of $L$ on $C^{\infty}(L)$ induce a $\C$-algebra homomorphism
\begin{equation}\label{eqn:dl-dr-L}
\dd\ell \otimes \dd r : U(\llll_{\C}) \otimes U(\llll_{\C})^{L_H} \longrightarrow \D(X),
\end{equation}
where $U(\llll_{\C})^{L_H}$ is the subalgebra of $L_H$-invariant elements in the enveloping algebra $U(\llll_{\C})$, and $\D(X)$ the full $\C$-algebra of differential operators on~$X$.
In particular, $\dd\ell$ is defined on the center $Z(\llll_{\C})$ of the enveloping algebra $U(\llll_{\C})$, and $\dd\ell(Z(\llll_{\C}))=\dd r(Z(\llll_{\C}))$.

The Casimir element of~$\llll$ gives rise to the Laplacian $\Delta_Y$ on the Riemannian symmetric space~$Y$.
More precisely, choose any $\Ad(L)$-invariant bilinear form on~$\llll$; this defines a Riemannian structure on~$Y$ and the Casimir element $C_L\in Z(\llll_{\C})$.
Then the following diagram commutes, where $q^{\ast}$ is the pull-back by the $L$-equivariant projection $q : X\rightarrow Y$.
$$\xymatrixcolsep{3pc}
\xymatrix{
C^{\infty}(X) \ar[r]^{\dd\ell(C_L)} & C^{\infty}(X)\\
C^{\infty}(Y) \ar[u]^{q^{\ast}} \ar[r]^{\Delta_Y} & C^{\infty}(Y) \ar[u]_{q^{\ast}}
}$$
More generally, any element $z\in Z(\llll_{\C})$ gives rise (similarly to \eqref{eqn:dl-dr-L}) to an $L$-invariant differential operator $\dd\ell_Y(z)$ on~$Y$ such that $\dd\ell(z)\circ q^{\ast}=q^{\ast}\circ\dd\ell_Y(z)$.

Since we have assumed $L$ to be connected, so is~$L_K$; therefore the adjoint action of $L_K$ on the center $Z(\llll_{\C}\cap\kk_{\C})$ is trivial, and $\dd r(Z(\llll_{\C}\cap\kk_{\C}))$ is well defined.
Geometrically, the algebra $\dd r(Z(\llll_{\C}\cap\kk_{\C}))$ corresponds to the algebra $\D_{L_K}(F)$ of $L_K$-invariant differential operators on the compact fiber $F=L_K/L_H$.
More precisely, similarly to \eqref{eqn:dl-dr-L}, we can define a map
$$\dd r_{\scriptscriptstyle F} : U(\llll_{\C}\cap\kk_{\C})^{L_H} \longrightarrow \D_{L_K}(F).$$
There is a natural injective homomorphism
$\iota : \D_{L_K}(F) \hookrightarrow \D_L(X)$
such that the following diagram commutes (see \cite[\S\,2.3]{kkdiffop}).
$$\xymatrixcolsep{3pc}
\xymatrix{
Z(\llll_{\C}\cap\kk_{\C}) \ar[d]^{\dd r_{\scriptscriptstyle F}} \ar@{}[r]|{\textstyle\subset} & U(\llll_{\C})^{L_H} \ar[d]^{\dd r}\\
\D_{L_K}(F) \ar@{^{(}->}[r]^{\iota} & \D_L(X)
}$$
When $X_{\C}$ is $L_{\C}$-spherical and $G$ simple, the map $\dd r_{\scriptscriptstyle F}$ is actually surjective (see \cite[Lem.\,2.6]{kkdiffop}), and so
\begin{equation} \label{eqn:dr-iota}
\dd r(Z(\llll_{\C}\cap\kk_{\C})) = \iota(\D_{L_K}(F)).
\end{equation}

\begin{remark}\label{rem:commutative}
If $X_{\C}$ is $L_{\C}$-spherical, then $\D_L(X)$ is commutative (Fact~\ref{fact:spherical}).
In particular,
\begin{itemize}
  \item the three subalgebras $\D_G(X)$, $\dd\ell(Z(\llll_{\C}))$, and $\dd r(Z(\llll_{\C}\cap\kk_{\C}))\!=\!\iota(\D_{L_K}\!(F))$ commute;
  \item $\D_G(X)$ and $\D_{L_K}(F)$ are commutative, and so $X_{\C}$ is $G_{\C}$-spherical and $F_{\C}=(L_{\C}\cap K_{\C})/(L_{\C}\cap H_{\C})$ is $(L_{\C}\cap K_{\C})$-spherical (Fact~\ref{fact:spherical}).
\end{itemize}
In fact, the following implications hold (see Example~\ref{ex:sph}).
\begin{changemargin}{-2cm}{0cm}
$\xymatrixcolsep{3pc}
\xymatrix{
\begin{minipage}{0.4\textwidth}
\text{\begin{tabular}{|c|}
\hline
{\txt{$F_{\C} \simeq (L_{\C}\cap K_{\C})/(L_{\C}\cap H_{\C})$\\ $(L_{\C}\cap K_{\C})$-spherical}}\\
\hline
\end{tabular}}
\end{minipage}
& \ar@{=>}[l]
\begin{minipage}{0.28\textwidth}
\text{\begin{tabular}{|c|}
\hline
{\txt{$X_{\C}\simeq L_{\C}/L_{\C}\cap H_{\C}$\\ $L_{\C}$-spherical}}\\
\hline
\end{tabular}}
\end{minipage}
\ar@{=>}[d] \ar@{=>}[r] &
\begin{minipage}{0.23\textwidth}
\text{\begin{tabular}{|c|}
\hline
{\txt{$X\simeq L/L_H$\\ $L$-real spherical}}\\
\hline
\end{tabular}}
\end{minipage}
\ar@{=>}[d]
\\
\begin{minipage}{0.25\textwidth}
\text{\begin{tabular}{|c|}
\hline
{\txt{$X=G/H$\\ symmetric space}}\\
\hline
\end{tabular}}
\end{minipage}
\ar@{=>}[r] &
\begin{minipage}{0.2\textwidth}
\text{\begin{tabular}{|c|}
\hline
{\txt{$X_{\C}=G_{\C}/H_{\C}$\\ $G_{\C}$-spherical}}\\
\hline
\end{tabular}}
\end{minipage}
\ar@{=>}[r] &
\begin{minipage}{0.23\textwidth}
\text{\begin{tabular}{|c|}
\hline
{\txt{$X=G/H$\\ $G$-real spherical}}\\
\hline
\end{tabular}}
\end{minipage}
}$
\end{changemargin}
\end{remark}

\begin{remark}
Using Lemma~\ref{lem:spher-finite-length}.(2), one can prove that if $X$ is $L$-real spherical, then $\dim\DD'(X_{\Gamma};\NN_{\nu})<+\infty$ for any $\nu\in\Hom_{\C\text{-}\mathrm{alg}}(Z(\llll_{\C}),\C)$ and any torsion-free cocompact discrete subgroup $\Gamma$ of~$L$.
\end{remark}

\section{Relations between Laplacians}

Recall that the Laplacian $\square_X$ is defined by the $G$-invariant pseudo-Riemannian structure of Lemma~\ref{lem:pseudo-Riem-struct} on $X=G/H$, and $\square_X\in\D_G(X)$.
In most cases the Laplacian $\square_X$ and the Casimir operator $C_L$ for the subalgebra $\llll$ are linearly independent, but the following proposition shows that the ``error term'' (after appropriate normalization) comes from the action of $\dd r(Z(\llll_{\C}\cap\kk_{\C}))$ on the fiber~$F$.

\begin{proposition} \label{prop:rel-Lapl}
In the general setting~\ref{gen-setting}, choose any $\Ad(L)$-invariant, nondegenerate, symmetric bilinear form on~$\llll$ and let $C_L\in Z(\llll)$ be the corresponding Casimir element.
If $X_{\C}=G_{\C}/H_{\C}$ is $G_{\C}$-spherical and $G$ simple, then there exists a nonzero $a\in\R$ such that
\begin{equation} \label{eqn:rel-Lapl}
\square_X \in a\,\dd\ell(C_L) + \dd r(Z(\llll_{\C}\cap\kk_{\C})).
\end{equation}
\end{proposition}

Even if $G$ is simple, $L$ need not be (see Table~\ref{table1}), and so the invariant bilinear form on~$\llll$ may not be unique.
For any choice of such form, Proposition~\ref{prop:rel-Lapl} holds for the corresponding Casimir element~$C_L$.

For simple~$G_{\C}$, Proposition~\ref{prop:rel-Lapl} is a consequence of \cite[Cor.\,1.7]{kkdiffop}; in each case of Table~\ref{table1}, the nonzero scalar $a\in\R$ is the one computed explicitly in \cite[\S\,6--7]{kkdiffop} for the corresponding compact real forms.

\begin{example} \label{ex:rel-Lapl-SO(2n,2)}
Let $G=\SO(2n,2)_0$.
If $(H,L)$ is either $(\SO(2n,1)_0,\U(n,1))$ (so that $X=G/H$ is the anti-de Sitter space $\AdS^{2n+1}$ of Example~\ref{ex:AdS-odd}) or $(\U(n,1),\SO(2n,1)_0)$, then $\dd\ell(C_G) = 2 \dd\ell(C_L) - \dd r(C_{L_K})$.
\end{example}

\begin{proof}[Proof of Proposition~\ref{prop:rel-Lapl} when $G_{\C}$ is simple]
Recall from Lem\-ma \ref{lem:pseudo-Riem-struct} and Example \ref{ex:pseudo-Riem-struct}.(1)--(2) that a $G$-invariant pseudo-Riemannian structure $g_X$ on $X=G/H$ is unique up to scale, and induced by the Killing form of~$\g$.
Let $C_G\in Z(\g)$ be the corresponding Casimir element.
Then $\square_X=\dd\ell(C_G)$.
Let $C_{G,\C}\in Z(\g_{\C})$ be the Casimir element of the complex simple Lie algebra~$\g_{\C}$, and $C_{L,\C}\in Z(\llll_{\C})$ the Casimir element of~$\llll_{\C}$ associated to the complex extension of the bilinear form on~$\llll$ defining~$C_L$.
Then $\dd\ell(C_{G,\C})$ and $\dd\ell(C_{L,\C})$ are holomorphic differential operators on the complex manifold $X_{\C}=G_{\C}/H_{\C}$, whose restrictions to the totally real submanifold $X=G/H$ are $\dd\ell(C_G)$ and $\dd\ell(C_L)$, respectively: see \cite[Lem.\,5.4]{kkdiffop}.
By \cite[Cor.\,1.7]{kkdiffop}, there is a nonzero $a\in\R$ such that
$$\dd\ell(C_{G,\C}) \in a\,\dd\ell(C_{L,\C}) + \dd r(Z(\llll_{\C}\cap\kk_{\C})).$$
We obtain \eqref{eqn:rel-Lapl} by restricting to $X=G/H$.
\end{proof}

By the classification of Table~\ref{table1}, the remaining case is $(G,H,L)=\linebreak (\SO(8,\C),\SO(7,\C),\mathrm{Spin}(7,1))$ up to covering.
In this case, where $G_{\C}$ is not simple because $G$ itself has a complex structure, there exist two linearly independent $G$-invariant second-order differential operators on $X=G/H$, and not all of their linear combinations are contained in the vector space $\C\,\dd\ell(C_L) + \dd r(Z(\llll_{\C}\cap\kk_{\C}))$ (see Remark~\ref{rem:necessary-cond}).
However, we now check that the Laplacian $\square_X$ with respect to the $G$-invariant pseudo-Riemannian structure on~$X$ (which is unique up to scale) belongs to this vector space.
We start with the following general lemma.

\begin{lemma} \label{lem:Casimir-complex}
Let $\g$ be the Lie algebra of a complex semisimple Lie group~$G$, and $\g\otimes_{\R}\C \simeq \g^{\mathrm{holo}} \oplus \g^{\mathrm{anti}}$ the decomposition corresponding to the sum
\begin{equation} \label{eqn:hol-decomp}
(T_eG)_{\C} \simeq T_e^{1,0}G \oplus T_e^{0,1}G
\end{equation}
of the holomorphic and anti-holomorphic tangent spaces at the origin.
We denote by $C_{G,\R}$ the Casimir element of~$\g$ (regarded as a real Lie algebra), and by $C^{\mathrm{holo}}$ and $C^{\mathrm{anti}}$ those of $\g^{\mathrm{holo}}$ and $\g^{\mathrm{anti}}$ (regarded as complex Lie algebras), respectively.
Then $C_{G,\R} = C^{\mathrm{holo}} + C^{\mathrm{anti}}$ in $Z(\g\otimes_{\R}\C)$.
\end{lemma}

\begin{proof}
Let $B : \g\times\g\to\C$ be the Killing form of~$\g$ as a complex Lie algebra, and $B_{\R} : \g\times\g\to\R$ the Killing form of~$\g$ as a real Lie algebra (forgetting the complex structure).
Then
$$B_{\R}(X,Y) = 2\,\mathrm{Re}\,B(X,Y)$$
for all $X,Y\in\g$.
Let $J$ be the complex structure of~$G$, and $\kk$ the Lie algebra of a maximal compact subgroup of~$G$.
The Cartan decomposition $\g=\kk+J\kk$ holds.
The Killing form $B$ takes real values on $\kk\times\kk$, where it is in fact negative definite.
We choose a basis $\{X_1,\dots,X_n\}$ of $\kk$ over~$\R$ such that $B(X_k,X_{\ell})=-\delta_{k,\ell}$ for $1\leq k,\ell\leq n$.
Then $\{X_1,\dots,X_n,JX_1,\dots,JX_n\}$ is a basis of $\g$ over~$\R$, and
$$B_{\R}(X_k,X_{\ell}) = - B_{\R}(JX_k,JX_{\ell}) = - 2 \delta_{k,\ell}$$
and $B(X_k,JX_{\ell})=0$ for all $1\leq k,\ell\leq n$.
The Casimir elements $C_G$ and~$C_{G,\R}$ with respect to the two Killing forms $B$ and~$B_{\R}$ are given by
$$C_G = - \sum_{k=1}^n X_k^2, \quad\quad\quad C_{G,\R} = \frac{1}{2} \sum_{k=1}^n \big((JX_k)^2 - X_k^2\big).$$
We note that $C_G$ does not change if we replace $J$ by $-J$.

On the other hand, corresponding to the decomposition \eqref{eqn:hol-decomp} of the complexified tangent space, we have a decomposition of the complexification of~$\g$ into two complex Lie algebras:
\begin{eqnarray}
\g \otimes_{\R} \C & \simeq & \g^{\mathrm{holo}} \oplus \g^{\mathrm{anti}}\label{eqn:decomp-g-C}\\
X & \mapsto & \frac{1}{2} \, \big(X - \sqrt{-1}JX\big) + \frac{1}{2} \, \big(X + \sqrt{-1}JX\big),\nonumber
\end{eqnarray}
where $\g^{\mathrm{holo}}$ and $\g^{\mathrm{anti}}$ are the eigenspace of the original complex structure $J$ for the eigenvalues $\sqrt{-1}$ and $-\sqrt{-1}$, respectively.
The Casimir elements $C^{\mathrm{holo}}$ and $C^{\mathrm{anti}}$ of the complex Lie algebras $\g^{\mathrm{holo}}$ and $\g^{\mathrm{anti}}$ are given by
$$C^{\mathrm{holo}} = - \frac{1}{4} \sum_{k=1}^n \big(X_k - \sqrt{-1}JX_k\big)^2, \quad\quad\quad C^{\mathrm{anti}} = - \frac{1}{4} \sum_{k=1}^n \big(X_k + \sqrt{-1}JX_k\big)^2.$$
Therefore $C_{G,\R} = C^{\mathrm{holo}} + C^{\mathrm{anti}}$ in $Z(\g\otimes_{\R}\C)$.
\end{proof}

\begin{proof}[Proof of Proposition~\ref{prop:rel-Lapl} for $(G,H,L)\!=\!(\SO(8,\C),\SO(7,\C),\mathrm{Spin}(7,1))$]
The complexification of the triple $(G,H,L)$ is given by
$$(G_{\C}, H_{\C}, L_{\C}) = \big(\SO(8,\C)\times\SO(8,\C), \SO(7,\C)\times\SO(7,\C),\mathrm{Spin}(8,\C)\big),$$
and its compact real form by
$$(G_U, H_U, L_U) = \big(\SO(8)\times\SO(8), \SO(7)\times\SO(7),\mathrm{Spin}(8)\big).$$
For $i=1,2$, let $C_{G_U}^{(i)}$ be the Casimir element of the $i$-th factor of~$G_U$.
The Casimir element of~$G_U$ is given by $C_{G_U} = C_{G_U}^{(1)} + C_{G_U}^{(2)}$.
By \cite[Prop.\,7.4.(1)]{kkdiffop}, we have $\dd\ell(C_{G_U}) = 6\dd\ell(C_{L_U}) - 4\dd r(C_{L\cap K})$.
Let $C_{G,\R}$, $C^{\mathrm{holo}}$, and $C^{\mathrm{anti}}$ be as in Lemma~\ref{lem:Casimir-complex}, and define $C_{L,\R}$ similarly for~$L$.
Then $\dd\ell(C_{G_U})$ and $\dd\ell(C_{G,\R})$ extend to the same holomorphic differential operator on~$X_{\C}$, and similarly for $\dd\ell(C_{L_U})$ and $\dd\ell(C_{L,\R})$.
Thus $\dd\ell(C^{\mathrm{holo}}+C^{\mathrm{anti}})$ is equal to $6\dd\ell(C_L) - 4\dd r(C_{L\cap K})$, and the Laplacian $\square_X = \dd\ell(C_{G,\R})$ satisfies \eqref{eqn:rel-Lapl} with $a=6$ by Lemma~\ref{lem:Casimir-complex}.
\end{proof}

\begin{remark} \label{rem:rel-Lapl-GxG}
Let $X:=({}^{\backprime}G\times\!{}^{\backprime}G)/\Diag({}^{\backprime}G)$ and $L:={}^{\backprime}G\times\!{}^{\backprime}K$, where ${}^{\backprime}G$ is a noncompact reductive Lie group and ${}^{\backprime}K$ a maximal compact subgroup as in Example~\ref{ex:group-manifold}.
For $i=1,2$, let $C_L^{(i)}$ be the Casimir element of the $i$-th factor of~$L$.
Then $C_L=C_L^{(1)}+C_L^{(2)}$.
We have $\square_X=\dd\ell(C_L^{(1)})$ and $\dd\ell(C_L^{(2)})=\dd r(C_L^{(2)})\in\dd r(Z(\llll_{\C}\cap\kk_{\C}))$, hence \eqref{eqn:rel-Lapl} holds with $a=1$, even though $X_{\C}$ is not necessarily $L_{\C}$-spherical (see Example~\ref{ex:sph}.(4)).
\end{remark}

\section{The maps $\pp_{\tau,\Gamma}$} \label{subsec:p-tau}

We continue with the general setting~\ref{gen-setting}.
Let $\Gamma$ be a torsion-free discrete subgroup of~$L$.
The $L$-equivariant fibration $q : X = G/H \simeq L/L_H \overset{F}{\longrightarrow} L/L_K = Y$ of \eqref{eqn:bundleXY} induces a fibration
\begin{equation} \label{eqn:X-Gamma-Y-Gamma}
q_{\Gamma} : X_{\Gamma} \overset{F}{\longrightarrow} Y_{\Gamma}
\end{equation}
with compact fiber $F\simeq L_K/L_H$.

Let $\F=\A$, $C^{\infty}$, $L^2$, or~$\DD'$.
For any $(\tau,V_{\tau})\in\Disc(L_K/L_H)$, we now introduce a projection map $\pp_{\tau,\Gamma} : \F(X_{\Gamma})\rightarrow (V_{\tau}^{\vee})^{L_H}\otimes\F(Y_{\Gamma},\V_{\tau})$, where $(V_{\tau}^{\vee})^{L_H}$ denotes the space of $L_H$-fixed vectors in the contragredient representation of $(\tau,V_{\tau})$, as in Section~\ref{subsec:bundle}.

For $v\in V_{\tau}$, the map $k\mapsto\tau(k)v$ from $L_K$ to~$V_{\tau}$ is right-$L_H$-invariant.
Let $\dd k$ be the Haar probability measure on the compact group~$L_K$.
Consider the map
\begin{eqnarray*}
\F(X_{\Gamma}) \otimes V_{\tau} & \longrightarrow & \hspace{0.6cm} \F(Y_{\Gamma},\V_{\tau})\\
f \hspace{0.5cm} \otimes v \hspace{0.19cm} & \longmapsto & \int_{L_K} f(\cdot k) \, \tau(k) v \, \dd k,
\end{eqnarray*}
where $f\in\F(X_{\Gamma})$ is seen as a right-$L_H$-invariant element of $\F(\Gamma\backslash L)$.
We fix an $L_K$-invariant inner product on $V_\tau$, and write $(V_{\tau})'$ for  the orthogonal complement of $(V_{\tau})^{L_H}$ in $V_\tau$, which is the direct sum of irreducible representations of $L_H$ apart from the trivial representation. 
Let $\dd h$ be the Haar probability measure on the compact group $L_H$.
By the Schur orthogonality relations, the integral $\int_{L_H}  \tau(h) v \, \dd h$ vanishes for all $v \in (V_{\tau})'$.
By the Fubini theorem, the integral $\int_{L_K} f(\cdot k) \, \tau(k) v \, \dd k = \int_{L_K } f(\cdot k) \, (\, \int_{L_H} \tau(kh) v \, \dd h) \, \dd k$
vanishes if $v$ belongs to $(V_{\tau})'$. Hence the integral induces a homomorphism
\begin{equation} \label{eqn:p-tau}
\pp_{\tau,\Gamma} : \F(X_{\Gamma}) \longrightarrow ((V_{\tau})^{L_H})^{\vee}\otimes\F(Y_{\Gamma},\V_{\tau}) \simeq (V_{\tau}^{\vee})^{L_H}\otimes\F(Y_{\Gamma},\V_{\tau}).
\end{equation}
The map $\pp_{\tau,\Gamma}$ is surjective and continuous for $\F=\A$, $C^{\infty}$, $L^2$, or~$\DD'$ in each topology.
If $\ii_{\tau,\Gamma} : ((V_{\tau})^{L_H})^{\vee}\otimes\F(Y_{\Gamma},\V_{\tau})\to\F(X_{\Gamma})$ is as in \eqref{eqn:i-tau}, then $\pp_{\tau,\Gamma}\circ\ii_{\tau,\Gamma}=\mathrm{id}$ on $\F(Y_{\Gamma},\V_{\tau})$ by the Schur orthogonality relation for the compact group~$L_K$.
When $\Gamma=\{e\}$ is trivial we shall simply write $\pp_{\tau}$ for~$\pp_{\tau,\Gamma}$.
Then $\pp_{\tau}\circ p_{\Gamma}^{\ast} = {p'_{\Gamma}}^{\!\!\ast}\circ\pp_{\tau,\Gamma}$ (see Observation~\ref{obs:extend-i-tau}), where $p_{\Gamma}^{\ast} : \F(X_{\Gamma})\to\DD'(X)$ and ${p'_{\Gamma}}^{\!\!\ast} : \F(Y_{\Gamma},\V_{\tau})\to\DD'(Y,\V_{\tau})$ denote the maps induced by the natural projections $p_{\Gamma} : X\to X_{\Gamma}$ and $p'_{\Gamma} : Y\to Y_{\Gamma}$, respectively.

We refer to Remark~\ref{rem:L_K-bundle} for a more general construction.

\section{Conditions (Tf), (A), (B) on higher-rank operators} \label{subsec:cond-Tf-A-B}

In order to establish the theorems of Chapters \ref{sec:intro} and~\ref{sec:method}, we shall use the following conditions (Tf), (A), (B), which we prove are satisfied in the setting of these theorems.

For $\tau\in\Disc(L_K/L_H)$, recall the algebra homomorphism 
$$\dd\ell^{\tau} : Z(\llll_{\C}) \longrightarrow \D_L(Y,\V_{\tau})$$
from Section~\ref{subsec:bundle}, where $\D_L(Y,\V_{\tau})$ is the $\C$-algebra of $L$-invariant matrix-valued differential operators acting on $\F(Y,\V_{\tau})$; this map generalizes the natural map $\dd\ell: Z(\llll_{\C}) \to \D_L(Y)$ of \eqref{eqn:dl-dr}, see Example~\ref{ex:trivial-tau}.
We note that $\F(Y,\V_{\tau};\NN_{\nu})\neq\{ 0\}$ only if $\nu\in\Hom_{\C\text{-}\mathrm{alg}}(Z(\llll_{\C}),\C)$ vanishes on $\Ker(\dd\ell^{\tau})$.

\begin{definition} \label{def:cond-Tf}
In the general setting~\ref{gen-setting}, we say that the quadruple $(G,L,H,L_K)$ satisfies \emph{condition~(Tf)} if for every $\tau\in\Disc(L_K/L_H)$ there exist maps
$$\nnu(\cdot,\tau) : \Hom_{\C\text{-}\mathrm{alg}}(\D_G(X),\C) \longrightarrow \Hom_{\C\text{-}\mathrm{alg}}(Z(\llll_{\C}),\C)$$
and
$$\llambda(\cdot,\tau) : \Hom_{\C\text{-}\mathrm{alg}}(Z(\llll_{\C})/\Ker(\dd\ell^{\tau}),\C) \longrightarrow \Hom_{\C\text{-}\mathrm{alg}}(\D_G(X),\C)$$
with the following properties for $\F=C^{\infty}$ and~$\DD'$:
\begin{enumerate}
  \item for any $\lambda\in\Hom_{\C\text{-}\mathrm{alg}}(\D_G(X),\C)$,
  $$\pp_{\tau}\big(\F(X;\M_{\lambda})\big)\ \subset\ (V_{\tau}^{\vee})^{L_H} \otimes \F(Y,\V_{\tau};\NN_{\nnu(\lambda,\tau)}) \,;$$
  \item for any $\nu\in\Hom_{\C\text{-}\mathrm{alg}}(Z(\llll_{\C})/\Ker(\dd\ell^{\tau}),\C)$,
  $$\ii_{\tau}\big((V_{\tau}^{\vee})^{L_H} \otimes \F(Y,\V_{\tau};\NN_{\nu})\big)\ \subset\ \F(X;\M_{\llambda(\nu,\tau)}).$$
\end{enumerate}
We shall call $\nnu$ and~$\llambda$ \emph{transfer maps}.
\end{definition}

\begin{remark} \label{rem:lambda-circ-nu}
If the quadruple $(G,L,H,L_K)$ satisfies condition~(Tf), then the transfer maps $\nnu$ and~$\llambda$ are inverse to each other, in the sense that
\begin{enumerate}
  \item for any $\tau$ and $\lambda$ such that $\pp_{\tau}(C^{\infty}(X;\M_{\lambda}))\neq\{0\}$ we have
  $$\llambda\big(\nnu(\lambda,\tau),\tau\big) = \lambda,$$
  \item for any $\tau$ and $\nu$ such that $C^{\infty}(Y,\V_{\tau};\NN_{\nu})\neq\{0\}$ we have
  $$\nnu\big(\llambda(\nu,\tau),\tau\big) = \nu.$$
\end{enumerate}
\end{remark}

We also introduce two other conditions, (A) and~(B), which relate representations of the real reductive Lie group~$G$ and of its reductive subgroup~$L$.
We denote by $\g$ and~$\llll$ the respective Lie algebras of $G$ and~$L$.
As above, a $\g_{\C}$-module $V$ is called $Z(\g_{\C})$-finite if the annihilator $\mathrm{Ann}_{Z(\g_{\C})}(V)$ of $V$ in $Z(\g_{\C})$ has finite codimension in $Z(\g_{\C})$; equivalently, the action of $Z(\g_{\C})$ on~$V$ factors through the action of a finite-dimensional $\C$-algebra.

\begin{definition} \label{def:cond-A-B}
In the general setting~\ref{gen-setting}, we say that the quadruple $(G,L,H,L_K)$ satisfies
\begin{itemize}
  \item \emph{condition~(A)} if $\ii_{\tau}(\vartheta)$ is $Z(\g_{\C})$-finite for any $\tau\in\Disc(L_K/L_H)$ and any $Z(\llll_{\C})$-finite $\llll$-module $\vartheta\subset C^{\infty}(Y,\V_{\tau})$,
  \item \emph{condition~(B)} if $\pp_{\tau}(V)$ is $Z(\llll_{\C})$-finite for any $\tau\in\Disc(L_K/L_H)$ and any $Z(\g_{\C})$-finite $\g$-module $V\subset C^{\infty}(X)$.
\end{itemize}
\end{definition}

The following two propositions are used as a stepping stone to the proof of the theorems of Chapters \ref{sec:intro} and~\ref{sec:method}, which are stated either in the main setting~\ref{spher-setting} of these theorems or in the group manifold case.

\begin{proposition} \label{prop:cond-Tf-A-B-satisfied}
In the general setting~\ref{gen-setting}, suppose that $X_{\C}=G_{\C}/H_{\C}$ is $L_{\C}$-spherical and $G$ simple.
Then conditions (Tf), (A), (B) are satisfied for the quadruple $(G,L,H,L_K)$.
\end{proposition}

\begin{proposition} \label{prop:cond-Tf-A-B-satisfied-GxG}
Let ${}^{\backprime}G$ be a noncompact reductive Lie group and ${}^{\backprime}K$ a maximal compact subgroup of~${}^{\backprime}G$.
Let
$$(G,H,L)=({}^{\backprime}G\times\!{}^{\backprime}G,\Diag({}^{\backprime}G),{}^{\backprime}G\times\!{}^{\backprime}K)$$
and $K=L_K={}^{\backprime}K\times\!{}^{\backprime}K$, as in Example~\ref{ex:group-manifold}.
Then conditions (Tf), (A), (B) are satisfied for the quadruple $(G,L,H,L_K)$.
\end{proposition}

Propositions \ref{prop:cond-Tf-A-B-satisfied} and~\ref{prop:cond-Tf-A-B-satisfied-GxG} state in particular the existence of transfer maps $\nnu$ and~$\llambda$.
We now explain why they are true, based on \cite{kkdiffop}; a formal proof will be given in Section~\ref{subsec:proof-cond-Tf-A-B-satisfied}.

\begin{remark} \label{rem:AB-need-spherical}
In Proposition~\ref{prop:cond-Tf-A-B-satisfied}, we cannot remove the $L_{\C}$-sphericity assumption: see Section~\ref{subsec:ex-not-AB}.
\end{remark}

\section[Conditions ($\widetilde{\mathrm{A}}$) and~($\widetilde{\mathrm{B}}$)]{Conditions ($\widetilde{\mathrm{A}}$) and~($\widetilde{\mathrm{B}}$), and their relation to conditions (Tf), (A), (B)} \label{subsec:cond-tilA-tilB}

Under some mild assumptions (which are satisfied in all but one case of Table~\ref{table1}), conditions (Tf), (A), (B) are consequences of two conditions ($\widetilde{\mathrm{A}}$) and~($\widetilde{\mathrm{B}}$) that only depend on complexifications of $X$ and~$Y$ and of the Lie algebras, as we now explain.

We first observe that the $\C$-algebra $\D_G(X)$ depends only on the pair of complexified Lie algebras $(\g_{\C},\h_{\C})$, and is also isomorphic to the $\C$-algebra $\D_{G_{\C}}(X_{\C})$ of $G_{\C}$-invariant holomorphic differential operators on the complex manifold $X_{\C}=G_{\C}/H_{\C}$, where $G_{\C}$ is any connected complex Lie group with Lie algebra~$\g_{\C}$, and $H_{\C}$ a connected subgroup with Lie algebra~$\h_{\C}$.
This consideration allows us to use results for other real forms, \eg compact real forms as in \cite{kkdiffop}.

In \cite[\S\,1.4]{kkdiffop} we introduced two conditions, ($\widetilde{\mathrm{A}}$) and~($\widetilde{\mathrm{B}}$), in the general setting of reductive Lie algebras $\g_{\C}\supset\h_{\C},\llll_{\C},\rr_{\C}$ over~$\C$ such that
\begin{equation}\label{eqn:glhr}
\g_{\C} = \h_{\C} + \llll_{\C} \quad\quad\mathrm{and}\quad\quad \llll_{\C}\cap\h_{\C} \subset \rr_{\C} \subset \llll_{\C}.
\end{equation}
(The notation $(\g_{\C},\h_{\C},\llll_{\C},\rr_{\C})$ here corresponds to the notation $(\widetilde{\g}_{\C},\widetilde{\h}_{\C},\g_{\C},\kk_{\C})$ in \cite[\S\,1.4]{kkdiffop}.)
Let $G_{\C}\supset H_{\C},L_{\C}$ be connected complex reductive Lie groups with Lie algebras $\g_{\C},\h_{\C},\llll_{\C}$, respectively.
We may regard $\D_{G_{\C}}(X_{\C})$, $\dd r(Z(\rr_{\C}))$, and $\dd\ell(Z(\llll_{\C}))$ as subalgebras of $\D_{L_{\C}}(X_{\C})$.
We note that the elements of $\D_{G_{\C}}(X_{\C})$ and $\dd\ell(Z(\llll_{\C}))$ naturally commute because $L_{\C}\subset G_{\C}$, and similarly for the right action $\dd r(Z(\rr_{\C}))$ and the left action $\dd\ell(Z(\llll_{\C}))$.

\begin{definition}[{\cite[\S\,1.4]{kkdiffop}}] \label{def:cond-tilA-tilB}
The quadruple $(\g_{\C},\llll_{\C},\h_{\C},\rr_{\C})$ satisfies
\begin{itemize}
  \item \emph{condition ($\widetilde{A}$)} if $\D_{G_{\C}}(X_{\C}) \subset \langle \dd\ell(Z(\llll_{\C})), \dd r(Z(\rr_{\C})) \rangle$,
  \item \emph{condition ($\widetilde{B}$)} if $\dd\ell(Z(\llll_{\C})) \subset \langle \D_{G_{\C}}(X_{\C}), \dd r(Z(\rr_{\C})) \rangle$,
\end{itemize}
where $\langle\cdot\rangle$ denotes the $\C$-algebra generated by two subalgebras.
\end{definition}

Going back to the general setting~\ref{gen-setting}, we now take $G_{\C}$ to be the complexification of a real reductive Lie group~$G$, and $\rr_{\C}=\llll_{\C}\cap\kk_{\C}$ where $K$ is a maximal compact subgroup of~$G$ such that $L_K:=L\cap K$ is a maximal compact subgroup of~$L$ containing $L_H:=L\cap H$.
In this case, the subalgebra $\D_{L_K}(F)$ of invariant differential operators coming from the compact fiber $F=L_K/L_H$ is equal to $\dd r(Z(\rr_{\C}))$ (see Section~\ref{subsec:DLX}), and conditions ($\widetilde{\mathrm{A}}$) and~($\widetilde{\mathrm{B}}$) mean that this subalgebra is sufficiently large in $\D_L(X)$ so that the two other subalgebras $\D_G(X)$ and $\dd\ell(Z(\llll_{\C}))$ are comparable modulo $\D_{L_K}(F)$.

\begin{example} \label{ex:rank-1}
Suppose $X_{\C}=G_{\C}/H_{\C}$ is $G_{\C}$-spherical.
If $\rank G/H=1$, then $\D_G(X)$ is generated by the Laplacian $\square_X$ (see Sections \ref{subsec:DGH}--\ref{subsec:pseudo-Riem-struct}).
In particular, in this case the relation \eqref{eqn:rel-Lapl} implies that condition ($\widetilde{\mathrm{A}}$) holds for the quadruple $(\g_{\C},\llll_{\C},\h_{\C},\llll_{\C}\cap\kk_{\C})$.
\end{example}

In the case that $\rank G/H>1$, the $\C$-algebra $\D_G(X)$ is not generated only by the Laplacian~$\square_X$, and so \eqref{eqn:rel-Lapl} does not imply condition ($\widetilde{\mathrm{A}}$).
The following is a direct consequence of \cite[Cor.\,1.12 \& Rem.\,1.13]{kkdiffop}.

\begin{fact}[\cite{kkdiffop}] \label{fact:cond-tilA-tilB-satisfied}
In the general setting~\ref{gen-setting}, assume that $X_{\C}=G_{\C}/H_{\C}$ is $L_{\C}$-spherical.
\begin{enumerate}
  \item If the complexified Lie algebra $\g_{\C}$ is simple, then condition ($\widetilde{\mathrm{A}}$) holds for the quadruple $(\g_{\C},\llll_{\C},\h_{\C},\llll_{\C}\cap\kk_{\C})$.
  \item If the real Lie algebra $\g$ is simple (in particular, if the complexification $\g_{\C}$ is simple), then condition ($\widetilde{\mathrm{B}}$) holds for the quadruple $(\g_{\C},\llll_{\C},\h_{\C},\llll_{\C}\cap\kk_{\C})$.
\end{enumerate}
\end{fact}

\begin{proof}[Proof of Fact~\ref{fact:cond-tilA-tilB-satisfied}]
Without loss of generality, after possibly replacing $H$ by some conjugate in~$G$, we may and do assume that $H\cap K$ is a maximal compact subgroup of~$H$.
If $G_U$ is a maximal compact subgroup of~$G_{\C}$, then $K:=G\cap G_U$, $H_U:=H_{\C}\cap G_U$, and $L_U:=L_{\C}\cap G_U$ are also maximal compact subgroups of $G$, $H_{\C}$, and $L_{\C}$, respectively, and $L_U\cap H_U=L_H=L\cap H$.
Since $X_{\C}$ is $L_{\C}$-spherical, $L$ acts transitively on $G/H$ and $L_U$ acts transitively on $G_U/H_U$ by \cite[Lem.\,5.1]{kob94}.
Moreover, $\llll_U\cap\kk$ is a maximal proper Lie subalgebra of~$\llll_U$ containing $\h_U\cap\llll_U$ because $\llll\cap\kk$ is a maximal proper Lie subalgebra of~$\llll$ containing $\llll\cap\h$.
Thus we can apply \cite[Cor.\,1.12 \& Rem.\,1.13]{kkdiffop} for compact Lie groups.
(The notation $(\g_{\C},\llll_{\C},\h_{\C},\llll_{\C}\cap\kk_{\C})$ here corresponds to the notation $(\widetilde{\g}_{\C},\g_{\C},\widetilde{\h}_{\C},\kk_{\C})$ in \cite{kkdiffop}.)
\end{proof}

\begin{remark} \label{rem:necessary-cond}
In Fact~\ref{fact:cond-tilA-tilB-satisfied}.(1)--(2), we cannot relax the condition of $L_{\C}$-sphericity of~$X_{\C}$ to $G_{\C}$-sphericity: see Examples \ref{ex:not-spher-A-B-nonsym} and~\ref{ex:not-spher-B}.
\end{remark}

Conditions ($\widetilde{\mathrm{A}}$) and~($\widetilde{\mathrm{B}}$) still hold in the group manifold case, even though the complexification $G_{\C}$ is not simple, and even though $G_{\C}/H_{\C}$ is not necessarily $L_{\C}$-spherical (see Example~\ref{ex:sph}.(4)).

\begin{lemma} \label{lem:tilA-tilB-group-manifold}
In the setting of Proposition~\ref{prop:cond-Tf-A-B-satisfied-GxG}, conditions ($\widetilde{\mathrm{A}}$) and~($\widetilde{\mathrm{B}}$) hold for the quadruple $(\g_{\C},\llll_{\C},\h_{\C},\llll_{\C}\cap\kk_{\C})$.
\end{lemma}

\begin{proof}
This follows immediately from the fact that
\begin{align*}
\D_{G_{\C}}(X_{\C}) & = \dd\ell(Z({}^{\backprime}\g_{\C})) = \dd r(Z({}^{\backprime}\g_{\C})),\\
\dd\ell(Z(\llll_{\C})) & = \langle\dd\ell(Z({}^{\backprime}\g_{\C})),\dd r(Z({}^{\backprime}\kk_{\C}))\rangle,\\
\dd\ell(Z(\llll_{\C}\cap\kk_{\C})) & = \dd\ell(Z({}^{\backprime}\kk_{\C})).\qedhere
\end{align*}
\end{proof}

When $(\g_{\C}, \h_{\C}, \llll_{\C}) = (\so(8,\C)\otimes_{\R}\C ,\, \so(7,\C)\otimes_{\R}\C ,\, \spin(8,\C))$, condition ($\widetilde{\mathrm{A}}$) does not hold for the quadruple $(\g_{\C}, \h_{\C}, \llll_{\C}, \llll_{\C}\cap\kk_{\C})$, see \cite[Prop.\,7.5]{kkdiffop}.
Instead, we consider the following weaker condition, where $R$ is the $\C$-subalgebra of $\D_{L_{\C}}(X_{\C})$ generated by $\dd\ell(Z(\llll_{\C}))$ and $\dd r(Z(\llll_{\C}\cap\kk_{\C}))$:

\smallskip

\emph{Condition ($\widetilde{\mathrm{A}}'$):} $\D_{L_{\C}}(X_{\C})$ is finitely generated as an $R$-module.

\smallskip
\noindent
Then the following holds.

\begin{fact}[see {\cite[Prop.\,7.5]{kkdiffop}}] \label{fact:cond-tilA'-complex-case}
Condition~($\widetilde{\mathrm{A}}'$) holds for the quadruple $(\g_{\C}, \h_{\C}, \llll_{\C}) = (\so(8,\C)\otimes_{\R}\C ,\, \so(7,\C)\otimes_{\R}\C ,\, \spin(8,\C))$.
\end{fact}

Propositions \ref{prop:cond-Tf-A-B-satisfied} and~\ref{prop:cond-Tf-A-B-satisfied-GxG} are a consequence of Facts \ref{fact:cond-tilA-tilB-satisfied} and~\ref{fact:cond-tilA'-complex-case} and Lemma~\ref{lem:tilA-tilB-group-manifold}, together with the following lemma.

\begin{lemma} \label{lem:A-tilA-B-tilB}
In the general setting~\ref{gen-setting},
\begin{enumerate}
  \item condition ($\widetilde{A}$) or ($\widetilde{A}'$) for the quadruple $(\g_{\C},\llll_{\C},\h_{\C},\llll_{\C}\cap\kk_{\C})$ implies condition~(A) for the quadruple $(G,L,H,L_K)$;
  \item condition~($\widetilde{B}$) for the quadruple $(\g_{\C},\llll_{\C},\h_{\C},\llll_{\C}\cap\kk_{\C})$ implies condition~(B) for the quadruple $(G,L,H,L_K)$ whenever $X_{\C}$ is $G_{\C}$-spherical;
  \item conditions ($\widetilde{\mathrm{A}}$) and~($\widetilde{B}$) together for the quadruple $(\g_{\C},\llll_{\C},\h_{\C},\llll_{\C}\cap\kk_{\C})$ imply condition~(Tf) for the quadruple $(G,L,H,L_K)$;
  \item condition~(Tf) holds for the quadruple
  $$(G, L, H, L_K) = (\SO(8,\C), \Spin(7,1), \SO(7,\C), \Spin(7)).$$
\end{enumerate}
\end{lemma}

In view of Lemma~\ref{lem:A-tilA-B-tilB}.(3), the relation \eqref{eqn:rel-Lapl} between Laplacians, which partially implies condition ($\widetilde{\mathrm{A}}$) (see Example~\ref{ex:rank-1}), is part of the underlying structure of the existence of transfer maps.

We summarize the relations among conditions (Tf), (A), ($\widetilde{\mathrm{A}}$), (B), ($\widetilde{\mathrm{B}}$) in the following table.
The double arrows $\Rightarrow$ or $\Leftarrow$ indicate that the conditions in the complex setting (symbolically written as~$X_{\C}$) imply those in the real setting (symbolically written as~$X$) when $X_{\C}$ is $G_{\C}$-spherical.
\begin{center}
\begin{tabular}{ccccc|c}
$X$ & & $X_{\C}$ & & $X$ & Applications to $C^{\infty}(X)$\!\!\\
\hline
\multirow{2}{*}{(Tf)} & \multirow{2}{*}{$\overset{\text{Lem.\,\ref{lem:A-tilA-B-tilB}.(3)}}{\Longlongleftarrow}$} & ($\widetilde{\mathrm{A}}$) & $\overset{\text{Lem.\,\ref{lem:A-tilA-B-tilB}.(1)}}{\Longlongrightarrow}$ & (A) & $L\uparrow G$ (Prop.~\ref{prop:condA})\\
& & ($\widetilde{\mathrm{B}}$) & $\overset{\text{Lem.\,\ref{lem:A-tilA-B-tilB}.(2)}}{\Longlongrightarrow}$ & (B) & $G\downarrow L$ (Th.~\ref{thm:condB})
\\
\hspace{-0.1cm}Def.\,\ref{def:cond-Tf}\hspace{-0.1cm} & & \hspace{-0.1cm}Def.\,\ref{def:cond-tilA-tilB}\hspace{-0.1cm} & & \hspace{-0.1cm}Def.\,\ref{def:cond-A-B} &
\end{tabular}
\end{center}

\begin{remark}
It is natural to expect that representations $\pi$ of~$G$ occurring in $C^{\infty}(X)$ and representations $\vartheta$ of~$L$ occurring in $C^{\infty}(Y,\V_{\tau})$ should be closely related via $\ii_{\tau}$ and~$\pp_{\tau}$.
Lemma~\ref{lem:A-tilA-B-tilB}.(3) shows that this is the case under conditions ($\widetilde{\mathrm{A}}$) and~($\widetilde{\mathrm{B}}$).
In the future paper \cite{kkI}, we shall use Lemma~\ref{lem:A-tilA-B-tilB}.(3) to find the branching laws for the restriction to~$L$ of infinite-dimensional representations of~$G$ realized in $\DD'(X)$.
\end{remark}

\section[Proof of Lemma~5.19 and Propositions 5.10--5.11]{Proof of Lemma~\ref{lem:A-tilA-B-tilB} and Propositions \ref{prop:cond-Tf-A-B-satisfied}--\ref{prop:cond-Tf-A-B-satisfied-GxG}} \label{subsec:proof-cond-Tf-A-B-satisfied}


We now complete these proofs by discussing how the structure of $\D_L(X)$ in terms of the three subalgebras $\D_G(X)$, $\dd\ell(Z(\llll_{\C}))$, $\dd r(Z(\llll_{\C}\cap\kk_{\C}))$ (given by conditions ($\widetilde{\mathrm{A}}$) and ($\widetilde{\mathrm{B}}$)) controls representation-theoretic properties of the groups $G$ and~$L$ (given by conditions (A) and~(B)).

\begin{proof}[Proof of Lemma~\ref{lem:A-tilA-B-tilB}.(1)]
Suppose that condition ($\widetilde{\mathrm{A}}$) or ($\widetilde{\mathrm{A}}'$) holds for the quadruple $(\g_{\C},\llll_{\C},\h_{\C},\llll_{\C}\cap\kk_{\C})$.
Let $\tau\in\Disc(L_K/L_H)$ and let $\vartheta$ be a $Z(\llll_{\C})$-finite $\llll$-submodule of $C^{\infty}(Y,\V_{\tau})$.
By Schur's lemma, $Z(\llll_{\C}\cap\kk_{\C})$ acts on the irreducible representation $\tau$ of $\llll_{\C}\cap\kk_{\C}$ as scalars, hence also on $\ii_{\tau}(\vartheta)$ via $\dd r$.
Therefore, the annihilator of $\ii_{\tau}(\vartheta)$ has finite codimension in the $\C$-algebra $R$ generated by $\dd\ell(Z(\llll_{\C}))$ and $\dd r(Z(\llll_{\C}\cap\kk_{\C}))$.
If condition~($\widetilde{\mathrm{A}}$) holds, then the $\C$-algebra $R$ contains $\D_G(X)$, and so the annihilator of $\ii_{\tau}(\vartheta)$ has finite codimension in the subalgebra $\dd\ell(Z(\g_{\C}))\subset\D_G(X)$.
If condition~($\widetilde{\mathrm{A}}'$) holds, then we can write $\D_L(X)=\sum_{j=1}^k Ru_j$ for some $u_1,\dots,u_k\in\D_L(X)$, and so we can inflate $\ii_{\tau}(\vartheta)$ to
$$\ii_{\tau}(\vartheta)^{\sim} := \D_{L_{\C}}(X_{\C}) \, \ii_{\tau}(\vartheta) = \sum_{j=1}^k u_j \, \ii_{\tau}(\vartheta),$$
which is an $\llll$-module as well as a $\D_{L_{\C}}(X_{\C})$-module.
Since the annihilator of $\ii_{\tau}(\vartheta)^{\sim}$ has finite codimension in~$R$, so does it as a $\D_{L_{\C}}(X_{\C})$-module.
Therefore the annihilator of $\ii_{\tau}(\vartheta)^{\sim}$ has finite codimension in the subalgebra $\dd\ell(Z(\g_{\C}))$.
Thus condition~(A) holds for the quadruple $(G,L,H,L_K)$.
\end{proof}

\begin{proof}[Proof of Lemma~\ref{lem:A-tilA-B-tilB}.(2)]
Suppose that condition~($\widetilde{\mathrm{B}}$) holds for the quadruple $(\g_{\C},\llll_{\C},\h_{\C},\llll_{\C}\cap\kk_{\C})$.
Let $V$ be a $Z(\g_{\C})$-finite $\g$-submodule of $C^{\infty}(X)$ and let $\tau\in\Disc(L_K/L_H)$.
If $X_{\C}$ is $G_{\C}$-spherical, then $\D_G(X)$ is finitely generated as a $\dd\ell(Z(\g_{\C}))$-module (see Section~\ref{subsec:DGH}).
Take $D_1,\dots,D_k\in\D_G(X)$ such that $\D_G(X)=\sum_{j=1}^k \dd\ell(Z(\g_{\C}))\cdot D_j$ and let $\tilde{V}:=\sum_{j=1}^k D_j\cdot V\subset C^{\infty}(X)$.
Then $\tilde{V}$ is a $Z(\g_{\C})$-finite $\D_G(X)$-module; therefore, it is $\D_G(X)$-finite.
Since $Z(\llll_{\C}\cap\kk_{\C})$ acts on $\pp_{\tau}(C^{\infty}(X))$ via $\dd r$ as scalars, the action on $\pp_{\tau}(\tilde{V})$ of the $\C$-algebra generated by $\D_G(X)$ and $\dd r(Z(\llll_{\C}\cap\kk_{\C}))$ factors through the action of a finite-dimensional algebra, and so does the action of $Z(\llll_{\C})$ due to condition~($\widetilde{\mathrm{B}}$).
Thus condition~(B) holds for the quadruple $(G,L,H,L_K)$.
\end{proof}

The proof of Lemma~\ref{lem:A-tilA-B-tilB}.(3) is based on the following observation.

\begin{remark} \label{rem:D-tau}
In the general setting~\ref{gen-setting}, let $D$ be an element of $\dd\ell(Z(\llll_{\C}))$ or $\dd r(Z(\llll_{\C}\cap\kk_{\C}))$.
For any $\tau\in\Disc(L_K/L_H)$, the differential operator $D$ on~$X$ induces an $\curlyEnd((V_{\tau}^{\vee})^{L_H}\otimes\V_{\tau})$-valued differential operator $D^{\tau}$ acting on the sections of the vector bundle $(V_{\tau}^{\vee})^{L_H}\otimes\V_{\tau}$ over~$Y$ such that
\begin{equation} \label{eqn:equiviota}
D(\ii_{\tau}(\varphi)) = \ii_{\tau}(D^{\tau}\varphi)
\end{equation}
for any $\varphi \in (V_{\tau}^{\vee})^{L_H}\otimes\F(Y,\V_{\tau})$, where $\F=\A$, $C^{\infty}$, or~$\DD'$: see Definition-Proposition \ref{def-prop:D-tau}.
The operator $D^{\tau}$ is $L$-invariant.
See Example~\ref{ex:D-tau-Z-l} for the case $D\in\dd\ell(Z(\llll_{\C}))$ and Example~\ref{ex:D-tau-Z-r} for the case $D\in\dd r(Z(\llll_{\C}\cap\kk_{\C}))$.
\end{remark}

\begin{proof}[Proof of Lemma~\ref{lem:A-tilA-B-tilB}.(3)]
Fix $\tau\in\Disc(L_K/L_H)$ and let $\F=\A$, $C^{\infty}$, or~$\DD'$.

Let $\lambda\in\Hom_{\C\text{-}\mathrm{alg}}(\D_G(X),\C)$.
If condition~($\widetilde{\mathrm{B}}$) holds for the quadruple $(\g_{\C},\llll_{\C},\h_{\C},\llll_{\C}\cap\kk_{\C})$, then for any $z\in Z(\llll_{\C})$, we can write $\dd\ell(z) = \sum_j \dd r(Q_j)P_j$ in $\D_L(X)$, for some $P_j\in\D_G(X)$ and $Q_j\in Z(\llll_{\C}\cap\kk_{\C})$, $1\leq j\leq m$.
By Schur's lemma, each $Q_j$ acts on $\pp_{\tau}(\F(X))$ via $\dd r$ as a scalar $c_j\in\C$ (depending on~$\tau$).
For any $F\in\F(X;\M_{\lambda})$, we then have
$$\dd\ell^{\tau}(z)\big(\pp_{\tau}(F)) = \Big(\sum_j c_j\,\lambda(P_j)\Big) \, \pp_{\tau}(F).$$
This proves the existence of a map $\nnu(\cdot,\tau)$ as in condition~(Tf) for the quadruple $(G,L,H,L_K)$.

Conversely, let $\nu\in\Hom_{\C\text{-}\mathrm{alg}}(Z(\llll_{\C}),\C)$.
Any $z\in Z(\llll_{\C})$ acts as the scalar $\nu(z)$ on $\F(Y,\V_{\tau};\NN_{\nu})$ via the operator $\dd\ell^{\tau}$, hence also on $\ii_{\tau}((V_{\tau}^{\vee})^{L_H}\otimes\F(Y,\V_{\tau};\NN_{\nu}))$ by \eqref{eqn:equiviota}.
On the other hand, any $z'\in Z(\llll_{\C}\cap\kk_{\C})$ acts as a scalar on $\ii_{\tau}((V_{\tau}^{\vee})^{L_H}\otimes\F(Y,\V_{\tau};\NN_{\nu}))$ via $\dd r$ by Remark~\ref{rem:D-tau}.
Thus, if condition~($\widetilde{\mathrm{A}}$) holds for the quadruple $(\g_{\C},\llll_{\C},\h_{\C},\llll_{\C}\cap\kk_{\C})$, then any element of $\D_G(X)$ acts as a scalar on $\ii_{\tau}((V_{\tau}^{\vee})^{L_H}\otimes\F(Y,\V_{\tau};\NN_{\nu}))$.
By construction, the scalar depends only on $\nu$ modulo $\mathrm{Ker}(\dd\ell^{\tau})$, proving the existence of a map $\llambda(\cdot,\tau)$ as in condition~(Tf) for the quadruple $(G,L,H,L_K)$.
\end{proof}

\begin{proof}[Proof of Lemma~\ref{lem:A-tilA-B-tilB}.(4)]
The compact real form of the complexification of the triple $(G,L,H)$ is given by
$$(G_U, L_U, H_U) = \big(\SO(8)\times\SO(8), \Spin(8), \SO(7)\times\SO(7)\big),$$
and there is a unique maximal connected proper subgroup of~$L_U$ containing $L_H = L_U\cap H_U$, namely $L_K\simeq\Spin(7)$.
This means that the $L_U$-equivariant fiber bundle $G_U/H_U\to L_U/L_H$ and the $L$-equivariant fiber bundle $G/H\to L/L_H$ have the same fiber $F=L_K/L_H$.
Then Lemma~\ref{lem:A-tilA-B-tilB}.(4) follows from \cite[Prop.\,4.8 \& Th.\,4.9]{kkdiffop} in the compact case via holomorphic continuation.
\end{proof}

\begin{proof}[Proof of Proposition~\ref{prop:cond-Tf-A-B-satisfied}]
By Facts \ref{fact:cond-tilA-tilB-satisfied} and~\ref{fact:cond-tilA'-complex-case}, conditions ($\widetilde{\mathrm{A}}$) or ($\widetilde{\mathrm{A}}'$) and ($\widetilde{\mathrm{B}}$) hold for the quadruple $(\g_{\C},\llll_{\C},\h_{\C},\llll_{\C}\cap\kk_{\C})$ when $X_{\C}=G_{\C}/H_{\C}$ is $L_{\C}$-spherical and $G$ simple.
Therefore condition (A) (\resp (B), \resp (Tf)) holds for the quadruple $(G,L,H,L_K)$ by Lemma~\ref{lem:A-tilA-B-tilB}.(1) (\resp (2), \resp (3)--(4)).
\end{proof}

\begin{proof}[Proof of Proposition~\ref{prop:cond-Tf-A-B-satisfied-GxG}]
By Lemma~\ref{lem:tilA-tilB-group-manifold}, conditions ($\widetilde{\mathrm{A}}$) and ($\widetilde{\mathrm{B}}$) hold for the quadruple $(\g_{\C},\llll_{\C},\h_{\C},\llll_{\C}\cap\kk_{\C})$ in the group manifold setting of Proposition~\ref{prop:cond-Tf-A-B-satisfied-GxG}.
Moreover, $X_{\C}=G_{\C}/H_{\C}$ is $G_{\C}$-spherical for $(G_{\C},H_{\C}) = ({}^{\backprime}G_{\C}\times\!{}^{\backprime}G_{\C},\Diag({}^{\backprime}G_{\C}))$.
Therefore condition (A), (B), and (Tf) hold for the quadruple $(G,L,H,L_K)$ by Lemma~\ref{lem:A-tilA-B-tilB}.
\end{proof}

\section{Explicit transfer maps}

When $X_{\C}=G_{\C}/H_{\C}$ is $L_{\C}$-spherical and $G$ simple, the maps $\nnu$ and~$\llambda$ of condition~(Tf) can be given explicitly.
Let $\mathfrak{t}_{\C}$ be a Cartan subalgebra of~$\llll_{\C}$, and $W(\llll_{\C})$ the Weyl group of the root system $\Delta(\llll_{\C},\mathfrak{t}_{\C})$.
Let $\jj_{\C}\subset\g_{\C}$ and $W=W(\g_{\C},\jj_{\C})$ be as in Section~\ref{subsec:DGH}.
The notation $(G,H,L,\jj_{\C},W,\mathfrak{t}_{\C},W(\llll_{\C}))$ here corresponds to the notation $(\widetilde{G},\widetilde{H},G,\widetilde{\aaa}_{\C},\widetilde{W},\jj_{\C},W(\g_{\C}))$ in \cite[Th.\,4.9]{kkdiffop}.

\begin{proposition}\label{prop:specXY-precise}
In the general setting~\ref{gen-setting}, suppose that $X_{\C}=G_{\C}/H_{\C}$ is $L_{\C}$-spherical and $G$ simple.
For any $\tau\in\Disc(L_K/L_H)$, there is an affine map $S_{\tau} : \jj_{\C}^*\to\mathfrak{t}_{\C}^*$ such that the diagram
$$\xymatrix{\jj_{\C}^* \ar[d] \ar[r]^{S_{\tau}} & \mathfrak{t}_{\C}^* \ar[d]\\
\jj_{\C}^*/W & \mathfrak{t}_{\C}^*/W(\llll_{\C})\\
\Hom_{\C\text{-}\mathrm{alg}}(\D_G(X),\C) \ar[u]^{\Psi^*}_{\text{\rotatebox{90}{$\sim$}}} \ar@{-->}[r]^{\nnu(\cdot,\tau)}
& \Hom_{\C\text{-}\mathrm{alg}}(Z(\llll_{\C}),\C) \ar[u]_{\Phi^*}^{\text{\rotatebox{-90}{$\sim$}}}
}$$
induces a map
$$\nnu(\cdot,\tau) : \Hom_{\C\text{-}\mathrm{alg}}(\D_G(X),\C) \longrightarrow \Hom_{\C\text{-}\mathrm{alg}}(Z(\llll_{\C}),\C).$$
Moreover, there is a unique map
$$\llambda(\cdot,\tau) : \Hom_{\C\text{-}\mathrm{alg}}(Z(\llll_{\C})/\Ker(\dd\ell^{\tau}),\C) \longrightarrow \Hom_{\C\text{-}\mathrm{alg}}(\D_G(X),\C)$$
such that $\nnu(\llambda(\nu,\tau))=\nu$ for all $\nu\in\Hom_{\C\text{-}\mathrm{alg}}(Z(\llll_{\C})/\Ker(\dd\ell^{\tau}),\C)$.
\end{proposition}

In \cite[Th.\,4.9 \& Prop.\,4.13]{kkdiffop} we proved the proposition when $G$ is compact, with an explicit expression of the affine map $S_{\tau}$ in terms of the highest weight of~$\tau$.
We now briefly explain how to reduce to this case.

\begin{proof}[Proof of Proposition~\ref{prop:specXY-precise}]
Without loss of generality, we may and do assume that the subgroups $H$ and~$L$ are preserved by a Cartan involution $\theta$ of~$G$.
Let $G_{\C}\supset H_{\C},L_{\C}$ be connected complex Lie groups containing $G\supset H,L$ as real forms.
We can find a compact real form $G_U$ of~$G_{\C}$ such that $H_U:=G_U\cap H_{\C}$ and $L_U:=G_U\cap L_{\C}$ are maximal compact subgroups of $H_{\C}$ and~$L_{\C}$, respectively.
We set $X_{\C} := G_{\C}/H_{\C}$ and $Y_{\C} := L_{\C}/(L_{\C}\cap K_{\C})$, as well as $X_U:=G_U/H_U$ and $Y_U:=L_U/L_K$.
Since $L_U$ acts transitively on~$X_U$ and $L_U\cap H_U=L_H$, we have an $L_U$-equivariant fibration $L_K/L_H\to X_U\to Y_U$, which may be thought of as a compact real form of the $L_{\C}$-equivariant fibration $(L_{\C}\cap K_{\C})/(L_{\C}\cap H_{\C})\to X_{\C}\to Y_{\C}$ which is the complexification of the $L$-equivariant fibration $L_K/L_H\to X\to Y$.
We note that the fiber $L_K/L_H$ is common to these two real fibrations.
Via the holomorphic extension of invariant differential operators, we have natural $\C$-algebra isomorphisms $\D_G(X) \simeq \D_{G_{\C}}(X_{\C}) \simeq \D_{G_U}(X_U)$ and $\D_L(Y,\V_{\tau}) \simeq \D_{L_{\C}}(Y_{\C},\V_{\tau}^{\C}) \simeq \D_{L_U}(Y_U,\V_{\tau,U})$ for all $\tau\in\Disc(L_K/L_H)$, where $\V_{\tau,U}:=L_U\times_{L_K}\tau$ is an $L_U$-equivariant vector bundle over~$Y_U$ associated to~$\tau$, and $\V_{\tau}^{\C} := L_{\C} \times_{L_{\C}\cap K_{\C}} \tau^{\C}$ is an $L_{\C}$-equivariant holomorphic vector bundle over~$Y_{\C}$ associated to a holomorphic extension $\tau^{\C}$ of~$\tau$ by the Weyl unitary trick as in \cite[Lem.\,5.4]{kkdiffop}.
Thus the proposition for~$G$ follows from the proposition for~$G_U$, which is established in \cite[Th.\,4.9]{kkdiffop}.
\end{proof}

In this setting, explicit formulas for $\nnu(\lambda,\tau)$ in each case of Table~\ref{table1} are given in \cite[\S\,6--7]{kkdiffop}: see for instance Example~\ref{ex:SO-4n-2} for the case $G/H=\SO(4n,2)_0/\U(2n,1)$.

\part{Proof of the theorems of Chapter~\ref{sec:intro}} \label{part:main-proofs}

In this Part~\ref{part:main-proofs}, we provide proofs of the theorems of Chapters \ref{sec:intro} and~\ref{subsec:method-transfer}.

We start, in Chapter~\ref{sec:strategy}, by establishing the essential self-adjointness of the pseudo-Riemannian Laplacian $\square_{X_{\Gamma}}$ (Theorems \ref{thm:selfadj} and~\ref{thm:GxG}.(2)).
For this we reduce to the Riemannian case using the relation \eqref{eqn:rel-Lapl} between the pseudo-Riemannian Laplacian $\square_X$ and the Casimir element $C_L$; this relation is part of the underlying structure of the existence of transfer maps, as mentioned just after Lemma~\ref{lem:A-tilA-B-tilB}.

In Chapter~\ref{sec:transfer-Gamma}, using the transfer maps $\nnu$ and~$\llambda$ of Chapter~\ref{sec:dliota}, we complete the proofs of Theorems \ref{thm:transfer} and~\ref{thm:transfer-GxG}, hence of Theorems \ref{thm:analytic} and~\ref{thm:GxG}.(1), as well as the proof of Theorem~\ref{thm:Fourier}.

In Chapter~\ref{sec:transfer} we derive two consequences of conditions (A) and~(B) of Section~\ref{subsec:cond-Tf-A-B}; in particular, we prove (Theorem~\ref{thm:condB}) that any infinite-dimensional representation of~$G$ realized in $\DD'(X)$ decomposes discretely as a representation of the subgroup~$L$ under condition~(B).

In Chapter~\ref{sec:transfer-Gamma-type-I-II} we show (Theorem~\ref{thm:transfer-spec}) that the discrete spectrum of type~$\I$ and type~$\II$ from Chapter~\ref{sec:type-I-II} is compatible for the pseudo-Riemannian locally homogeneous space $X_{\Gamma}$ and for the vector bundle $\V_{\tau}$ over the Riemannian locally symmetric space~$Y_{\Gamma}$, as a counterpart of discrete decomposability results (Theorem~\ref{thm:condB}) for the restriction of representations of $G$ to its subgroup~$L$, which contains the discrete group~$\Gamma$.

In Chapter~\ref{sec:proof-mainII} we complete the proof of Theorem~\ref{thm:mainII}, which states the existence of an infinite discrete spectrum of type~$\II$ when $\Gamma$ is cocompact or arithmetic in~$L$: this is deduced from Theorem~\ref{thm:transfer-spec} and from the classical Riemannian case (Fact~\ref{fact:inf-spec-Riem}).

\chapter{Essential self-adjointness of the Laplacian} \label{sec:strategy}

In this chapter we address Questions~\ref{problems} by establishing the following.

\begin{proposition} \label{prop:relationsLapl}
In the general setting~\ref{gen-setting}, consider a $G$-invariant pseudo-Rieman\-nian structure on~$X$ (see Lemma~\ref{lem:pseudo-Riem-struct}) and let $\square_X$ be the corresponding Laplacian.
Choose any $\Ad(L)$-invariant, nondegenerate, symmetric bilinear form on~$\llll$ and let $C_L\in Z(\llll)$ be the corresponding Casimir element.
If \eqref{eqn:rel-Lapl} holds for some nonzero $a\in\R$, then for any torsion-free discrete subgroup $\Gamma$ of~$L$,
\begin{enumerate}
  \item the closure of the pseudo-Riemannian Laplacian~$\square_{X_{\Gamma}}$ defined on $C^{\infty}_c(X_{\Gamma})$ in the graph norm gives a self-adjoint operator on $L^2(X_{\Gamma})$,
  \item for any $\lambda\in\Spec_d(X_{\Gamma})$, the space $L^2(X_{\Gamma};\M_{\lambda})$ contains real analytic functions as a dense subset,
  \item the pseudo-Riemannian Laplacian $\square_{X_{\Gamma}}$ on~$X_{\Gamma}$ has infinitely many $L^2$-eigenva\-lues as soon as the Riemannian Laplacian $\Delta_{Y_{\Gamma}}$ on~$Y_{\Gamma}$ has (\eg if $\Gamma$ is a uniform or arithmetic lattice in~$L$), and $\square_{X_{\Gamma}}$ has absolutely continuous spectrum as soon as $\Delta_{Y_{\Gamma}}$ has.
\end{enumerate}
\end{proposition}

As mentioned above, the idea is to use the relation \eqref{eqn:rel-Lapl} to derive Proposition~\ref{prop:relationsLapl} from the corresponding results for the Riemannian Laplacian on~$Y_{\Gamma}$.

\section{Proof of Theorems \ref{thm:selfadj} and~\ref{thm:GxG}.(2)} \label{subsec:proof-selfadj}

Postponing the proof of Proposition~\ref{prop:relationsLapl} till Section~\ref{subsec:main-thm-Laplacian}, we now prove the self-adjointness of the pseudo-Rieman\-nian Laplacian $\square_{X_{\Gamma}}$ in the setting~\ref{spher-setting} (Theorem~\ref{thm:selfadj}) and in the group manifold case (Theorem~\ref{thm:GxG}.(2)).

\begin{proof}[Proof of Theorem~\ref{thm:selfadj} assuming Proposition~\ref{prop:relationsLapl}]
In the setting~\ref{spher-setting}, the group $L$ acts transitively on~$X$, and by Proposition~\ref{prop:rel-Lapl} there exists a nonzero $a\in\R$ such that \eqref{eqn:rel-Lapl} holds.
Thus the Laplacian $\square_{X_{\Gamma}}$ is essentially self-adjoint by Proposition~\ref{prop:relationsLapl}.(1).
\end{proof}

\begin{proof}[Proof of Theorem~\ref{thm:GxG}.(2) assuming Proposition~\ref{prop:relationsLapl}]
By Remark \ref{rem:rel-Lapl-GxG}, there exists a nonzero $a\in\R$ such that \eqref{eqn:rel-Lapl} holds, and so $\square_{X_{\Gamma}}$ is essentially self-adjoint by Proposition~\ref{prop:relationsLapl}.(1).
\end{proof}

\begin{remark}
If $\rank X=1$ (as in examples (i), (i)$'$, (iii), (v), (v)$'$, (vi), (vii),~(ix) of Table~\ref{table1}), then the $\C$-algebra $\D_G(X)$ is generated by the Laplacian $\square_X$, and so for any discrete subgroup $\Gamma$ of $G$ acting properly discontinuously and freely on~$X$, we may identify $\Spec_d(X_{\Gamma})$ with the discrete spectrum of the Laplacian~$\square_{X_{\Gamma}}$.
In this case, based on the relation \eqref{eqn:rel-Lapl} (which holds for some nonzero~$a$ by Proposition~\ref{prop:rel-Lapl}), one can use the same approach as in the proof of Proposition~\ref{prop:relationsLapl}.(2)--(3) to obtain the abundance of real analytic joint eigenfunctions (Theorem~\ref{thm:transfer}) and the existence of an infinite discrete spectrum of type~$\II$ under certain assumptions (Theorem~\ref{thm:mainII}).
To prove these results in the general case, allowing for $\rank X>\nolinebreak 1$, we shall use conditions (Tf) and~(B) of Section~\ref{subsec:cond-Tf-A-B}: see Sections \ref{subsec:proof-transfer} and~\ref{sec:proof-mainII}.
\end{remark}

\section[A decomposition of $L^2(X_{\Gamma})$ using discrete series for~$F$]{A decomposition of $L^2(X_{\Gamma})$ using discrete series for the fiber~$F$}\label{subsec:L2decomp}

In preparation for the proof of Proposition~\ref{prop:relationsLapl}, we introduce a useful decomposition of~$L^2(X_{\Gamma})$.

We work again in the general setting~\ref{gen-setting}.
Recall from \eqref{eqn:X-Gamma-Y-Gamma} that the $L$-equivariant fibration $q : X=G/H\simeq L/L_H\to L/L_K=Y$ induces a fibration of the quotient $X_{\Gamma}$ over the Riemannian locally symmetric space $Y_{\Gamma}=\Gamma\backslash L/L_K$ with compact fiber $F=L_K/L_H$.

By the Frobenius reciprocity theorem, for any irreducible representation $(\tau,V_{\tau})$ of the compact group~$L_K$, the space $(V_{\tau})^{L_H}$ of $L_H$-fixed vectors in~$V_H$ is nonzero if and only if $\tau$ belongs to $\Disc(L_K/L_H)$; in this case we consider the (finite-dimensional) unitary representation $(V_{\tau}^{\vee})^{L_H} \otimes V_{\tau}$ of~$L_K$ (a multiple of~$V_{\tau}$).
There is a unitary equivalence of $L_K$-modules (isotypic decomposition of the regular representation of~$L_K$):
\begin{equation}\label{eqn:LKLH}
L^2(L_K/L_H) \,\simeq\quad \sumplus{\tau\in\Disc(L_K/L_H)}\ (V_{\tau}^{\vee})^{L_H} \otimes V_{\tau}
\end{equation}
(Hilbert direct sum).
We shall use this decomposition to compare spectral analysis on the pseudo-Riemannian locally homogeneous space $X_{\Gamma}$ and on the Riemannian locally symmetric space~$Y_{\Gamma}$.

For $\F=\A$, $C^{\infty}$, $L^2$, or~$\DD'$, and for $\tau\in\Disc(L_K/L_H)$, we denote by $\F(Y_{\Gamma},\V_{\tau})$ the space of analytic, smooth, square-integrable, or distribution sections of the Hermitian vector bundle
$$\V_{\tau} := \Gamma\backslash L \times_{\scriptscriptstyle{L_K}} V_{\tau} \ \longrightarrow\ Y_{\Gamma},$$
respectively.
Recall the continuous linear maps
$$\left\{
\begin{array}{ccl}
\ii_{\tau,\Gamma} & : & (V_{\tau}^{\vee})^{L_H} \otimes \F(Y_{\Gamma},\V_{\tau}) \longhookrightarrow \F(X_{\Gamma}),\\
\pp_{\tau,\Gamma} & : & \F(X_{\Gamma}) \longtwoheadrightarrow (V_{\tau}^{\vee})^{L_H} \otimes \F(Y_{\Gamma},\V_{\tau})
\end{array}
\right.$$
from \eqref{eqn:i-tau} and \eqref{eqn:p-tau}, respectively, for $\F=\A$, $C^{\infty}$, $L^2$, or~$\DD'$.

For $\F=L^2$, the maps $\ii_{\tau,\Gamma} : (V_{\tau}^{\vee})^{L_H}\otimes L^2(Y_{\Gamma},\V_{\tau})\rightarrow L^2(X_{\Gamma})$ are isometric embeddings, whose images are orthogonal to each other for varying $\tau\in\Disc(L_K/L_H)$, and they induce a unitary operator
\begin{equation} \label{eqn:XYL2}
\underline{\ii}_{\Gamma}\ :\quad\, \sumplus{\tau\in\Disc(L_K/L_H)}\ (V_{\tau}^{\vee})^{L_H} \otimes L^2(Y_{\Gamma},\V_{\tau}) \ \overset{\sim}{\longrightarrow}\ L^2(X_{\Gamma}).
\end{equation}
We endow $\Gamma\backslash L$ with an $L$-invariant Radon measure such that the pull-back of the submersion $\Gamma\backslash L \to X_{\Gamma} \simeq \Gamma\backslash L/L_H$ induces an isometry between Hilbert spaces
\begin{equation} \label{eqn:L2XGammaL}
L^2(X_{\Gamma}) \overset{\sim}{\longrightarrow} L^2(\Gamma\backslash L)^{L_H}.
\end{equation}
For $\tau\in\Disc(L_K/L_H)$, the map $\ii_{\tau,\Gamma}\circ\pp_{\tau,\Gamma} : L^2(X_{\Gamma})\rightarrow L^2(X_{\Gamma})$ is the orthogonal projection onto $\ii_{\tau,\Gamma}((V_{\tau}^{\vee})^{L_H}\otimes L^2(Y_{\Gamma},\V_{\tau}))$.
If $\chi_{\tau} : L_K\to\C^*$ denotes the character of~$\tau$, then $\ii_{\tau,\Gamma}\circ\pp_{\tau,\Gamma}$ is induced by the integral operator
\begin{eqnarray}\label{eqn:itau-circ-ptau}
L^2(\Gamma\backslash L) & \longrightarrow & \hspace{1.4cm} L^2(\Gamma\backslash L)\\
f \hspace{0.5cm} & \longmapsto & (\dim\tau)\,\int_{L_K} \chi_{\tau}(k) \, f(\cdot\,k) \, \dd k,\nonumber
\end{eqnarray}
which leaves the subspace $L^2(\Gamma\backslash L)^{L_H}\simeq L^2(X_{\Gamma})$ invariant.
By construction, $\pp_{\tau,\Gamma}$ is the adjoint of~$\ii_{\tau,\Gamma}$: for any $f\in L^2(X_{\Gamma})$ and $\varphi\in (V_{\tau}^{\vee})^{L_H}\otimes L^2(Y_{\Gamma},\V_{\tau})$,\begin{equation}\label{eqn:piadjoint}
(\pp_{\tau,\Gamma}(f),\varphi)_{L^2(Y_{\Gamma},\V_{\tau})} = (f,\ii_{\tau,\Gamma}(\varphi))_{L^2(X_{\Gamma})}.
\end{equation}

In general, the following holds.

\begin{lemma} \label{lem:p-tau-i-tau}
In the general setting~\ref{gen-setting}, let $\Gamma$ be a torsion-free discrete subgroup of~$L$.
Let $\F=\A$, $C^{\infty}$, $L^2$, or~$\DD'$.
\begin{enumerate}
  \item For any $\tau\in\Disc(L_K/L_H)$, we have $\pp_{\tau,\Gamma}\circ\ii_{\tau,\Gamma}=\mathrm{id}$ on $\F(Y_{\Gamma},\V_{\tau})$.
  \item The maps $\ii_{\tau,\Gamma}$ induce an injective linear map
  $$\underline{\ii}_{\Gamma} := \bigoplus_{\tau}\,\ii_{\tau,\Gamma} \,:\, \bigoplus_{\tau\in\Disc(L_K/L_H)} (V_{\tau}^{\vee})^{L_H} \otimes \F(Y_{\Gamma},\V_{\tau}) \longhookrightarrow \F(X_{\Gamma}),$$
  which is continuous with dense image.
   More precisely, for any increasing exhaustive sequence $(D_j)_{j\in\N}$ of finite subsets of $\Disc(L_K/L_H)$, the finite sums $\sum_{\tau \in D_j}  (\ii_{\tau,\Gamma}\circ\pp_{\tau,\Gamma})(f)$ converge to $f$ in $\F(X_\Gamma)$ with respect to each topology for $\F=\A$, $C^{\infty}$, $L^2$, or~$\DD'$ as $j\to +\infty$.
\end{enumerate}
\end{lemma}

We shall simply write the convergence in Lemma~\ref{lem:p-tau-i-tau}.(2) as
\begin{equation} \label{eqn:f=sum_ipf}
f = \sum_{\tau\in\Disc(L_K/L_H)} (\ii_{\tau,\Gamma}\circ\pp_{\tau,\Gamma})(f)
\quad \textrm{in}\  \F(X_\Gamma).
\end{equation}

\begin{proof}[Proof of Lemma~\ref{lem:p-tau-i-tau}]
(1) This follows from the Schur orthogonality relation for the compact group~$L_K$.

(2)  We start by reviewing the  Peter--Weyl theorem for the compact Lie group~$L_K$.
For any $\tau\in\widehat{L_K}$, we define two continuous linear maps
$$\left\{
\begin{array}{ccl}
\ii_{\tau}^{e} & : & V_{\tau}^{\vee} \otimes V_{\tau} \longhookrightarrow \F(L_K),\\
\pp_{\tau}^{e} & : & \F(L_K) \longtwoheadrightarrow V_{\tau}^{\vee} \otimes V_{\tau}
\end{array}
\right.$$
as follows: $\ii_{\tau}^e$ is the map sending any $v'\otimes v\in V_{\tau}^{\vee} \otimes V_{\tau}$ to the matrix coefficient
$$\left(k \longmapsto \langle v',\tau(k^{-1})v\rangle \right) \in \F(L_K),$$
and $\pp_{\tau}^e$ is the map sending any $\psi\in\F(L_K)$ to
$$\bigg(v \longmapsto (\dim\tau) \, \int_{L_K} \psi(k^{-1}) \, \tau(k) v \, \dd k\bigg) \in \End(V_{\tau}) \simeq V_{\tau}^{\vee} \otimes V_{\tau}.$$
Then $\ii_{\tau}^e$ takes values in $\A(L_K)$.
The composition $\pp_{\tau}^e\circ\ii_{\tau}^e$ is the identity on $V_{\tau}^{\vee}\otimes V_{\tau}$.
For any $\psi\in\F(L_K)$ and any increasing exhaustive sequence $((\widehat{L_K})_j)_{j\in\N}$ of finite subsets of $\widehat{L_K}$, the finite sum $\sum_{\tau \in (\widehat{L_K})_j} (\ii_{\tau}^e\circ\pp_{\tau}^e)(\psi)$ converges to $\psi$ in $\F(L_K)$ with respect to each topology for $\F=\A$, $C^{\infty}$, $L^2$, or~$\DD'$ as $j\to +\infty$, see \eg \cite[(2.3.1)--(2.3.4)]{kob98b}; this convergence can be written as\footnote{F 10/6: Kept the numbered equation, but slightly rephrased the sentence (I would prefer to avoid ``We shall write it as'' if we never write it again in the rest of the paper). Is this fine with you?}
\begin{equation}\label{Peter-Weyl-convergence}
\sum_{\tau \in \widehat{L_K}} (\ii_{\tau}^e\circ\pp_{\tau}^e)(\psi) = \psi \quad \textrm{in} \ \F(L_K).
\end{equation}

Now consider the closed subgroup $L_H$ of~$L_K$. 
By taking the projection to the space of right-$L_H$-invariant vectors via the integration over the compact subgroup $L_H$, we see that the algebraic direct sum
$$\bigoplus_{\tau\in\Disc(L_K/L_H)} (V_{\tau}^{\vee})^{L_H} \otimes V_{\tau}$$
is contained in $\A(L_K/L_H)\simeq\A(L_K)^{L_H}$ and dense in $\F(L_K/L_H)\simeq\F(L_K)^{L_H}$.
Explicitly, for any $\tau\in\Disc(L_K/L_H)$, the above maps $\ii_{\tau}^e$ and $\pp_{\tau}^e$ induce continuous linear maps
$$\left\{
\begin{array}{ccl}
\ii_{\tau}^0 & : & (V_{\tau}^{\vee})^{L_H} \otimes V_{\tau} \longhookrightarrow \F(L_K/L_H),\\
\pp_{\tau}^0 & : & \F(L_K/L_H) \longtwoheadrightarrow (V_{\tau}^{\vee})^{L_H} \otimes V_{\tau}.
\end{array}
\right.$$
Then $\ii_{\tau}^0$ takes values in $\A(L_K/L_H)$.
The composition $\pp_{\tau}^0\circ\ii_{\tau}^0$ is the identity on $(V_{\tau}^{\vee})^{L_H}\otimes V_{\tau}$, and any $\psi\in\F(L_K/L_H)$ can be approximated by finite sums of the form $\sum_{\tau} (\ii_{\tau}^0\circ\pp_{\tau}^0)(\psi)$ for $\F=\A$, $C^{\infty}$, $L^2$, or~$\DD'$.

The maps
$$\left\{
\begin{array}{lcl}
\mathrm{id}\otimes\ii_{\tau}^0 & : & \F(\Gamma\backslash L) \otimes (V_{\tau}^{\vee})^{L_H} \otimes V_{\tau} \longrightarrow \F\big(\Gamma\backslash L) \otimes \F(L_K/L_H),\\
\mathrm{id}\otimes\pp_{\tau}^0 & : & \F\big(\Gamma\backslash L) \otimes \F(L_K/L_H) \longrightarrow \F(\Gamma\backslash L) \otimes (V_{\tau}^{\vee})^{L_H} \otimes V_{\tau}
\end{array}
\right.$$
induce continuous linear maps
$$\left\{
\begin{array}{lcl}
\mathrm{id}\otimes\ii_{\tau}^0 & : & (V_{\tau}^{\vee})^{L_H} \otimes \F(\Gamma\backslash L, V_{\tau}) \longrightarrow \F\big(\Gamma\backslash L, \F(L_K/L_H)\big),\\
\mathrm{id}\otimes\pp_{\tau}^0 & : & \F\big(\Gamma\backslash L, \F(L_K/L_H)\big) \longrightarrow (V_{\tau}^{\vee})^{L_H} \otimes \F(\Gamma\backslash L, V_{\tau}),
\end{array}
\right.$$
for which we use the same notation.
For a direct product of two real analytic manifolds $M$ and~$N$ with Radon measures, there is a natural isomorphism
$$\F(M,\F(N)) \simeq \F(M\times N),$$
where $\DD'(M,\DD'(N))$ is defined to be the dual space of $C_c^{\infty}(M,C_c^{\infty}(N))$.
In particular, if the compact Lie group $L_K$ acts on $M$ and~$N$, and if one of the actions is free, then we obtain a natural isomorphism $\F(M,\F(N))^{L_K}\simeq\F(M\times_{L_K} N)$.
Applying this observation to $M=\Gamma\backslash L$ and $N=L_K/L_H$, we see that the following diagram commutes in restriction to the $\ii_{\tau}$ arrows or to the $\pp_{\tau}$ arrows:
$$\xymatrixcolsep{-4pc}
\hspace{-0.3cm}\xymatrix{
\F(X_{\Gamma}) \simeq \F\big(\Gamma\backslash L, \F(L_K/L_H)\big)^{L_K} \ar@{->>}@<-7ex>[d]^{\pp_{\tau,\Gamma}} \ar@{}[r]|{\hspace{1cm}\textstyle\subset} & \F\big(\Gamma\backslash L, \F(L_K/L_H)\big) \ar@<1ex>[d]^{\mathrm{id}\otimes\pp_{\tau}^0}\\
(V_{\tau}^{\vee})^{L_H} \otimes \F(Y_{\Gamma},\V_{\tau}) \simeq (V_{\tau}^{\vee})^{L_H} \otimes \F\big(\Gamma\backslash L, V_{\tau}\big)^{L_K} \hspace{1.9cm} \ar@{^{(}->}@<9ex>[u]^{\ii_{\tau,\Gamma}} \ar@{}[r]|{\hspace{1cm}\textstyle\subset} & (V_{\tau}^{\vee})^{L_H} \otimes \F\big(\Gamma\backslash L, V_{\tau}\big) \ar@<1ex>[u]^{\mathrm{id}\otimes\ii_{\tau}^0}
}$$
Since any $\psi\in\F(L_K/L_H)$ can be approximated\footnote{F 8/6: Should we change this formulation into something similar to the ``More precisely'' in the statement of Lemma~\ref{lem:p-tau-i-tau}?} by finite sums of the form $\sum_{\tau} (\ii_{\tau}^0\circ\pp_{\tau}^0)(\psi)$, we obtain that any $f\in\F(X_{\Gamma})$ can be approximated\footnote{F 8/6: Should we change this formulation into something similar to the ``More precisely'' in the statement of Lemma~\ref{lem:p-tau-i-tau}?} by finite sums of the form $\sum_{\tau} (\ii_{\tau,\Gamma}\circ\pp_{\tau,\Gamma})(f)$.
In particular, the map $\underline{\ii}_{\Gamma}=\bigoplus_{\tau}\,\ii_{\tau,\Gamma}$ has dense image in $\F(X_{\Gamma})$.
\end{proof}

\begin{remark} \label{rem:L_K-bundle}
The pairs of maps $(\ii_{\tau,\Gamma},\pp_{\tau,\Gamma})$ and $(\ii_{\tau}^0,\pp_{\tau}^0)$ from the proof of Lemma~\ref{lem:p-tau-i-tau} are part of a more general construction.
Namely, let $Q$ be a manifold endowed with a free action of a compact group~$L_K$.
For any closed subgroup $L_H$ of~$L_K$, we have a fibration $Q/L_H\to Q/L_K$ with compact fiber $L_K/L_H$, and for any $\tau\in\Disc(L_K/L_H)$ and $\F=\A$, $C^{\infty}$, or~$\DD'$ we obtain natural continuous linear maps
$$\left\{ \begin{array}{lcl}
\ii_{\tau}^Q & : & (V_{\tau}^{\vee})^{L_H} \otimes \F(Q/L_K,\V_{\tau}) \longrightarrow \F(Q/L_H),\\
\pp_{\tau}^Q & : & \F(Q/L_H) \longrightarrow (V_{\tau}^{\vee})^{L_H} \otimes \F(Q/L_K,\V_{\tau}),
\end{array}\right.$$
with $\ii_{\tau}^Q$ injective and $\pp_{\tau}^Q$ surjective.
The maps $(\ii_{\tau,\Gamma},\pp_{\tau,\Gamma})$ and $(\ii_{\tau}^0,\pp_{\tau}^0)$ correspond to $Q=\Gamma\backslash L$ and $Q=L_K$, respectively.
\end{remark}

Recall that by Schur's lemma, the center $Z(\llll_{\C}\cap\kk_{\C})$ acts on the representation space of any irreducible representation $\tau$ of~$L_K$ as scalars, yielding a $\C$-algebra homomorphism $\Psi_{\tau} : Z(\llll_{\C}\cap\kk_{\C})\to\C$.

\begin{lemma} \label{lem:Im-i-tau}
In the general setting~\ref{gen-setting}, let $\Gamma$ be a torsion-free discrete subgroup of~$L$.
Let $\F=\A$, $C^{\infty}$, $L^2$, or~$\DD'$.
For any $\tau\in\Disc(L_K/L_H)$, the image of~$\ii_{\tau,\Gamma}$ is characterized by a system of differential equations as follows:
$$\mathrm{Image}\,(\ii_{\tau,\Gamma}) = \big\{ f\in\F(X_{\Gamma}) : \dd r(z)f = \Psi_{\tau^{\vee}}(z)f \quad\text{for all }z\in Z(\llll_{\C}\cap\kk_{\C})\big\}.$$
\end{lemma}

\begin{proof}
Let $z\in Z(\llll_{\C}\cap\kk_{\C})$.
Then $\dd r(z)$ acts on the subspace $(V_{\tau}^{\vee})^{L_H}\otimes V_{\tau}$ of $\A(L_K/L_H)$, see \eqref{eqn:LKLH}, by the scalar $\Psi_{\tau^{\vee}}(z)$, hence the inclusion $\subset$ holds.

To prove the opposite inclusion, we observe that $L_K$ is connected.
Therefore $\tau\in\widehat{L_K}$ is uniquely determined by its infinitesimal character~$\Psi_{\tau}$.
Then the inclusion $\supset$ follows from Lemma~\ref{lem:p-tau-i-tau}.(2).
\end{proof}

We now consider the following condition for a differential operator $D$ on~$X$:
\begin{equation} \label{eqn:D-commute-dr}
D \circ \dd r(z) = \dd r(z) \circ D \quad\quad \text{for all }z\in Z(\llll_{\C}\cap\kk_{\C}).
\end{equation}

\begin{definition-proposition}[Operators $D^{\tau}$ and $D^{\tau}_{\Gamma}$] \label{def-prop:D-tau}
In the general setting~\ref{gen-setting}, let $\tau\in\Disc(L_K/L_H)$.
\begin{enumerate}
  \item Any differential operator $D$ on~$X$ satisfying \eqref{eqn:D-commute-dr} induces a matrix-valued differential operator $D^{\tau}$ acting on $(V_{\tau}^{\vee})^{L_H}\otimes C^{\infty}(Y,\V_{\tau})$ such that $D(\ii_{\tau}(\varphi)) = \ii_{\tau}(D^{\tau}\varphi)$ for all $\varphi\in (V_{\tau}^{\vee})^{L_H}\otimes C^{\infty}(Y,\V_{\tau})$, as in~\eqref{eqn:equiviota}.
  \item If $D$ is $L$-invariant, then $D^{\tau}$ induces a differential operator $D^{\tau}_{\Gamma}$ acting on $(V_{\tau}^{\vee})^{L_H}\otimes C^{\infty}(Y_{\Gamma},\V_{\tau})$ for any torsion-free subgroup $\Gamma$ of~$L$.
  \item If $X_{\C}$ is $L_{\C}$-spherical, then any $D\in\D_L(X)$ satisfies \eqref{eqn:D-commute-dr} and we obtain a $\C$-algebra homomorphism
  \begin{eqnarray*}
  \D_L(X) & \longrightarrow & \D_L(Y_{\Gamma},\V_{\tau}).\\
  D & \longmapsto & D^{\tau}_{\Gamma}
  \end{eqnarray*}
\end{enumerate}
\end{definition-proposition}

\begin{proof}
(1) By \eqref{eqn:D-commute-dr}, the differential operator $D$ preserves the image of~$\ii_{\tau}$, which is described in Lemma~\ref{lem:Im-i-tau}.
Therefore $D$ induces an endomorphism $D^{\tau}$ of $(V_{\tau}^{\vee})^{L_H}\otimes C^{\infty}(Y,\V_{\tau})$, because $\ii_{\tau}$ is injective.
Since $D^{\tau}$ does not increase the support, it is a differential operator.

(2) Clear as in \eqref{eqn:D-gamma}.

(3) If $X_{\C}$ is $L_{\C}$-spherical, then the $\C$-algebra $\D_L(X)$ is commutative (Fact~\ref{fact:spherical}), hence \eqref{eqn:D-commute-dr} holds for any $D\in\D_L(X)$.
Moreover, in that setting $(V_{\tau}^{\vee})^{L_H}$ is one-dimensional (see Section~\ref{subsec:bundle}), hence (3) follows from~(2).
\end{proof}

\begin{example} \label{ex:D-tau-Z-l}
If $D=\dd\ell(z)$ for $z\in Z(\llll_{\C})$, then $D^{\tau}$ is a differential operator of the form $D^{\tau} = \mathrm{id}\otimes\dd\ell^{\tau}(z)$ where $\dd\ell^{\tau} : Z(\llll_{\C})\to\D_L(Y,\V_{\tau})$ is a homomorphism as in Section~\ref{subsec:bundle}.
\end{example}

\begin{example} \label{ex:D-tau-Z-r}
If $D=\dd r(z)$ for $z\in Z(\llll_{\C}\cap\kk_{\C})$, then the differential operator $D_{\Gamma}^{\tau}$ acts on $\ii_{\tau,\Gamma}((V_{\tau}^{\vee})^{L_H}\otimes L^2(Y_{\Gamma},\V_{\tau}))$ as the scalar $\Psi_{\tau^{\vee}}(z)$, which is independent of~$\Gamma$.
Indeed, by Schur's lemma, $D$ acts on~$V_{\tau}$ (seen as a subspace of $\A(L_K/L_H)$ by \eqref{eqn:LKLH}) as the scalar $\Psi_{\tau^{\vee}}(z)$, and $D_{\Gamma}^{\tau}$ acts on $\ii_{\tau,\Gamma}((V_{\tau}^{\vee})^{L_H}\otimes L^2(Y_{\Gamma},\V_{\tau}))$ by the same scalar.
\end{example}

\section[Proof of Proposition~6.1]{Proof of Proposition~\ref{prop:relationsLapl}} \label{subsec:main-thm-Laplacian}


Suppose that \eqref{eqn:rel-Lapl} holds for some nonzero $a\in\R$, \ie there exists $z\in Z(\llll_{\C}\cap\kk_{\C})$ such that $\square_X=a\,\dd\ell(C_L)+\dd r(z)$.

Let $\Gamma$ be a torsion-free discrete subgroup of~$L$.
For any $\tau\in\Disc(L_K/L_H)$, the image $\dd\ell(C_L)$ of the Casimir element $C_L\in Z(\llll_{\C})$ induces an elliptic differential operator $\dd\ell(C_L)_{\Gamma}^{\tau}$ on $C^{\infty}(Y_{\Gamma},\V_{\tau})$, and also on $(V_{\tau}^{\vee})^{L_H}\otimes C^{\infty}(Y_{\Gamma},\V_{\tau})$ (see Section~\ref{subsec:L2decomp}).
When $\tau$ is the trivial one-dimensional representation of~$L_K$, the space $C^{\infty}(Y_{\Gamma},\V_{\tau})$ identifies with $C^{\infty}(Y_{\Gamma})$ (Example~\ref{ex:trivial-tau}) and $\dd\ell(C_L)_{\Gamma}^{\tau}$ is the usual Laplacian $\Delta_{Y_{\Gamma}}$ on the Riemannian locally symmetric space~$Y_\Gamma$.

We endow the principal bundle $L\to Y=L/L_K$ with the canonical connection, which induces a connection on the principal $L_H$-bundle $\Gamma\backslash L\to Y_{\Gamma}$.
Then the curvature tensor $\Omega^{\tau}$ on the associated bundle $\V_{\tau}\to Y_{\Gamma}$ is an $\curlyEnd(\V_{\tau})$-valued $2$-form on~$Y_{\Gamma}$ given by
$$\Omega^{\tau}\big(\dd\ell(g^{-1})u,\dd\ell(g^{-1})v\big) = \dd\tau([u,v])$$
for $u,v\in\llll\cap\p$, where $\llll=\llll\cap\kk+\llll\cap\p$ is the Cartan decomposition of the Lie algebra~$\llll$, and $\dd\ell(g^{-1})$ identifies the tangent space $T_{\Gamma gL_K}Y_{\Gamma}$ with $\llll\cap\p$ for every $g\in L$.
In particular, the operator norm of the curvature tensor $\Omega^{\tau}$ is bounded on~$Y_{\Gamma}$.
Then a similar argument to \cite{gaf54,wol7273} applied to Clifford bundles implies that $\dd\ell(C_L)_{\Gamma}^{\tau}$ on $C_c^{\infty}(Y_{\Gamma},\V_{\tau})$ extends uniquely to a self-adjoint operator on $L^2(Y_{\Gamma},\V_{\tau})$ (see \cite{ric86}).

On the other hand, the element $z\in Z(\llll_{\C}\cap\kk_{\C})$ acts on $(V_{\tau}^{\vee})^{L_H}\otimes L^2(Y_{\Gamma},\V_{\tau})$ via $\dd r$ as the scalar $\Psi_{\tau^{\vee}}(z)$ (see Example~\ref{ex:D-tau-Z-r}), and so by \eqref{eqn:rel-Lapl} we may write
\begin{equation}\label{eqn:action-lapl-decomp}
\square_{X_{\Gamma}} \circ \ii_{\tau,\Gamma} = a\,\dd\ell(C_L)_{\Gamma}^{\tau} + c(\tau)
\end{equation}
on $(V_{\tau}^{\vee})^{L_H}\otimes L^2(Y_{\Gamma},\V_{\tau})$, where $c(\tau):=\Psi_{\tau^{\vee}}(z)$.

Since $\square_{X_{\Gamma}}$\,and $\dd\ell(C_L)_{\Gamma}^{\tau}$ are symmetric operators on $L^2(X_{\Gamma})$\,and\,$(V_{\tau}^{\vee})^{L_H}\otimes L^2(Y_{\Gamma},\V_{\tau})$ respectively, and since $a\in\R$, the function $c : \Disc(L_K/L_H)\rightarrow\C$ in \eqref{eqn:action-lapl-decomp} actually takes values in~$\R$.
Therefore the Laplacian~$\square_{X_{\Gamma}}$ has a self-adjoint extension on $L^2(X_{\Gamma})$, with domain equal to the image, under the unitary operator $\ii$ of \eqref{eqn:XYL2}, of the set of $\sum_{\tau\in\Disc(L_K/L_H)} \varphi_{\tau}$ with
$$\sum_{\tau} \Vert \varphi_{\tau}\Vert^2 < +\infty \quad\mathrm{and}\quad \sum_{\tau} \big\Vert a\,\dd\ell(C_L)_{\Gamma}^{\tau}\,\varphi_{\tau} + c(\tau)\varphi_{\tau}\big\Vert^2 < +\infty.$$
This proves Proposition~\ref{prop:relationsLapl}.(1).

For any $\lambda\in\Spec_d(X_{\Gamma})$, let $t_{\lambda}\in\C$ be the corresponding eigenvalue of~$\square_{X_{\Gamma}}$.
(See Remark~\ref{rem:eigenv-Lapl} for an explicit formula for $\lambda\mapsto t_{\lambda}$ when $X=G/H$ is a reductive symmetric space.)
By \eqref{eqn:equiviota}, the Laplacian~$\square_{X_{\Gamma}}$ commutes with the projection operator $\pp_{\tau} : L^2(X_{\Gamma})\rightarrow (V_{\tau}^{\vee})^{L_H}\otimes L^2(Y_{\Gamma},\V_{\tau})$ for any $\tau\in\Disc(L_K/L_H)$.
Therefore, by \eqref{eqn:action-lapl-decomp} we have the following direct sum decomposition of the $L^2$-eigenspace $\operatorname{Ker}(\square_{X_{\Gamma}}-t_{\lambda})$ as a Hilbert space:
$$\operatorname{Ker}(\square_{X_{\Gamma}}-t_{\lambda}) \,\simeq\ \ \sumplus{\tau\in\Disc(L_K/L_H)}\ \ii_{\tau,\Gamma}\big(a\operatorname{Ker}\big(\dd\ell(C_L)_{\Gamma}^{\tau} - t_{\lambda} + c(\tau)\big)\big).$$
The right-hand side contains the algebraic sum
$$\bigoplus_{\tau\in\Disc(L_K/L_H)}\, \ii_{\tau,\Gamma}\big(\operatorname{Ker}\big(a\,\dd\ell(C_L)_{\Gamma}^{\tau} - t_{\lambda} + c(\tau)\big)\big)$$
as a dense subspace.
Since $a\neq 0$, the differential operator $a\,\dd\ell(C_L)_{\Gamma}^{\tau} - t_{\lambda} + c(\tau)$ is nonzero and elliptic.
By the elliptic regularity theorem, this subspace consists of analytic functions.
Therefore, $L^2(X_{\Gamma};\M_{\lambda})$ contains real analytic eigenfunctions as a dense subset for any $\lambda\in\Spec_d(X_{\Gamma})$, proving Proposition~\ref{prop:relationsLapl}.(2).

Finally, \eqref{eqn:action-lapl-decomp} implies that $\ii_{\tau,\Gamma}(\varphi_{\tau})$ is an eigenfunction of~$\square_{X_{\Gamma}}$ for any eigenfunction $\varphi_{\tau}$ of $\dd\ell(C_L)_{\Gamma}^{\tau}$.
In particular, taking $\tau$ to be the trivial one-dimensional representation of~$L_K$, for which $\dd\ell(C_L)_{\Gamma}^{\tau}$ is the usual Laplacian $\Delta_{Y_{\Gamma}}$ on the Riemannian locally symmetric space~$Y_{\Gamma}$, we find that $\square_{X_{\Gamma}}$ has an infinite discrete spectrum (\resp has continuous spectrum) as soon as $\Delta_{Y_{\Gamma}}$ does, proving Proposition~\ref{prop:relationsLapl}.(3).

\chapter{Transfer of Riemannian eigenfunctions and spectral decomposition} \label{sec:transfer-Gamma}

In this chapter, using Propositions \ref{prop:cond-Tf-A-B-satisfied} and~\ref{prop:cond-Tf-A-B-satisfied-GxG}, we complete the proof of Theorems \ref{thm:transfer} and~\ref{thm:transfer-GxG}, which describe joint eigenfunctions on the pseudo-Riemannian locally symmetric space~$X_{\Gamma}$ by means of the regular representation of the subgroup $L$ on $C^{\infty}(\Gamma\backslash L)$.
In particular, this establishes Theorems \ref{thm:analytic} and~\ref{thm:GxG}.(1).
We then prove Theorem~\ref{thm:Fourier} concerning the spectral decomposition of $L^2$-eigenfunctions on~$X_{\Gamma}$ with respect to the systems of differential equations for $\D_G(X)$.

\section{The transfer maps $\nnu$ and~$\llambda$ in the presence of a discrete group~$\Gamma$}

We work again in the general setting~\ref{gen-setting}.
Recall from \eqref{eqn:X-Gamma-Y-Gamma} that for any torsion-free discrete subgroup $\Gamma$ of~$L$, the quotient $X_{\Gamma}$ fibers over the Riemannian locally symmetric space $Y_{\Gamma}=\Gamma\backslash L/L_K$ with compact fiber $F=L_K/L_H$.
Since any $L$-invariant differential operator on~$X$ induces a differential operator on~$X_{\Gamma}$ via the covering map $X\to X_{\Gamma}$, and since condition (Tf) (Definition~\ref{def:cond-Tf}) is formulated in terms of two algebras of \emph{$L$-invariant} differential operators on~$X$ (namely $\dd\ell(Z(\llll_{\C}))$ and $\D_G(X)$), the following holds by definition of the transfer maps.

\begin{remark} \label{rem:specXYGamma}
In the general setting~\ref{gen-setting}, suppose condition~(Tf) is satisfied for the quadruple $(G,L,H,L_K)$, with transfer maps $\nnu$ and~$\llambda$.
Let $\F=\A$, $C^{\infty}$, $L^2$, or~$\DD'$, and let $\Gamma$ be a torsion-free discrete subgroup of~$L$.
Then
\begin{enumerate}
  \item for any $(\lambda,\tau)\in\Hom_{\C\text{-}\mathrm{alg}}(\D_G(X),\C)\times\Disc(L_K/L_H)$,
  $$\pp_{\tau,\Gamma}\big(\F(X_{\Gamma};\M_{\lambda})\big)\ \subset\ (V_{\tau}^{\vee})^{L_H}\otimes\F(Y_{\Gamma},\V_{\tau};\NN_{\nnu(\lambda,\tau)}) \,;$$
  \item for any $(\nu,\tau)\in\Hom_{\C\text{-}\mathrm{alg}}(Z(\llll_{\C})/\Ker(\dd\ell^{\tau}),\C)\times\Disc(L_K/L_H)$ with $\DD'(Y,\V_{\tau};\NN_{\nu})\neq\{0\}$,
  $$\ii_{\tau,\Gamma}\big((V_{\tau}^{\vee})^{L_H}\otimes\F(Y_{\Gamma},\V_{\tau};\NN_{\nu})\big)\ \subset\ \F(X_{\Gamma};\M_{\llambda(\nu,\tau)}).$$
\end{enumerate}
\end{remark}

In particular, the maps $\nnu$ and~$\llambda$ given by condition (Tf) can be used to transfer discrete spectrum from the Riemannian locally symmetric space $Y_{\Gamma}$ to the pseudo-Riemannian space~$X_{\Gamma}$.
Indeed, taking $\tau$ to be the trivial one-dimensional representation, there is a natural isomorphism of Hilbert spaces $L^2(Y_{\Gamma},\V_{\tau})\simeq L^2(Y_{\Gamma})$ (Example~\ref{ex:trivial-tau}), and we obtain an embedding
$$\Spec_d^{Z(\llll_{\C})}(Y_{\Gamma}) \ \longhookrightarrow\ \Spec_d(X_{\Gamma})$$
by Remarks \ref{rem:specXYGamma}.(2) and~\ref{rem:lambda-circ-nu}.(2), where $\Spec_d^{Z(\llll_{\C})}(Y_{\Gamma})$ is defined as in Notation~\ref{def:Ntype-I-II} with $\g$ replaced by~$\llll$.
We shall see (Theorem~\ref{thm:transfer-spec}.(3)) that if condition~(B) holds (Definition~\ref{def:cond-A-B}), then the image of this embedding is actually contained in $\Spec_d(X_{\Gamma})_{\II}$, which implies Theorem~\ref{thm:mainII} (see Chapter~\ref{sec:proof-mainII}).

\section[Proof of Theorems 2.3--2.4 and Corollary~2.5]{Proof of Theorems \ref{thm:transfer}--\ref{thm:transfer-GxG} (Transfer of Riemannian eigenfunctions) and Corollary~\ref{cor:L2d(XGamma,Mlambda)-orth}} \label{subsec:proof-transfer}


Theorems \ref{thm:transfer} and~\ref{thm:transfer-GxG} are refinements of Theorems \ref{thm:analytic} and \ref{thm:GxG}.(1).
They are consequences of Propositions \ref{prop:cond-Tf-A-B-satisfied} and~\ref{prop:cond-Tf-A-B-satisfied-GxG}, of Lemma~\ref{lem:p-tau-i-tau}, of Remark~\ref{rem:specXYGamma}, and of the following regularity property for vector-bundle-valued eigenfunctions on the Riemannian locally symmetric space~$Y_{\Gamma}$.

\begin{lemma} \label{lem:regY}
Let $Y=L/L_K$ be a Riemannian symmetric space, where $L$ is a real reductive Lie group and $L_K$ a maximal compact subgroup.
Let $\Gamma$ be a torsion-free discrete subgroup of~$L$.
For any $\nu\in\Hom_{\C\text{-}\mathrm{alg}}(Z(\llll_{\C}),\C)$ and $\tau\in\widehat{L_K}$,
$$L^2(Y_{\Gamma},\V_{\tau};\NN_{\nu}) \subset \DD'(Y_{\Gamma},\V_{\tau};\NN_{\nu}) = C^{\infty}(Y_{\Gamma},\V_{\tau};\NN_{\nu}) = \A(Y_{\Gamma},\V_{\tau};\NN_{\nu}).$$
\end{lemma}

\begin{proof}
Let $\llll = \llll\cap\kk + \llll\cap\p$ be a Cartan decomposition corresponding to the maximal compact subgroup $L_K$ of~$L$.
As in the proof of Lemma~\ref{lem:pseudo-Riem-struct}, we choose an $\Ad(L)$-invariant nondegenerate symmetric bilinear form $B$ on~$\llll$ which is positive definite on $\llll\cap\p$ and negative definite on $\llll\cap\kk$, and for which $\llll\cap\p$ and $\llll\cap\kk$ are orthogonal.
Let $C_L\in Z(\llll)$ be the Casimir element defined by~$B$.
Let $\F=C^{\infty}$, $L^2$, or~$\DD'$.
Then $\dd\ell(C_L)_{\Gamma}^{\tau}$ is a (matrix-valued) elliptic differential operator acting on $\F(Y_{\Gamma},\V_{\tau})$, and therefore $\F(Y_{\Gamma},\V_{\tau};\NN_{\nu})$ is contained in $\A(Y_{\Gamma},\V_{\tau})$, by the elliptic regularity theorem.
\end{proof}

\begin{proof}[Proof of Theorems \ref{thm:transfer} and~\ref{thm:transfer-GxG}]
Condition~(Tf) for the quadruple $(G,L,H,L_K)$ is satisfied in the setting~\ref{spher-setting} by Proposition~\ref{prop:cond-Tf-A-B-satisfied}, and in the group manifold case by Proposition~\ref{prop:cond-Tf-A-B-satisfied-GxG}.
Thus there is a transfer map
$$\nnu : \Hom_{\C\text{-}\mathrm{alg}}(\D_G(X),\C) \times \Disc(L_K/L_H) \longrightarrow \Hom_{\C\text{-}\mathrm{alg}}(Z(\llll_{\C}),\C),$$
and by Remark~\ref{rem:specXYGamma}.(1), for any torsion-free discrete subgroup $\Gamma$ of~$L$ and any $(\lambda,\tau)\in\Hom_{\C\text{-}\mathrm{alg}}(\D_G(X),\C)\times\Disc(L_K/L_H)$,
$$\pp_{\tau,\Gamma}(\F(X_{\Gamma};\M_{\lambda})) \subset (V_{\tau}^{\vee})^{L_H} \otimes \F(Y_{\Gamma},\V_{\tau};\NN_{\nnu(\lambda,\tau)}).$$
Fix $\lambda\in\Hom_{\C\text{-}\mathrm{alg}}(\D_G(X),\C)$ and let $f\in\F(X_{\Gamma};\M_{\lambda})$, where $\F=\A$, $C^{\infty}$, $L^2$, or~$\DD'$.
By Lemma~\ref{lem:p-tau-i-tau}, we can approximate~$f$, in each topology depending on~$\F$, by a sequence of finite sums of the form $\sum_{\tau} \ii_{\tau,\Gamma}(\varphi_{\tau})$ where
$$\varphi_{\tau}=\pp_{\tau,\Gamma}(f)\in (V_{\tau}^{\vee})^{L_H}\otimes\F(Y_{\Gamma},\V_{\tau};\NN_{\nnu(\lambda,\tau)}).$$
By Lemma~\ref{lem:regY}, we have $\varphi_{\tau}\in (V_{\tau}^{\vee})^{L_H}\otimes (\F\cap\A)(Y_{\Gamma},\V_{\tau};\NN_{\nnu(\lambda,\tau)})$ for all~$\tau$.
Since $\ii_{\tau,\Gamma}(\varphi_{\tau})\in (\F\cap\A)(X_{\Gamma};\M_{\lambda})$ by Remark~\ref{rem:specXYGamma}.(2), we obtain Theorem~\ref{thm:transfer}.
\end{proof}

\begin{proof}[Proof of Corollary~\ref{cor:L2d(XGamma,Mlambda)-orth}]
Given $\tau\in\Disc(L_K/L_H)$, let $\Char^{Z(\llll_{\C})}(\widehat{L}(\tau)) \subset \Hom_{\C\text{-}\mathrm{alg}}(Z(\llll_{\C}),\C)$ be the set of $Z(\llll_{\C})$-infinitesimal characters of elements of $\widehat{L}(\tau)$ (see \eqref{eqn:def-Supp}).
The spaces $L^2_d(Y_{\Gamma},\V_{\tau};\NN_{\nu})$, for varying $\nu\in\Char^{Z(\llll_{\C})}(\widehat{L}(\tau))$, are orthogonal to each other.
The maps
$$\ii_{\tau,\Gamma} : (V_{\tau}^{\vee})^{L_H}\otimes L^2(Y_{\Gamma},\V_{\tau})\to L^2(X_{\Gamma}),$$
for $\tau\in\Disc(L_K/L_H)$, combine into a unitary operator $\underline{\ii}_{\Gamma}$ as in \eqref{eqn:XYL2}.~In~parti\-cular, the spaces $\ii_{\tau,\Gamma}((V_{\tau}^{\vee})^{L_H}\otimes L^2(Y_{\Gamma},\V_{\tau};\NN_{\nu}))$, for varying $\tau\in\nolinebreak\Disc(L_K/L_H)$ and $\nu\in\Char^{Z(\llll_{\C})}(\widehat{L}(\tau))$, are orthogonal to each other.
We conclude using Theorem~\ref{thm:transfer}, which states that for any $\lambda\in\Hom_{\C\text{-}\mathrm{alg}}(\D_G(X),\C)$ the algebraic direct sum $\bigoplus_{\tau\in\Disc(L_K/L_H)} \ii_{\tau,\Gamma}((V_{\tau}^{\vee})^{L_H}\otimes L^2(Y_{\Gamma},\V_{\tau};\NN_{\nnu(\lambda,\tau)}))$ is dense in the Hilbert space $L^2(X_{\Gamma};\M_{\lambda})$.
\end{proof}

\begin{remark}
The proofs show that Theorem~\ref{thm:transfer} and Corollary~\ref{cor:L2d(XGamma,Mlambda)-orth} hold without assuming that $X_{\C}$ is $L_{\C}$-spherical: it is sufficient to assume, in the general setting~\ref{gen-setting}, that condition (Tf) is satisfied for the quadruple $(G,L,H,L_K)$, with transfer maps $\nnu$ and~$\llambda$.
\end{remark}

\section[Proof of Theorem~1.9 (Spectral decomposition)]{Proof of Theorem~\ref{thm:Fourier} (Spectral decomposition)} \label{subsec:proof-Fourier}


Proposition~\ref{prop:cond-Tf-A-B-satisfied} and Remark~\ref{rem:specXYGamma} imply that in the setting~\ref{spher-setting}, any spectral information on the vector bundles $\V_{\tau}$ over the Riemannian locally symmetric space~$Y_{\Gamma}$ transfers to spectral information on the pseudo-Riemannian locally symmetric space $X_{\Gamma}$, and vice versa, leading to Theorem~\ref{thm:Fourier} on the spectral decomposition of smooth functions on~$X_{\Gamma}$ by joint eigenfunctions of $\D_G(X)$.

\begin{proof}[Proof of Theorem~\ref{thm:Fourier}]
Since the reductive Lie group $L$ is of type I in the sense of von Neumann algebras \cite{har53}, the Mautner theorem \cite{mau50} states that the right regular representation of $L$ on $L^2(\Gamma\backslash L)$ decomposes uniquely into a direct integral of irreducible unitary representations:
\begin{equation} \label{eqn:L2-Gamma-L-decomp}
L^2(\Gamma\backslash L) \simeq \int_{\widehat{L}} \HHH_{\vartheta} \, \dd m(\vartheta),
\end{equation}
where $\dd m$ is a Borel measure on the unitary dual $\widehat{L}$ of~$L$ with respect to the Fell topology and $\HHH_{\vartheta}$ is a (possibly infinite) multiple of the irreducible unitary representation~$\vartheta$ depending measurably on $\vartheta\in\widehat{L}$.
We note that the Fr\'echet space $\HHH_{\vartheta}^{\infty}$ of smooth vectors is realized in $C^{\infty}(\Gamma\backslash L)$ by a Sobolev-type theorem.
There is a measurable family of continuous maps
$$\mathcal{T}_{\vartheta} : L^2(\Gamma\backslash L)^{\infty} \longrightarrow \HHH_{\vartheta}^{\infty}\subset C^{\infty}(\Gamma\backslash L),$$
for $\vartheta\in\widehat{L}$, such that for any $f\in L^2(\Gamma\backslash L)^{\infty}$ we can write $f = \int_{\widehat{L}} \mathcal{T}_{\vartheta}f \, \dd m(\vartheta)$ with
$$\Vert f\Vert_{L^2(\Gamma\backslash L)}^2 = \int_{\widehat{L}} \Vert\mathcal{T}_{\vartheta}f\Vert_{L^2(\Gamma\backslash L)}^2\ \dd m(\vartheta).$$
Let $(\widehat{L})_{L_H}$ be the subset of~$\widehat{L}$ consisting of (equivalence classes of) irreducible unitary representations of~$L$ with nonzero $L_H$-fixed vectors, where $L_H=L\cap H$.
Via the identification $L^2(X_{\Gamma})\simeq L^2(\Gamma\backslash L)^{L_H}$ of \eqref{eqn:L2XGammaL}, we can view any $f\in C_c^{\infty}(X_{\Gamma})$ as an $L_H$-fixed element of $L^2(\Gamma\backslash L)^{\infty}$ and write
\begin{equation} \label{eqn:f-Ttheta}
f = \int_{(\widehat{L})_{L_H}} \mathcal{T}_{\vartheta}f \ \dd m(\vartheta), 
\end{equation}
where $\Vert f\Vert_{L^2(X_{\Gamma})}^2 = \int_{(\widehat{L})_{L_H}} \Vert\mathcal{T}_{\vartheta}f\Vert_{L^2(X_{\Gamma})}^2\,\dd m(\vartheta)$.

Let $K$ be a maximal compact subgroup of~$G$ such that $L_K:=L\cap K$ is a maximal compact subgroup of~$L$ containing~$L_H$.
Using \eqref{eqn:f-Ttheta}, Lemma~\ref{lem:p-tau-i-tau} and \eqref{eqn:f=sum_ipf}, we obtain, for any $f\in C_c^{\infty}(X_{\Gamma})$, the decomposition
\begin{eqnarray}
f & = & \sum_{\tau\in\Disc(L_K/L_H)} (\ii_{\tau,\Gamma}\circ\pp_{\tau,\Gamma})(f)\nonumber\\
& = & \sum_{\tau\in\Disc(L_K/L_H)} (\ii_{\tau,\Gamma}\circ\pp_{\tau,\Gamma})\left(\int_{(\widehat{L})_{L_H}} \mathcal{T}_{\vartheta}f\ \dd m(\vartheta)\right)\nonumber\\
& = & \sum_{\tau\in\Disc(L_K/L_H)} \int_{(\widehat{L})_{L_H}}  (\ii_{\tau,\Gamma}\circ\pp_{\tau,\Gamma})(\mathcal{T}_{\vartheta}f)\ \dd m(\vartheta).\label{eqn:f-i-tau-p-tau-T-theta}
\end{eqnarray}

By Proposition~\ref{prop:cond-Tf-A-B-satisfied}, condition~(Tf) holds for the quadruple\linebreak $(\g_{\C},\llll_{\C},\h_{\C},\llll_{\C}\cap\kk_{\C})$, with transfer maps $\nnu$ and~$\llambda$.
Consider the map
$$\Lambda : (\widehat{L})_{L_H}\times\Disc(L_K/L_H) \longrightarrow \Hom_{\C\text{-}\mathrm{alg}}(\D_G(X),\C)$$
given by $\Lambda(\vartheta,\tau)=\llambda(\chi_{\vartheta},\tau)$, where $\chi_{\vartheta}$ is the infinitesimal character of~$\vartheta$, which we see as an element of $\Hom_{\C\text{-}\mathrm{alg}}(Z(\llll_{\C})/\Ker(\dd\ell),\C)$.
Since there are at most finitely many elements $\vartheta$ of $(\widehat{L})_{L_H}$ with the same infinitesimal character, $\Lambda^{-1}(\lambda)$ is at most countable for every~$\lambda$.
By Remark~\ref{rem:specXYGamma}, for any $f\!\in\!C_c^{\infty}(X_{\Gamma})$ and any $(\tau,\vartheta)\in (\widehat{L})_{L_H}\!\times\Disc(L_K/L_H)$ we have $(\ii_{\tau,\Gamma}\circ\pp_{\tau,\Gamma})(\mathcal{T}_{\vartheta}f)\in C^{\infty}(X_{\Gamma};\M_{\Lambda(\vartheta,\tau)})$.
We set
\begin{equation} \label{eqn:F-lambda-f}
\mathtt{F}_{\lambda} f := \sum_{(\vartheta,\tau)\in\Lambda^{-1}(\lambda)} (\ii_{\tau,\Gamma}\circ\pp_{\tau,\Gamma})(\mathcal{T}_{\vartheta}f)\ \in C^{\infty}(X_{\Gamma};\M_{\Lambda(\vartheta,\tau)}).
\end{equation}

We define a measure $\dd\mu$ on $\Hom_{\C\text{-}\mathrm{alg}}(\D_G(X),\C)$ as the pushforward by $\Lambda$ of the direct product of the measures on $\Disc(L_K/L_H)$ and $(\widehat{L})_{L_H}$.
Then \eqref{eqn:f-i-tau-p-tau-T-theta} yields that any $f\in C_c^{\infty}(X_{\Gamma})$ is expanded into joint eigenfunctions of $\D_G(X)$ as
\begin{eqnarray*}
f & = & \int_{\Hom_{\C\text{-}\mathrm{alg}}(\D_G(X),\C)}\ \sum_{(\vartheta,\tau)\in\Lambda^{-1}(\lambda)} (\ii_{\tau,\Gamma}\circ\pp_{\tau,\Gamma})(\mathcal{T}_{\vartheta}f) \ \dd\mu(\lambda)\\
& = & \int_{\Hom_{\C\text{-}\mathrm{alg}}(\D_G(X),\C)}\ \mathtt{F}_{\lambda} f \ \dd\mu(\lambda).
\end{eqnarray*}

Using the definition \eqref{eqn:F-lambda-f} of~$\mathtt{F}_{\lambda} f$ and the fact that the images of the maps~$\ii_{\tau}$ are orthogonal to each other for varying $\tau\in\Disc(L_K/L_H)$ (see Section~\ref{subsec:L2decomp}), we have
\begin{eqnarray*}
& & \int_{\Hom_{\C\text{-}\mathrm{alg}}(\D_G(X),\C)}\ \Vert\mathtt{F}_{\lambda} f\Vert_{L^2(X_{\Gamma})}^2 \ \dd\mu(\lambda)\\
& = & \int_{\Hom_{\C\text{-}\mathrm{alg}}(\D_G(X),\C)}\ \Bigg\Vert \sum_{(\vartheta,\tau)\in\Lambda^{-1}(\lambda)} (\ii_{\tau,\Gamma}\circ\pp_{\tau,\Gamma})(\mathcal{T}_{\vartheta}f) \Bigg\Vert_{L^2(X_{\Gamma})}^2 \ \dd\mu(\lambda)\\
& = & \int_{\Hom_{\C\text{-}\mathrm{alg}}(\D_G(X),\C)}\ \sum_{(\vartheta,\tau)\in\Lambda^{-1}(\lambda)} \big\Vert(\ii_{\tau,\Gamma}\circ\pp_{\tau,\Gamma})(\mathcal{T}_{\vartheta}f)\big\Vert_{L^2(X_{\Gamma})}^2 \ \dd\mu(\lambda),
\end{eqnarray*}
which is equal to
\begin{eqnarray*}
& & \int_{(\widehat{L})_{L_H}} \sum_{\tau\in\Disc(L_K/L_H)} \big\Vert(\ii_{\tau,\Gamma}\circ\pp_{\tau,\Gamma})(\mathcal{T}_{\vartheta}f)\big\Vert_{L^2(X_{\Gamma})}^2 \ \dd m(\vartheta)\\
& = & \int_{(\widehat{L})_{L_H}} \Vert \mathcal{T}_{\vartheta}f\Vert_{L^2(X_{\Gamma})}^2\ \dd m(\vartheta)
\end{eqnarray*}
by the definition of the measure $\dd\mu$ and by \eqref{eqn:XYL2}.
Thus \eqref{eqn:f-Ttheta} implies
$$\Vert f\Vert_{L^2(X_{\Gamma})}^2 = \int_{\Hom_{\C\text{-}\mathrm{alg}}(\D_G(X),\C)} \Vert\mathtt{F}_{\lambda} f\Vert_{L^2(X_{\Gamma})}^2 \ \dd\mu(\lambda).$$

If $X_{\Gamma}$ is compact, then the integral \eqref{eqn:L2-Gamma-L-decomp} is a Hilbert direct sum by a result of Gelfand--Piatetski-Schapiro (see \eg \cite[Ch.\,IV, \S\,3.1]{war72}), and accordingly the inversion formula in \eqref{eqn:f-Ttheta} is a discrete sum.
Therefore the measure $\dd\mu$ is supported on a countable subset of $\Hom_{\C\text{-}\mathrm{alg}}(\D_G(X),\C)$.
\end{proof}

\chapter{Consequences of conditions (A) and~(B) on representations of $G$ and~$L$} \label{sec:transfer}

In this chapter we analyze infinite-dimensional representations realized on function spaces on a reductive homogeneous space $X=G/H$ on which a proper subgroup $L$ of~$G$ acts spherically.
The results do not involve any discrete group~$\Gamma$.
Under conditions (A) and~(B) of Section~\ref{subsec:cond-Tf-A-B}, we establish some relations between infinite-dimensional representations of $G$ and of~$L$ (Proposition~\ref{prop:condA} and Theorem~\ref{thm:condB}).
We shall use these results in Chapter~\ref{sec:transfer-Gamma-type-I-II} to prove Theorem~\ref{thm:transfer-spec}, a key tool from which we shall derive Theorems \ref{thm:mainII} and~\ref{thm:GxG}.(3) in Chapter~\ref{sec:proof-mainII}, and Theorem~\ref{thm:Specd-lambda} in Chapter~\ref{sec:proof-Specd-lambda}.

\section{A consequence of condition~(A) in a real spherical setting}\label{subsec:transfer-spec-condA}

Condition~(A) (see Definition~\ref{def:cond-A-B}) controls the smallest $G$-module containing a given irreducible module of the subgroup~$L$.
In this chapter we show that (Harish-Chandra's) discrete series representations of the subgroup~$L$ generate a $G$-module of finite length, which is actually a direct sum of discrete series representations for $G/H$.

We work again in the general setting~\ref{gen-setting}.
Let $\tau\in\widehat{L_K}$.
Since $L_K$ is compact, any irreducible unitary representation $\vartheta$ of~$L$ that occurs in the regular representation on $L^2(Y,\V_{\tau})$ is a Harish-Chandra discrete series representation.
By Frobenius reciprocity, $[\vartheta|_{L_K}:\tau]\neq 0$, and there are at most finitely many discrete series representations $\vartheta$ of~$L$ with $[\vartheta|_{L_K}:\tau]\neq 0$ by the classification of Harish-Chandra's discrete series representations in terms of minimal $K$-type (the Blattner parameter).
We denote by $L^2_d(Y,\V_{\tau})$ the sum of irreducible unitary representations of~$L$ occurring in $L^2(Y,\V_{\tau})$.
Then we have a unitary isomorphism
$$L^2_d(Y,\V_{\tau}) \simeq \bigoplus_{\vartheta\in\Disc(L)} [\vartheta:\tau] \, \vartheta \quad \mathrm{(finite\ sum)}.$$
Note that $L^2_d(Y,\V_{\tau})=\{ 0\}$ if $\tau$ is the trivial one-dimensional representation.

\begin{proposition}\label{prop:condA}
In the general setting~\ref{gen-setting}, suppose that $X=G/H$ is $G$-real spherical and that condition~(A) holds for the quadruple $(G,L,H,L_K)$.
For any $\tau\in\Disc(L_K/L_H)$, there exist at most finitely many discrete series representations $\pi_j$ for $G/H$ such that
$$\left \{
\begin{array}{c}
  (V_{\tau}^{\vee})^{L_H} \otimes L^2_d(Y,\V_{\tau}) \subset \pp_{\tau}\big(\bigoplus_j \pi_j\big),\\
  \ii_{\tau}((V_{\tau}^{\vee})^{L_H} \otimes L^2_d(Y,\V_{\tau})) \subset \bigoplus_j \pi_j.
\end{array}
\right.$$
\end{proposition}

(By a little abuse of notation, we also write $\pi_j$ for the vector space on which it is realized.)

\begin{proof}
Let $\tau\in\Disc(L_K/L_H)$.
We write $\vartheta$ for the unitary representation of~$L$ on the Hilbert space $(V_{\tau}^{\vee})^{L_H}\otimes L^2_d(Y,\V_{\tau})$.
Since $Y=L/L_K$ is a Riemannian symmetric space, $\vartheta$ is the direct sum of at most finitely many (possibly zero) irreducible unitary representations of~$L$.
The underlying $(\llll,L_K)$-module $\vartheta_{L_K}$ is $Z(\llll_{\C})$-finite by Schur's lemma, and so $\ii_{\tau}(\vartheta_{L_K})$ is $Z(\g_{\C})$-finite by condition~(A).
Since the actions of $G$ and $Z(\g_{\C})$ commute, the $G$-span of $\ii_{\tau}(\vartheta_{L_K})$ in $(L^2\cap C^{\infty})(X)$ is $Z(\g_{\C})$-finite.
Let $V$ be the closure in $L^2(X)$ of this $G$-span.
The group $G$ acts as a unitary representation on the Hilbert space~$V$.
Since $V$ is still $Z(\g_{\C})$-finite in the distribution sense, so is the underlying $(\g,K)$-module $V_K$ in the usual sense.
Since $X$ is $G$-real spherical, $V_K$ is of finite length as a $(\g,K)$-module by Lemma~\ref{lem:spher-finite-length}.(2).
Therefore, $V$ is a direct sum of at most finitely many irreducible unitary representations $\pi_j$ of~$G$, which are discrete series representations for $G/H$ by definition.
Taking the completion in $L^2(X)$, we obtain $\ii_{\tau}(\vartheta)\subset\bigoplus_j\pi_j$.
Then we also have $\vartheta\subset\pp_{\tau}(\bigoplus_j \pi_j)$ because $\pp_{\tau}\circ\ii_{\tau}$ is the identity on $(V_{\tau}^{\vee})^{L_H}\otimes L^2(Y,\V_{\tau})$.
\end{proof}

\section{A consequence of condition~(B)}

Condition~(B) (see Definition~\ref{def:cond-A-B}) controls the restriction of representations of~$G$ to a subgroup $L$ of~$G$.
We now formulate this property in terms of the notion of \emph{discrete decomposability}.

Recall that any unitary representation of a Lie group can be decomposed into irreductible unitary representations by means of a direct integral of Hilbert spaces, and this decomposition is unique up to equivalence if the group is reductive.
However, in contrast to the case of compact groups, the decomposition may involve continuous spectrum.
For instance, the Plancherel formula for a real reductive Lie group~$G$ (\ie the irreducible decomposition of the regular representation of $G$ on $L^2(G)$) always contains continuous spectrum if $G$ is noncompact.
Discrete decomposability is a ``compact-like property'' for unitary representations: we recall the following terminology for unitary representations and its algebraic analogue for $(\g,K)$-modules.

\begin{definition} \label{def:discr-dec}
Let $G$ be a real reductive Lie group.

A unitary representation of~$G$ is \emph{discretely decomposable} if it is isomorphic to a Hilbert direct sum of irreducible unitary representations of~$G$.

A $(\g,K)$-module~$V$ is \emph{discretely decomposable} if there exists an increasing filtration $\{V_n\}_{n\in\N}$ such that $V=\bigcup_{n\in\N} V_n$ and each $V_n$ is a $(\g,K)$-module of finite length.
\end{definition}

Here we consider the question whether or not a representation $\pi$ of~$G$ is discretely decomposable when restricted to a reductive subgroup~$L$.
When $V$ is the underlying $(\g,K)$-module $\pi_K$ of a unitary representation $\pi$ of~$G$ of finite length, the fact that $V$ is discretely decomposable as an $(\llll,L_K)$-module is equivalent to the fact that it is isomorphic to an algebraic direct sum of irreducible $(\llll,L_K)$-modules (see \cite[Lem.\,1.3]{kob98c}).
In this case the restriction of the unitary representation is discretely decomposable, \ie $\pi$ decomposes into a Hilbert direct sum of irreducible unitary representations of the subgroup~$L$ \cite[Th.\,2.7]{kob00}.

The definition of discrete decomposability for an $(\llll,L_K)$-module~$V$ makes sense even when $V$ is not unitarizable.
In the following theorem, we do not assume $V$ to come from a unitary representation of~$G$.

\begin{theorem} \label{thm:condB}
In the general setting~\ref{gen-setting}, suppose that condition~(B) holds for the quadruple $(G,L,H,L_K)$.
Then
\begin{enumerate}
  \item any irreducible $(\g,K)$-module $\pi_K$ occurring as a subquotient of the space $\DD'(X)$ of distributions on~$X$ is discretely decomposable as an $(\llll,L_K)$-module;
  \item for any irreducible unitary representation $\pi$ of~$G$ realized in $\DD'(X)$, the restriction $\pi|_L$ decomposes discretely into a Hilbert direct sum of irreducible unitary representations of~$L$.
  Further, if $\pi$ is a discrete series representation for $G/H$, then any irreducible summand of the restriction $\pi|_L$ is a Harish-Chandra discrete series representation of the subgroup~$L$; in particular, $\pi^{L_K}=\{ 0\}$ if $L$ is noncompact;
  \item if $X$ is $L$-real spherical (\resp if $X_{\C}$ is $L_{\C}$-spherical), then the multiplicities in the branching law of the $(\g,K)$-module~$\pi_K$ when restricted to $(\llll,L_K)$ are finite (\resp uniformly bounded) in (1) and~(2).
\end{enumerate}
\end{theorem}

\begin{remark}
Theorem~\ref{thm:condB}.(2) applies, not only to discrete series representations for~$X$, but also to irreducible unitary representations $\pi$ of~$G$ that contribute to the continuous spectrum in the Plancherel formula of $L^2(X)$.
Branching problems without the assumption that $L$ acts properly on~$X$ were discussed in \cite{kob17}.
\end{remark}

\begin{remark}
The main goal of \cite{kob98b} was to find a sufficient condition for an irreducible unitary representation $\pi$ of~$G$ to be discretely decomposable when restricted to a subgroup~$L$, in terms of representation-theoretic invariants of~$\pi$ (\eg asymptotic $K$-support that was introduced by Kashiwara--Vergne \cite{kv79}).
Theorem~\ref{thm:condB} treats a special case, but provides another sufficient condition without using such invariants.
We note that all the concrete examples of explicit branching laws in \cite{kob93b,kob94} arise from the setup of Theorem~\ref{thm:condB}.
\end{remark}

\begin{proof}[Proof of Theorem~\ref{thm:condB}]
Let $\pi_K$ be an irreducible $(\g,K)$-module realized in $\DD'(X)$.
By Lemma~\ref{lem:spher-finite-length}.(1), it is actually realized in $C^{\infty}(X)$.
Any element of $Z(\g_{\C})$ acts on~$\pi_K$ as a scalar, hence $\pp_{\tau}(\pi_K)$ is $Z(\llll_{\C})$-finite for any $\tau\in\Disc(L_K/L_H)$, by condition~(B).
Since $Y$ is $L$-real spherical, this implies that $\pp_{\tau}(\pi_K)$ is of finite length as an $(\llll,L_K)$-module for any~$\tau$, by Lemma~\ref{lem:spher-finite-length}.(2).

\smallskip

(1) The irreducible $(\g,K)$-module $\pi_K$ contains at least one nonzero $(\llll,L_K)$-module of finite length, namely $\ii_{\tau}\circ\pp_{\tau}(\pi_K)$ for any~$\tau$ with $\pp_{\tau}(\pi_K)\neq\{ 0\}$.
Therefore, the irreducible $(\g,K)$-module $\pi_K$ is discretely decomposable as an $(\llll,L_K)$-module by \cite[Lem.\,1.5]{kob98c}.
 
\smallskip

(2) Let $\pi$ be an irreducible unitary representation of~$G$ realized in $\DD'(X)$.
The underlying $(\g,K)$-module $\pi_K$ is discretely decomposable as an $(\llll,L_K)$-module by~(1); therefore the first statement follows.
For the second statement, suppose that $\pi$ is a discrete series representation for $G/H$.
Then the underlying $(\g,K)$-module $\pi_K$ is dense in~$\pi$, and so is $\pp_{\tau}(\pi_K)$ in the Hilbert space $\pp_{\tau}(\pi)$.
Therefore the unitary representation $\pp_{\tau}(\pi)$ is a finite sum of (Harish-Chandra) discrete series representations of~$L$ for any $\tau\in\Disc(L_K/L_H)$, because $\pp_{\tau}(\pi_K)$ is of finite length as an $(\llll,L_K)$-module.
(The fact that any irreducible summand of $\pp_{\tau}(\pi)$ is a Harish-Chandra discrete series representation of~$L$ is also deduced from the general result \cite[Th.\,8.6]{kob98a}.)
The inclusion
$$\pi \subset\ \ \sumplus{\tau\in\Disc(L_K/L_H)}\ \ii_{\tau}\circ\pp_{\tau}(\pi)$$
then implies that $\pi$ decomposes discretely into a Hilbert direct sum of discrete series representations of~$L$.
If $L$ is noncompact, then there is no discrete series representation for the Riemannian symmetric space $L/L_K$, and therefore $\pp_{\tau}(\pi)^{L_K}=\{ 0\}$ for all $\tau\in\Disc(L_K/L_H)$.
Thus $\pi^{L_K}=\{ 0\}$.

\smallskip

(3) If $X$ is $L$-real spherical, then
$$\dim \Hom_{\llll,L_K}(\vartheta_{L_K},C^{\infty}(X)) < +\infty$$
for any irreducible $(\llll,L_K)$-module~$\vartheta_{L_K}$ by Lemma~\ref{lem:spher-finite-length}.(2); in particular,
$$\dim \Hom_{\llll,L_K}(\vartheta_{L_K},\pi_K) < +\infty$$
in~(1).
If $X_{\C}$ is $L_{\C}$-spherical, then Fact~\ref{fact:spherical} implies the existence of a constant $C>0$ such that
$$\dim \Hom_{\llll,L_K}(\vartheta_{L_K},C^{\infty}(X)) \leq C$$
for any irreducible $(\llll,L_K)$-module~$\vartheta_{L_K}$; in particular,
$$\dim \Hom_{\llll,L_K}(\vartheta_{L_K},\pi_K) \leq C$$
in~(1).
To control the multiplicities of the branching law in~(2), we use the inequality
$$\dim \Hom_L(\vartheta,\pi|_L) \leq \dim \Hom_{\llll,L_K}(\vartheta_{L_K},\pi_K).$$
(In our setting where restrictions are discretely decomposable, the dimensions actually coincide \cite[Th.\,2.7]{kob00}.)
\end{proof}

\begin{remark}
In the general setting~\ref{gen-setting} suppose that condition~(B) holds for the quadruple $(G,L,H,L_K)$.
If we drop the assumption that $X$ is $L$-real spherical, then the conclusion of Theorem~\ref{thm:condB}.(3) does not necessarily hold.
For example, if
$$(G,H,L) = \big({}^{\backprime}G\times\!{}^{\backprime}G, \Diag({}^{\backprime}G), {}^{\backprime}G\times\{ e\} \big)$$
where ${}^{\backprime}G$ is a noncompact semisimple Lie group, then the multiplicities in the branching law are infinite in Theorem~\ref{thm:condB}.(1)--(2) for any infinite-dimensional irreducible $(\g,K)$-module~$\pi_K$.
\end{remark}

\section{Examples where condition (A) or (B) fails} \label{subsec:ex-not-AB}

In the general setting~\ref{gen-setting}, the maps $\ii_{\tau}=\ii_{\tau,\{e\}}$ and $\pp_{\tau}=\pp_{\tau,\{e\}}$ of \eqref{eqn:i-tau} and \eqref{eqn:p-tau} are well defined, and conditions (A) and~(B) make sense for the quadruple $(G,L,H,L_K)$.
However, as mentioned in Remarks \ref{rem:AB-need-spherical} and~\ref{rem:necessary-cond}, if $X_{\C}$ is not $L_{\C}$-spherical, then neither (A) nor (B) holds in general.

Indeed, let ${}^{\backprime}G$ be a real linear reductive group.
The tensor product representation $\pi_1^{\infty}\otimes\pi_2^{\infty}$ is never $Z({}^{\backprime}\g_{\C})$-finite when $\pi_1$ and~$\pi_2$ are infinite-dimensional irreducible unitary representations of~${}^{\backprime}G$, as is derived from \cite[Th.\,6.1]{ko15}.
We use this fact to give an example where condition (A) or (B) fails when $\tau$ is the trivial one-dimensional representation of~$L_K$.

\begin{example} \label{ex:not-spher-A-B-nonsym}
Let $X={}^{\backprime}G\times\!{}^{\backprime}G$ and $G={}^{\backprime}G\times\!{}^{\backprime}G\times\!{}^{\backprime}G$.
We now view $X$ as a $G$-homogeneous space $G/H$ in two different ways, and in each case we find a closed subgroup $L$ of~$G$ such that $X_{\C}$ is not $L_{\C}$-spherical and condition (A) or~(B) fails for the quadruple $(G,L,H,L_K)$, where $L_K$ is a maximal closed subgroup of~$L$.

(1) The group $G$ acts transitively on~$X$ by
$$(g_1,g_2,g_3) \cdot (x,y) := (g_1xg_3^{-1}, g_2yg_3^{-1}).$$
The stabilizer of $e$ is the diagonal $H$ of ${}^{\backprime}G\times\!{}^{\backprime}G\times\!{}^{\backprime}G$, and so $X\simeq G/H$.
Note that $X_{\C}:=G_{\C}/H_{\C}$ is $G_{\C}$-spherical for ${}^{\backprime}G=\SL(2,\R)$ or $\SL(2,\C)$.
Let $L := {}^{\backprime}G\times\!{}^{\backprime}G\times\{e\}$.
Then $L$ acts transitively and freely (in particular, properly) on~$X$.
The group $L_H := L\cap H$ is trivial.
Let ${}^{\backprime}K$ be a maximal compact subgroup of~${}^{\backprime}G$, so that $L_K := {}^{\backprime}K\times\!{}^{\backprime}K\times\{e\}$ is a maximal compact subgroup of~$L$.
For $j=1,2$, take any ${}^{\backprime}K$-type $\tau_j$ of~$\pi_j$, and let $\tau := \tau_1^{\vee}\boxtimes\tau_2^{\vee} \in \widehat{L_K} \simeq \Disc(L_K/L_H)$.
The matrix coefficients of the outer tensor product representation $\vartheta := \pi_1^{\infty}\boxtimes\pi_2^{\infty}$ of~$L$ give a realization of $\vartheta$ in $C^{\infty}(Y,\V_{\tau})$ (where $Y:=L/L_K$) on which the center $Z(\llll_{\C})$ acts as scalars; in other words, $\vartheta\subset C^{\infty}(Y,\V_{\tau};\NN_{\nu})$ for some $\nu\in\Hom_{\C\text{-}\mathrm{alg}}(Z(\llll_{\C}),\C)$.
On the other hand, $\ii_{\tau}(\vartheta)$ is not $(1\otimes 1\otimes Z({}^{\backprime}\g_{\C}))$-finite, because the third factor ${}^{\backprime}G$ of~$G$ acts diagonally on~$X$ from the right, hence $\ii_{\tau}(\vartheta)$ is neither $\D_G(X)$-finite nor $Z(\g_{\C})$-finite.
In particular, condition~(A) fails.
We also note that the fact that $\ii_{\tau}(\vartheta)$ is not $\D_G(X)$-finite implies that $\ii_{\tau}(C^{\infty}(Y,\V_{\tau};\NN_{\nu}))$ is not contained in $C^{\infty}(X,\M_{\lambda})$ for any $\lambda\in\Hom_{\C\text{-}\mathrm{alg}}(\D_G(X),\C)$.

(2) The group $G$ acts transitively on~$X$
$$(g_1,g_2,g_3) \cdot (x,y) := (g_1x, g_2yg_3^{-1}).$$
The stabilizer of $e$ is $H:=\{e\}\times\Diag({}^{\backprime}G)$, and so $X\simeq G/H$.
Let $L := \Diag({}^{\backprime}G)\times\!{}^{\backprime}G$.
Then $L$ acts transitively and freely (in particular, properly) on~$X$.
Let ${}^{\backprime}K$ be a maximal compact subgroup of~${}^{\backprime}G$, so that $L_K := \Diag({}^{\backprime}K)\times\!{}^{\backprime}K$ is a maximal compact subgroup of~$L$.
The space $V\subset C^{\infty}({}^{\backprime}G\times\!{}^{\backprime}G)$ of matrix coefficients of $\pi_1^{\infty}\boxtimes\pi_2^{\infty}$ as above is a subrepresentation of the regular representation $C^{\infty}({}^{\backprime}G\times\!{}^{\backprime}G)$ for the action of $({}^{\backprime}G\times\!{}^{\backprime}G) \times ({}^{\backprime}G\times\!{}^{\backprime}G)$ by left and right multiplication, and restricts to a representation of~$G$.
Then $V$ is $Z(\g_{\C})$-finite.
However, $\pp_{\tau}(V)$ is not $Z(\llll_{\C})$-finite because the first factor ${}^{\backprime}G$ of~$L$ acts diagonally on $\pi_1^{\infty}\boxtimes\pi_2^{\infty}$.
Thus condition~(B) fails.
\end{example}

In Example~\ref{ex:not-spher-A-B-nonsym}.(2), condition~(B) fails but $X_{\C}$ is not $G_{\C}$-spherical.
Here is another example showing that condition~(B) may fail even when $X=G/H$ is a reductive symmetric space.

\begin{example} \label{ex:not-spher-B}
Let
$$\left\{\begin{array}{ccl}
G & = & \SO(2n,2)_0 \times \SO(2n,2)_0,\\
H & = & \Diag(\SO(2n,2)_0),\\
L & = & \SO(2n,1)_0 \times \U(n,1).
\end{array}\right.$$
Let $K$ be a maximal compact subgroup of~$G$ such that $L_K:=L\cap K$ is a maximal compact subgroup of~$L$.
The group $L$ acts transitively and properly on $X=G/H$, and $X_{\C}$ is $G_{\C}$-spherical but not $L_{\C}$-spherical.
For any holomorphic discrete series representation $\pi_1$ of $\SO(2n,2)_0$, the outer tensor product $\pi_1\boxtimes\pi_1^{\vee}$ is a discrete series representation for $G/H$.
It is not discretely decomposable when we restrict it to~$L$, in fact the first factor $\pi_1|_{\SO(2n,1)_0}$ involves continuous spectrum \cite{kob98a,oo96}.
Therefore, by Theorem~\ref{thm:condB}.(2), condition~(B) does not hold for the quadruple $(G,L,H,L_K)$.
In particular, by Lemma~\ref{lem:A-tilA-B-tilB}.(2), condition~($\widetilde{\mathrm{B}}$) does not hold for the quadruple $(\g_{\C},\llll_{\C},\h_{\C},\llll_{\C}\cap\kk_{\C})$.
Moreover, the conclusion of Lemma~\ref{lem:A-tilA-B-tilB}.(3) does not hold because at least one of the summands in the $L$-decomposition
$$\pi_1\boxtimes\pi_1^{\vee} \simeq \sumplus{\tau\in\Disc(L_K/L_H)} \pp_{\tau}(\pi_1\boxtimes\pi_1^{\vee})$$
contains continuous spectrum for $L$-irreducible decomposition, and thus $Z(\llll_{\C})$ cannot act on it as scalar multiplication.
\end{example}

\section{Existence problem of discrete series representations}

Proposition~\ref{prop:condA} and Theorem~\ref{thm:condB} have the following consequence, which is not needed in the proof but might be interesting in its own right.
Recall that the condition $\Disc(G/H)=\emptyset$ is equivalent to the rank condition \eqref{eqn:rank} \emph{not} being satisfied.

\begin{corollary}\label{cor:Disc-empty}
In the general setting~\ref{gen-setting}, assume that $X_{\C}=G_{\C}/H_{\C}$ is $G_{\C}$-spherical.
\begin{enumerate}
  \item If condition~(A) holds for the quadruple $(G,L,H,L_K)$ (in particular, if condition~($\widetilde{\mathrm{A}}$) holds for the quadruple $(\g_{\C},\llll_{\C},\h_{\C},\llll_{\C}\cap\kk_{\C})$) and if $\Disc(G/H)=\emptyset$, then $\Disc(L/L_H)=\emptyset$.
  \item If condition~(B) holds for the quadruple $(G,L,H,L_K)$ (in particular, if condition~($\widetilde{\mathrm{B}}$) holds for the quadruple $(\g_{\C},\llll_{\C},\h_{\C},\llll_{\C}\cap\kk_{\C})$) and if $\Disc(L/L_H)=\emptyset$, then $\Disc(G/H)=\emptyset$.
\end{enumerate}
\end{corollary}

Discrete series representations for $L/L_H$ form a subset of Harish-Chandra's discrete series representations for~$L$ because $L_H$ is compact.
However, this subset may be strict: $\Disc(L/L_H)=\emptyset$ does not necessarily imply $\Disc(L/\{ e\})=\emptyset$.

\begin{example}
$\Disc(\SO(2n,1)/\U(n))=\emptyset$ if and only if $n$ is odd, but $\Disc(\SO(2n,1)/\{ e\})\neq\emptyset$ for all $n\in\N_+$.
\end{example}

\begin{proof}[Proof of Corollary~\ref{cor:Disc-empty}]
\hspace{-0.08cm}Suppose condition\,(A) holds.\,For any $\vartheta\!\in\!\Disc(L/L_H)$, Proposition~\ref{prop:condA} implies the existence of finitely many $\pi_j\in\Disc(G/H)$ such that $\ii_{\tau}(\vartheta)\subset\bigoplus_j \pi_j$.
Thus $\Disc(L/L_H)\neq\emptyset$ implies\linebreak $\Disc(G/H)\neq\emptyset$.

Suppose condition~(B) holds.
By Theorem~\ref{thm:condB}.(2)--(3), any $\pi\in\Disc(G/H)$ splits into a direct sum of Harish-Chandra discrete series representations $\vartheta_j$ of~$L$.
Since $\pi$ is realized on a closed subspace of $L^2(G/H)\simeq L^2(L/L_H)$, such $\vartheta_j$ are realized on closed subspaces of $L^2(L/L_H)$, \ie belong to $\Disc(L/L_H)$.
Thus $\Disc(G/H)\neq\emptyset$ implies $\Disc(L/L_H)\neq\emptyset$.
\end{proof}

\chapter{The maps $\ii_{\tau,\Gamma}$ and $\pp_{\tau,\Gamma}$ preserve type~$\I$ and type~$\II$} \label{sec:transfer-Gamma-type-I-II}

In Section~\ref{subsec:def-type-I-II} we decomposed joint $L^2$-eigenfunctions of $\D_G(X)$ on \emph{pseudo-Rieman\-nian} locally homogeneous spaces $X_{\Gamma}=\Gamma\backslash G/H$ into type~$\I$ and type~$\II$: those of type~$\I$ arise from distribution vectors of discrete series representations for $X=G/H$, and those of type~$\II$ are defined by taking an orthogonal complement in $L^2(X_{\Gamma})$.
In this chapter we introduce an analogous notion for Hermitian vector bundles $\V_{\tau}$ over the \emph{Riemannian} locally symmetric space $Y_{\Gamma}=\Gamma\backslash L/L_K$, by using Harish-Chandra's discrete series representations for~$L$.
We prove (Theorem~\ref{thm:transfer-spec}) that the transfer maps $\nnu$ and~$\llambda$ preserve discrete spectrum of type~$\I$ and of type~$\II$ between $L^2(X_{\Gamma})$ and $L^2(Y_{\Gamma},\V_{\tau})$ if conditions (Tf), (A), (B) of Section~\ref{subsec:cond-Tf-A-B} are satisfied and $X$ is $G$-real spherical.
This is the case in the main setting~\ref{spher-setting} by Proposition~\ref{prop:cond-Tf-A-B-satisfied}, and in the group manifold case by Proposition~\ref{prop:cond-Tf-A-B-satisfied-GxG}.
The results of this chapter play a key role in the proof of Theorem~\ref{thm:Specd-lambda} (see Section~\ref{subsec:proof-Specd-lambda}).

\section{Type~$\I$ and type~$\II$ for Hermitian bundles}\label{subsec:type-I-II-Herm}

Let $Y=L/L_K$ be a Riemannian symmetric space, where $L$ is a real reductive Lie group and $L_K$ a maximal compact subgroup of~$L$.
Let $\Gamma$ be a torsion-free discrete subgroup of~$L$.
As in Section~\ref{subsec:intro-main-result}, for $(\tau,V_{\tau})\in\widehat{L_K}$ and $\F=\A$, $C^{\infty}$, $L^2$, or~$\DD'$, we denote by $\F(Y_{\Gamma},\V_{\tau})$ the space of analytic, smooth, square-integrable, or distribution sections of the Hermitian vector bundle $\V_{\tau} := \Gamma\backslash L \times_{\scriptscriptstyle L_K} V_{\tau}$ over~$Y_{\Gamma}$.
We may regard $L^2(Y,\V_{\tau})$ as a subspace of $\DD'(Y,\V_{\tau})$, where $\DD'(Y,\V_{\tau})$ is endowed with the topology coming from Remark~\ref{rem:top-distrib}.
Consider the linear map
$${p'_{\Gamma}}^{\ast} : L^2(Y_{\Gamma},\V_{\tau}) \longrightarrow \DD'(Y,\V_{\tau})$$
induced by the natural projection $p'_{\Gamma} : Y\rightarrow Y_{\Gamma}$.
For any $\C$-algebra homomorphism $\nu : Z(\llll_{\C})\to\C$, we define $L^2(Y_{\Gamma},\V_{\tau};\NN_{\nu})_{\I}$ to be the preimage, under~${p'_{\Gamma}}^{\ast}$, of the closure of $L^2(Y,\V_{\tau};\NN_{\nu})$ in $\DD'(Y,\V_{\tau})$.
Given~$\tau$, there are at most finitely many $\nu$ such that $L^2(Y,\V_{\tau};\NN_{\nu})\neq\{ 0\}$ by the Blattner formula for discrete series representations \cite{hs75}, hence there are at most finitely many $\nu$ such that $L^2(Y_{\Gamma},\V_{\tau};\NN_{\nu})_{\I}\neq\{ 0\}$.
On the other hand, there may exist countably many~$\nu$ such that $L^2(Y_{\Gamma},\V_{\tau};\NN_{\nu})_{\II}$ is nonzero.
As in Lemma~\ref{lem:L2Iclosed}, the subspace $L^2(Y_{\Gamma},\V_{\tau};\NN_{\nu})_{\I}$ is closed in the Hilbert space $L^2(Y_{\Gamma},\V_{\tau};\NN_{\nu})$; we define $L^2(Y_{\Gamma},\V_{\tau};\NN_{\nu})_{\II}$ to be its orthogonal complement in $L^2(Y_{\Gamma},\V_{\tau};\NN_{\nu})$, so that we have the Hilbert space decomposition
$$L^2(Y_{\Gamma},\V_{\tau};\NN_{\nu}) = L^2(Y_{\Gamma},\V_{\tau};\NN_{\nu})_{\I} \oplus L^2(Y_{\Gamma},\V_{\tau};\NN_{\nu})_{\II}.$$
We also set
$$L^2_d(Y_{\Gamma},\V_{\tau}) := L^2_d(Y_{\Gamma},\V_{\tau})_{\I} \oplus L^2_d(Y_{\Gamma},\V_{\tau})_{\II},$$
where
\begin{eqnarray*}
L^2_d(Y_{\Gamma},\V_{\tau})_{\I} & := & \bigoplus_{\nu}\, L^2(Y_{\Gamma},\V_{\tau};\NN_{\nu})_{\I} \quad\,\,\, \mathrm{(finite\ sum)},\\
L^2_d(Y_{\Gamma},\V_{\tau})_{\II} & := & \sumplus{\nu}\, L^2(Y_{\Gamma},\V_{\tau};\NN_{\nu})_{\II} \ \mathrm{(Hilbert\ completion)}.
\end{eqnarray*}

\begin{notation} \label{def:Vtype-I-II}
For any $\tau\in\widehat{L_K}$ and any $i=\I,\II$, we set
$$\Spec_d^{Z(\llll_{\C})}(Y_{\Gamma},\V_{\tau})_i := \big\{ \nu\in\Hom_{\C\text{-}\mathrm{alg}}(Z(\llll_{\C}),\C) : L^2(Y_{\Gamma},\V_{\tau};\NN_{\nu})_i \neq \{ 0\}\big\}$$
and
\begin{align*}
\Spec_d^{Z(\llll_{\C})}(Y_{\Gamma},\V_{\tau}) & := \big\{ \nu\in\Hom_{\C\text{-}\mathrm{alg}}(Z(\llll_{\C}),\C) : L^2(Y_{\Gamma},\V_{\tau};\NN_{\nu}) \neq \{ 0\}\big\}\\
& \ = \, \Spec_d^{Z(\llll_{\C})}(Y_{\Gamma},\V_{\tau})_{\I} \cup \Spec_d^{Z(\llll_{\C})}(Y_{\Gamma},\V_{\tau})_{\II}.
\end{align*}
\end{notation}

\section{The maps $\ii_{\tau,\Gamma}$ and~$\pp_{\tau,\Gamma}$ preserve type~$\I$ and type~$\II$}\label{subsec:compatibility-I-II}

We now go back to the general setting~\ref{gen-setting} of the paper.
For a torsion-free discrete subgroup $\Gamma$ of~$L$, recall the notation $L^2(X_{\Gamma};\M_{\lambda})_i$ and $L^2_d(X_{\Gamma})_i$ as well as $\Spec_d(X_{\Gamma})_i$ from Section~\ref{subsec:def-type-I-II}, and conditions (Tf), (A), (B) from Section~\ref{subsec:cond-Tf-A-B}.
For $(\tau,V_{\tau})\in\Disc(L_K/L_H)$ and $\nu\in\Hom_{\C\text{-}\mathrm{alg}}(Z(\llll_{\C}),\C)$, recall the notation $L^2(Y_{\Gamma},\V_{\tau};\NN_{\nu})_i$ and $L^2(Y_{\Gamma},\V_{\tau})_i$ as well as $\Spec_d^{Z(\llll_{\C})}(Y_{\Gamma},\V_{\tau})$ from Section~\ref{subsec:type-I-II-Herm}.

The following theorem shows that the pseudo-Riemannian spectrum\linebreak $\Spec_d(X_{\Gamma})_i$ (for $i=\I$ or~$\II$) is obtained from the Riemannian spectrum $\Spec_d^{Z(\llll_{\C})}(Y_{\Gamma},\V_{\tau})$ via the transfer maps $\nnu$ and~$\llambda$ of condition~(Tf).

\begin{theorem}\label{thm:transfer-spec}
In the general setting~\ref{gen-setting}, suppose that condition (Tf) is satisfied for the quadruple $(G,L,H,L_K)$, with transfer maps $\nnu$ and~$\llambda$.
Let $\Gamma$ be a torsion-free discrete subgroup of~$L$.
\begin{enumerate}
  \item If condition~(A) holds and $X$ is $G$-real spherical, then for any $(\nu,\tau)\in\Hom_{\C\text{-}\mathrm{alg}}(Z(\llll_{\C})/\Ker(\dd\ell^{\tau}),\C)\times\Disc(L_K/L_H)$,
  $$\ii_{\tau,\Gamma}\big((V_{\tau}^{\vee})^{L_H}\otimes L^2_d(Y_{\Gamma},\V_{\tau};\NN_{\nu})_{\I}\big) \subset L^2_d(X_{\Gamma};\M_{\llambda(\nu,\tau)})_{\I} \,;$$
  in particular, for any $\tau\in\Disc(L_K/L_H)$,
  $$\ii_{\tau,\Gamma}\big((V_{\tau}^{\vee})^{L_H}\otimes L^2_d(Y_{\Gamma},\V_{\tau})_{\I}\big) \subset L^2_d(X_{\Gamma})_{\I}.$$
  \item If condition~(B) holds, then for any $(\lambda,\tau)\in\Hom_{\C\text{-}\mathrm{alg}}(\D_G(X),\C)\times\linebreak\Disc(L_K/L_H)$,
  $$\pp_{\tau,\Gamma}\big(L^2_d(X_{\Gamma};\M_{\lambda})_{\I}\big) \subset (V_{\tau}^{\vee})^{L_H}\otimes L^2_d(Y_{\Gamma},\V_{\tau};\NN_{\nnu(\lambda,\tau)})_{\I} \,;$$
  in particular, for any $\tau\in\Disc(L_K/L_H)$,
  $$\pp_{\tau,\Gamma}\big(L^2_d(X_{\Gamma})_{\I}\big) \subset (V_{\tau}^{\vee})^{L_H}\otimes L^2_d(Y_{\Gamma},\V_{\tau})_{\I}.$$
  \item If condition~(B) holds, then for any $(\nu,\tau)\!\in\!\Hom_{\C\text{-}\mathrm{alg}}(Z(\llll_{\C})/\Ker(\dd\ell^{\tau}),\C)\times\Disc(L_K/L_H)$,
  $$\ii_{\tau,\Gamma}\big((V_{\tau}^{\vee})^{L_H}\otimes L^2_d(Y_{\Gamma},\V_{\tau};\NN_{\nu})_{\II}\big) \subset L^2_d(X_{\Gamma};\M_{\llambda(\nu,\tau)})_{\II} \,;$$
  in particular, for any $\tau\in\Disc(L_K/L_H)$,
  $$\ii_{\tau,\Gamma}\big((V_{\tau}^{\vee})^{L_H}\otimes L^2_d(Y_{\Gamma},\V_{\tau})_{\II}\big) \subset L^2_d(X_{\Gamma})_{\II}.$$
  \item If condition~(A) holds and $X$ is $G$-real spherical, then for any $(\lambda,\tau)\in\Hom_{\C\text{-}\mathrm{alg}}(\D_G(X),\C)\times\Disc(L_K/L_H)$,
  $$\pp_{\tau,\Gamma}\big(L^2_d(X_{\Gamma};\M_{\lambda})_{\II}\big) \subset (V_{\tau}^{\vee})^{L_H}\otimes L^2_d(Y_{\Gamma},\V_{\tau};\NN_{\nnu(\lambda,\tau)})_{\II} \,;$$
  in particular, for any $\tau\in\Disc(L_K/L_H)$,
  $$\pp_{\tau,\Gamma}\big(L^2_d(X_{\Gamma})_{\II}\big) \subset (V_{\tau}^{\vee})^{L_H}\otimes L^2_d(Y_{\Gamma},\V_{\tau})_{\II}.$$
  \item If conditions (A) and~(B) both hold and $X$ is $G$-real spherical, then for any $\lambda\in\Hom_{\C\text{-}\mathrm{alg}}(\D_G(X),\C)$,
  \begin{eqnarray*}
  L^2_d(X_{\Gamma};\M_{\lambda})_{\I} & \simeq & \ \sumplus{\tau\in\Disc(L_K/L_H)}\,  \ii_{\tau,\Gamma}\big((V_{\tau}^{\vee})^{L_H}\otimes L^2_d(Y_{\Gamma},\V_{\tau};\NN_{\nnu(\lambda,\tau)})_{\I}\big),\\
  L^2_d(X_{\Gamma};\M_{\lambda})_{\II} & \simeq & \ \sumplus{\tau\in\Disc(L_K/L_H)}\,  \ii_{\tau,\Gamma}\big((V_{\tau}^{\vee})^{L_H}\otimes L^2_d(Y_{\Gamma},\V_{\tau};\NN_{\nnu(\lambda,\tau)})_{\II}\big),
  \end{eqnarray*}
  where $\simeq$ denotes unitary equivalence between Hilbert spaces.
\end{enumerate}
\end{theorem}

Theorem~\ref{thm:transfer-spec} will be proved in Section~\ref{subsec:proof-compat-I-II}.

The following is a direct consequence of Theorem~\ref{thm:transfer-spec}.(5).

\begin{corollary}\label{cor:Specd-lambda-Wtau}
In the general setting~\ref{gen-setting}, suppose that $X$ is $G$-real spherical and that conditions (Tf), (A), (B) hold.
Let $\Gamma$ be a torsion-free discrete subgroup of~$L$.
Then for $i=\I$ or~$\II$, we have
$$\Spec_d(X_{\Gamma})_i = \bigcup_{\tau\in\Disc(L_K/L_H)} \big\{ \llambda(\nu,\tau) : \nu \in \Spec_d^{Z(\llll_{\C})}(Y_{\Gamma},\V_{\tau})_i \big\}$$
hence
$$\Spec_d(X_{\Gamma}) = \bigcup_{\tau\in\Disc(L_K/L_H)} \big\{ \llambda(\nu,\tau) : \nu \in \Spec_d^{Z(\llll_{\C})}(Y_{\Gamma},\V_{\tau}) \big\}.$$
\end{corollary}

\begin{remark}
There exist elements $\tau\in\widehat{L_K}$ such that $L^2_d(Y_{\Gamma},\V_{\tau})_{\I}=\{ 0\}$ for any~$\Gamma$: indeed, by the Blattner formula for discrete series representations \cite{hs75}, there are some ``small'' representations $\tau$ of~$L_K$ for which $L^2(Y,\V_{\tau})=L^2(L/L_K,\V_{\tau})$ has no discrete series representation (for instance the one-dimensional trivial representation $\tau$).
On the other hand, there exist elements $\tau\in\widehat{L_K}$ such that $L^2_d(Y_{\Gamma},\V_{\tau})_{\I}\neq\{ 0\}$ as soon as $\rank L=\rank L_K$, because in that case $L^2_d((\Gamma\times\{ e\})\backslash (L\times L)/\Diag(L))_{\I}\neq\{ 0\}$ by \cite{kk16}.
In the setting of Theorem~\ref{thm:transfer-spec}, we have a refinement of this statement: there exists $\tau\in\Disc(L_K/L_H)$ such that $L^2_d(Y_{\Gamma},\V_{\tau})_{\I}\neq\{ 0\}$.
\end{remark}

\section[Proof of Theorem~9.2]{Proof of Theorem~\ref{thm:transfer-spec}} \label{subsec:proof-compat-I-II}


For any torsion-free discrete subgroup $\Gamma$ of~$L$, we denote by
$$p_{\Gamma}^{\ast} : L^2(X_{\Gamma}) \longrightarrow \DD'(X)$$
the linear map induced by the natural projection $p_{\Gamma} : X\rightarrow X_{\Gamma}$.
We shall use the following observation.

\begin{observation} \label{obs:extend-i-tau}
In the general setting~\ref{gen-setting}, for any $\tau\in\Disc(L_K/L_H)$, the maps $\ii_{\tau} : (V_{\tau}^{\vee})^{L_H}\otimes L^2(Y,\V_{\tau})\rightarrow L^2(X)$ and $\pp_{\tau} : L^2(X)\rightarrow (V_{\tau}^{\vee})^{L_H}\otimes L^2(Y,\V_{\tau})$ extend to continuous maps $\ii_{\tau} : (V_{\tau}^{\vee})^{L_H}\otimes \DD'(Y,\V_{\tau})\rightarrow\DD'(X)$ and $\pp_{\tau} : \DD'(X)\rightarrow (V_{\tau}^{\vee})^{L_H}\otimes\DD'(Y,\V_{\tau})$ (for the topology given by Remark~\ref{rem:top-distrib}), and for any torsion-free discrete subgroup $\Gamma$ of~$L$ the following diagram commutes.
\end{observation}
$$\xymatrix{
(V_{\tau}^{\vee})^{L_H}\otimes L^2(Y,\V_{\tau}) \ar@{}[d]|{\text{\rotatebox{-90}{$\textstyle\subset$}}} \ar@{^{(}->}[r]^-{\ii_{\tau}}
& L^2(X) \ar@{}[d]|{\text{\rotatebox{-90}{$\textstyle\subset$}}} \ar[r]^-{\pp_{\tau}}
& (V_{\tau}^{\vee})^{L_H}\otimes L^2(Y,\V_{\tau}) \ar@{}[d]|{\text{\rotatebox{-90}{$\textstyle\subset$}}}\\
(V_{\tau}^{\vee})^{L_H}\otimes\DD'(Y,\V_{\tau}) \ar@{^{(}->}[r]^-{\ii_{\tau}}
& \DD'(X) \ar[r]^-{\pp_{\tau}}
& (V_{\tau}^{\vee})^{L_H}\otimes\DD'(Y,\V_{\tau})\\
(V_{\tau}^{\vee})^{L_H}\otimes L^2(Y_{\Gamma},\V_{\tau}) \ar@{_{(}->}[u]_{{p'_{\Gamma}}^{\!\!\ast}} \ar@{^{(}->}[r]^-{\ii_{\tau,\Gamma}}
& L^2(X_{\Gamma}) \ar@{_{(}->}[u]_{{p_{\Gamma}}^{\!\!\ast}} \ar[r]^-{\pp_{\tau,\Gamma}}
& (V_{\tau}^{\vee})^{L_H}\otimes L^2(Y_{\Gamma},\V_{\tau}) \ar@{_{(}->}[u]_{{p'_{\Gamma}}^{\!\!\ast}}
}
$$

\begin{proof}[Proof of Theorem~\ref{thm:transfer-spec}.(1)]
By Remark~\ref{rem:specXYGamma}.(2), it is sufficient to check that
$$\ii_{\tau,\Gamma}((V_{\tau}^{\vee})^{L_H}\otimes L^2_d(Y_{\Gamma},\V_{\tau})_{\I})\subset L^2_d(X_{\Gamma})_{\I}$$
for all $\tau\in\Disc(L_K/L_H)$.
Let $\varphi\in (V_{\tau}^{\vee})^{L_H}\otimes L^2_d(Y_{\Gamma},\V_{\tau})_{\I}$.
By definition of $L^2_d(Y_{\Gamma},\V_{\tau};\NN_{\nu})_{\I}$, the element ${p'_{\Gamma}}^{\ast}(\varphi)\in (V_{\tau}^{\vee})^{L_H}\otimes\DD'(Y,\V_{\tau})$ can be written as a limit of elements $\tilde{\varphi}_j\in (V_{\tau}^{\vee})^{L_H}\otimes L^2_d(Y,\V_{\tau};\NN_{\nu})$.
By Observation~\ref{obs:extend-i-tau},
$$p_{\Gamma}^{\ast}\big(\ii_{\tau,\Gamma}(\varphi)\big) = \ii_{\tau}\big({p'_{\Gamma}}^{\!\!\ast}(\varphi)\big) = \ii_{\tau}\big(\lim_j \tilde{\varphi}_j\big) = \lim_j \ii_{\tau}(\tilde{\varphi}_j).$$
If condition~(A) holds and $X$ is $G$-real spherical, then by Proposition~\ref{prop:condA} each $\ii_{\tau}(\tilde{\varphi}_j)$ is contained in a finite direct sum of discrete series representations for $G/H$.
Therefore, $\ii_{\tau,\Gamma}(\varphi)\in L^2_d(X_{\Gamma})_{\I}$.
\end{proof}

\begin{proof}[Proof of Theorem~\ref{thm:transfer-spec}.(2)]
By Remark~\ref{rem:specXYGamma}.(1), it is sufficient to prove that
$$\pp_{\tau,\Gamma}(L^2_d(X_{\Gamma})_{\I})\subset (V_{\tau}^{\vee})^{L_H}\otimes L^2_d(Y_{\Gamma},\V_{\tau})_{\I}$$
for all $\tau\in\Disc(L_K/L_H)$.
Consider $\pi\in\Disc(G/H)$, with representation space $V_{\pi}\subset L^2(X)$.
If condition~(B) holds, then by Theorem~\ref{thm:condB}.(2) the restriction $\pi|_L$ is discretely decomposable and we have a unitary equivalence of $L$-modules
$$\pi|_L \simeq\, \sumplus{\vartheta\in\widehat{L}}\ n_{\vartheta}(\pi) \, \vartheta,$$
where $\vartheta$ is a Harish-Chandra discrete series representation of~$L$.
Note that the multiplicities $n_{\vartheta}(\pi)$ are possibly infinite.
Let $\tau\in\Disc(L_K/L_H)$.
Since $\pp_{\tau} : L^2(X)\rightarrow (V_{\tau}^{\vee})^{L_H}\otimes L^2(Y,\V_{\tau})$ is an $L$-homomorphism, $\pp_{\tau}(n_{\vartheta}(\pi)\,\vartheta)$ is a multiple of discrete series representation for $(V_{\tau}^{\vee})^{L_H}\otimes L^2(Y,\V_{\tau})$, which must vanish for all but finitely many~$\vartheta$, because there are at most finitely many discrete series representations for $L^2(Y,\V_{\tau})$.
Thus $\pp_{\tau}(V_{\pi})$ is a finite sum of Harish-Chandra discrete series representations:
$$\pp_{\tau}(V_{\pi}) \simeq \bigoplus_{\vartheta\in\widehat{L}}\ n'_{\vartheta}(\pi) \, \vartheta \subset (V_{\tau}^{\vee})^{L_H} \otimes L^2_d(Y,\V_{\tau}),$$
where $n'_{\vartheta}(\pi)\leq n_{\vartheta}(\pi)$.
In particular,
$$\pp_{\tau}(\overline{V_{\pi}}) \subset \overline{(V_{\tau}^{\vee})^{L_H}\otimes L^2_d(Y,\V_{\tau})},$$
where $\overline{\ \cdot\ }$ denotes the closure in $\DD'(X)$ and in $(V_{\tau}^{\vee})^{L_H}\otimes\DD'(Y,\V_{\tau})$, respectively.

Suppose $f\in L^2_d(X_{\Gamma};\M_{\lambda})_{\I}$.
By definition, there exist at most finitely many $\pi_1,\dots,\pi_k\in\Disc(G/H)$ such that
$$p_{\Gamma}^{\ast}(f) \in \bigoplus_{j=1}^k \overline{V_{\pi_j}} \quad (\subset \DD'(X)),$$
where $p_{\Gamma}^{\ast} : L^2(X_{\Gamma})\to\DD'(X)$ is the pull-back of the projection $p_{\Gamma} : X\to X_{\Gamma}$.
By Observation~\ref{obs:extend-i-tau},
$${p'_{\Gamma}}^{\!\!\ast} (\pp_{\tau,\Gamma}f) = \pp_{\tau} (p_{\Gamma}^{\ast}f) \subset \sum_{j=1}^k \pp_{\tau}(\overline{V_{\pi_j}}) \subset \overline{(V_{\tau}^{\vee})^{L_H}\otimes L^2_d(Y,\V_{\tau})}.$$
This shows that $\pp_{\tau,\Gamma}(L^2_d(X_{\Gamma})_{\I})\subset (V_{\tau}^{\vee})^{L_H}\otimes L^2_d(Y_{\Gamma},\V_{\tau})_{\I}$.
\end{proof}

\begin{proof}[Proof of Theorem~\ref{thm:transfer-spec}.(3)]
Let $(\nu,\tau)\in\Hom_{\C\text{-}\mathrm{alg}}(Z(\llll_{\C})/\Ker(\dd\ell^{\tau}),\C)\times\linebreak\Disc(L_K/L_H)$ and $\varphi\in (V_{\tau}^{\vee})^{L_H}\otimes L^2_d(Y_{\Gamma},\V_{\tau};\NN_{\nu})_{\II}$.
By Remark~\ref{rem:specXYGamma}.(2), we have $\ii_{\tau,\Gamma}(\varphi)\in L^2_d(X_{\Gamma};\M_{\llambda(\nu,\tau)})$.
If condition~(B) holds, then by Theorem~\ref{thm:transfer-spec}.(2) we have $\pp_{\tau,\Gamma}(f)\in (V_{\tau}^{\vee})^{L_H}\otimes L^2_d(Y_{\Gamma},\V_{\tau})_{\I}$ for all\linebreak $f\in L^2_d(X_{\Gamma};\M_{\llambda(\nu,\tau)})_{\I}$.
Moreover, by Remarks \ref{rem:specXYGamma}.(1) and~\ref{rem:lambda-circ-nu} we have $\pp_{\tau,\Gamma}(f)\in (V_{\tau}^{\vee})^{L_H}\otimes L^2(Y_{\Gamma},\V_{\tau};\NN_{\nu})$.
Therefore, $(\pp_{\tau,\Gamma}(f),\varphi)_{L^2(X_{\Gamma})}=0$ by definition of $L^2_d(Y_{\Gamma},\V_{\tau})_{\I}$.
By duality (see \eqref{eqn:piadjoint}), we obtain $(f,\ii_{\tau,\Gamma}(\varphi))_{L^2(X_{\Gamma})}=\nolinebreak 0$.
Since $f$ is arbitrary, this shows that $\ii_{\tau,\Gamma}(\varphi)\in L^2_d(X_{\Gamma};\M_{\llambda(\nu,\tau)})_{\II}$.
\end{proof}

\begin{proof}[Proof of Theorem~\ref{thm:transfer-spec}.(4)]
Let $(\lambda,\tau)\in\Hom_{\C\text{-}\mathrm{alg}}(\D_G(X),\C)\times\linebreak\Disc(L_K/L_H)$ and $f\in L^2(X_{\Gamma};\M_{\lambda})_{\II}$.
By Remark~\ref{rem:specXYGamma}.(1), we have
$$\pp_{\tau,\Gamma}(f)\in (V_{\tau}^{\vee})^{L_H}\otimes L^2(Y_{\Gamma},\V_{\tau};\NN_{\nnu(\lambda,\tau)}).$$
If condition~(A) holds and $X$ is $G$-real spherical, then by Theorem~\ref{thm:transfer-spec}.(1) and Remarks \ref{rem:specXYGamma}.(2) and~\ref{rem:lambda-circ-nu} we have $\ii_{\tau,\Gamma}(\varphi)\in L^2(X_{\Gamma};\M_{\lambda})_{\I}$ for all $\varphi\in (V_{\tau}^{\vee})^{L_H}\otimes L^2(Y_{\Gamma},\V_{\tau};\NN_{\nnu(\lambda,\tau)})_{\I}$.
Therefore, $(f,\ii_{\tau,\Gamma}(\varphi))_{L^2(X_{\Gamma})}$ by definition of $L^2(X_{\Gamma};\M_{\lambda})_{\I}$.
By duality (see \eqref{eqn:piadjoint}), we obtain $(\pp_{\tau,\Gamma}(f),\varphi)_{L^2(Y_{\Gamma},\V_{\tau})}=0$.
Since $\varphi$ is arbitrary, this shows that $\pp_{\tau,\Gamma}(f)\in (V_{\tau}^{\vee})^{L_H}\otimes L^2(Y_{\Gamma},\V_{\tau};\NN_{\nnu(\lambda,\tau)})_{\II}$.
\end{proof}

\begin{proof}[Proof of Theorem~\ref{thm:transfer-spec}.(5)]
Suppose $X$ is $G$-real spherical and conditions (A) and~(B) both hold, and let $\lambda\in\Hom_{\C\text{-}\mathrm{alg}}(\D_G(X),\C)$.
Applying \eqref{eqn:XYL2} to the closed subspace $L^2(X_{\Gamma};\M_{\lambda})$ of the Hilbert space $L^2(X_{\Gamma})$, we have an isomorphism of Hilbert spaces
$$L^2(X_{\Gamma};\M_{\lambda}) \,\simeq\,\ \sumplus{\tau\in\Disc(L_K/L_H)}\ \ii_{\tau,\Gamma} \circ \pp_{\tau,\Gamma}\big(L^2(X_{\Gamma};\M_{\lambda})\big).$$
For $i=\I$ or~$\II$, it follows from Theorem~\ref{thm:transfer-spec}.(2) or~(4) that
$$\pp_{\tau,\Gamma}\big(L^2(X_{\Gamma};\M_{\lambda})_i\big) \subset (V_{\tau}^{\vee})^{L_H} \otimes L^2(Y_{\Gamma},\V_{\tau};\NN_{\nnu(\lambda,\tau)})_i$$
for all $\tau\in\Disc(L_K/L_H)$.
Therefore,
$$L^2(X_{\Gamma};\M_{\lambda})_i \,\subset\ \ \sumplus{\tau\in\Disc(L_K/L_H)}\, \ii_{\tau,\Gamma}\big((V_{\tau}^{\vee})^{L_H} \otimes L^2(Y_{\Gamma},\V_{\tau};\NN_{\nnu(\lambda,\tau)})_i\big).$$
Conversely, in view of Remark~\ref{rem:lambda-circ-nu}, it follows from Theorem~\ref{thm:transfer-spec}.(1) or~(3) that $\ii_{\tau,\Gamma}\big((V_{\tau}^{\vee})^{L_H} \otimes L^2(Y_{\Gamma},\V_{\tau};\NN_{\nnu(\lambda,\tau)})_i\big) \subset L^2(X_{\Gamma};\M_{\lambda})_i$.
The result follows.
\end{proof}

\chapter{Infinite discrete spectrum of type~$\II$} \label{sec:proof-mainII}

We have established Theorems \ref{thm:selfadj} and~\ref{thm:GxG}.(2) in Chapter~\ref{sec:strategy}, and Theorems \ref{thm:Fourier} and~\ref{thm:transfer} in Chapter~\ref{sec:transfer-Gamma}.
We now prove Theorem~\ref{thm:mainII} and Theorem~\ref{thm:GxG}.(3) by using Theorem~\ref{thm:transfer-spec} and the following classical fact in the Riemannian case (see Notation~\ref{def:Ntype-I-II} with $\g$ replaced by~$\llll$).

\begin{fact} \label{fact:inf-spec-Riem}
Let $Y=L/L_K$ be a Riemannian symmetric space, where $L$ is a real reductive Lie group and $L_K$ a maximal compact subgroup.
If $\Gamma$ is a torsion-free uniform lattice or a torsion-free arithmetic subgroup of~$L$, then $\Spec_d^{Z(\llll_{\C})}(Y_{\Gamma})$ is infinite.
\end{fact}

We give a proof of Fact~\ref{fact:inf-spec-Riem} for the sake of completeness.

\begin{proof}
The usual Laplacian $\Delta_{Y_{\Gamma}}$ on the Riemannian manifold~$Y_{\Gamma}$ has an infinite discrete spectrum, and for any eigenvalue $s$ of~$\Delta_{Y_{\Gamma}}$ the corresponding eigenspace $W_s:=\mathrm{Ker}(\Delta_{Y_{\Gamma}}-s)\subset L^2(Y_{\Gamma})$ is finite-dimensional: this holds in the compact case (\ie when $\Gamma$ is a uniform lattice in~$L$) by general results on compact Riemannian manifolds, and in the arithmetic case by work of Borel--Garland \cite{bg83}.
The center $Z(\llll_{\C})$ of the enveloping algebra $U(\llll_{\C})$, acting on $L^2(Y_{\Gamma})$ via $\dd\ell$, preserves $W_s$ and thus defines a finite-dimensional commutative subalgebra of $\End(W_s)$, which is generated by normal operators.
Therefore, the action of $Z(\llll_{\C})$ on~$W_s$ can be jointly diagonalized and $W_s$ is the direct sum of joint eigenspaces of $Z(\llll_{\C})$.
Thus the fact that the discrete spectrum of $\Delta_{Y_{\Gamma}}$ is infinite implies that $\Spec_d^{Z(\llll_{\C})}(Y_{\Gamma})$ is infinite.
\end{proof}

\begin{proposition} \label{prop:conseq-Tf-B}
In the general setting~\ref{gen-setting}, if conditions (Tf) and~(B) hold, then $\Spec_d(X_{\Gamma})_{\II}$ is infinite for any torsion-free discrete subgroup $\Gamma$ of~$L$ for which $\Spec_d^{Z(\llll_{\C})}(Y_{\Gamma})$ is infinite.
\end{proposition}

\begin{proof}
Let $\Gamma$ be a torsion-free discrete subgroup of~$L$.
Recall from Example~\ref{ex:trivial-tau} that when $\tau$ is the trivial one-dimensional representation of~$L_K$, the map $\ii_{\tau,\Gamma}$ of \eqref{eqn:i-tau} is the pull-back $q_{\Gamma}^*$ of the projection map $q_{\Gamma} : X_{\Gamma}\to Y_{\Gamma}$.
Applying Remark~\ref{rem:specXYGamma}.(2) with this trivial~$\tau$, we see that if condition (Tf) holds, then for any $\nu\in\Spec_d^{Z(\llll_{\C})}(Y_{\Gamma})$ there exists $\lambda\in\Hom_{\C\text{-}\mathrm{alg}}(\D_G(X),\C)$ (depending on~$\nu$) such that
$$q_{\Gamma}^{\ast} \, L^2(Y_{\Gamma};\NN_{\nu}) \subset L^2(X_{\Gamma};\M_{\lambda}).$$
Recall that the whole discrete spectrum of the Riemannian locally symmetric space $Y_{\Gamma}$ is of type~$\II$ (Remark~\ref{rem:type-I-II}.(3)).
If moreover condition~(B) holds, then Theorem~\ref{thm:transfer-spec}.(3) implies
$$q_{\Gamma}^{\ast} \, L^2(Y_{\Gamma};\NN_{\nu}) \subset q_{\Gamma}^{\ast} \, L^2_d(Y_{\Gamma})_{\II} \subset L^2_d(X_{\Gamma})_{\II}.$$
Thus $\lambda\in\Spec_d(X_{\Gamma})_{\II}$.
On the other hand, Remarks \ref{rem:specXYGamma}.(1) and~\ref{rem:lambda-circ-nu}.(2) imply that $\lambda$ determines~$\nu$.
Therefore, if $\Spec_d^{Z(\llll_{\C})}(Y_{\Gamma})$ is infinite, then so is $\Spec_d(X_{\Gamma})_{\II}$.
\end{proof}

\begin{proof}[Proof of Theorem~\ref{thm:mainII} and Theorem~\ref{thm:GxG}.(3)]
By Propositions\linebreak \ref{prop:cond-Tf-A-B-satisfied} and \ref{prop:cond-Tf-A-B-satisfied-GxG}, conditions (Tf) and~(B) are satisfied in the setting of Theorems \ref{thm:mainII} and~\ref{thm:GxG}.
Proposition~\ref{prop:conseq-Tf-B} then implies that $\Spec_d(X_{\Gamma})_{\II}$ is infinite whenever $\Spec_d^{Z(\llll_{\C})}(Y_{\Gamma})$ is; this is the case whenever $\Gamma$ is cocompact or arithmetic in~$L$ by Fact~\ref{fact:inf-spec-Riem}.
\end{proof}

\part{Representation-theoretic description of the discrete spectrum} \label{part:repr-spec}

In this Part~\ref{part:repr-spec}, we give a proof of Theorem~\ref{thm:Specd-lambda}, which describes the discrete spectrum of type~$\I$ and type~$\II$ of standard pseudo-Riemannian locally homogeneous spaces $X_{\Gamma}=\Gamma\backslash G/H$ with $\Gamma\subset L\subset G$ in terms of the representation theory of the reductive group~$L$ via the transfer map~$\llambda$.
For this we use the machinery developed in Part~\ref{part:main-proofs}, in particular Theorem~\ref{thm:transfer-spec}.

Before that, in Chapter~\ref{sec:conj} we find an upper estimate (Proposition~\ref{prop:SpecZg}) for the set $\Spec^{Z(\g_{\C})}(X_{\Gamma})$ of joint eigenvalues of differential operators on~$X_{\Gamma}$ coming from the center $Z(\g_{\C})$.
We also give conjectural refinements of this (Conjectures \ref{conj:unitar} and~\ref{conj:Specd-Zg}) as statements relating $L^2$-eigenvalues and unitary representations.
Evidence for these conjectures is provided in Section~\ref{subsec:AdS3} in the special case of standard $3$-dimensional anti-de Sitter manifolds~$X_{\Gamma}$, for which we show that the discrete spectrum of type~$\I$ (\resp type~$\II$) of the Laplacian $\square_{X_{\Gamma}}$ is nonpositive (\resp nonnegative).

Theorem~\ref{thm:Specd-lambda} is proved in Chapter~\ref{sec:proof-Specd-lambda}, based on a partial solution (Theorem~\ref{thm:SpecZg-Hcpt}) to Conjecture~\ref{conj:Specd-Zg}.

\chapter{A conjectural picture} \label{sec:conj}

We begin in this section with some preliminary set-up and a general conjectural picture (expressed as Conjectures \ref{conj:unitar} and~\ref{conj:Specd-Zg}) which we shall prove in some special cases.

More precisely, let $X=G/H$ be a reductive homogeneous space and $\Gamma$ a discrete subgroup of~$G$ acting properly discontinuously and freely on~$X$ (not necessarily standard in the sense of Section~\ref{subsec:intro-stand}).
Any eigenfunction $f$ on $X_{\Gamma}=\Gamma\backslash G/H$ generates a representation $U_f$ of $G$ in $\DD'(X)$.
If $X$ is $G$-real spherical, then by Lemma~\ref{lem:spher-finite-length}, this representation is of finite length, and so it would be natural to study eigenfunctions $f$ on~$X_{\Gamma}$ by the representation theory of~$G$.
However, even a basic question such as the unitarizability of~$U_f$ for $L^2$-eigenfunctions $f$ on~$X_{\Gamma}$ is not clear.
This is the object of Conjectures \ref{conj:unitar} and~\ref{conj:Specd-Zg} below.

Conjecture~\ref{conj:Specd-Zg} is true in the case when $H$ is compact (Theorem~\ref{thm:SpecZg-Hcpt}) or $X=G/H$ is a group manifold (Proposition~\ref{prop:SpecG}), as we shall prove in Sections \ref{subsec:proof-SpecZg-Hcpt} and~\ref{subsec:proof-Specd-lambda}, respectively.
In Section~\ref{subsec:AdS3} we examine in detail the case of the $3$-dimensional anti-de Sitter space, by using the classification of the unitary dual of $\SL(2,\R)$, and we provide a proof Proposition~\ref{prop:AdS-I-II}.

\section{$Z(\g_{\C})$-infinitesimal support for sets of admissible representations}\label{subsec:notation-type-I-II}

We start by introducing some general notation.

Let $G$ be a real linear reductive Lie group.
We denote by $\widehat{G}_{ad}$ the set of infinitesimal equivalence classes of irreducible admissible representations of~$G$.
Here we say that two admissible representations $\pi,\pi'$ are \emph{infinitesimally equivalent} if the underlying $(\g,K)$-modules $\pi_K,\pi'_K$ are isomorphic.
Recall that every $(\g,K)$-module $V$ of finite length always admits a \emph{globalization}, \ie a continuous representation $\pi$ of~$G$ on a complete, locally convex topological vector space such that $\pi_K\simeq V$.
Moreover, $V$ is irreducible if and only if $\pi$ is (Fact~\ref{fact:gK}).
Thus $\widehat{G}_{ad}$ is naturally identified with the set of equivalence classes of irreducible $(\g,K)$-modules.

We note that a globalization of a $(\g,K)$-module is not unique.
However, if an admissible representation $\pi$ of~$G$ of finite length is realized on a Banach space, then the smooth representation $\pi^{\infty}$ on the space of all smooth vectors (see Section~\ref{subsec:g-K-modules}) is determined only by the underlying $(\g,K)$-module.
Such $\pi^{\infty}$ is characterized by a property of \emph{moderate growth} of matrix coefficients.
By the Casselman--Wallach globalization theory \cite[II, Chap.\,2]{wal88}, there is a natural equivalence between the category of $(\g,K)$-modules $\pi_K$ of finite length and the category of admissible (smooth) representations $\pi^{\infty}$ of~$G$ of finite length that are of moderate growth.

Let $\widehat{G}$ be the unitary dual of~$G$, \ie the set of unitary equivalence classes of irreducible unitary representations of~$G$.
By a theorem of Harish-Chandra, there is a bijection between $\widehat{G}$ and the set of irreducible unitarizable $(\g,K)$-modules (see \eg \cite[I, Th.\,3.4.(2)]{wal88}).
Thus we may regard $\widehat{G}$ as a subset of~$\widehat{G}_{ad}$.
More directly, the correspondence $\pi\mapsto\pi^{\infty}$ yields an injection $\widehat{G}\hookrightarrow\nolinebreak\widehat{G}_{ad}$.

Recall that Schur's lemma holds for irreducible admissible smooth representations: the center $Z(\g_{\C})$ acts as scalars on the representation space $\pi^{\infty}$ for any $\pi\in\widehat{G}_{ad}$, yielding a $\C$-algebra homomorphism $\chi_{\pi} : Z(\g_{\C})\to\C$ called the \emph{$Z(\g_{\C})$-infinitesimal character} of~$\pi$.
For any subset $S$ of~$\widehat{G}_{ad}$, we denote by
\begin{equation}\label{eqn:def-Supp}
\Char(S) \equiv \Char^{Z(\g_{\C})}(S) \subset \Hom_{\C\text{-}\mathrm{alg}}(Z(\g_{\C}),\C)
\end{equation}
the set of $Z(\g_{\C})$-infinitesimal characters $\chi_{\pi}$ of elements $\pi\in S$.
By the Langlands classification \cite{kz82,lan89} of $\widehat{G}_{ad}$, the fiber of the projection $\widehat{G}_{ad}\to\Char(\widehat{G}_{ad})$ is finite.

For any closed subgroup $H$ of~$G$, we set
\begin{eqnarray*}
(\widehat{G}_{ad})_H & := & \big\{ \pi\in\widehat{G}_{ad} : \Hom_G\big(\pi^{\infty},\DD'(G/H)\big) \neq \{ 0\}\big\} \\
& = & \big\{ \pi\in\widehat{G}_{ad}  :  \Hom_G\big(\pi^{\infty},C^{\infty}(G/H)\big) \neq \{ 0\}\big\} \\
& = & \big\{ \pi\in\widehat{G}_{ad}  :  (\pi^{-\infty})^H \neq \{ 0\}\big\},
\end{eqnarray*}
where $(\pi^{-\infty})^H$ is the set of $H$-invariant elements in the space of distribution vectors of~$\pi$ (see Section~\ref{subsec:H+-infty}), and $(\widehat{G})_H:=\widehat{G}\cap (\widehat{G}_{ad})_H$.
If $H$ is unimodular, then $G/H$ carries a $G$-invariant Radon measure and $G$ acts on $L^2(G/H)$ as a unitary representation.
As in Section~\ref{subsec:method-transfer-map}, we denote by $\Disc(G/H)$ the set of $\pi\in\widehat{G}$ such that $\Hom_G(\pi,L^2(G/H))\neq\{0\}$.
Clearly,
\begin{equation} \label{eqn:Disc-G-hat-H}
\Disc(G/H) \subset (\widehat{G})_H \subset (\widehat{G}_{ad})_H.
\end{equation}
Since the action of $Z(\g_{\C})$ on $C^{\infty}(G/H)$ factors through the homomorphism\linebreak $\dd\ell : Z(\g_{\C})\to\D_G(X)$ of \eqref{eqn:dl-dr}, we have the following constraint on\linebreak $\Char((\widehat{G}_{ad})_H)$:
\begin{equation} \label{eqn:supp-ad}
\Char\big((\widehat{G}_{ad})_H\big) \subset \Hom_{\C\text{-}\mathrm{alg}}(Z(\g_{\C})/\Ker(\dd\ell),\C).
\end{equation}

\begin{remark}
If $H$ is compact, then $(\widehat{G})_H$ coincides with the set of $\pi\in\widehat{G}$ such that $\pi^H\neq\{ 0\}$ by the Frobenius reciprocity.
Furthermore, $\Disc(G/H)=(\widehat{G})_H=(\widehat{G}_{ad})_H$ if $G$ is compact.
On the other hand, for noncompact~$G$, the inclusions in \eqref{eqn:Disc-G-hat-H} are strict in general.
For instance, if $G/H$ is a reductive symmetric space, then the set $\Disc(G/H)$ is countable (possibly empty), whereas $(\widehat{G})_H$ contains continuously many elements \cite{osh79}.
\end{remark}

\section{A conjecture: $L^2$-eigenfunctions and unitary representations}

We now assume that $H$ is a reductive subgroup of the real reductive Lie group~$G$.
We consider a discrete subgroup $\Gamma$ of~$G$ acting properly discontinuously and freely on $X=G/H$ (not necessarily standard in the sense of Section~\ref{subsec:intro-stand}).
We assume $H$ to be noncompact and $\Gamma$ to be infinite.
Then $L^2(X_{\Gamma})$ is not a subspace of $L^2(X)$ or $L^2(\Gamma\backslash G)$ on which $G$ acts unitarily.
Nevertheless, when $X$ is $G$-real spherical, we expect $L^2$-eigenfunctions on~$X_{\Gamma}$ to be related to irreducible \emph{unitary} representations of~$G$, as follows.

For $f\in\DD'(X_{\Gamma})$, we denote by $U_f$ the minimal $G$-invariant closed subspace of $\DD'(X)$ containing $p_{\Gamma}^*f$, where $p_{\Gamma} : X\to X_{\Gamma}$ is the natural projection.
We know from Lemma~\ref{lem:spher-finite-length} that if $f\in\DD'(X_{\Gamma};\M_{\lambda})$ for some $\lambda\in\Hom_{\C\text{-}\mathrm{alg}}(\D_G(X),\C)$, then $U_f$ is of finite length as a $G$-module.

\begin{conjecture} \label{conj:unitar}
Let $X=G/H$ be a reductive homogeneous space which is $G$-real spherical, and $\Gamma$ a discrete subgroup of~$G$ acting properly discontinuously and freely on~$X$.
For any $\lambda\in\Hom_{\C\text{-}\mathrm{alg}}(\D_G(X),\C)$ and any nonzero $f\in L^2(X_{\Gamma};\M_{\lambda})$, the representation $U_f$ of~$G$ contains an irreducible unitary representation as a subrepresentation.
\end{conjecture}

Conjecture~\ref{conj:unitar} concerns the unitarity of representations.
We now reformulate it in terms of spectrum, by using the $Z(\g_{\C})$-infinitesimal character instead of the $\C$-algebra $\D_G(X)$ of invariant differential operators.
The conjectural statement \eqref{eqn:Spec-subset-Supp} below asserts that any joint $L^2$-eigenvalue of $\dd\ell(Z(\g_{\C}))$ should occur as the $Z(\g_{\C})$-infinitesimal character of some irreducible unitary representation of~$G$.
We also refine it by considering the type ($\I$ or $\II$) of the spectrum.
Recall Notation~\ref{def:Ntype-I-II} for $\Spec_d^{Z(\g_{\C})}(X_{\Gamma})_i$.

\begin{conjecture} \label{conj:Specd-Zg}
Let $X=G/H$ be a reductive homogeneous space which is $G$-real spherical.
For any discrete subgroup $\Gamma$ of~$G$ acting properly discontinuously and freely on~$X$,
\begin{equation} \label{eqn:Spec-subset-Supp}
\Spec_d^{Z(\g_{\C})}(X_{\Gamma}) \subset \Char(\widehat{G}).
\end{equation}
More precisely,
\begin{itemize}
  \item[(0)] $\Spec_d^{Z(\g_{\C})}(X_{\Gamma})\subset\Char((\widehat{G})_H)$,
  \item[(1)] $\Spec_d^{Z(\g_{\C})}(X_{\Gamma})_{\I}\subset\Char(\Disc(G/H))$,
  \item[(2)] $\Spec_d^{Z(\g_{\C})}(X_{\Gamma})_{\II}\subset\Char((\widehat{G})_H\smallsetminus\Disc(G/H))$.
\end{itemize}
\end{conjecture}

Note that (1) is clear from the definitions, so the point of the conjecture is (0) and~(2).
Statement~(0) is nontrivial because $p^*(L^2(X_{\Gamma}))\not\subset L^2(\Gamma\backslash G)$, and statement~(2) is nontrivial because $p_{\Gamma}^*(L^2(X_{\Gamma}))\not\subset L^2(X)$.
Here $p : \Gamma\backslash G\to X_{\Gamma}$ and $p_{\Gamma} : X\to X_{\Gamma}$ are the natural projections.

In the sequel, we provide evidence for Conjecture~\ref{conj:Specd-Zg}.(0):
\begin{itemize}
  \item a weaker assertion holds by dropping unitarity (Proposition~\ref{prop:SpecZg});
  \item it is true if Conjecture~\ref{conj:unitar} is (Proposition~\ref{prop:conj-unitar-implies-Specd});
  \item it is compatible with the essential self-adjointness of the Laplacian (Question~\ref{problems}.(c)), see Conjecture~\ref{conj:Specd-R} below;
  \item it holds if $H$ is compact (Theorem~\ref{thm:SpecZg-Hcpt}) or if $G/H$ is a group manifold (Proposition~\ref{prop:SpecG}).
\end{itemize}

\begin{proposition} \label{prop:SpecZg}
Let $X=G/H$ be a reductive homogeneous space which is $G$-real spherical.
For any discrete subgroup $\Gamma$ of~$G$ acting properly discontinuously and freely on~$X$,
$$\Spec^{Z(\g_{\C})}(X_{\Gamma}) \subset \Char\big((\widehat{G}_{ad})_H\big).$$
In particular,
$$\Spec_d^{Z(\g_{\C})}(X_{\Gamma}) \subset \Char\big((\widehat{G}_{ad})_H\big).$$
\end{proposition}

\begin{proof}
Suppose $\nu\in\Spec^{Z(\g_{\C})}(X_{\Gamma})$.
This means there exists a nonzero $f\in\DD'(X_{\Gamma};\NN_{\nu})$.
Recall the natural projection $p_{\Gamma} : X\rightarrow X_{\Gamma}$.
Since the pull-back $p_{\Gamma}^{\ast} : \DD'(X_{\Gamma})\to\DD'(X)$ preserves weak solutions to $(\NN_{\nu})$, we have $p_{\Gamma}^{\ast}f\in\DD'(X;\NN_{\nu})$.
Let $U_f$ be the minimal $G$-invariant closed subspace of $\DD'(X;\NN_{\nu})$ containing $p_{\Gamma}^{\ast}f$.
As a $G$-module, $U_f$ is of finite length by Lemma~\ref{lem:spher-finite-length}.
In particular, there exists an irreducible submodule $\pi_f$ of~$U_f$ that is realized in $\DD'(X)$.
Since $\pi_f\in (\widehat{G}_{ad})_H$, we conclude $\nu\in\Char((\widehat{G}_{ad})_H)$.
\end{proof}

Proposition~\ref{prop:SpecZg} shows that the set $\Spec^{Z(\g_{\C})}(X_{\Gamma})$, which is originally defined as a subset of the algebraic variety $\Hom_{\C\text{-}\mathrm{alg}}(Z(\g_{\C}),\C)$ of dimension equal to $\rank\g_{\C}$, is in fact contained in the subvariety $\Char((\widehat{G}_{ad})_H)$ of dimension $\rank X\leq\rank\g_{\C}$ if $X_{\C}=G_{\C}/H_{\C}$ is $G_{\C}$-spherical.
We refer to Table~\ref{table1} for examples with $\rank X$.
The proof of Proposition~\ref{prop:SpecZg} also shows the following.

\begin{proposition} \label{prop:conj-unitar-implies-Specd}
If $X_{\C}$ is $G_{\C}$-spherical, then Conjecture~\ref{conj:unitar} implies Conjecture~\ref{conj:Specd-Zg}.(0).
\end{proposition}

\begin{proof}
Suppose $\nu\in\Spec^{Z(\g_{\C})}(X_{\Gamma})$.
For a nonzero $f\in L^2(X_{\Gamma};\NN_{\nu})$, we consider the $G$-module $U_f\subset\DD'(X;\NN_{\nu})$ as in the proof of Proposition~\ref{prop:SpecZg}.
By definition, $Z(\g_{\C})$ acts on~$U_f$ as scalars via $\dd\ell$.
Since $X_{\C}$ is $G_{\C}$-spherical, $\D_G(X)$ is finitely generated as a $\dd\ell(Z(\g_{\C}))$-module (see Section~\ref{subsec:DGH}) and we can enlarge $U_f$ to a $\D_G(X)$-module $\tilde{U}_f$ as in the proof of Lemma~\ref{lem:A-tilA-B-tilB}.(2); then $\tilde{U}_f$ is also a $G$-module.
We note that $U_f=\tilde{U}_f$ if $\dd\ell : Z(\g_{\C})\to\D_G(X)$ is surjective.
Since the action of $\D_G(X)$ on $\tilde{U}_f$ factors through the action of a finite-dimensional commutative algebra, there is a joint eigenfunction $h\in\tilde{U}_f$ for the action of $\D_G(X)$.
If Conjecture~\ref{conj:unitar} is true, then the $G$-module $\tilde{U}_f$ contains an irreducible subrepresentation $\pi_h$ of~$G$ which is unitary.
By construction, $\pi_h\in (\widehat{G})_H$.
Since $Z(\g_{\C})$ acts on the enlarged space $\tilde{U}_f$ by the same scalar~$\nu$, we conclude that $\nu\in\Char((\widehat{G})_H)$.
\end{proof}

\section{Real spectrum}\label{subsec:real-spec}

The spectrum of any self-adjoint operator is real.
Therefore, an affirmative answer to Question~\ref{problems}.(c) on the self-adjoint extension of $(\square_{X_{\Gamma}},C^{\infty}_c(X_{\Gamma}))$ would imply the following.

\begin{conjecture} \label{conj:Specd-R}
For any reductive symmetric space $X=G/H$ and any discrete subgroup $\Gamma$ of~$G$ acting properly discontinuously and freely on~$X$, we have $\Spec_d(\square_{X_{\Gamma}})\subset\R$.
\end{conjecture}

\begin{proposition}
If $\rank X=1$, then Conjecture~\ref{conj:Specd-Zg}.(0) (hence Conjecture~\ref{conj:unitar}) implies Conjecture~\ref{conj:Specd-R}.
\end{proposition}

\begin{proof}
Let $X=G/H$ be a reductive symmetric space and $\Gamma$ a discrete subgroup $\Gamma$ of~$G$ acting properly discontinuously and freely on~$X$.
The Casimir operator $C_G\in Z(\g_{\C})$ acts as a real scalar on any irreducible \emph{unitary} representation on~$G$ in the space of smooth vectors (see Parthasarathy \cite{par80}), hence also in the space of distribution vectors.
On the other hand, $C_G$ acts as the Laplacian $\square_X$ on $X=G/H$.
If $\rank X=1$, then the $\C$-algebra $\D_G(X)$ is generated by $\dd\ell(C_G)$, and so the inclusion $\Spec_d^{Z(\g_{\C})}(X_{\Gamma})\subset\Char((\widehat{G})_H)$ implies $\Spec_d(\square_{X_{\Gamma}})\subset\R$.
\end{proof}

\begin{remark}
As in Section~\ref{subsec:intro-spec-decomp}, let $(\square_{X_{\Gamma}},\mathcal{S})$ be the closure of\linebreak $(\square_{X_{\Gamma}},C_c^{\infty}(X_{\Gamma}))$, and $(\square_{X_{\Gamma}}^*,\mathcal{S}^*)$ the adjoint of $(\square_{X_{\Gamma}},\mathcal{S})$, so that $C_c^{\infty}(X_{\Gamma})\subset\mathcal{S}\subset\mathcal{S}^*\subset L^2(X_{\Gamma})$.
Since $(\square_{X_{\Gamma}},C_c^{\infty}(X_{\Gamma}))$ is a symmetric operator, if $\square_{X_{\Gamma}}f=\lambda\,f$ for some nonzero $f\in\mathcal{S}$, then $\lambda\in\R$: indeed, writing $f=\lim_j f_j$ and $\square_{X_{\Gamma}}=\lim_j \square_{X_{\Gamma}}$ where $f_j\in C_c^{\infty}(X_{\Gamma})$, we have
$$\lambda (f,f) = (\square_{X_{\Gamma}}f,f) = \lim_j (\square_{X_{\Gamma}}f,f_j) = \lim_j (f,\square_{X_{\Gamma}}f_j) = (f,\square_{X_{\Gamma}}f) = \overline{\lambda} (f,f).$$
However, our definition of $\Spec_d(\square_{X_{\Gamma}})$ (see the beginning of Chapter~\ref{sec:intro}) uses the larger space~$\mathcal{S}^*$, and the above argument does not imply $\Spec_d(\square_{X_{\Gamma}})\subset\nolinebreak\R$.
\end{remark}

In Chapter~\ref{sec:strategy} we proved (Theorems \ref{thm:selfadj} and~\ref{thm:GxG}.(2)) that the pseudo-Riemannian Laplacian $\square_{X_{\Gamma}}$ extends uniquely to a self-adjoint operator on $L^2(X_{\Gamma})$ in the main setting~\ref{spher-setting} and in the group manifold case, and so Conjecture~\ref{conj:Specd-R} holds in these settings.

\section{The case of compact~$H$} \label{subsec:refine-SpecZg-Hcpt}

Conjecture~\ref{conj:Specd-Zg} is true for compact~$H$, as given by the following theorem, which will be proved in Section~\ref{subsec:proof-SpecZg-Hcpt}.

\begin{theorem} \label{thm:SpecZg-Hcpt}
Let $G$ be a real reductive Lie group and $H$ a compact subgroup of~$G$ such that $X=G/H$ is $G$-real spherical.
For any torsion-free discrete subgroup $\Gamma$ of~$G$,
\begin{eqnarray*}
\Spec_d^{Z(\g_{\C})}(X_{\Gamma}) & = & \Char\big(\Disc(\Gamma\backslash G) \cap (\widehat{G})_H\big),\\
\Spec_d^{Z(\g_{\C})}(X_{\Gamma})_{\I} & = & \Char\big(\Disc(\Gamma\backslash G) \cap \Disc(G/H)\big),\\
\Spec_d^{Z(\g_{\C})}(X_{\Gamma})_{\II} & = & \Char\big(\Disc(\Gamma\backslash G) \cap ((\widehat{G})_H \smallsetminus \Disc(G/H))\big).
\end{eqnarray*}
\end{theorem}

Let us now state a variant of Theorem~\ref{thm:SpecZg-Hcpt}, involving a \emph{maximal} compact subgroup $K$ of~$G$ instead of~$H$.
The Riemannian symmetric space $G/K$ is clearly $G$-real spherical.
We use the notation of Section~\ref{subsec:type-I-II-Herm} with $Y=G/K$ instead of $L/L_K$.
In particular, for $(\tau,V_{\tau})\in\widehat{K}$, we define a Hermitian vector bundle $\V_{\tau} := \Gamma\backslash G \times_K V_{\tau}$ over $Y_{\Gamma}=\Gamma\backslash G/K$ and define $\Spec_d^{Z(\g_{\C})}(Y_{\Gamma},\V_{\tau})_i$ for $i=\I$ or~$\II$ and $\Spec_d^{Z(\g_{\C})}(Y_{\Gamma},\V_{\tau})$ similarly to Notation~\ref{def:Vtype-I-II} for $Y_{\Gamma}=\Gamma\backslash G/K$.
We also define the following two subsets of~$\widehat{G}$:
\begin{itemize}
  \item $\widehat{G}(\tau)$ is the set of $\vartheta\in\widehat{G}$ such that $[\vartheta|_K:\tau]\neq 0$,
  \item $\Disc(G/K;\tau)$ is the set of $\vartheta\in\widehat{G}(\tau)$ that are Harish-Chandra discrete series representations for~$G$.
\end{itemize}
With this notation, we shall prove the following in Section~\ref{subsec:proof-SpecZg-Hcpt}.

\begin{theorem}\label{thm:Spectau}
Let $G$ be a real reductive Lie group and $K$ a maximal compact subgroup of~$G$.
For any torsion-free discrete subgroup $\Gamma$ of~$G$ and any $(\tau,V_{\tau})\in\widehat{K}$, setting $Y_{\Gamma}:=\Gamma\backslash G/K$, we have
\begin{eqnarray*}
\Spec_d^{Z(\g_{\C})}(Y_{\Gamma},\V_{\tau}) & = & \Char^{Z(\g_{\C})}\big(\Disc(\Gamma\backslash G) \cap \widehat{G}(\tau)\big),\\
\Spec_d^{Z(\g_{\C})}(Y_{\Gamma},\V_{\tau})_{\I} & = & \Char^{Z(\g_{\C})}\big(\Disc(\Gamma\backslash G) \cap \Disc(G/K;\tau)\big),\\
\Spec_d^{Z(\g_{\C})}(Y_{\Gamma},\V_{\tau})_{\II} & = & \Char^{Z(\g_{\C})}\big(\Disc(\Gamma\backslash G) \cap (\widehat{G}(\tau) \smallsetminus \Disc(G/K;\tau))\big).
\end{eqnarray*}
\end{theorem}

Theorem~\ref{thm:Spectau} actually implies Theorem~\ref{thm:SpecZg-Hcpt}.
Indeed, let $H$ be a compact subgroup of~$G$ such that $X=G/H$ is $G$-real spherical.
Consider a maximal compact subgroup $K$ of~$G$ containing~$H$.
Then the fibration $K/H\to X_{\Gamma}\to Y_{\Gamma}$ induces a decomposition
$$L^2(X_{\Gamma}) = \sumplus{\tau\in\Disc(K/H)} (V_{\tau}^{\vee})^H \otimes L^2(Y_{\Gamma},\V_{\tau})$$
(see \eqref{eqn:XYL2} with $(G,K,H)$ instead of $(L,L_K,L_H)$) and bijections
$$\Spec_d^{Z(\g_{\C})}(X_{\Gamma})_i = \bigcup_{\tau\in\Disc(K/H)} \Spec_d^{Z(\g_{\C})}(Y_{\Gamma},\V_{\tau})_i$$
for $i=\I$ or~$\II$ and
$$\Spec_d^{Z(\g_{\C})}(X_{\Gamma}) = \bigcup_{\tau\in\Disc(K/H)} \Spec_d^{Z(\g_{\C})}(Y_{\Gamma},\V_{\tau}).$$
Since
$$(\widehat{G})_H = \bigcup_{\tau\in\Disc(K/H)} \widehat{G}(\tau) \quad\mathrm{and}\quad
\Disc(G/H) = \bigcup_{\tau\in\Disc(K/H)} \Disc(G/K;\tau),$$
we see that Theorem~\ref{thm:Spectau} implies Theorem~\ref{thm:SpecZg-Hcpt}.

\section{The case of group manifolds}

Conjecture~\ref{conj:Specd-Zg} is also true in the case that $X=G/H$ is a group manifold $({}^{\backprime}G\times\!{}^{\backprime}G)/\Diag({}^{\backprime}G)$ and $\Gamma\subset{}^{\backprime}G\times\!{}^{\backprime}K$ as in Example~\ref{ex:group-manifold}.
In this case the $\C$-algebra homomorphism $\dd\ell : Z({}^{\backprime}\g_{\C})\to\D_G(X)$ of \eqref{eqn:dl-dr} is bijective, and so the statement is as follows.

\begin{proposition}\label{prop:SpecG}
Let $X=G/H$ be a group manifold $({}^{\backprime}G\times\!{}^{\backprime}G)/\Diag({}^{\backprime}G)$, where ${}^{\backprime}G$ is a real linear reductive Lie group contained in a connected complexification~${}^{\backprime}G_{\C}$.
We identify the $\C$-algebra $\D_G(X)$ with $Z({}^{\backprime}\g_{\C})$ via $\dd\ell$.
\begin{enumerate}
  \item For any discrete subgroup $\Gamma$ of $G={}^{\backprime}G\times\!{}^{\backprime}G$ acting properly discontinuously and freely on~$X$,
  $$\Spec_d(X_{\Gamma})_{\I} \subset \Char(\Disc({}^{\backprime}G)).$$
  \item Moreover, in the standard case where $\Gamma\subset L:={}^{\backprime}G\times\!{}^{\backprime}K$,
\begin{eqnarray*}
\Spec_d(X_{\Gamma}) & \subset & \Char(\widehat{{}^{\backprime}G}),\\
\Spec_d(X_{\Gamma})_{\II} & \subset & \Char(\widehat{{}^{\backprime}G} \smallsetminus \Disc({}^{\backprime}G)).
\end{eqnarray*}
\end{enumerate}
\end{proposition}

Statement~(1) is immediate from the definition.
Statement~(2) will be proved in Section~\ref{subsec:proof-Specd-lambda}, by reducing to Theorem~\ref{thm:Spectau} and using Theorem~\ref{thm:transfer-spec}.

\begin{remark} \label{rem:intersect-Spec-I-II}
In contrast to the case ${}^{\backprime}G=\SL(2,\R)$ (see Proposition~\ref{prop:AdS-I-II}), in general $\Spec_d(X_{\Gamma})_{\I}$ and $\Spec_d(X_{\Gamma})_{\II}$ may have nonempty intersection for $X=({}^{\backprime}G\times\!{}^{\backprime}G)/\Diag({}^{\backprime}G)$.
The proof of Proposition~\ref{prop:AdS-I-II} (see Section~\ref{subsec:AdS3} just below) uses the fact that for ${}^{\backprime}G=\SL(2,\R)$, irreducible unitary representations having the same infinitesimal character as the trivial one-dimensional representation are Harish-Chandra's discrete series representations; this does not hold for more general real reductive groups.
In fact, the unitarization of Zuckerman's derived functor modules $A_{\q}(0)$ for $\theta$-stable parabolic subalgebra $\q$ ($\subset\g_{\C}$) are such examples if the normalizer of $\q$ in~$G$ is noncompact \cite{vog84}.
\end{remark}

\section[The example of $X=\AdS^3$: proof of Proposition~1.13]{The example of $X=\AdS^3$: proof of Proposition~\ref{prop:AdS-I-II}} \label{subsec:AdS3}


Let $X$ be the $3$-dimensional anti-de Sitter space
$$\AdS^3 = G/H = \SO(2,2)/\SO(2,1) \simeq (\SL(2,\R)\times\SL(2,\R))/\Diag(\SL(2,\R)).$$
Then $\rank X=1$, and so the $\C$-algebra $\D_G(X)$ of $G$-invariant differential operators on~$X$ is generated by the Laplacian~$\square_X$.
Thus, for any discrete subgroup $\Gamma$ of $\SO(2,2)$ acting properly discontinuously and freely on~$X$, we may identify $\Spec_d(X_{\Gamma})$ with the discrete spectrum of the Laplacian~$\square_{X_{\Gamma}}$.

We already know from Theorem~\ref{thm:mainII} (proved in Chapter~\ref{sec:proof-mainII}) that\linebreak $\Spec_d(X_{\Gamma})_{\II}$ is infinite whenever $\Gamma$ is cocompact or arithmetic in~$L$.
We now prove the other statements of Proposition~\ref{prop:AdS-I-II} using Proposition~\ref{prop:SpecG}.
For this, recall that the irreducible unitary representations of ${}^{\backprime}G:=\SL(2,\R)$ are classified up to unitary equivalence in the following list:
\begin{center}
\begin{tabular}{cl}
$\mathbf{1}$ & trivial one-dimensional representation,\tabularnewline
$\pi_{i\nu,\delta}$ & unitary principal series representations\tabularnewline
& \quad ($\nu\geq 0$ for $\delta=+$, or $\nu>0$ for $\delta=-$),\tabularnewline
$\pi_{\lambda}$ & complementary series representations ($0<\lambda<1$),\tabularnewline
$\varpi_n^+$ & holomorphic discrete series representations ($n\in\N_+$),\tabularnewline
$\varpi_n^-$ & antiholomorphic discrete series representations ($n\in\N_+$),\tabularnewline
$\varpi_0^+$ & limit of holomorphic discrete series representations,\tabularnewline
$\varpi_0^-$ & limit of antiholomorphic discrete series representations.
\end{tabular}
\end{center}
We use the following parametrization: the smooth representations $\mathbf{1}$, $\pi_{i\nu,\delta}$, $\pi_{\lambda}$, $\varpi_n^{\pm}$, and $\varpi_0^{\pm}$ are subrepresentations of the unnormalized principal series representations $C^{\infty}({}^{\backprime}G/{}^{\backprime}P,\mathcal{L}_{s,\varepsilon})$ of ${}^{\backprime}G=\SL(2,\R)$ with $(s,\varepsilon) = (0,+)$, $(1+\nolinebreak i\nu,\delta)$, $(\lambda,+)$, $(n+1,(-1)^{n+1})$, and $(1,-1)$, respectively.
Here $\mathcal{L}_{s,\varepsilon}$ is a ${}^{\backprime}G$-equivariant line bundle over the real flag manifold ${}^{\backprime}G/{}^{\backprime}P$ associated to a one-dimensional representation of the parabolic subgroup ${}^{\backprime}P := \big\{ \big(\begin{smallmatrix} a & b\\ 0 & a^{-1}\end{smallmatrix}\big) \,:\, a\in\R^*,\ b\in\R\big\}$ given by
$$\begin{pmatrix} a & b\\ 0 & a^{-1}\end{pmatrix} \longmapsto
\left\{ \begin{array}{ll}
|a|^s & \text{for }\varepsilon=+,\\
|a|^s \, \mathrm{sgn}(a) & \text{for }\varepsilon=-.
\end{array}\right.$$
We normalize the Harish-Chandra isomorphism
$$\Hom_{\C\text{-}\mathrm{alg}}(Z({}^{\backprime}\g_{\C}),\C) \simeq \C/(\Z/2\Z), \quad \chi_{\lambda} \longleftrightarrow \lambda$$
so that the infinitesimal character of the trivial one-dimensional representation $\mathbf{1}$ is equal to $\chi_{\lambda}$ for $\lambda=1\in\C/(\Z/2\Z)$.
Then the infinitesimal character of $C^{\infty}({}^{\backprime}G/{}^{\backprime}P,\mathcal{L}_{s,\varepsilon})$ is $s-1\in\C/(\Z/2\Z)$, and therefore the infinitesimal characters of $\pi_{i\nu,\delta}$, $\pi_{\lambda}$, and $\varpi_n^{\pm}$ are given by $i\nu$, $\lambda-1$, and~$n$ respectively.

As subsets of $\C/(\Z/2\Z)$, we have
\begin{eqnarray*}
\Char\big(\Disc({}^{\backprime}G)\big) & = & \{ n : n\in\N_+\},\\
\Char\big(\widehat{{}^{\backprime}G}\smallsetminus\Disc({}^{\backprime}G)\big) & = & \{ i\nu : \nu\geq 0\} \cup \{ \mu : 0\leq\mu\leq 1\}.
\end{eqnarray*}
Therefore, by Proposition~\ref{prop:SpecG},
$$\Spec_d(X_{\Gamma})_{\I} \subset \{ n : n\in\N_+\}$$
for any discrete subgroup $\Gamma$ of $G=\SO(2,2)$ acting properly discontinuously and freely on~$X$, and
$$\Spec_d(X_{\Gamma})_{\II} \subset \{ i\nu : \nu\geq 0\} \cup \{ \mu : 0\leq\mu\leq 1\}$$
whenever $\Gamma\subset L:=\U(1,1)$.

Let $B_K$ be the Killing form on $\ssl(2,\R)$.
We use $2B_K$ to normalize the Lorentzian metric on $X\simeq\SL(2,\R)$, the Casimir element $C\in Z(\ssl(2,\C))$, and the Laplace--Beltrami operator~$\square_X$.
Then the norm of the root vector is equal to one, and $\chi_{\lambda}(C)=\frac{1}{4}(\lambda^2-1)$.
Therefore, via the bijection
\begin{eqnarray*}
\Hom_{\C\text{-}\mathrm{alg}}(Z({}^{\backprime}\g_{\C}),\C) (\simeq \C/(\Z/2\Z)) & \overset{\sim}{\longrightarrow} & \hspace{1.2cm} \C\\
\chi_{\lambda} \hspace{2.4cm} & \longmapsto & \chi_{\lambda}(C) = \frac{1}{4}(\lambda^2-1),
\end{eqnarray*}
we have
$$\Spec_d(X_{\Gamma})_{\I} \subset \Big\{\frac{1}{4}(n^2-1) : n\in\N_+\Big\} = \Big\{\frac{1}{4}k(k+2) : k\in\N\Big\}$$
for any discrete subgroup $\Gamma$ of $G=\SO(2,2)$ acting properly discontinuously and freely on~$X$, and
$$\Spec_d(X_{\Gamma})_{\II} \subset (-\infty,0]$$
whenever $\Gamma\subset L=\U(1,1)$.
By \cite[Th.\,3.8 \& 9.9]{kk16}, the set $\Spec_d(X_{\Gamma})_{\I}$ is infinite as soon as $\Gamma$ is sharp (a strong form of proper discontinuity, see \cite[Def.\,4.2]{kk16}); this includes the case that $X_{\Gamma}$ is standard; more precisely, there exists $k_0\in\N$ such that
$$\Spec_d(X_{\Gamma})_{\I} \supset \Big\{\frac{1}{4}k(k+2) : k\in\N,\ k\geq k_0\Big\}$$
if $-1\notin\Gamma$, and the same holds with $\N$ replaced by $2\N$ if $-1\in\Gamma$.

Finally, $0$ is contained in $\Spec_d(X_{\Gamma})_{\II}$ if and only if the trivial one-dimen\-sional representation~$\mathbf{1}$ contributes to the $L^2$-spectrum, which happens if and only if the constant function on~$X_{\Gamma}$ is square-integrable (see Proposition~\ref{prop:type-I-II} below for details), namely, $\mathrm{vol}(X_{\Gamma})<+\infty$.
This completes the proof of Proposition~\ref{prop:AdS-I-II}.

\begin{remark}
We may compare the example of $X=\AdS^3$ with the classical Riemannian example of $X=G/H=\SL(2,\R)/\SO(2)$.
In the latter case, $\Spec_d(X_{\Gamma})_{\I} = \emptyset$ and $\Spec_d(X_{\Gamma})_{\II} \subset (-\infty,0]$ for any discrete subgroup $\Gamma$ of~$G$ (see Remark~\ref{rem:Gamma-torsion}), and Selberg's $\frac{1}{4}$ Conjecture asserts that $\Spec_d(X_{\Gamma})_{\II} \subset (-\infty,-1/4]$ if $\Gamma$ is a congruence subgroup, namely, complementary series representations $\pi_{\lambda}$ do not contribute to the discrete spectrum.
\end{remark}

\chapter{The discrete spectrum in terms of group representations} \label{sec:proof-Specd-lambda}

The goal of this section is to prove Theorem~\ref{thm:Specd-lambda}, in the main setting~\ref{spher-setting}.
For this we provide a proof of Theorems \ref{thm:SpecZg-Hcpt} and~\ref{thm:Spectau}, which describe the discrete spectrum of type $\I$ and~$\II$ in terms of representations of~$G$ into spaces of functions (or of sections of vector bundles) on the two $G$-spaces $\Gamma\backslash G$ and $X=G/H$.
Recall that conditions (Tf), (A), (B) hold in the setting \ref{spher-setting} (Proposition~\ref{prop:cond-Tf-A-B-satisfied}); therefore Theorem~\ref{thm:Specd-lambda} follows from Corollary~\ref{cor:Specd-lambda-Wtau} and from Theorem~\ref{thm:Spectau} with $(G,H,K)$ replaced by $(L,L_H,L_K)$.

Recall from Section~\ref{subsec:def-type-I-II} that the definition of discrete spectrum of type~$\I$ is built on the $L^2$-analysis of~$X$, whereas the definition of type~$\II$ relies on the $L^2$-inner product on~$X_{\Gamma}$, which is not related in general to that of~$X$.
A key idea in the proof of Theorem~\ref{thm:SpecZg-Hcpt} is to introduce a $G$-intertwining operator $T_f \equiv T(\cdot,p^*f)$ from a $G$-submodule of $\DD'(\Gamma\backslash G)$ into a $G$-submodule of $\DD'(X)$ for every \emph{tempered} (Definition~\ref{def:tempered}) joint eigenfunction $f\in\DD'(X_{\Gamma};\NN_{\nu})$ (see Lemma~\ref{lem:Tf}), and to study carefully the dependence of the intertwining operator $T_f$ on the properties of~$f$ such as being an eigenfunction of type~$\I$ or type~$\II$ (see Proposition~\ref{prop:type-I-II}).
Here $p : \Gamma\backslash G\to X_{\Gamma}$ is the natural projection, as given by the following diagram.
\begin{equation} \label{eqn:p-p-Gamma}
\xymatrix{& \ar[dl]_{p_{\Gamma}} G \ar[dr]^p & \\ \Gamma\backslash G \ar[dr]^p & & X = G/H \ar[dl]_{p_{\Gamma}}\\
& X_{\Gamma} = \Gamma\backslash G/H & }
\end{equation}

\section{Representations $V_{\pi}$ and~$W_{\pi}$}

Let $G$ be a real reductive Lie group.
The differential of the inversion $g\mapsto g^{-1}$ of~$G$ gives rise to an antiautomorphism $\eta$ of the enveloping algebra $U(\g_{\C})$, defined by $Y_1\cdots Y_m \mapsto (-Y_m)\cdots (-Y_1)$ for all $Y_1,\dots,Y_m\in\nolinebreak\g_{\C}$.
This antiautomorphism induces an involutive automorphism of the commutative subalgebra $Z(\g_{\C})$.
Note that
\begin{equation}\label{eqn:r-eta}
\dd r(z) = \dd\ell \circ \eta(z)
\end{equation}
on $\DD'(G)$ for all $z\in Z(\g_{\C})$, where $\dd\ell,\dd r : U(\g_{\C})\to\D(G)$ are the $\C$-algebra homomorphisms given by the differentiation from the left or the right, respectively (see \eqref{eqn:dl-dr} with $H=\{ e\}$).

Let $\Gamma$ be an arbitrary torsion-free discrete subgroup of~$G$, and let $\nu\in\linebreak\Hom_{\C\text{-}\mathrm{alg}}(Z(\g_{\C}),\C)$.
As in Section~\ref{subsec:type-I-II-Zg}, we define the system $(\NN_{\nu})$ of differential equations on $\Gamma\backslash G$ by $\dd\ell(z)_{\Gamma}\,\varphi=\nu(z)\varphi$ for all $z\in Z(\g_{\C})$.
For $\F=\A$, $C^{\infty}$, $L^2$ or~$\DD'$, consider the regular representation of~$G$ on $\F(\Gamma\backslash G)$, given by $g\cdot f=f(\cdot\,g)$ for all $g\in G$ and $f\in\F(\Gamma\backslash G)$.
By \eqref{eqn:r-eta}, any $z\in Z(\g_{\C})$ acts on $\F(\Gamma\backslash G;\NN_{\nu})$ by $\nu\circ\eta(z)$.
Twisting by~$\eta$, we define the set
\begin{equation} \label{eqn:hat-G-nu}
\widehat{G}_{\nu} := \{ \pi\in\widehat{G} : \chi_{\pi}=\nu\circ\eta\},
\end{equation}
which is finite.

Similarly to Lemma~\ref{lem:L2Iclosed}, the subspace $L^2(\Gamma\backslash G;\NN_{\nu})$ is closed in the Hilbert space $L^2(\Gamma\backslash G)$.
Moreover, the system $(\NN_{\nu})$ on $\Gamma\backslash G$ is right-$G$-invariant.
Thus we obtain a unitary representation of~$G$ on $L^2(\Gamma\backslash G;\NN_{\nu})$.
This unitary representation is a finite direct sum of isotypic unitary representations of~$G$:
\begin{equation}\label{eqn:Vi}
L^2(\Gamma\backslash G;\NN_{\nu}) \simeq \bigoplus_{\pi\in\widehat{G}_{\nu}} V_{\pi}.
\end{equation}
The representation $V_{\pi}$ is unitarily equivalent to $\Hom_G(\pi,L^2(\Gamma\backslash G))\otimes\pi$, and the multiplicity of $\pi$ in~$V_{\pi}$ is $\dim_{\C}\Hom_G(\pi,L^2(\Gamma\backslash G))$, which may be $0$ or $+\infty$ (since we do not impose any assumption on~$\Gamma$ such as $\mathrm{vol}(\Gamma\backslash G)<+\infty$).

According to the decomposition \eqref{eqn:Vi}, we have a finite direct sum decomposition
\begin{equation}\label{eqn:Vidistr}
\left\{ \begin{array}{lll}
L^2(\Gamma\backslash G;\NN_{\nu})^{\infty} & \simeq & \bigoplus_{\pi\in\widehat{G}_{\nu}} V_{\pi}^{\infty},\\
L^2(\Gamma\backslash G;\NN_{\nu})^{-\infty} & \simeq & \bigoplus_{\pi\in\widehat{G}_{\nu}} V_{\pi}^{-\infty}.
\end{array}\right.
\end{equation}
By a Sobolev-type theorem, $L^2(\Gamma\backslash G)^{\infty}\subset C^{\infty}(\Gamma\backslash G)$; in particular,
$$V_{\pi}^{\infty} \subset L^2(\Gamma\backslash G;\NN_{\nu})^{\infty} \subset C^{\infty}(\Gamma\backslash G).$$
Therefore, the sesquilinear continuous map $C^{\infty}(\Gamma\backslash G)\to\C$ sending $f$ to $\overline{f(\Gamma e)}$ induces elements of $L^2(\Gamma\backslash G;\NN_{\nu})^{-\infty}$ and of $V_{\pi}^{-\infty}$ for $\pi\in\widehat{G}_{\nu}$, which will be denoted by $\delta$ and~$\delta_{\pi}$, respectively.
Then
$$\delta = \sum_{\pi\in\widehat{G}_{\nu}} \delta_{\pi}$$
according to the decomposition \eqref{eqn:Vidistr}.
Clearly $\delta\in L^2(\Gamma\backslash G;\NN_{\nu})^{-\infty}$ is a cyclic vector.

Let $H$ be a reductive subgroup of~$G$ such that $X=G/H$ is $G$-real spherical.
Similarly to $V_{\pi}$ for $\Gamma\backslash G$, we now introduce a $G$-module $W_{\pi}$ for $X=G/H$.
A difference is that we consider the space $\DD'(X)$ fo distributions on~$X$ rather than $L^2(X)$.
For any $\nu\in\Spec^{Z(\g_{\C})}(X_{\Gamma})$ and any $\pi\in\widehat{G}_{\nu}$, let
\begin{equation} \label{eqn:W}
W_{\pi} \equiv W_{\pi}(H) = \overline{\sum_A A(\pi_K)} \ \ \subset\ \DD'(X),
\end{equation}
where $A$ ranges through $\Hom_{\g,K}(\pi_K,\DD'(X))$ and $\overline{\,\cdot\,}$ denotes the closure in $\DD'(X)$.
For any~$A$, the image $A(\pi_K)$ is contained in $\DD(X;\NN_{\nu\circ\eta})$ which is a $G$-module of finite length, hence $W_{\pi}$ is a $G$-submodule of $\DD'(X)$.
Moreover, $N_{\pi} := \dim\Hom_{\g,K}(\pi_K,\DD'(X)) < +\infty$, and the underlying $(\g,K)$-module $(W_{\pi})_K$ is isomorphic to a direct sum of $N_{\pi}$ copies of~$\pi_K$.
We note that $\DD'(X;\NN_{\nu\circ\eta})$ is not always completely reducible, and the quotient $(\DD'(X;\NN_{\nu\circ\eta})/W_{\pi})_K$ may contain an irreducible submodule which is isomorphic to~$\pi_K$.
We have $W_{\pi}\neq\{0\}$ if and only if $\pi\in(\widehat{G})_H$.
Moreover, $W_{\pi}\cap\bigoplus_{\pi'\in\widehat{G}_{\nu}\smallsetminus\{\pi\}} W_{\pi'}=\{0\}$.

Here is a brief summary concerning the two representations $V_{\pi}$ and~$W_{\pi}$.

\begin{lemma} \label{lem:recap-V-W-pi}
Suppose $X=G/H$ is $G$-real spherical.
\begin{enumerate}
  \item The group~$G$ acts on~$V_{\pi}$ as a unitary representation, and $V_{\pi}$ is the maximal $G$-invariant closed subspace of $L^2(\Gamma\backslash G)$ which is isotypic to~$\pi$.
  \item The group~$G$ acts on~$W_{\pi}$ as a continuous representation of finite length, and $W_{\pi}$ is the maximal $G$-invariant closed subspace of $\DD'(X)$ whose underlying $(\g,K)$-module is a multiple of~$\pi_K$.
\end{enumerate}
\end{lemma}

\begin{proof}
Statement~(1) is clear.
By Lemma~\ref{lem:spher-finite-length}, the regular representation of~$G$ on the complete, locally convex topological space
$$\DD'(X;\NN_{\nu\circ\eta}) = \{ F\in\DD'(X)  : \dd\ell(z)F = \nu\circ\eta(z)F \ \textrm{for all} \ z\in Z(\g_{\C})\} $$
is of finite length (but not necessarily completely reducible).
Thus statement~(2) follows from Fact~\ref{fact:gK}.
\end{proof}

We shall relate $V_{\pi}$ and~$W_{\pi}$ in Proposition~\ref{prop:type-I-II}.

\section{Intertwining operators associated with eigenfunctions on $X_{\Gamma}$}

Recall the projections $p : \Gamma\backslash G\to X_{\Gamma}$ and $p_{\Gamma} : X\to X_{\Gamma}$ from \eqref{eqn:p-p-Gamma}.
Given $\nu\in\Hom_{\C\text{-}\mathrm{alg}}(Z(\g_{\C}),\C)$ and a joint eigenfunction $f\in L^2(X_{\Gamma};\NN_{\nu})$, we can consider two $G$-modules: the $G$-submodule generated by $p^*f\in\DD'(\Gamma\backslash G;\NN_{\nu})$ in the right regular representation on $\DD'(\Gamma\backslash G)$, and the $G$-submodule generated by $p_{\Gamma}^*f\in\DD'(X;\NN_{\nu})$ in the left regular representation on $\DD'(X)$.
We do not expect these two $G$-modules to be isomorphic to each other.
Instead, in Lemma~\ref{lem:Tf} below we construct a $G$-intertwining operator
$$T(\cdot,p^*f) : L^2(\Gamma\backslash G;\NN_{\nu})^{-\infty} \longrightarrow \DD'(X;\NN_{\nu\circ\eta})$$
for each tempered eigenfunction $f\in\DD'(X_{\Gamma};\NN_{\nu})$.
This intertwining operator $T(\cdot,p^*f)$ depends on the eigenfunction~$f$, and we shall formulate this dependency in terms of representation theory in Lemma~\ref{lem:Tf}.(3), which will play a crucial role in proving Theorem~\ref{thm:SpecZg-Hcpt} in Section~\ref{subsec:proof-SpecZg-Hcpt}.
Here we use the following terminology.

\begin{definition} \label{def:tempered}
Let $\nu\in\Hom_{\C\text{-}\mathrm{alg}}(Z(\g_{\C}),\C)$.
An eigenfunction $f\in\DD'(X_{\Gamma};\NN_{\nu})$ is called \emph{tempered} if $p^*f\in\DD'(\Gamma\backslash G;\NN_{\nu})$ belongs to $L^2(\Gamma\backslash G;\NN_{\nu})^{-\infty}$.
\end{definition}

(Recall that $G$ acts on $\HHH:=L^2(\Gamma\backslash G;\NN_{\nu})$ as a unitary representation; $\HHH^{\infty}\subset\HHH\subset\HHH^{-\infty}$ is the Gelfand triple \eqref{eqn:Gelfand} associated with~$\HHH$.)

If $H$ is compact, then any $L^2$-eigenfunction $f$ is tempered because $p^*f\in\linebreak L^2(\Gamma\backslash G;\NN_{\nu})\subset L^2(\Gamma\backslash G;\NN_{\nu})^{-\infty}$.

\begin{lemma}\label{lem:Tf}
Suppose $X=G/H$ is $G$-real spherical.
Taking matrix coefficients for distribution vectors of the unitary representations of $G$ on $L^2(\Gamma\backslash G;\NN_{\nu})$ induces a sesquilinear map
$$T : L^2(\Gamma\backslash G;\NN_{\nu})^{-\infty} \times \big(L^2(\Gamma\backslash G;\NN_{\nu})^{-\infty}\big)^H \longrightarrow \DD'(X;\NN_{\nu\circ\eta})$$
with the following properties.
Let $f\in\DD'(X_{\Gamma};\NN_{\nu})$ be any tempered eigenfunction, namely $p^*f\in\DD'(\Gamma\backslash G)$ belongs to $(L^2(\Gamma\backslash G;\NN_{\nu})^{-\infty})^H$.
Then
\begin{enumerate}
  \item the map $T_f := T(\cdot,p^*f) : L^2(\Gamma\backslash G;\NN_{\nu})^{-\infty} \rightarrow \DD'(X;\NN_{\nu\circ\eta})$ is a continuous $G$-homomorphism.
  \item $T(\delta,p^*f) = \overline{p_{\Gamma}^*f}$, where $\overline{p_{\Gamma}^*f}$ is the complex conjugate of $p_{\Gamma}^{\ast}f$;
  \item writing $p^*f=\sum_{\pi\in\widehat{G}_{\nu}} F_{\pi}$ according to the decomposition \eqref{eqn:Vidistr}, we have, for all $\pi\in\widehat{G}_{\nu}$,
  \begin{eqnarray}
  T(\delta,F_{\pi}) = T(\delta_{\pi},p^*f) = T(\delta_{\pi},F_{\pi}),\label{eqn:T-delta-F-pi}\\
  T\big(L^2(\Gamma\backslash G;\NN_{\nu})^{-\infty},F_{\pi}\big) = T(V_{\pi}^{-\infty},p^*f) \subset W_{\pi}.\label{eqn:T-V-W-pi}
  \end{eqnarray}
\end{enumerate}
\end{lemma}

\begin{proof}
(1) Let $\varpi$ be a unitary representation of~$G$ on the Hilbert space $\HHH = L^2(\Gamma\backslash G;\NN_{\nu})$, and consider the continuous map
$$T : \HHH^{-\infty} \times \HHH^{-\infty} \longrightarrow \DD'(G)$$
of \eqref{eqn:matdist} sending $(u,F)\in\HHH^{-\infty}\times\HHH^{-\infty}$ to the corresponding matrix coefficient for distribution vectors.
If $\varpi(h)F=F$ for all $h\in H$, then $T(u,F)$ is invariant under the right action of~$H$ by \eqref{eqn:T-bi-equiv}, and therefore $T$ induces a map
$$\HHH^{-\infty} \times (\HHH^{-\infty})^H \longrightarrow \DD'(X),$$
which we still denote by~$T$.
By \eqref{eqn:T-bi-equiv} again, $T(\cdot,F)$ is a continuous $G$-homomorphism from $\HHH^{-\infty}$ to $\DD'(X)$.
We conclude by taking $F:=p^*f$ and using the fact that the center $Z(\g_{\C})$ acts on $\HHH^{-\infty}$ by the scalar $\nu\circ\eta$.

(2) By the definition \eqref{eqn:matdist} of~$T$ and the definition of~$\delta$, the element $T_f(\delta)\in\DD'(G)$ sends a test function $\varphi\in C_c^{\infty}(G)$ to
$$\int_G \varphi(g) \ \overline{p^*f(\Gamma g)} \ \dd g = \int_G \varphi(g) \ \overline{p_{\Gamma}^*p^*f(g)} \ \dd g.$$
We define $\varphi^H\in C_c^{\infty}(X)$ by $\varphi^H(gH) := \int_H \varphi(gh) \, \dd h$.
Then $\overline{p_{\Gamma}^*f}\in\DD'(X)$, regarded as an $H$-invariant distribution on~$G$, sends $\varphi$ to
$$\int_X \varphi^H(x) \ \overline{p_{\Gamma}^*f(x)} \ \dd x = \int_G \varphi(g) \ \overline{p^*p_{\Gamma}^*f(g)} \ \dd g.$$
We conclude using the equality $p_{\Gamma}^*\circ p^* = p^*\circ p_{\Gamma}^*$ from Diagram~\eqref{eqn:p-p-Gamma}.

(3) Let us first check that $T(V_{\pi}^{-\infty},p^*f)\subset W_{\pi}$.
By~(1),
$$T_f : L^2(\Gamma\backslash G;\NN_{\nu})^{-\infty} \longrightarrow \DD'(X;\NN_{\nu\circ\eta})$$
is a $G$-intertwining operator and $V_{\pi}$ is a unitary representation of~$G$ which is a multiple of the single irreducible unitary representation~$\pi$ (in particular, it is discretely decomposable).
Since $\DD'(X;\NN_{\nu\circ\eta})$ is of finite length by Lemma~\ref{lem:spher-finite-length}, the image $T_f(V_{\pi}^{-\infty})$ is a direct sum of finitely many copies of the $(\g,K)$-module $\pi_K$ by Lemma~\ref{lem:noext} below.
By definition \eqref{eqn:W} of~$W_{\pi}$, we conclude $T_f(V_{\pi}^{-\infty})\subset W_{\pi}$.
The equalities in \eqref{eqn:T-delta-F-pi} and \eqref{eqn:T-V-W-pi} follow from the fact that $T(u,v)=0$ for all $u\in V_{\pi}^{-\infty}$ and $v\in V_{\pi'}^{-\infty}$ with $\pi\not\simeq\pi'$.
\end{proof}

In general, the space of distribution vectors of a unitary representation $\pi$ could be huge.
However, it behaves in a reasonable way if $\pi$ is discretely decomposable (Definition~\ref{def:discr-dec}), as follows.

\begin{lemma}\label{lem:noext}
Let $U$ be a discretely decomposable unitary representation of a real reductive Lie group~$G$ and $V$ a continuous representation of~$G$ of finite length on a complete, locally convex vector space.
Suppose $T : U^{-\infty}\to\nolinebreak V$ is a continuous $G$-homomorphism and let $W$ be the closure of $T(U^{-\infty})$ in~$V$.
Then the underlying $(\g,K)$-module $W_K$ is completely reducible and unitarizable.
More precisely, if $U \simeq {\sum_{\pi\in\widehat{G}}}^{\oplus} m_{\pi}\,\pi$ (Hilbert direct sum) where $m_{\pi}\in\N\cup\{\infty\}$, then $W_K$ is isomorphic to a finite direct sum $\bigoplus_{\pi\in\widehat{G}} n_{\pi}\,\pi$ where $n_{\pi}\leq m_{\pi}$ for all~$\pi$.
\end{lemma}

\begin{proof}
We first consider the case where the unitary representation $U$ has finite length.
Since the underlying $(\g,K)$-module $U_K$ is a direct sum of finitely many irreducible, unitarizable $(\g,K)$-modules, so is $T(U_K)$.
Since $U_K$ is dense in $U^{-\infty}$, so is $T(U_K)$ in~$W$.
Thus the two $(\g,K)$-modules $T(U_K)$ and~$W_K$ coincide by Fact~\ref{fact:gK}.
In particular, $W_K$ is completely reducible and unitarizable.

Suppose now that $U$ is a general discretely decomposable unitary representation of~$G$.
Let $(U_N)_N$ be an increasing sequence of closed $G$-invariant subspaces such that $U=\overline{\bigcup_N U_N}$ and that the unitary representation $U_N$ is of finite length for any~$N$.
We regard $(U_N)^{-\infty}$ as a subspace of $U^{-\infty}$ by using the orthogonal decomposition $U=U_N\oplus (U_N)^{\perp}$.

If $W'$ is a closed $G$-invariant proper subspace of~$W$ such that $T(U_N^{-\infty})\subset W'$ for all~$N$, then
$$T(U^{-\infty}) \,\subset\, \overline{\bigcup_N T(U_N^{-\infty})} \,\subset\, W',$$
and so $W'=W$ since $\overline{T(U^{-\infty})}=W$.
This shows that $\overline{T(U_N^{-\infty})}=W$ for some~$N$.
Then the conclusion of the lemma follows from the case of finite length.
\end{proof}

\section[A preliminary result on Harish-Chandra discrete series]{A preliminary result on Harish-Chandra discrete series representations}

In order to prove Theorems \ref{thm:SpecZg-Hcpt} and~\ref{thm:Spectau}, we will need the following.

\begin{proposition}\label{prop:disc-series-sph-Hcpt}
Suppose $X=G/H$ is $G$-real spherical, with $H$ compact.
Let $(\pi,\HHH)$ be a Harish-Chandra discrete series representation of~$G$.
If $\psi\in\Hom_{\g,K}(\pi_K,\DD'(X))$, then $\psi(\HHH_K)\subset L^2(X)$.
\end{proposition}

\begin{remark}
This is not true anymore if we drop one of the assumptions on~$H$, as follows.

1) $H$ is noncompact.
For instance, let $X=G/H$ be the reductive symmetric space $\SO(n+1,1)/\SO(n,1)$ with $n\geq 3$.
Then there exists $(\pi,\HHH)\in\Disc(G/H)$ such that $\Hom_{\g,K}(\pi_K,\DD'(X))$ contains two linearly independent elements $\psi_1,\psi_2$ with $\psi_1(\HHH_K)\subset L^2(X)$ and $\psi_2(\HHH_K)\cap L^2(X)=\{0\}$ (see \cite{mo84,osh79}).

2) $H$ is compact but $G/H$ is not $G$-real spherical.
For instance, take $H=\{e\}$, any $(\pi,\HHH)\in\Disc(G)$, and any $w\in\HHH^{-\infty}\smallsetminus\HHH$.
Define a $(\g,K)$-homomorphism $\psi : \HHH_K\to C^{\infty}(G)$ by $v\mapsto (w,\pi(g)^{-1}v)$, with the notation of Section~\ref{subsec:g-K-modules}.
Then $\psi(\HHH_K)\cap L^2(G)=\{0\}$.
\end{remark}

Proposition~\ref{prop:disc-series-sph-Hcpt} relies on the following lemma.

\begin{lemma}\label{lem:Vtau}
For $(\tau,V_{\tau})\in\widehat{K}$, let $\mathcal{V}_{\tau}$ be the $G$-equivariant Hermitian vector bundle $G\times_K V_{\tau}$ over the Riemannian symmetric space $G/K$.
Let $(\pi,\HHH)$ be a (Harish-Chandra) discrete series representation of~$G$, with underlying $(\g,K)$-module $(\pi_K,\HHH_K)$.
Then $\psi(\HHH_K)\subset L^2(G/K,\mathcal{V}_{\tau})$ for all $\psi\in\Hom_{\g,K}(\pi_K,\DD'(G/K,\mathcal{V}_{\tau}))$.
\end{lemma}

\begin{proof}[Proof of Lemma~\ref{lem:Vtau}]
By the elliptic regularity theorem, the image of the $(\g,K)$-homomorphism $\psi : \pi_K\to\DD'(G/K,\mathcal{V}_{\tau})$ is contained in\linebreak $C^{\infty}(G/K,\mathcal{V}_{\tau})$.
Let $\HHH^{\infty}$ be the Fr\'echet space of smooth vectors of the unitary representation $(\pi,\HHH)$, as in Section~\ref{subsec:H+-infty}.
By the Casselman--Wallach globalization theorem \cite[II, Ch.\,2]{wal88}, the $(\g,K)$-homomorphism~$\psi$ extends to a continuous $G$-homomorphism $\HHH^{\infty}\to C^{\infty}(G/K,\mathcal{V}_{\tau})$, still denoted by~$\psi$.
We identify $C^{\infty}(G/K,\mathcal{V}_{\tau})$ with the space $C^{\infty}(G,V_{\tau})^K$ of smooth maps $f : G\to V_{\tau}$ satisfying $f(gk)=\tau(k)^{-1}f(g)$ for all $g\in G$ and all $k\in K$.
Thus we can define a linear map
$$\Phi : \HHH^{\infty} \longrightarrow V_{\tau}$$
by $u\mapsto\psi(u)(e)$; it is a $K$-homomorphism because $\psi$ is a $G$-homomorphism.
Taking the adjoint of~$\Phi$, we have a $K$-homomorphism $A : V_{\tau}\to\HHH^{-\infty}$ such that
$$(\Phi(u),v)_{V_{\tau}} = (u,A(v))_{\HHH}$$
for all $u\in\HHH^{\infty}$ and $v\in V_{\tau}$, where $(\cdot,\cdot)_{V_{\tau}}$ and $(\cdot,\cdot)_{\HHH}$ are the respective inner products of the Hilbert spaces $V_{\tau}$ and~$\HHH$.
We note that the image of~$A$ is contained in $(\HHH^{-\infty})_K=\HHH_K$ because $A$ is a $K$-homomorphism.
Then for any $g\in G$, $u\in\HHH$, and $v\in V_{\tau}$, we have
$$(\psi(u)(g),v)_{V_{\tau}} = (\Phi(\pi(g^{-1})u),v)_{\HHH} = (\pi(g)^{-1}u,A(v))_{\HHH}.$$
The right-hand side is the matrix coefficient associated with $u,A(v)\in\HHH$, and so it is square-integrable on~$G$.
Since $v$ is arbitrary, we conclude that $\psi(u)\in L^2(G/K,\mathcal{V}_{\tau})$ for all $u\in\HHH^{\infty}$.
\end{proof}

\begin{proof}[Proof of Proposition~\ref{prop:disc-series-sph-Hcpt}]
Since $H$ is compact, we can take a maximal compact subgroup $K$ of~$G$ containing~$H$.
For $(\tau,V_{\tau})\in\Disc(K/H)$, we set $\ell_{\tau}:=\dim_{\C}((V_{\tau}^{\vee})^H)\geq 1$ and let $\ii_{\tau} : (V_{\tau}^{\vee})^{L_H}\otimes\DD'(G/K,\V_{\tau})\hookrightarrow\DD'(G/H)$ and $\pp_{\tau} : \DD'(G/H)\twoheadrightarrow(V_{\tau}^{\vee})^{L_H}\otimes\DD'(G/K,\V_{\tau})$ be the natural $G$-homomorphisms associated with the $G$-equivariant fiber bundle $G/H\to G/K$ with compact fiber $K/H$ as in \eqref{eqn:i-tau} and \eqref{eqn:p-tau} with $\Gamma=\{e\}$.
We note that $\pp_{\tau}\circ\ii_{\tau}=\mathrm{id}$ for all~$\tau$ and $\sum_{\tau} \ii_{\tau}\circ\pp_{\tau}=\mathrm{id}$.
Then
$$\sum_{\tau\in\Disc(K/H)} \ell_{\tau} \dim_{\C} \Hom_{\g,K}(\pi_K,\DD'(G/K,\mathcal{V}_{\tau})) = \dim_{\C} \Hom_{\g,K}(\pi_K,\DD'(G/H)).$$
Since $G/H$ is $G$-real spherical, the right-hand side is finite-dimensional by Lemma~\ref{lem:spher-finite-length}.
Since $\ell_{\tau}\neq 0$ for $\tau\in\Disc(K/H)$, the following subset of~$\widehat{K}$ is finite:
$$\widehat{K}_H(\pi) := \{ \tau\in\Disc(K/H) : \Hom_{\g,K}(\pi_K,\DD'(G/K,\mathcal{V}_{\tau})) \neq \{0\} \}.$$
Thus any $\psi\in\Hom_{\g,K}(\pi_K,\DD'(G/H))$ is decomposed into a finite sum
$$\psi = \bigoplus_{\tau\in\widehat{K}_H(\pi)} \ii_{\tau} \circ \pp_{\tau} \circ \psi.$$
Since $\pp_{\tau}\circ\psi(\HHH_K)\subset (V_{\tau}^{\vee})^{L_H}\otimes L^2(G/K,\V_{\tau})$ by Lemma~\ref{lem:Vtau} and since $\ii_{\tau} : (V_{\tau}^{\vee})^{L_H}\otimes L^2(G/K,\V_{\tau})\rightarrow L^2(G/H)$ is an isometric embedding for any~$\tau$ (see \eqref{eqn:XYL2} with $\Gamma=\{e\}$ and with $(L,L_H)$ replaced by $(G,H)$), we have
$$\psi(\HHH_K)\subset\bigoplus_{\tau\in\widehat{K}_H(\pi)} \ii_{\tau}\big((V_{\tau}^{\vee})^{L_H}\otimes L^2(G/K,\V_{\tau})\big)\subset L^2(G/H). \qedhere$$
\end{proof}

\section[Proof of Theorems 11.9 and~11.10]{Proof of Theorems \ref{thm:SpecZg-Hcpt} and~\ref{thm:Spectau}}\label{subsec:proof-SpecZg-Hcpt}


In this section we complete the proof of Theorems \ref{thm:SpecZg-Hcpt} and~\ref{thm:Spectau}, where $H$ is assumed to be compact.
Before entering the details of the argument, let us briefly clarify the point.

By definition, eigenfunctions of type~$\I$ on $X_{\Gamma}=\Gamma\backslash G/H$ are given by discrete series representations for $X=G/H$, whereas eigenfunctions of type~$\II$ are orthogonal to them in the Hilbert space $L^2(X_{\Gamma})$.
However, this orthogonality in $L^2(X_{\Gamma})$ is not a priori reflected in $L^2(X)$ because the image of $p_{\Gamma}^* : L^2(X_{\Gamma})\to\DD'(X)$ is not contained in $L^2(X)$, and in particular $p_{\Gamma}^*$ is not an isometry.

On the other hand, $p^* : L^2(X_{\Gamma})\to L^2(\Gamma\backslash G)$ is an isometry since $H$ is compact, and the orthogonality of type~$\I$ and type~$\II$ in $L^2(X_{\Gamma})$ is preserved in $L^2(\Gamma\backslash G)$.
We carry out the proof of Theorems \ref{thm:SpecZg-Hcpt} and~\ref{thm:Spectau} by connecting the two maximal isotypic $G$-submodules $V_{\pi}\subset L^2(\Gamma\backslash G)$ and $W_{\pi}\subset\DD'(X)$ for each $\pi\in\widehat{G}$ (see \eqref{eqn:Vi} and \eqref{eqn:W}) through intertwining operators $T_f$ which are defined in Lemma~\ref{lem:Tf} for each joint eigenfunction $f\in\DD'(X_{\Gamma};\NN_{\nu})$.

We start by proving the first equality in Theorem~\ref{thm:SpecZg-Hcpt}.

\begin{lemma}\label{lem:Nunitary}
Assume $H$ is compact.
For any $\nu\in\Hom_{\C\text{-}\mathrm{alg}}(Z(\g_{\C}),\C)$,
\begin{equation}\label{eqn:L2XGammaNnu}
L^2(X_{\Gamma};\NN_{\nu}) \simeq\bigoplus_{\pi\in\widehat{G}_{\nu}} \Hom_G(\pi,L^2(\Gamma\backslash G)) \otimes \pi^H.
\end{equation}
In particular,
$$\Spec_d^{Z(\g_{\C})}(X_{\Gamma}) = \Char\big(\Disc(\Gamma\backslash G) \cap (\widehat{G})_H\big).$$
\end{lemma}

\begin{proof}
Since $H$ is compact, we may identify $L^2(X_{\Gamma};\NN_{\nu})$ with the subspace $L^2(\Gamma\backslash G;\NN_{\nu})^H$ of $H$-fixed vectors in the regular representation\linebreak $L^2(\Gamma\backslash G;\NN_{\nu})$.
Taking $H$-fixed vectors in the isomorphism of unitary representations \eqref{eqn:Vi}, we obtain \eqref{eqn:L2XGammaNnu}.
\end{proof}

For $(\pi,V)\in\widehat{G}$, we denote by $\pi_K$ the underlying $(\g,K)$-module on the space $V_K$ of $K$-finite vectors in~$V$.
Suppose $X=G/H$ is $G$-real spherical.
We denote by $U_{\pi}$ the closure of $\sum A(V)$ in $\DD'(X)$, where $A$ ranges through $\Hom_G(\pi,L^2(X))$, which is a submodule of $W_{\pi}\equiv W_{\pi}(H)$ (see \eqref{eqn:W}).

\begin{lemma}\label{lem:Upi-Wpi}
\begin{enumerate}
  \item $U_{\pi}\neq\{ 0\}$ if and only if $\pi\in\Disc(G/H)$.
  \item $U_{\pi}\subset W_{\pi}$.
  \item If $H$ is compact, then $U_{\pi}=W_{\pi}$.
\end{enumerate}
\end{lemma}

\begin{proof}
(1) and~(2) are clear.
(3) is a consequence of Proposition~\ref{prop:disc-series-sph-Hcpt}.
\end{proof}

Suppose now that $H$ is compact, the pull-back of the projection $p : \Gamma\backslash G\to X_{\Gamma}=\Gamma\backslash G/H$ gives an isometric embedding of Hilbert spaces:
$$\xymatrixcolsep{3pc}
\xymatrix{
p^{\ast} : L^2(X_{\Gamma}) \ar@{}[d]|{\displaystyle\cup} \ar@{^{(}->}[r] & L^2(\Gamma\backslash G) \ar@{}[d]|{\displaystyle\cup}\\
L^2(X_{\Gamma};\NN_{\nu}) \ar@{^{(}->}[r] & L^2(\Gamma\backslash G;\NN_{\nu}).
}$$
In particular, we may regard $L^2(X_{\Gamma};\NN_{\nu})$ as a subspace of $(L^2(\Gamma\backslash G;\NN_{\nu})^{-\infty})^H$, and apply Lemma~\ref{lem:Tf} to $F=p^{\ast}f$ for all $L^2$-eigenfunctions $f\in L^2(X_{\Gamma};\NN_{\nu})$.

For $\nu\in\Hom_{\C\text{-}\mathrm{alg}}(Z(\g_{\C}),\C)$, recall from \eqref{eqn:hat-G-nu} that $\widehat{G}_{\nu}$ is the set of irreducible unitary representations of~$G$ with infinitesimal character $\nu\circ\eta$.
We now define the following disjoint subsets of the set $\widehat{G}_{\nu}$ of \eqref{eqn:hat-G-nu}:
\begin{eqnarray*}
\I(\nu) & := & \widehat{G}_{\nu} \cap \Disc(G/H),\\
\II(\nu) & := & \widehat{G}_{\nu} \cap ((\widehat{G})_H\smallsetminus \Disc(G/H)).
\end{eqnarray*}

We note that $\Char((\widehat{G})_H)$ and $\Char(\Disc(G/H))$ are both invariant under~$\eta$.

\begin{lemma}\label{lem:GHdual}
Let $G$ be a real reductive Lie group and $H$ a closed unimodular subgroup of~$G$.
\begin{enumerate}
  \item The involution $\pi\mapsto\pi^{\vee}$ of the unitary dual~$\widehat{G}$ leaves $\Disc(G/H)$ and~$(\widehat{G})_H$ invariant.
  \item The involution $\nu\mapsto\nu\circ\eta$ of $\Hom_{\C\text{-}\mathrm{alg}}(Z(\g_{\C}),\C)$ leaves\linebreak $\Char(\Disc(G/H))$ and $\Char((\widehat{G})_H)$ invariant.
\end{enumerate}
\end{lemma}

Here $\pi^{\vee}$ denotes the contragredient unitary representation of~$\pi$.

\begin{proof}
For a unitary representation $(\pi,\mathcal{H})$, we form the conjugate representation $(\overline{\pi},\overline{\mathcal{H}})$ by giving $\mathcal{H}$ the conjugate complex structure.
Then $\overline{\pi}$ is unitarily equivalent to the contragredient representation~$\pi^{\vee}$.
Thus statement~(1) is clear from the counterpart in $(\overline{\pi},\overline{\mathcal{H}})$.

If $\pi$ has $Z(\g_{\C})$-infinitesimal character $\chi_{\pi}$, then the contragredient representation $\pi^{\vee}$ has $Z(\g_{\C})$-infinitesimal character $\chi_{\pi}\circ\eta$.
Since $\pi\in\widehat{G}$ if and only if $\pi^{\vee}\in\widehat{G}$, the set $\Char(\widehat{G})$ is preserved by the involution $\nu\mapsto\nu\circ\eta$.
Thus statement~(2) follows from statement~(1).
\end{proof}

Recall from Lemma~\ref{lem:recap-V-W-pi} that $V_{\pi}\equiv V_{\pi}(\Gamma)$ is a maximal $G$-invariant closed subspace of $L^2(\Gamma\backslash G)$ which is isotropic to $\pi\in\widehat{G}$, and that $W_{\pi}\equiv W_{\pi}(H)$ is a maximal $G$-invariant closed subspace of $\DD'(X)$ whose underlying $(\g,K)$-module is a multiple of $\pi_K$.
The following proposition shows that we can determine whether or not $f\in L^2(X_{\Gamma};\NN_{\nu})$ is of type~$\I$ or of type~$\II$ by means of the $G$-submodule generated by $p^*f$ in $\DD'(\Gamma\backslash G)$, or equivalently by means of the $G$-submodule generated by $p_{\Gamma}^*f$ in $\DD'(G/H)$.
This proposition is a key to the second and third equalities in Theorem~\ref{thm:SpecZg-Hcpt}.

\begin{proposition}\label{prop:type-I-II}
Suppose $G/H$ is $G$-real spherical, with $H$ compact.
Let $\nu\in\Hom_{\C\text{-}\mathrm{alg}}(Z(\g_{\C}),\C)$ and $i=\I$ or~$\II$.
Then the following three conditions on $f\in L^2(X_{\Gamma};\NN_{\nu})$ are equivalent:
\begin{enumerate}[(i)]
  \item $f\in L^2(X_{\Gamma};\NN_{\nu})_i$;
  \item $\overline{p_{\Gamma}^{\ast}f} \in \bigoplus_{\pi\in i(\nu)} W_{\pi}$;
  \item $\overline{p^{\ast}f} \in \bigoplus_{\pi\in i(\nu)} V_{\pi}$.
\end{enumerate}
\end{proposition}

\begin{proof}[Proof of Proposition~\ref{prop:type-I-II}]
We first prove the equivalence (ii)$\,\Leftrightarrow\,$(iii) for $i=\I$ and~$\II$.
Consider the decomposition $f=\sum_{\pi} f_{\pi}$ such that $p^*f=\sum_{\pi} p^*f_{\pi}\in L^2(\Gamma\backslash G)$ is the decomposition of \eqref{eqn:Vidistr}, with $p^*f_{\pi}\in V_{\pi}$ for all~$\pi$.
Then condition~(iii) is equivalent to $p^*f_{\pi}=0$ for all $\pi\notin i(\nu)$.
On the other hand, by Lemma~\ref{lem:Tf}.(2), we have $\overline{p_{\Gamma}^*f_{\pi}}=T(\delta,p^*f_{\pi})\in W_{\pi}$.
Since $\overline{p_{\Gamma}^*f} = \sum_{\pi} \overline{p_{\Gamma}^*f_{\pi}}$, condition~(ii) is equivalent to $\overline{p_{\Gamma}^*f_{\pi}}=0$ for all $\pi\notin i(\nu)$, \ie to $p_{\Gamma}^*f_{\pi}=0$ for all $\pi\notin i(\nu)$.
Since both $p^*$ and $p_{\Gamma}^*$ are injective, the equivalence (ii)$\,\Leftrightarrow\,$(iii) is proved.

For $i=\I$, the equivalence (i)$\,\Leftrightarrow\,$(ii) follows from Lemma~\ref{lem:typeI-U-W}.

Finally, we prove the equivalence (i)$\,\Leftrightarrow\,$(iii) for $i=\II$.
Condition~(i) is equivalent to $f$ being orthogonal to $L^2(X_{\Gamma};\NN_{\nu})_{\I}$ in $L^2(X_{\Gamma})$; since $p^{\ast} : L^2(X_{\Gamma})\hookrightarrow L^2(\Gamma\backslash G)$ is an isometry, this is equivalent to $p^{\ast}f$ being orthogonal to $\bigoplus_{\pi\in\I(\nu)} V_{\pi}$ in $L^2(\Gamma\backslash G)$, which is equivalent to~(iii).
\end{proof}

\begin{lemma}\label{lem:typeI-U-W}
Suppose $X=G/H$ is $G$-real spherical, with $H$ compact.
Then
\begin{eqnarray*}
L^2(X_{\Gamma};\NN_{\nu})_{\I} & = & (p_{\Gamma}^{\ast})^{-1} \bigoplus_{\pi\in\I(\nu\circ\eta)} U_{\pi}\\
& = & (p_{\Gamma}^{\ast})^{-1} \bigoplus_{\pi\in\I(\nu\circ\eta)\cap\Disc(\Gamma\backslash G)} W_{\pi}.
\end{eqnarray*}
\end{lemma}

\begin{proof}[Proof of Lemma~\ref{lem:typeI-U-W}]
The first equality holds by definition of type~$\I$.
To check the second equality, we recall that $U_{\pi}=W_{\pi}$ if $\pi\in\Disc(G/H)$ (Lemma~\ref{lem:Upi-Wpi}.(3)), and so
\begin{equation}\label{eqn:typeI-W}
L^2(X_{\Gamma};\NN_{\nu})_{\I} = (p_{\Gamma}^{\ast})^{-1} \bigoplus_{\pi\in\I(\nu\circ\eta)} W_{\pi}.
\end{equation}
By Lemma~\ref{lem:Tf}, the map $T(\cdot,p^{\ast}f) : V_{\pi}^{-\infty}\to W_{\pi}$ is a continuous $G$-homo\-morphism for $\pi\in\widehat{G}_{\nu}$.
By Lemma~\ref{lem:noext}, the $(\g,K)$-module $(W_{\pi})_K$ is a multiple of $\pi_K$.
By Lemma~\ref{lem:Tf} again,
$$\overline{p_{\Gamma}^{\ast}f} = \sum_{\pi\in\widehat{G}_{\nu}} T(\delta,(p^{\ast}f)_{\pi}) = \sum_{\pi\in\widehat{G}_{\nu}} T(\delta_{\pi},p^{\ast}f) \in \bigoplus_{\pi\in\widehat{G}_{\nu}} W_{\pi}.$$
Thus we only need to consider $\pi$ satisfying $(p^{\ast}f)_{\pi}\neq 0$ in the right-hand side of \eqref{eqn:typeI-W}.
In particular, $\pi$ belongs to $\Disc(\Gamma\backslash G)$.
This completes the proof.
\end{proof}

\begin{proof}[Proof of Theorem~\ref{thm:SpecZg-Hcpt}]
By Lemma~\ref{lem:GHdual} and the equivalence (i)$\,\Leftrightarrow\,$(iii) in Proposition~\ref{prop:type-I-II}, we have
$$L^2(X_{\Gamma};\NN_{\nu})_i = (p^{\ast})^{-1} \Bigg(\bigoplus_{\pi\in i(\nu)} V_{\pi}\Bigg)$$
for $i=\I$ or~$\II$.
Since $V_{\pi}\subset L^2(\Gamma\backslash G)$, we have $L^2(X_{\Gamma};\NN_{\nu})_i\neq\{0\}$ if and only if $i(\nu)\cap\Disc(\Gamma\backslash G)\neq\emptyset$.
\end{proof}

\begin{proof}[Proof of Theorem~\ref{thm:Spectau}]
The argument works similarly to that of Theorem~\ref{thm:SpecZg-Hcpt}.
More precisely, consider the two projections
$$\xymatrix{\Gamma\backslash G \ar[dr]^q & & Y = G/K \ar[dl]_{q_{\Gamma}}\\
& Y_{\Gamma} = \Gamma\backslash G/K & }$$
Fix $\tau\in\widehat{K}$ and $\pi\in\Disc(\Gamma\backslash G)$.
Recall $V_{\pi} \equiv V_{\pi,\tau}(\Gamma) \subset L^2(\Gamma\backslash G)$ from \eqref{eqn:Vi} and define $W_{\pi} \equiv W_{\pi,\tau} \subset V_{\tau}^{\vee} \otimes \DD'(G/K,\V_{\tau};\NN_{\nu})$ by
$$W_{\pi} := \sum_A A(\pi_K),$$
where $A$ ranges through $V_{\tau}^{\vee} \otimes \Hom_{\g,K}(\pi_K,\DD'(G/K,\V_{\tau};\NN_{\nu}))$.
Then, analogously to Proposition~\ref{prop:type-I-II}, the following holds.

\begin{proposition}\label{prop:type-I-II-Y}
Suppose $(\tau,V_{\tau})\in\widehat{K}$, and $\nu\in\Hom_{\C\text{-}\mathrm{alg}}(Z(\g_{\C}),\C)$, and $i=\I$ or~$\II$.
Then the following three conditions on $f\in L^2(X_{\Gamma},\V_{\tau};\NN_{\nu})$ are equivalent:
\begin{enumerate}[(i)]
  \item $f\in L^2(Y_{\Gamma},\V_{\tau};\NN_{\nu})_i$;
  \item $\overline{q_{\Gamma}^{\ast}f} \in \bigoplus_{\pi\in i(\nu)} W_{\pi}$;
  \item $\overline{q^{\ast}f} \in \bigoplus_{\pi\in i(\nu)} V_{\pi}$.
\end{enumerate}
\end{proposition}

For the proof of Proposition~\ref{prop:type-I-II-Y}, we use a sesquilinear map
$$T : L^2(\Gamma\backslash G;\NN_{\nu})^{-\infty} \times \DD'(Y_{\Gamma},\V_{\tau};\NN_{\nu}) \longrightarrow \DD'(Y_{\Gamma},\V_{\tau};\NN_{\nu\circ\eta})$$
defined similarly to Lemma~\ref{lem:Tf}.
The proof is parallel to that of Lemma~\ref{lem:Tf} and Proposition~\ref{prop:type-I-II}, so we omit the details.
\end{proof}

\section[Proof of Theorem~2.7 and Proposition~11.11]{Proof of Theorem~\ref{thm:Specd-lambda} and Proposition~\ref{prop:SpecG}} \label{subsec:proof-Specd-lambda}


We are now ready to give a proof of Theorem~\ref{thm:Specd-lambda} describing $\Spec_d(X_{\Gamma})_i$ by the data of the Riemannian locally symmetric space $Y_{\Gamma}=\Gamma\backslash L/L_K$.

\begin{proof}[Proof of Theorem~\ref{thm:Specd-lambda}]
By Proposition~\ref{prop:cond-Tf-A-B-satisfied}, conditions (Tf), (A), (B) hold for the quadruple $(G,L,H,L_K)$.
Moreover, since $X_{\C}$ is $L_{\C}$-spherical (in particular, $G_{\C}$-spherical), $X$ is $G$-real spherical.
Therefore, we can apply Corollary~\ref{cor:Specd-lambda-Wtau} and obtain Theorem~\ref{thm:Specd-lambda} via the transfer map~$\llambda$ from the corresponding results for the (vector-bundle-valued) Riemannian results given in Theorem~\ref{thm:Spectau}, with $(G,H,K)$ replaced by $(L,L_H,L_K)$.
\end{proof}

\begin{proof}[Proof of Proposition~\ref{prop:SpecG}]
Statement~(1) is immediate from the definition of discrete spectrum of type~$\I$.
To check~(2), we set $L_K:={}^{\backprime}K\times\!{}^{\backprime}K$, which is a maximal compact subgroup of $L={}^{\backprime}G\times\!{}^{\backprime}K$, and $L_H:=\Diag({}^{\backprime}K)$.
By the Peter--Weyl theorem,
$$\Disc(L_K/L_H) = \big\{ ({}^{\backprime}\tau)^{\vee} \boxtimes {}^{\backprime}\tau : {}^{\backprime}\tau\in\widehat{{}^{\backprime}K}\big\} .$$
Using the notation of Section~\ref{subsec:refine-SpecZg-Hcpt}, for $\tau=({}^{\backprime}\tau)^{\vee}\boxtimes{}^{\backprime}\tau\in\Disc(L_K/L_H)$ we have
$$\widehat{L}(\tau) = \big\{ {}^{\backprime}\vartheta \boxtimes {}^{\backprime}\tau : {}^{\backprime}\vartheta \in \widehat{{}^{\backprime}G}\text{ such that } \Hom_{{}^{\backprime}K}(({}^{\backprime}\tau)^{\vee},{}^{\backprime}\vartheta|_{{}^{\backprime}K} )\neq \{0\}\big\} ,$$
and $\Disc(L/L_K;\tau)$ is the set of ${}^{\backprime}\vartheta\boxtimes{}^{\backprime}\tau\in\widehat{L}(\tau)$ such that ${}^{\backprime}\vartheta$ is a Harish-Chandra discrete series representation of~${}^{\backprime}G$.
By Proposition~\ref{prop:cond-Tf-A-B-satisfied-GxG}, conditions (Tf), (A), (B) hold for the quadruple $(G,L,H,L_K)$.
The transfer map $\llambda$ of condition (Tf) is given by
$$\llambda(\vartheta,\tau) = \chi_{{}^{\backprime}\vartheta}$$
for $\tau=({}^{\backprime}\tau)^{\vee}\boxtimes{}^{\backprime}\tau\in\Disc(L_K/L_H)$ and $\vartheta={}^{\backprime}\vartheta\boxtimes{}^{\backprime}\tau\in\widehat{L}(\tau)$, via the isomorphism $\dd\ell : Z({}^{\backprime}\g_{\C})\overset{\sim}{\longrightarrow}\D_G(X)$.
We can apply Theorem~\ref{thm:transfer-spec} and obtain statement~(2) of Proposition~\ref{prop:SpecG} via the transfer map~$\llambda$ from the corresponding results for the (vector-bundle-valued) Riemannian results given in Theorem~\ref{thm:Spectau}, with $(G,H,K)$ replaced by $(L,L_H,L_K)$.
\end{proof}

\backmatter

\end{document}